\documentclass[a4paper,reqno]{amsart}
\usepackage[utf8]{inputenc}
\usepackage[T1]{fontenc}
\usepackage{amsfonts}
\usepackage[foot]{amsaddr}
\usepackage{amsmath}
\usepackage{todonotes}
\numberwithin{equation}{section}
\usepackage{xcolor}
\DeclareUnicodeCharacter{00A0}{~}

\usepackage[hidelinks]{hyperref}
\usepackage[margin=1.2in]{geometry}
\usepackage{enumerate}
\usepackage{graphicx}
\usepackage{caption}
\usepackage{subcaption}

\usepackage{array}
\usepackage{mathrsfs}  
\usepackage{accents}
\usepackage{cleveref}

\usepackage{amsrefs}

\usepackage{stmaryrd}
\usepackage{MnSymbol}

\usepackage{tikz}
\usepackage{tikz-cd}
\tikzcdset{column sep/normal=1.1cm}
\usetikzlibrary{calc}
\usepackage{color, mathtools,epsfig, graphicx}
\usetikzlibrary{positioning}
\usetikzlibrary{positioning, shapes, shadows, arrows}

\newcommand\Pone{\textrm{P}_{\textrm{I}} }
\newcommand\Ptwo{\textrm{P}_{\textrm{II}} }

\newcommand\Pfour{\textrm{P}_{\textrm{IV}}}

\newcommand\Psix{\textrm{P}_{\textrm{VI}}}

\newcommand{\Z}{\mathbb{Z}}
\newcommand{\C}{\mathbb{C}}
\newcommand{\p}{\mathbb{P}}
\newcommand{\E}{\mathcal{E}}
\newcommand{\F}{\mathcal{F}}
\newcommand{\h}{\mathcal{H}}
\newcommand{\pain}[1]{\operatorname{P}_{\mathrm{#1}}}
\newcommand{\defeq}{\vcentcolon=}

\newcommand{\Pic}{\operatorname{Pic}}

\DeclareMathOperator{\Spec}{Spec}
\DeclareMathOperator{\Proj}{Proj}
\DeclareMathOperator{\Aut}{Aut}
\DeclareMathOperator{\SL}{SL}
\usepackage{amsthm}

\newtheorem{theorem}{Theorem}[section]

\newtheorem{lemma}[theorem]{Lemma}

\newtheorem{definition}[theorem]{Definition}
\newtheorem{proposition}[theorem]{Proposition}
\newtheorem{conjecture}[theorem]{Conjecture}

\newtheorem{remark}[theorem]{Remark}
\newtheorem{corollary}[theorem]{Corollary}
\newtheorem{problem}[theorem]{Problem}

\newtheorem{maintheorem}{Theorem}

\numberwithin{equation}{section}

\title[On monodromy of monodromy surfaces]{On monodromy of monodromy surfaces}

\author{Pieter Roffelsen$^{1}$}

\address{$^{1}$School of Mathematics and Statistics F07, The University of Sydney, NSW 2006, Australia.}

\author{Alexander Stokes$^{2}$}

\address{$^{2}$Waseda Institute for Advanced Study, Waseda University, 1-21-1 Nishi Waseda, Shinjuku-ku, Tokyo 169-0051, Japan.}

\date{\today}

\begin{document}

\begin{abstract}
Different instances of the Riemann-Hilbert correspondence relate initial value spaces of Painlev\'e equations to affine varieties known as monodromy surfaces, built from monodromy invariants for associated linear ODEs.
We show that these affine varieties admit realisations as embedded affine del Pezzo surfaces, characterised by their degree and a prescribed divisor at infinity, first in this paper for cases associated with Painlev\'e equations $\rm{VI},\rm{IV},\rm{II},\rm{I}$.
 We prove that the monodromy groups of these monodromy surfaces form the finite parts of the affine Weyl symmetry groups of the corresponding Painlev\'e equations.
We realise the monodromy group in each case: analytically as permutations of lines induced by continuation along loops in parameter space, Galois-theoretically in terms of the function field of the incidence variety of lines, and combinatorially via their intersection graph. 
This in particular yields an interpretation of the parameter spaces of Painlev\'e equations, modulo symmetries, as moduli spaces of categories of embedded affine varieties.
Further, it shows that, despite the affine Weyl group symmetries becoming trivial when conjugated by the corresponding Riemann-Hilbert map, the finite Weyl group part survives as monodromy of the monodromy surface.
\end{abstract}

\maketitle


\section{Introduction}
Symmetry groups play a fundamental role in the theory of integrable systems. In the case of Painlev\'e equations, these are extended affine Weyl groups, that act by B\"acklund transformations on solutions.
Geometrically, for each Painlev\'e equation $\pain{J}$ these are automorphisms of families $\mathcal{X}^{\rm{J}} \to \mathscr{T}_{\rm{J}}\times \mathscr{A}_{\rm{J}}$ of smooth open complex algebraic surfaces, the fibres of which are the Okamoto initial value spaces $\mathcal{X}_{t,a}^{\rm{J}}$ at time $t\in\mathscr{T}_{\rm{J}}$ for $\pain{J}$ with parameters $a\in\mathscr{A}_{\rm{J}}$.

Each Painlev\'e equation $\pain{J}$ governs isomonodromic deformation within one or more classes of linear ODEs. 
Integrability of Painlev\'e equations is via transcendental integrals of motion constructed from generalised monodromy data of these linear ODEs.
The space of these forms an affine variety called a monodromy surface, which is the fibre of a family $\mathcal{M}^{\rm{J}}\to\mathscr{W}_{\rm{J}}$.

In each case we have a map 
\begin{equation}\label{eq:rhmap}
    \operatorname{RH}^{\rm{J}}_{t,a} : \mathcal{X}^{\rm{J}}_{t,a} \longrightarrow \mathcal{M}^{\rm{J}}_{w(a)},
\end{equation}
from the initial value space to the monodromy surface, essentially induced by the generalised Riemann-Hilbert correspondence \cite{malgrange91}. We refer to \eqref{eq:rhmap} as the Riemann-Hilbert map, which is known to be a non-algebraic biholomorphism for generic $a$ and $t$ in each case.

While affine Weyl group symmetries give automorphisms of the family $\mathcal{X}^{\rm{J}}$, direct conjugation by the Riemann-Hilbert map only induces trivial automorphisms of monodromy surfaces in all known cases.
This leads to the following fundamental problem. 

\begin{problem}\label{main:problem}
What role do the symmetry groups of the Painlev\'e equations play on the monodromy surface side of the Riemann-Hilbert correspondence?
\end{problem}

We resolve \Cref{main:problem}, first in this paper for the cases $\rm{J}=\rm{VI},\rm{IV},\rm{II},\rm{I}$.
We derive compactifications from $\mathcal{M}^{\rm{J}}\to\mathscr{W}_{\rm{J}}$ to families of del Pezzo surfaces 
$\overline{\mathcal{M}^{\rm{J}}} \to \mathscr{W}_{\rm{J}}$.
 Then, we define a corresponding category $\mathfrak{C}_{\rm{J}}$ of embedded affine del Pezzo surfaces associated with each Painlevé equation and its linear problem, for which $\mathcal{M}^{\rm{J}}\to\mathscr{W}_{\rm{J}}$ provides a universal family.
Permutations of exceptional curves 
induced by deforming over loops in the moduli of $\mathfrak{C}_{\rm{J}}$ form the monodromy of the monodromy surface, which yields the resolution of \Cref{main:problem}: the monodromy group of $\mathfrak{C}_{\rm{J}}$ is precisely the underlying finite Weyl group of the symmetry group of $\pain{J}$.

Moreover, viewing the monodromy surfaces in terms of the categories $\mathfrak{C}_{\rm{J}}$ we get a definite notion of lines on the monodromy surfaces.
From the literature, one can observe that these are of great significance in the asymptotic analysis of Painlev\'e transcendents. 
This is encapsulated by the following conjecture (see also \cites{ohyamastrasbourg, ramiswebseminaronpainlevetalk}). 

\begin{conjecture}[Ramis \cite{ramis}]\label{conjectureramis}
	Every line of every monodromy variety (for all the differential Painlevé equations and for all the values of the parameters, generic or not) is a one parameter family of “truncated” solutions.
\end{conjecture}

We confirm \Cref{conjectureramis} in each case by careful comparison of the models of monodromy surfaces as embedded affine varieties with the asymptotic results in the literature. Moreover, in the cases $\rm{J}=\rm{VI},\rm{IV},\rm{II}$, we identify additional distinguished curves relevant in the asymptotic theory.

\subsection{Background}

The monodromy group $W(E_6)$ of smooth cubic surfaces acts by permutations of the 27 lines induced by deformations over loops in their moduli, as well as by automorphisms of the Schl\"afli graph that encodes their intersections. 
This was famously established through the work of Schl\"afli, Kantor, Cartan, Coble and du Val, see \cite[Sec. 8 \& 9]{dolgachevclassicalAG}. 
The monodromy groups of smooth del Pezzo surfaces similarly act by permutations of lines and form finite Weyl groups \cites{dolgachevclassicalAG,fangdelpezzo}.
These monodromy groups, as well as those of many other classical enumerative problems,  can be realised in a Galois-theoretic way, as established by Jordan \cite{jordan1870} and revisited by Harris \cite{harris1979}, see \cite{sottileyahl}.

At the same time, the $\operatorname{SL}_2(\C)$-character variety of the four-punctured sphere can be regarded as a hypersurface in $\mathbb{C}^7$, forming a four-parameter family of affine cubic surfaces, as discovered in the work of Vogt \cite{vogt} and Fricke and Klein \cite{frickeklein}.
In turn, the Riemann-Hilbert correspondence \cites{deligne70,jimbo1982} relates these surfaces to rank two Fuchsian ODEs on the Riemann sphere with four singularities, and isomonodromic deformation of the latter is governed by the sixth Painlev\'e equation \cite{fuchs1905}.
There are monodromy surfaces similarly associated to the other Painlev\'e equations, parametrising equivalence classes of generalised monodromy data for associated linear ODEs. These can also be realised as families of affine cubic surfaces \cites{putsaito,chekhovdecorated}. 

Furthermore, monodromy varieties are of fundamental interest in mathematical physics, algebraic geometry and representation theory.
They are connected to, e.g., 
Cherednik algebras \cites{oblomkovDAHA,oblomkov, mazzoccoconfluences,chekhovconj}, Teichm\"uller theory \cite{chekhovmazzocco}, mirror symmetry for log Calabi-Yau surfaces \cites{grosshackingkeel,quantummirror,mirrorsymmetryforpainleve}, and arithmetic and algebraic dynamics \cites{CantatLoray,Cantat,rebeloroeder1,rebeloroeder2}.
They also play a crucial role in the classification of algebraic solutions of ODEs \cites{dubrovinmazzocco,mazzocco2001,boalchschwarzlist,boalchicosahedral,hitchinponcelet,litt}.
We note that monodromy varieties are  also studied in the literature under the names monodromy manifolds, Betti moduli spaces, wild character varieties and decorated character varieties \cites{putsaito, benedettadecoratedlocalsystems,boalchgeometryandbraiding,chekhovdecorated}.

\subsection{Main results}
For every $\rm{J}\in\{\rm{VI},\rm{IV},\rm{II},\rm{I}\}$,
we realise the monodromy surface as a family of embedded affine varieties
\begin{equation}\label{eq:family}
    \mathcal{M}^{\rm{J}} \to \mathscr{W}_{\rm{J}},
\end{equation} 
such that the divisor at infinity $\overline{\mathcal{M}^{\rm{J}}_w}\setminus \mathcal{M}^{\rm{J}}_w$ consists of only smooth points in $\overline{\mathcal{M}^{\rm{J}}_w}$, see Definition \ref{def:monodromysurfacefamilyPVI} ($\rm{J}=\rm{VI}$), Definition \ref{def:monodromysurfacePIV} ($\rm{J}=\rm{IV}$),
Definition \ref{def:monodromysurfacePII} ($\rm{J}=\rm{II}$) and
Definition \ref{def:monodromysurfacePI} ($\rm{J}=\rm{I}$).

As affine varieties, forgetting about the embedding, the surfaces are isomorphic to the affine cubic surfaces appearing in \cites{putsaito,chekhovdecorated}. 
Further, the two well-known linear problems related to $\pain{II}$, due to Flaschka-Newell and Jimbo-Miwa, lead to the same family \eqref{eq:family} of monodromy surfaces.

We derive the action of the extended affine Weyl group $\widetilde{W}_{\rm{J}}$ of symmetries of $\pain{J}$ on $\mathcal{M}^{\rm{J}}\to\mathscr{W}_{\rm{J}}$ induced by conjugation by the Riemann-Hilbert map \eqref{eq:rhmap}.
This extends the results on $\pain{VI}$ of Inaba-Iwasaki-Saito \cite{inabaiwasakisaito} (see \Cref{thm:PVIRHsymmetries}) and Lisovyy-Tykhyy \cite{lisovyyalgebraic} (see \Cref{prop:PVI:dynkinautosunderRH}) to $\pain{IV}$ and $\pain{II}$ as follows. 
\begin{maintheorem}[Symmetries conjugated by Riemann-Hilbert]\label{mainthm:conjugatedsymmetries}
    For $\rm{J}\in \{\rm{VI},\rm{IV},\rm{II}\}$, the action of $\widetilde{W}_{\rm{J}}$ on $\mathcal{M}^{\rm{J}}\to\mathscr{W}_{\rm{J}}$ induced by conjugation by the Riemann-Hilbert map is given in 
    Tables
    \ref{tab:PVI:symmetry:ws} ($\rm{J}=\rm{VI}$), 
    \ref{tab:PIV:symmetry:dynkinautos} ($\rm{J}=\rm{IV}$) and \ref{tab:PII:symmetriesunderRH}  ($\rm{J}=\rm{II}$).
    In particular, the entire unextended affine Weyl group $W_{\rm{J}}$ acts trivially in each case. 
\end{maintheorem}
As mentioned, the case ($\rm{J}=\rm{VI}$) in Theorem \ref{mainthm:conjugatedsymmetries} was already established in the literature \cites{inabaiwasakisaito,lisovyyalgebraic}. 
Note that for $\pain{II}$ with the Riemann-Hilbert map coming from the Flaschka-Newell linear problem, the corresponding case of \Cref{mainthm:conjugatedsymmetries} can be deduced from the work of Kapaev \cite{kapaev2004}. 
We prove the result for the one coming from the Jimbo-Miwa linear problem.
The two Riemann-Hilbert maps give biholomorphisms from the initial value space to the same monodromy surface, and we conjecture them to coincide, see \Cref{rem:commutingRHsPII}. 
For the case of $\pain{IV}$, we establish \Cref{mainthm:conjugatedsymmetries} for the Riemann-Hilbert map coming from a rank 2 linear problem. 
There is another Riemann-Hilbert map coming from a rank 3 linear problem \cite{puttop2013piv} associated with $\pain{IV}$, and the relation between these is unknown, see \Cref{rem:commutingRHsPIV}.

Further, we identify several automorphisms of the families $\mathcal{M}^{\rm{J}}\to \mathscr{W}_{\rm{J}}$, $\rm{J}\in\{ \rm{IV}, \rm{II}, \rm{I}\}$ for which the corresponding transcendental symmetry of $\pain{J}$ seems to not be known, see \Cref{rem:mysterysymmetryPIV,rem:mysterysymmetryPII,rem:mysterysymmetryPI}.

For each Painlev\'e equation $\pain{J}$, $\rm{J}\in\{\rm{VI},\rm{IV},\rm{II},\rm{I}\}$, 
    we define a corresponding category $\mathfrak{C}_{\rm{J}}$ of embedded affine del Pezzo surfaces.
    
    \begin{definition}[Category $\mathfrak{C}_{\rm{J}}$, $\rm{J}\in \{\rm{VI},\rm{IV},\rm{II},\rm{I}\}$]
        Define the category $\mathfrak{C}_{\mathrm{J}}$ with
    \begin{itemize}
        \item an object being an embedded affine del Pezzo surface $\mathcal{V}\subset\mathbb{A}^{d_{\rm{J}}}$ of degree $d_{\rm{J}}$ such that the complement in its projective completion, $\overline{\mathcal{V}}\setminus\mathcal{V}\subseteq \mathbb{P}^{d_{\rm{J}}}$, is a cycle of curves consisting of only smooth points of $\overline{\mathcal{V}}$, with the degree and cycle as specified in \Cref{tab:categorytable}.
        \item a morphism $L$ between objects $\mathcal{V}_1\subset\mathbb{A}^{d_{\rm{J}}}$ and $\mathcal{V}_2\subset\mathbb{A}^{d_{\rm{J}}}$ being an affine linear map $L\in \operatorname{End}(\mathbb{A}^{d_{\rm{J}}})$ such that $L(\mathcal{V}_1)\subseteq \mathcal{V}_2$.
    \end{itemize}
    \end{definition}

\begingroup
\setlength{\tabcolsep}{15pt} 
\renewcommand{\arraystretch}{1.5} 
\begin{table}[t]
    \centering
\begin{tabular}{c|c|c|c|c}
\shortstack{Painlev\'e\\ Equation} &
   \shortstack{Symmetry \\
          Group} & 
   \shortstack{Degree $d_{\mathrm{J}}$\\
      of del Pezzo} & 
    \shortstack{Divisor at \\
          Infinity} & 
    \shortstack{Monodromy of\\
          Monodromy}  \\
\hline\hline
 $\pain{VI}$ & $\widetilde{W}(D_4^{(1)})$ & $3$ & triangle  of lines & $W(D_4)$\\
   \hline
  $\pain{IV}$ & $\widetilde{W}(A_2^{(1)})$ & $4$ & rectangle of lines & $W(A_2)$\\
   \hline
   $\pain{II}$ & $\widetilde{W}(A_1^{(1)})$ & $6$ & triangle of conics & $W(A_1)$\\
   \hline
  $\pain{I}$ & $\widetilde{W}(A_0^{(1)})$ & $5$ & pentagon of lines & $W(A_0)$\\
        \hline
	\end{tabular}
    \caption{Embedded affine del Pezzo surfaces and their monodromy for Painlev\'e equations $\pain{VI},\pain{IV},\pain{II},\pain{I}.$}
    \label{tab:categorytable}
\end{table}
\endgroup

\begin{maintheorem}[Moduli space of $\mathfrak{C}_{\rm{J}}$] \label{mainthm:categoryandmodulispace}
    For every $\rm{J}\in\{\rm{VI},\rm{IV},\rm{II},\rm{I}\}$,
 any object in $\mathfrak{C}_{\rm{VI}}$ is isomorphic to a fibre $\mathcal{M}_w^{\rm{J}}$ of the family $\mathcal{M}^{\rm{J}}\to \mathscr{W}_{\rm{J}}$.
 Two fibres are isomorphic in $\mathfrak{C}_{\rm{VI}}$ if and only if their parameters are related by the action of $\widetilde{W}_{\rm{J}}$.
 Consequently, there is a canonical bijection 
 $$ \operatorname{Iso}(\mathfrak{C}_{\rm{VI}}) \to \mathscr{O}_{\rm{J}}(\C),$$
from isomorphism classes of objects to complex points of $\mathscr{O}_{\rm{J}}$, where the scheme $\mathscr{O}_{\rm{J}}= \mathscr{W}_{\rm{J}} \sslash \,\widetilde{W}_{\rm{J}} $ is the quotient of the parameter space by the action of the symmetry group.
\end{maintheorem}
The moduli spaces  $\mathscr{O}_{\mathrm{J}}$ are given explicitly in equation \eqref{eq:moduli_map} ($\rm{J}=\rm{VI}$), equation \eqref{eq:mapPhiPIV} ($\rm{J}=\rm{IV}$), equation \eqref{eq:mapPhiPII} ($\rm{J}=\rm{II}$). For $\rm{J}=\rm{I}$, the moduli space is a point $\mathscr{O}_{\rm{I}}=\operatorname{\Spec}\C$.

We also characterise the lines on the monodromy surfaces for $\pain{IV},\pain{II}$ and $\pain{I}$ in terms of special Stokes data and quantisation conditions on the associated linear problem, and degenerate asymptotic expansions of Painlev\'e transcendents, see \Cref{prop:curves_explained,prop:monodromydatalinescurvesexplainedPII} and \Cref{rem:monodromylinesexplainedPI}. 
The case of $\pain{VI}$ was already established in the literature, see \Cref{rem:linesmonodromyPVI}.

Our final main results are on monodromy groups of the monodromy surfaces.
For each $\rm{J}$, we give three realisations of the monodromy group of the family $\mathcal{M}^{\rm{J}}\to\mathscr{W}_{\rm{J}}$.
The first is its analytic monodromy group
$$ \operatorname{Mon}^{\operatorname{an}}(\mathcal{M}^{\rm{J}}) = \operatorname{Im} \Big( \pi_1( \mathscr{W}_{\rm{J}} \setminus \mathscr{W}_{\rm{J}}^{\operatorname{sing}}; w_*) \to \mathfrak{S}_{\operatorname{lines}} \Big) ,$$
of permutations of lines induced by deformation over loops in the parameter space $\mathscr{W}_{\rm{J}}$ avoiding the locus where $\mathcal{M}^{\rm{J}}_w$ is singular, which, as an abstract group, is independent of choice of base $w_*$.

The second is its algebraic monodromy group 
$$ \operatorname{Mon}^{\operatorname{alg}}(\mathcal{M}^{\rm{J}}) = \operatorname{Gal}\left(K_{\Gamma_{\rm{J}}}/ \C(\mathscr{W}_{\rm{J}})\right),$$
defined in terms of the minimal field extension $K_{\Gamma_{\rm{J}}}$ of $\C(\mathscr{W}_{\rm{J}})$ required to write lines rationally in terms of parameters, coming from the incidence variety of lines on the monodromy surface $\mathcal{M}_w$:
$$\Gamma_{\rm{J}} \xrightarrow{\rho} \mathscr{W}_{\rm{J}}, \quad \Gamma_{\rm{J}}  = \left\{ (w,\ell) \in \mathscr{W}_{\rm{J}}\times \mathbb{G}(1,d_{\rm{J}}) ~|~ \ell \subset \overline{\mathcal{M}_w^{\rm{J}}}, \,\, 
\ell\not\subset \overline{\mathcal{M}_w^{\rm{J}}}\setminus \mathcal{M}_w^{\rm{J}}\right\},$$
where $\mathbb{G}(1,d_{\rm{J}})$ is the Grassmannian of lines in $\p^{d_{\rm{J}}}$.
In each case, $\Gamma_{\rm{J}}$ is reducible, with number of irreducible components equal to the length of the cycle of curves at infinity on $\overline{\mathcal{M}^{\rm{J}}_w}$, and the  field $K_{\Gamma_{\rm{J}}}$ is the compositum of the normal closures of the function fields of the irreducible components of $\Gamma_{\rm{J}}$ within the algebraic closure of $\C(\mathscr{W}_{\rm{J}})$.

The third realisation is the combinatorial monodromy group
$$\operatorname{Mon}^{\operatorname{com}} (\mathcal{M}^{\rm{J}}) = \operatorname{Fix}_{\operatorname{Aut}(\mathcal{G}_{\rm{J}})} (\mathcal{G}_{\rm{J}}^{\infty}),$$
of automorphisms of the intersection graph $\mathcal{G}_{\rm{J}}$ of lines on $\overline{\mathcal{M}^{\rm{J}}_w}$ and components of $\overline{\mathcal{M}_w^{\rm{J}}}\setminus \mathcal{M}_w^{\rm{J}} $, that leaves the subgraph $\mathcal{G}_{\rm{J}}^{\infty}$ generated by components of the divisor at infinity invariant.
There are natural inclusions of groups
$ \operatorname{Mon}^{\operatorname{an}}(\mathcal{M}^{\rm{J}}) \subseteq \operatorname{Mon}^{\operatorname{alg}}(\mathcal{M}^{\rm{J}}) \subseteq  \operatorname{Mon}^{\operatorname{com}}(\mathcal{M}^{\rm{J}}),$
and our third main result shows that these are isomorphisms.
\begin{maintheorem}[Monodromy of monodromy surfaces] \label{mainthm:mon}
For each Painlev\'e equation $\pain{J}$, $\rm{J}\in \{\rm{VI},\rm{IV},\rm{II},\rm{I}\}$, the monodromy group of the family $\mathcal{M}^{\rm{J}}$ forms the underlying finite Weyl group of the extended affine Weyl group of B\"acklund transformation symmetries.
That is, 
\begin{equation}\label{eq:groupisos}
    \begin{gathered}
    \operatorname{Mon}^{\operatorname{an}}(\mathcal{M}^{\rm{J}}) \cong \operatorname{Mon}^{\operatorname{alg}}(\mathcal{M}^{\rm{J}}) \cong \operatorname{Mon}^{\operatorname{com}}(\mathcal{M}^{\rm{J}})\cong W(X_{\rm{J}}), 
\end{gathered}
\end{equation}
where $W(X_{\rm{J}})$ is the finite Weyl group of type $X_{\rm{J}}$ as shown in \Cref{tab:categorytable}.
\end{maintheorem}
The last isomorphism in \eqref{eq:groupisos} is also natural, induced by a map 
$\mathscr{A}_{\rm{J}} \to \mathscr{U}_{\rm{J}}, $
from the parameter space of $\pain{J}$ to  the space $\mathscr{U}_{\rm{J}}$ of parameters in terms of which lines on $\mathcal{M}_w^{\rm{J}}$ are written rationally, corresponding to the isomorphism of field extensions 
$$K_{\Gamma_{\rm{J}}}/ \C(\mathscr{W}_{\rm{J}}) \cong \C(\mathscr{U}_{\rm{J}})/\C(\mathscr{W}_{\rm{J}}).$$
The action of a B\"acklund transformation on $\mathscr{A}_{\rm{J}}$ then induces monodromy via an element of $\operatorname{Gal}\left(\C(\mathscr{U}_{\rm{J}})/\C(\mathscr{W}_{\rm{J}})\right).$

We remark that the monodromy groups of the monodromy surfaces in \Cref{mainthm:mon} were formulated in order to resolve \Cref{main:problem}. The initial value spaces, monodromy surfaces, and their monodromy can be defined and studied over more general fields. 
However, it is unclear how to extend the corresponding Riemann-Hilbert maps to such settings.

\subsection{Outline}
The paper is structured as follows. Each of Sections \ref{sec:PVI}-\ref{sec:PI} is devoted to the monodromy surface corresponding to a particular Painlev\'e equation, $\Psix$ in \Cref{sec:PVI}, $\Pfour$ in \Cref{sec:PIV}, $\Ptwo$ in \Cref{sec:PII} and $\Pone$ in \Cref{sec:PI}.
In each case, we introduce the symmetry group and initial value space, derive the monodromy surface, and compute the action of the symmetry group under conjugation by the corresponding Riemann-Hilbert map(s).
Then we define the associated category of embedded affine del Pezzo surfaces, establish its moduli space, describe the lines on the surfaces and derive the corresponding monodromy group.

We establish the main Theorems \ref{mainthm:conjugatedsymmetries}, \ref{mainthm:categoryandmodulispace}, \ref{mainthm:mon} separately for each Painlev\'e equation in the corresponding section:
\begin{itemize}
    \item 
    Painlev\'e-VI
    \begin{itemize}
    \item \Cref{mainthm:conjugatedsymmetries}  : \Cref{thm:PVIRHsymmetries} and \Cref{prop:PVI:dynkinautosunderRH}.  
    \item \Cref{mainthm:categoryandmodulispace} : \Cref{prop:oblomkov} and \Cref{cor:modulispacePVI}. 
    \item \Cref{mainthm:mon} : \Cref{prop:combinatorialmonodromyPVI,prop:algebraicmonodromyPVI,prop:analyticmonodromyPVI}.
    \end{itemize}
    \item 
    Painlev\'e-IV
    \begin{itemize}
        \item \Cref{mainthm:conjugatedsymmetries} : \Cref{thm:PIVsymmetriesunderRH}.
        \item \Cref{mainthm:categoryandmodulispace} : \Cref{prop:oblomkovstylenormalformPIV} and \Cref{cor:modulispacePIV}.
        \item \Cref{mainthm:mon} : \Cref{prop:combinatorialmonodromyPIV,prop:algebraicmonodromyPIV,prop:analyticmonodromyPIV}.
    \end{itemize}
    \item 
    Painlev\'e-II
    \begin{itemize}
        \item \Cref{mainthm:conjugatedsymmetries} :  \Cref{thm:symmetryunderRHPII}.
        \item \Cref{mainthm:categoryandmodulispace} : \Cref{prop:oblomkovstylenormalformPII} and \Cref{cor:modulispacePII}.
        \item \Cref{mainthm:mon} : \Cref{prop:combinatorialmonodromyPII,prop:algebraicmonodromyPII,prop:analyticmonodromyPII}.
    \end{itemize}
    \item 
    Painlev\'e-I
    \begin{itemize}
        \item \Cref{mainthm:categoryandmodulispace} : 
 \Cref{prop:oblomkovstylenormalformPI} and \Cref{cor:modulispacePI}.
    \end{itemize}
\end{itemize}
 Theorems \ref{mainthm:conjugatedsymmetries} and \ref{mainthm:mon} are trivial in the case of Painlev\'e-I, since there are no nontrivial B\"acklund transformations.

We provide a conclusion in Section \ref{sec:conclusion}, which is followed by Appendix \ref{app:singularlocusofJimboFricke}, where explicit formulas describing the singular locus of the monodromy surface of $\Psix$ are collected,
and Appendix \ref{app:derivationofmonodromysurfaceFN}, which contains the derivation of the monodromy surface of $\pain{II}$ from the Flaschka-Newell linear problem.

\subsection{Acknowledgements}
We thank Frank Nijhoff for interesting discussions and for bringing \Cref{main:problem} to our attention. PR would like to thank the MAGMA group at the University of Sydney for support, as well as the organisers Bregje Pauwels and Alexander Sherman of Magma Mondays, which played an important role in the genesis of this project.
We also thank Marco Bertola, Davide Dal Martello, Davide Guzzetti, Nalini Joshi, Marta Mazzocco and John Roberts for useful discussions.
We further thank Jean-Pierre Ramis for useful discussions and suggestions, in particular related to the definition of algebraic monodromy groups in terms of incidence varieties.
PR's research was supported by the
Australian Research Council Discovery Projects \#DP200100210 and \#DP210100129, and by the Australian Government through the Office of National Intelligence NISDRG Grant
\#NI240100145.
AS's research was supported by the Japan Society for the Promotion of Science (JSPS) through KAKENHI grant \#24K22843.

\section{The sixth Painlev\'e equation} \label{sec:PVI}
We consider the sixth Painlev\'e equation ($\Psix$) in the form
\begin{equation} \label{eq:PVIscalar}
\begin{aligned}
y_{tt}=&\left(\frac{1}{y}+\frac{1}{y-1}+\frac{1}{y-t}\right)\frac{y_t^2}{2}-\left(\frac{1}{t}+\frac{1}{t-1}+\frac{1}{y-t}\right)y_t\\
&+\frac{y(y-1)(y-t)}{t^2(t-1)^2}\left(\alpha+\frac{\beta\, t}{y^2}+\frac{\gamma\,(t-1)}{(y-1)^2}+\frac{\delta \,t(t-1)}{(y-t)^2}\right),
\end{aligned}
\end{equation}
where
\begin{equation} \label{eq:scalarparamstotheta}
    \alpha=\tfrac{1}{2}(2\theta_\infty-1)^2,\quad \beta=-2\theta_0^2,\quad \gamma =2\theta_1^2,\quad \delta=\tfrac{1}{2}-2\theta_t^2.
\end{equation}
Denote the space of parameters by
\begin{equation*}
    \Theta_{\mathrm{VI}}=\{\theta = (\theta_0,\theta_t,\theta_1,\theta_\infty)\in\mathbb{C}^4\}.
\end{equation*}
We will make use of the alternative parametrisation by root variables \cite{sakai2001}, the space of which we denote by
\begin{equation*}
\mathscr{A}_{\mathrm{VI}} = \left\{a = (a_0,a_1,a_2,a_3,a_4) \in \mathbb{C}^5 ~|~ a_0 +a_1 +2 a_2 + a_3 + a_4 = 1 \right\}.
\end{equation*}
We relate $\Theta_{\mathrm{VI}}$ and $\mathscr{A}_{\mathrm{VI}}$ through
\begin{equation} \label{eq:rootvarstotheta}
    a_0 =  2 \theta_t, \quad 
    a_1 = 2 \theta_{\infty } -1 ,\quad 
    a_2 =1 -\theta_0-\theta_t-\theta_1 -\theta_{\infty }, \quad 
    a_3 = 2 \theta _1, \quad 
    a_4 =  2 \theta _0.
\end{equation}
It will be convenient to recast the $\pain{VI}$ equation \eqref{eq:PVIscalar} as the following non-autonomous Hamiltonian system:
\begin{equation} \label{eq:hamiltoniansystempvi}
    \left\{
    \begin{aligned}
        f_t &=+f \frac{\partial H_{\rm VI}}{\partial g} = \frac{f(f-1)(f-t)}{t (t-1)}\left(2\frac{g}{f}-\frac{a_4}{f}+\frac{1-a_0}{f-t} - \frac{a_3}{f-1}\right), \\
        g_t &= - f \frac{\partial H_{\rm VI}}{\partial f} = \frac{( t- f^2)g^2 - \left( (a_1+2a_2)f^2 +  t a_4\right)g - a_2(a_1+a_2) f^2  }{t(t-1)f},
    \end{aligned}
    \right.    
\end{equation} 
with Hamiltonian as given in \cite{KNY},
\begin{equation*}
            H_{\rm VI} = \frac{g(f-1)(f-t)}{t (t-1)}\left(\frac{1-a_0}{f-t} - \frac{a_3}{f-1}-\frac{a_4}{f}+\frac{g}{f} \right) + \frac{a_2(a_1+a_2)f}{t(t-1)}    .
\end{equation*}
Eliminating $g$ leads to $\pain{VI}$ for $y=f$ with parameters given by \eqref{eq:scalarparamstotheta}, \eqref{eq:rootvarstotheta}.

\subsection{B\"acklund transformations}

The B\"acklund transformation symmetries of $\pain{VI}$ form the extension of the affine Weyl group $W(D_4^{(1)})$ by the group of Dynkin diagram automorphisms $\Aut(D_4^{(1)})$.
We write this as $\widetilde{W}(D_4^{(1)}) = W(D_4^{(1)})\rtimes \Aut(D_4^{(1)})$, and it contains the affine Weyl group $W(F_4^{(1)})$ forming the symmetries discovered by Okamoto \cite{okamotovi}.

The affine Weyl group $W(D_4^{(1)})$ has a standard presentation generated by simple reflections, with relations encoded in the extended Dynkin diagram of type $D_4^{(1)}$, which we write as
 \begin{equation*}
 	W(D_4^{(1)}) =
 	\left\langle r_0, r_1,r_2,r_3,r_4 \quad {\bigg |}\quad 
 	r_{i}^{2} = 1,  \quad
 	\begin{aligned}
 		(r_{i} r_{j})^2 &= 1&  &\text{ when }
 		\raisebox{-0.15in}{\begin{tikzpicture}[
 				elt/.style={circle,draw=black!100,thick, inner sep=0pt,minimum size=1.5mm}]
 				\path   ( 0,0) 	node  	(ai) [elt] {}
 				( 0.5,0) 	node  	(aj) [elt] {};
 				\draw [black] (ai)  (aj);
 				\node at ($(ai.south) + (0,-0.2)$) 	{\small ${i}$};
 				\node at ($(aj.south) + (0,-0.2)$)  {\small ${j}$};
 		\end{tikzpicture}}\\
 		(r_{i} r_{j})^3 &= 1& &\text{ when }
 		\raisebox{-0.12in}{\begin{tikzpicture}[
 				elt/.style={circle,draw=black!100,thick, inner sep=0pt,minimum size=1.5mm}]
 				\path   ( 0,0) 	node  	(ai) [elt] {}
 				( 0.5,0) 	node  	(aj) [elt] {};
 				\draw [black, thick] (ai) -- (aj);
 				\node at ($(ai.south) + (0,-0.2)$) 	{\small ${i}$};
 				\node at ($(aj.south) + (0,-0.2)$)  {\small ${j}$};
 		\end{tikzpicture}}
 	\end{aligned} 
 	\right\rangle
    \qquad 
    \quad
       \raisebox{-20pt}{
\begin{tikzpicture}[scale=.4,elt/.style={circle,draw=black!100,thick, inner sep=0pt,minimum size=1.5mm}]
		\path 	(-1,-1) 	node 	(d0) [elt ] {}
			(-1,1) 	node 	(d1) [elt ] {}
		        ( 0,0) 	node  	(d2) [elt  ] {}
		        ( 1,-1) 	node  	(d3) [elt  ] {}
		        (1,1) 	node 	(d4) [elt ] {};
		\node at ($(d0.west) + (-.3,+0.0)$) 	{ \tiny ${0}$};
		\node at ($(d1.west) + (-.3,+0.0)$) 	{ \tiny ${1}$};
		\node at ($(d2.east) + (+.3,+0.0)$) 	{ \tiny ${2}$};
		\node at ($(d3.east) + (+.3,+0.0)$) 	{ \tiny ${3}$};
		\node at ($(d4.east) + (+.3,+0.0)$) 	{ \tiny ${4}$};

		\draw [black,line width=1pt ] (d0) -- (d2) -- (d1);
		\draw [black,line width=1pt ] (d3) -- (d2) -- (d4);
		\node at ($(d2.east) + (+0,-2)$) 	{ \small ${D_4^{(1)}}$};

	\end{tikzpicture} }
 \end{equation*}  
The group of Dynkin diagram automorphisms $\Aut(D_4^{(1)})\cong \mathfrak{S}_4$ is isomorphic to the symmetric group on 4 symbols, and acts by conjugation on the simple reflections $r_0,\dots,r_4$.
We indicate elements by the corresponding permutation, so
\begin{equation*}
    \widetilde{W}(D_4^{(1)}) = \left\langle r_0, r_1, r_2, r_3,r_4\right\rangle\rtimes\left\langle \sigma_{(03)}, \sigma_{(34)}, \sigma_{(14)} \right\rangle \cong W(D_4^{(1)}) \rtimes \operatorname{Aut}(D_4^{(1)}).
\end{equation*}
The B\"acklund transformations of system \eqref{eq:hamiltoniansystempvi} corresponding to these generators are given in \Cref{tab:PVI:symmetry:varsandparams}.

\begingroup

\setlength{\tabcolsep}{15pt} 
\renewcommand{\arraystretch}{1.5} 

\begin{table}[htb]
    \begin{equation*}
    \begin{array}{c||c|c||c|c|c|c|c||c|}
        w       & \tilde{f}     & \tilde{g}     & \tilde{a}_0   & \tilde{a}_1   & \tilde{a}_2   & \tilde{a}_3   &  \tilde{a}_4      & \tilde{t}     \\
        \hline\hline 
        r_0    & f     & g + \frac{a_0 f}{t-f}    & -a_0   & a_1   & a_2 + a_0   & a_3   &  a_4      & t  \\
        r_1    & f     & g     & a_0   & - a_1   & a_2 + a_1   & {a}_3   &  {a}_4      & t  \\
        r_2    & f+ \frac{a_2 f}{g}    & g+a_2     & a_0+a_2   & a_1 +a_2   & -{a}_2   & {a}_3 +a_2  &  {a}_4+a_2      & {t}  \\
        r_3    & {f}     & {g} - \frac{a_3 f}{f-1}     & {a}_0   & {a}_1   & {a}_2+a_3   & -{a}_3   &  {a}_4      & {t}  \\
        r_4    & {f}     & {g} -a_4    & {a}_0   & {a}_1   & {a}_2+a_4   & {a}_3   &  -{a}_4      & {t}   \\
        \hline \hline
        \sigma_{(03)}   & \frac{f}{t}     & {g}     & {a}_3   & {a}_1   & {a}_2   & {a}_0   &  {a}_4      & \frac{1}{t}     \\
        \sigma_{(34)}   & 1 - f     & g - \frac{g}{f}     & {a}_0   & {a}_1   & {a}_2   & {a}_4   &  {a}_3      & 1-t     \\
        \sigma_{(14)}   & \frac{1}{f}     & -g     & {a}_0   & {a}_4   & {a}_2   & {a}_3   &  {a}_1      & \frac{1}{t}     
    \end{array}
    \end{equation*}
    \caption{Symmetries of $\pain{VI}$ on $(f,g)$ variables and root variable parameters $a$.}
    \label{tab:PVI:symmetry:varsandparams}
\end{table}
\endgroup

\subsection{Initial value space}

We will use the realisation of the initial value space of $\pain{VI}$ provided by extending \eqref{eq:hamiltoniansystempvi} to $\p^1 \times \p^1$ and applying the desingularisation procedure.
This involves performing a sequence of eight point blow-ups starting from $\p^1 \times\p^1$, whose centres depend on $a$ and $t$. 
The locations of these points can be computed by carrying out the desingularisation procedure along the lines as has been done for $\pain{VI}$ in various forms in, e.g., \cites{okamoto1979,heujoshiradnovic,differenthamsdzhamay}.

The result of this is a family of Sakai surfaces, which we write as 
\begin{equation*}
\overline{\mathcal{X}} \rightarrow \mathscr{T}_{\mathrm{VI}} \times \mathscr{A}_{\mathrm{VI}}, 
\end{equation*}
where $\mathscr{A}_{\mathrm{VI}}$ is the space of root variables and $\mathscr{T}_{\mathrm{VI}}= \C \setminus \{0,1\}$ is the independent variable space of $\pain{VI}$, i.e. the space of $t$ where the ODE is nonsingular.
The fibre 
$\overline{\mathcal{X}}_{t,a}$ is a Sakai surface of surface type $D_4^{(1)}$, i.e. with unique effective anticanonical divisor $D_{t,a}$ whose irreducible components intersect according to the $D_4^{(1)}$ Dynkin diagram.
We set $\mathcal{X}_{t,a}$ to be the result of removing the support of $D_{t,a}$ from $\overline{\mathcal{X}}_{t,a}$, and get a family 
\begin{equation*}
\mathcal{X} \rightarrow \mathscr{T}_{\mathrm{VI}} \times \mathscr{A}_{\mathrm{VI}}.
\end{equation*}
For a chosen $a \in \mathscr{A}_{\mathrm{VI}}$, the initial value space of $\pain{VI}$ is as follows.
\begin{definition}[Initial value space of $\pain{VI}$]
    The initial value space at $t\in \mathscr{T}_{\mathrm{VI}}$ of the Hamiltonian form \eqref{eq:hamiltoniansystempvi} of $\pain{VI}$ with parameters $a\in \mathscr{A}_{\mathrm{VI}}$ is the fibre $\mathcal{X}_{t,a}$ of the family $\mathcal{X}_a\to \mathscr{T}_{\mathrm{VI}}$.
\end{definition}

B\"acklund transformations of $\pain{VI}$ become automorphisms of the family $\mathcal{X}$, and extend to automorphisms of $\overline{\mathcal{X}}$ \cites{sakai2001, BTsonmanifolds,inabaiwasakisaito}.
In particular, for each $w \in \widetilde{W}(D_4^{(1)})$, there is an automorphism 
$$w : \mathscr{T}_{\mathrm{VI}}\times \mathscr{A}_{\mathrm{VI}} \to \mathscr{T}_{\mathrm{VI}}\times \mathscr{A}_{\mathrm{VI}},  \quad (t,a) \mapsto (\tilde{t},\tilde{a}),$$
and an isomorphism
$$w : \mathcal{X}_{t,a}  \rightarrow \mathcal{X}_{\tilde{t},\tilde{a}}.$$

\subsection{Associated linear problem}
Painlev\'e VI was discovered by R. Fuchs \cite{fuchs1905} arising as the compatibility condition of a pair of linear ODEs. Here we use a system form of Fuchs' pair of ODEs that can be found in \cite{jimbomiwaII1981},
\begin{subequations}\label{eq:laxpvi}
\begin{align}
  &&  Y_z&=A\,Y, & A&=\frac{A_0}{z}+\frac{A_t}{z-t}+\frac{A_1}{z-1}, && \label{eq:laxpvi1}\\
   && Y_t&=B\,Y, & B&=-\frac{A_t}{z-t}, && \label{eq:laxpvi2}
\end{align}
\end{subequations}
where the matrices $	A_0,A_t, A_1\in \mathfrak{sl}_2(\mathbb{C})$ are such that
		\begin{equation*}
	A_\infty:=-(A_0+A_t+A_1)=\begin{bmatrix}
	    \theta_\infty & 0\\
        0 & -\theta_\infty
	\end{bmatrix},
	\end{equation*}
with corresponding spectra fixed by
	\begin{equation}\label{eq:specsA}
		\operatorname{Spec}(A_j)=\{+\theta_j,-\theta_j\}\qquad (j=0,t,1,\infty).
	\end{equation}
Using coordinates $\{f,g,k\}$ defined by
\begin{equation*}
 A_{12}=\frac{\theta_\infty k(z-f)}{z(z-t)(z-1},\quad A_{11}|_{z=f}=\frac{g}{f}-\left(\frac{\theta_0}{f}+\frac{\theta_t}{f-t}+\frac{\theta_1}{f-1}\right),
\end{equation*}
compatibility of the pair \eqref{eq:laxpvi},  namely the condition
\begin{equation*}
    A_t+AB=B_z+BA,
\end{equation*}
is equivalent to \cref{eq:hamiltoniansystempvi} and the auxiliary equation
\begin{equation*}
   \frac{k_t}{k}=(2 \theta_\infty-1) \frac{(f-t)}{t(t-1)}.
\end{equation*}

\subsection{Derivation of monodromy surface}
Compatibility of the pair \eqref{eq:laxpvi} ensures that the classical monodromy of the linear ODE \eqref{eq:laxpvi1} is preserved under the $\Psix$ flow.

Namely, let $Y(z,t)$ be a local solution of the pair of ODEs \eqref{eq:laxpvi}, say around a point $(z_0,t_0)$. Analytic continuation of $Y(z,t)$ with respect to $z$ along curves $\gamma\in\pi_1(\mathbb{CP}^1\setminus\{0,t,1,\infty\},z_0)$ leads to a monodromy representation
\begin{equation*}
    \pi_1(\mathbb{CP}^1\setminus\{0,t,1,\infty\},z_0)\rightarrow \SL_2(\mathbb{C}),
\end{equation*}
in the usual way \cite{fokas}. This representation is constant in $t$, due to \eqref{eq:laxpvi2}.

After choosing some free generators $\gamma_k$, $k=0,t,1$, of the fundamental goup, corresponding to simple loops around the singularities $z=0,t,1$ respectively, the monodromy representation is fixed by the values of the respective corresponding monodromy matrices $M_0$, $M_t$ and $M_1$. We further define $M_\infty$ by
  \begin{equation*}
        M_\infty M_1M_tM_0=I,
    \end{equation*}
which corresponds to a loop around infinity.

Freedom of multiplication of any local solution to \eqref{eq:laxpvi} from the right by elements of $\operatorname{GL}_2(\mathbb{C})$, leads to the diagonal action of $\operatorname{SL}_2(\mathbb{C})$ on $\{(M_0,M_t,M_1)\in \SL_2(\mathbb{C})^3\}$ by simultaneous conjugation. The corresponding GIT quotient
\begin{equation*}
\{(M_0,M_t,M_1)\in \SL_2(\mathbb{C})^3\}\sslash\SL_2(\mathbb{C}),
\end{equation*}
forms the character variety relevant to $\Psix$.

We will use the description of this character variety as a family of affine cubic surfaces, following Fricke and Klein \cite{frickeklein}.
Namely, consider the quotient $R$ of the ring of polynomials in the entries of $M_0$, $M_t$ and $M_1$ by the ideal generated by the three relations coming from unimodularity. 
A classical result due to Fricke and Klein, and earlier Vogt \cite{vogt}, is that the ring of invariants $R^{\SL_2(\C)}$ of the $\SL_2(\C)$ action is given by 
$$R^{\SL_2(\C)} \cong \C[x_1,x_2,x_3,y,v_0,v_t,v_1,v_{\infty}] / (k_1,k_2),$$ 
where 
\begin{equation} \label{eq:pvitracesxv}
    \begin{aligned}
   x_1&=\operatorname{Tr} M_tM_1, &
   x_2&=\operatorname{Tr} M_0M_1, &
   x_3&=\operatorname{Tr} M_0M_t, &
   y&=\operatorname{Tr} M_1M_0M_t\\
   v_0&=\operatorname{Tr} M_0, & 
   v_t&=\operatorname{Tr} M_t, &
   v_1&=\operatorname{Tr} M_1, &
   v_\infty&=\operatorname{Tr} M_1M_tM_0,
\end{aligned}
\end{equation}
and 
\begin{align*}
k_1 &:= -y - v_\infty +  v_0 x_1+v_t x_2+v_1 x_3-v_0v_tv_1,\\
k_2 &:= -y\, v_\infty + x_1x_2x_3+x_1^2+x_2^2+x_3^2-(v_tv_1 x_1+v_0v_1x_2+v_0v_t x_3)+v_0^2+v_t^2+v_1^2-4.
\end{align*}
Now, the first equation is linear in $\{y,x_1,x_2,x_3\}$, and using it to elimate $y$ we get 
$$R^{\SL_2(\C)} \cong \C[x_1,x_2,x_3,v_0,v_t,v_1,v_{\infty}] / (p_w),$$
where
\begin{equation}
    p_w:=x_1x_2x_3
+x_1^2+x_2^2+x_3^2
-w_1 x_1-w_2x_2-w_3 x_3+w_4,\label{eq:pvicubic}
\end{equation}
with 
\begin{equation}
    \label{eq:pvi-om}
\begin{gathered}
     w_1:=v_0v_\infty+v_t v_1,\quad
    w_2:=v_0v_1+v_tv_\infty,\quad
    w_3:=v_0v_t+v_1v_\infty,\\  w_4:=v_0^2+v_t^2+v_1^2+v_\infty^2+v_0v_tv_1v_\infty-4.  
\end{gathered}
\end{equation}

Since the local exponents of ODE \eqref{eq:laxpvi1} around the singularities are fixed as in \eqref{eq:specsA}, the traces $v_k$ are given in terms of the $\theta$-parameters by
\begin{equation*}
 v_k=2\cos(2\pi\theta_k),
 \quad (k=0,t,1,\infty).
\end{equation*}
In particular, the $\theta$-parameters fix the values of $w_k$, $1\leq k\leq 4$, in \eqref{eq:pvicubic}, so that we regard this character variety as a four-parameter family of affine surfaces.
\begin{definition} \label{def:monodromysurfacefamilyPVI}
    The monodromy surface of $\pain{VI}$ is the embedded affine cubic surface  $$\mathcal{M}_{w}= \Spec \C[x_1,x_2,x_3]/(p_w),$$ 
    where $w\in\mathscr{W}_{\rm VI} = \left\{ w=(w_1,w_2,w_3,w_4) \in \C^4 \right\}$ and $p_w$ is the polynomial in \eqref{eq:pvicubic}, known as the Jimbo-Fricke cubic surface.
    This gives the family
    \begin{equation*}
        \mathcal{M}\to \mathscr{W}_{\rm VI},
    \end{equation*}
    with fibre over any $w\in \mathscr{W}_{\rm VI}$ being $\mathcal{M}_{w}$.
\end{definition}

We regard $\mathcal{M}_w\subseteq \mathbb{A}^3$ as an embedded affine variety. We define its projective completion through the map,
        $$\mathbb{A}^3\rightarrow \mathbb{P}^3,\quad (x_1,x_2,x_3) \mapsto [1:x_1:x_2:x_3],$$
        yielding the projective cubic $$\overline{\mathcal{M}}_w=\Proj \C[X_0,X_1,X_2,X_3]/(\overline{p}_w),\quad \operatorname{deg} X_i=1,$$ where
$$\overline{p}_w:=X_1X_2X_3
+X_0(X_1^2+X_2^2+X_3^2)
-X_0^2\left(w_1 X_1+w_2X_2+w_3 X_3\right)+w_4 X_0^3.$$

\begin{remark} \label{rem:propertiesofMwPVI}
We have the following properties of $\mathcal{M}_w$.
    \begin{itemize}
        \item The hyperplane section at infinity of $\overline{\mathcal{M}}_w$ is a triangle of lines,
        \begin{equation*}
            X_0 =0,\quad X_1X_2X_3=0,
        \end{equation*}
        consisting of only smooth points of $\overline{\mathcal{M}}_w$.       
        \item The singular locus $\mathscr{W}_{\rm VI}^{\operatorname{sing}}\subset \mathscr{W}_{\rm VI}$, where $\mathcal{M}_w$ is singular, is given by $Q(w)=0$, where $Q\in \C[w_1,w_2,w_3,w_4]$ is of degree $6,6,6,5$ in $w_1,w_2,w_3,w_4$ respectively, given in \Cref{app:singularlocusofJimboFricke}.
\end{itemize}

\end{remark}

\subsection{Riemann-Hilbert map and symmetries} \label{subsec:pvi:riemannhilbertmapandsymmetries}

In the context of $\pain{VI}$, the Riemann-Hilbert correspondence induces a map from the initial value space to the associated monodromy surface,
\begin{equation*}
    \begin{tikzcd}
        \mathcal{X}_{t,a} \arrow[r,"\operatorname{RH}_{t,a}"] & \mathcal{M}_{w(a)},
    \end{tikzcd}
\end{equation*}
which we refer to as the Riemann-Hilbert map for $\pain{VI}$. It is defined as follows. 
Any point on $\mathcal{X}_{t,a}$ determines uniquely a local solution around $t\in \mathscr{T}_{\rm{VI}}$ of $\pain{VI}$ with parameters $a\in\mathscr{A}_{\rm{VI}}$. Plugging this solution into the linear problem defines a corresponding set of monodromy data determining a point on the affine cubic $\mathcal{M}_{w(a)}$, where $w(a) \in \mathscr{W}_{\rm VI}$ is determined by $a$ according to equations \eqref{eq:pvi-om}, \eqref{eq:pvitracesxv} and  \eqref{eq:rootvarstotheta}.
This is a biholomorphism of complex analytic surfaces for generic $a$ \cite{inabaiwasakisaito}.

The B\"acklund transformation symmetries of $\pain{VI}$, as isomorphisms between corresponding initial value spaces, can be conjugated by the Riemann-Hilbert map. 
For the symmetries in $W(D_4^{(1)})$, the parameters $w$ are invariant under the action on $\mathscr{A}_{\rm VI}$, and they conjugate to the identity map on $\mathcal{M}_{w}$.

\begin{theorem}[Inaba-Iwasaki-Saito \cite{inabaiwasakisaito}]\label{thm:PVIRHsymmetries}
For $g\in W(D_4^{(1)})$, write the corresponding action of the B\"acklund transformation on parameters as $g : \mathscr{A}_{\mathrm{VI}} \to \mathscr{A}_{\mathrm{VI}},  a \mapsto \tilde{a},$
and the isomorphism of initial value spaces as 
$g : \mathcal{X}_{t,a}  \rightarrow \mathcal{X}_{t,\tilde{a}}.$
Then $w=w(a)=w(\tilde{a})$ and $g$ conjugates to the identity under the Riemann-Hilbert map, in other words, the following diagram commutes:
\begin{equation*}
    \begin{tikzcd}
        \mathcal{X}_{t,a} \arrow[d,"\operatorname{RH}_{t,a}"] \arrow[r,"g"] & \mathcal{X}_{t,\tilde{a}} \arrow[d,"\operatorname{RH}_{t,\tilde{a}}"] \\
        \mathcal{M}_{w} \arrow[r,"\operatorname{Id}"] & \mathcal{M}_{w}
    \end{tikzcd}
\end{equation*}
\end{theorem}
As for the B\"acklund transformations corresponding to Dynkin diagram automorphisms, they become combinations of sign flips and permutations.

\begin{proposition}[Lisovyy-Tykhyy \cite{lisovyyalgebraic}] \label{prop:PVI:dynkinautosunderRH}
    Under the Riemann-Hilbert map, Dynkin diagram automorphisms in $\Aut(D_4^{(1)})\cong \mathfrak{S}_4$ become symmetries of the family $\mathcal{M}\rightarrow \mathscr{W}_{\rm VI}$. 
    For $\sigma\in\Aut(D_4^{(1)})$, the action $\sigma: (x_1,x_2,x_3)\mapsto(\tilde{x}_1,\tilde{x}_2,\tilde{x}_3)$, $(w_1,w_2,w_3,w_4)\mapsto(\tilde{w}_1,\tilde{w}_2,\tilde{w}_3,\tilde{w}_4)$ is as given in \Cref{tab:PVI:symmetry:ws}.
\end{proposition}

\begingroup

\setlength{\tabcolsep}{15pt} 
\renewcommand{\arraystretch}{1.5} 

\begin{table}[h]
    \begin{equation*}
    \begin{array}{c||c|c|c||c|c|c|c|}
        \sigma \in \Aut(D_4^{(1)})      & \tilde{x}_1    & \tilde{x}_2     & \tilde{x}_3   & \tilde{w}_1   & \tilde{w}_2   & \tilde{w}_3   &  \tilde{w}_4          \\
        \hline\hline 
        \sigma_{(03)}   & x_1     & x_3     & x_2   & w_1   & w_3   & w_2   &  w_4          \\
        \sigma_{(34)}   & x_3    & x_2     & x_1   & w_3   & w_2   & w_1   &  w_4          \\
        \sigma_{(04)}   & x_2    & x_1     & x_3   & w_2   & w_1   & w_3   &  w_4          \\
        \sigma_{(03)(14)}   & x_1     & -x_2     & -x_3   & w_1   & -w_2   & -w_3   &  w_4   \\
        \sigma_{(04)(13)}   & -x_1     & -x_2     & x_3   & -w_1   & -w_2   & w_3   &  w_4   \\
        \sigma_{(01)(34)}   & -x_1     & x_2     & -x_3   & -w_1   & w_2   & - w_3   &  w_4  
    \end{array}
    \end{equation*}
    \caption{Action of $\Aut(D_4^{(1)})$ symmetries of $\pain{VI}$ on $\mathcal{M}\rightarrow\mathscr{W}_{\rm VI}$}
    \label{tab:PVI:symmetry:ws}
\end{table}
\endgroup

\begin{remark}
    Note that $\langle\sigma_{(03)},\sigma_{(34)}\rangle \cong \mathfrak{S}_3$, and $\langle\sigma_{(03)(14)}, \sigma_{(04)(13)}\rangle\cong \Z_2\times\Z_2$, which recovers the description of $\mathfrak{S}_4$ as the semi-direct product of $\mathfrak{S}_3$ and the Klein 4-group.
    The subgroup $\langle\sigma_{(03)(14)}, \sigma_{(04)(13)}\rangle$ corresponds to the part of the extended affine Weyl group given by  translations in the weight lattice of $D_4$ modulo the root lattice.
\end{remark}

\subsection{Moduli space}
In this section, we show that the family $\mathcal{M}\rightarrow \mathscr{W}_{\rm VI}$ forms a universal family in a category of embedded affine varieties. This in particular describes the parameter space $\mathscr{A}_{\rm VI}$ of $\pain{VI}$, modulo the action of $W(D_4^{(1)})\rtimes \Aut(D_4^{(1)})$, as the corresponding moduli space.

\begin{definition}[Category $\mathfrak{C}_{\mathrm{VI}}$ of monodromy surfaces for $\pain{VI}$]
    Define the category $\mathfrak{C}_{\mathrm{VI}}$ with
    \begin{itemize}
        \item an object being an embedded affine cubic surface $\mathcal{V}\subset\mathbb{A}^3$, such that the complement in its projective completion, $\overline{\mathcal{V}}\setminus\mathcal{V}\subseteq \mathbb{P}^3$,  is the union of three lines which intersect like a triangle and consist of only smooth points of $\overline{\mathcal{V}}$.
        \item a morphism $L$ between objects $\mathcal{V}_1\subset\mathbb{A}^3$ and $\mathcal{V}_2\subset\mathbb{A}^3$ being an affine linear map $L\in \operatorname{End}(\mathbb{A}^3)$ such that $L(\mathcal{V}_1)\subseteq \mathcal{V}_2$.
    \end{itemize}
\end{definition}

Then $\mathcal{M}\to\mathscr{W}_{\rm VI}$ from \Cref{def:monodromysurfacefamilyPVI} is a universal family for $\mathfrak{C}_{\mathrm{VI}}$ in the following sense.

\begin{proposition} \label{prop:oblomkov}
   For any object $\mathcal{V}$ in $\mathfrak{C}_{\rm VI}$, there exists a, unique up to $\operatorname{Aut}(D_4^{(1)})$ action, $w\in \mathscr{W}_{\rm VI}$ such that $\mathcal{V}$ is isomorphic in $ \mathfrak{C}_{\rm VI}$ to $\mathcal{M}_w$. 
   The automorphism group of $\mathcal{M}_w$ is $\operatorname{Stab}_{\Aut(D_4^{(1)})}\left(w\right)$, which is trivial for generic $w$.
\end{proposition}
\begin{proof}
We note that existence can be found in \cite{oblomkov} and we repeat the argument here.
Let $\mathcal{V}\subset \mathbb{A}^3$ be an object in $\mathfrak{C}_{\rm VI}$ and denote its projective completion by $\overline{\mathcal{V}} \subset \mathbb{P}^3$. 
The hyperplane section $\overline{\mathcal{V}}\setminus \mathcal{V}$ is by assumption a triangle of lines defined by
\begin{equation*}
    L_1 L_2 L_3=0,\quad X_0=0,
\end{equation*}
where each $L_k\in \mathbb{C}[X_0,X_1,X_2,X_3]$ is homogeneous of degree 1. Since the three lines $L_k=0$, $1\leq k\leq 3$, pairwise intersect at  distinct points, the affine map
\begin{equation*}
 x_k\mapsto L_k(x_1,x_2,x_3)\qquad (1\leq k\leq 3),   
\end{equation*}
has full rank and application of its inverse puts $\mathcal{V}$ into the form
\begin{equation}\label{eq:cubic_mixed}
    x_1x_2x_3+a_1 x_1^2+a_2 x_2^2+a_3 x_3^2+b_1x_2x_3+b_2x_1x_3+b_3x_1x_2+l(x_1,x_2,x_3)=0,
\end{equation}
for some $a_{1,2,3},b_{1,2,3}\in\mathbb{C}$ and a  polynomial $l$ of degree at most one. Applying $x_k\mapsto x_k-b_k$, we may eliminate all mixed terms from \eqref{eq:cubic_mixed}, yielding
\begin{equation*}
    x_1x_2x_3+a_1 x_1^2+a_2 x_2^2+a_3 x_3^2-c_1 x_1-c_2 x_2-c_3 x_3+c_4=0,
\end{equation*}
for some $c_k\in\mathbb{C}$, $1\leq k\leq 4$.
Further scaling, 
$x_k\mapsto \gamma_k x_k$ for $k=1,2,3$, rescales the coefficients in the above cubic as
\begin{equation*}
    a_1\mapsto a_1 \frac{\gamma_1}{\gamma_2\gamma_3},\quad 
    a_2\mapsto a_2 \frac{\gamma_2}{\gamma_1\gamma_3},\quad
    a_3\mapsto a_3 \frac{\gamma_3}{\gamma_1\gamma_2},\quad
   c_1\mapsto c_1 \frac{1}{\gamma_2\gamma_3},\quad 
    c_2\mapsto c_2 \frac{1}{\gamma_1\gamma_3},\quad
    c_3\mapsto c_3 \frac{1}{\gamma_1\gamma_2}.
\end{equation*}
Since $\overline{\mathcal{V}}\setminus \mathcal{V}$ consists of only smooth points of $\overline{\mathcal{V}}$, the coefficients of the quadratic terms are nonzero and any of the four solutions $(\gamma_1,\gamma_2,\gamma_3)\in(\mathbb{C}^*)^3$ to
$$\gamma_1^2 = a_2 a_3, \quad \gamma_2^2 = a_1 a_3, \quad \gamma_3^2 = a_1 a_2,\quad \gamma_1\gamma_2\gamma_3=a_1a_2a_3,$$
put the cubic into the form 
$p_w=0$ as in \Cref{eq:pvicubic} 
with 
$$w_1 = \frac{c_1}{\gamma_2\gamma_3}, \quad 
w_2 = \frac{c_2}{\gamma_1\gamma_3}, \quad 
w_3 = \frac{c_3}{\gamma_1\gamma_2}, \quad 
w_4 = \frac{c_4}{\gamma_1\gamma_2 \gamma_3}.$$
This gives, for any $\mathcal{V}$, a $w \in \mathscr{W}_{\rm VI}$ and an isomorphism $\mathcal{V}\to \mathcal{M}_w$ determined up to $\Aut(D_4^{(1)})$ action, through the freedom of choice of enumeration of $L_1,L_2,L_3$ and the choice of $(\gamma_1,\gamma_2,\gamma_3)$ which correspond exactly to the actions on $x_i$ and $w_i$ in \Cref{tab:PVI:symmetry:ws}.

To show that $\mathcal{M}_w$ is isomorphic to $\mathcal{M}_{\tilde{w}}$ if and only if $w$ and $\tilde{w}$ are related by the action of some $\sigma\in \Aut(D_4^{(1)})$ given in \Cref{tab:PVI:symmetry:ws}, we compute directly as follows. 
Explicitly, up to permutation of coordinates via the action of $\langle \sigma_{(03)},\sigma_{(34)}\rangle \cong \mathfrak{S}_3$, to match the triangles at infinity $\overline{\mathcal{M}}_{w}\setminus \mathcal{M}_{w}$ and $\overline{\mathcal{M}}_{\tilde{w}}\setminus \mathcal{M}_{\tilde{w}}$ the isomorphism must be of the form $(x_1,x_2,x_3)\mapsto (A_1 x_1,A_2 x_2, A_3x_3)$ for some nonzero $A_1,A_2,A_3$. Requiring that $p_{w}=0$ is sent to $p_{\tilde{w}}=0$ gives 
\begin{equation*}
    A_1A_2A_3=1,\quad A_1^2=A_2^2=A_3^2=1, \qquad w_i = A_i \tilde{w}_i, ~i=1,2,3,\quad \tilde{w}_4=w_4, 
\end{equation*}
the solutions of which correspond to the remaining action of $\langle \sigma_{(03)(14)}, \sigma_{(04)(13)}\rangle$.
An automorphism of $\mathcal{M}_w$ then must come from the action of some $\sigma \in \Aut(D_4^{(1)})$ such that $\sigma(w)=w$, so generically the isomorphism $\mathcal{V}\to \mathcal{M}_w$ is unique for fixed $w$.
\end{proof}

This implies that the Dynkin diagram automorphisms as in \Cref{prop:PVI:dynkinautosunderRH} exhaust all symmetries of the Jimbo-Fricke cubic in the category $\mathfrak{C}_{\rm VI}$.
\begin{corollary}
    The symmetry group of the family of embedded affine varieties $\mathcal{M}\to\mathscr{W}_{\rm{VI}}$ is $\Aut(D_4^{(1)})\cong\mathfrak{S}_4$, acting as in \Cref{prop:PVI:dynkinautosunderRH}.
\end{corollary}

\begin{remark}
    As an affine variety, forgetting about the embedding, $\mathcal{M}_w$ admits an infinite group of symmetries,
    \begin{align*}
        &C_2*C_2*C_2=\langle r_1,r_2,r_3\,|\, r_1^2=r_2^2=r_3^2=1\rangle,\\
       & r_i:x_i\mapsto w_i-x_i-x_jx_k,\quad x_j\mapsto x_j,\quad x_k\mapsto x_k,\quad \{i,j,k\}=\{1,2,3\},
    \end{align*}
     which already appeared in \cite[Sec. ${\rm{IV}}$]{fricke1904} and, together with $\operatorname{Stab}_{\Aut(D_4^{(1)})}\left(w\right)$, generates the full symmetry group, as follows from \cite[Th. 2]{ElHuti}.
   But, beyond $\operatorname{Stab}_{\Aut(D_4^{(1)})}\left(w\right)$, these do not preserve the embedding and are thus not automorphisms of $\mathcal{M}_w$ in $\mathfrak{C}_{\rm{VI}}$.
    These symmetries were studied in the context of $\pain{VI}$ in \cites{IwasakiModular, CantatLoray}, and in particular used to classify algebraic solutions in \cites{dubrovinmazzocco,lisovyyalgebraic}.
\end{remark}



Taking the quotient of $\mathscr{W}_{\rm{VI}}$ by the $\Aut(D_4^{(1)})$ action, we obtain a classifying map for objects in $\mathfrak{C}_{\rm{VI}}$.
By \Cref{prop:oblomkov}, we get a mapping 
\begin{equation}\label{eq:moduli_map}
\begin{gathered}
     \Phi : \operatorname{ob}(\mathfrak{C}_{\mathrm{VI}}) \to \mathscr{O}_{\mathrm{VI}} := \mathscr{W}_{\rm{VI}}\sslash\Aut(D_4^{(1)}) = \Spec \C[o_1,o_2,o_3,o_4],\\
    o_1 = w_1^2+w_2^2+w_3^2, \quad o_2= w_1^2w_2^2+w_2^2w_3^2+w_3^2w_1^2, \quad o_3 = w_1w_2w_3, \quad o_4=w_4.
    \end{gathered}
    \end{equation}
    The singular locus $\mathscr{W}_{\rm{VI}}^{\operatorname{sing}}\subset \mathscr{W}_{\rm{VI}}$ descends to the quotient as $\mathscr{O}_{\rm{VI}}^{\operatorname{sing}}\subset \mathscr{O}_{\rm{VI}}$ and is given by $P(o)=0$, where $P\in \mathbb{Q}[o_1,o_2,o_3,o_4]$ is of degree $6,4,3,10$ in $o_1,o_2,o_3,o_4$ respectively, given in \Cref{app:singularlocusofJimboFricke}.

    Note that $\Phi$ induces a surjective map from $\mathscr{A}_{\rm{VI}}\cong \Theta_{\rm{VI}}$ to $\mathscr{O}_{\rm{VI}}(\C)$, and two points have the same image if and only if they are related by the action of $\widetilde{W}(D_4^{(1)})$. This map also appears in \cite[Rem. 11]{lisovyyalgebraic}.
    From \Cref{prop:oblomkov}, we get the following corollary.

\begin{corollary} \label{cor:modulispacePVI}
The scheme $\mathscr{O}_{\rm{VI}}$ over $\mathbb{C}$ is the moduli space of isomorphism classes of objects in $\mathfrak{C}_{\rm{VI}}$.
That is, $\Phi$ induces a canonical bijection
\begin{equation} \label{eq:canonicalbijectionPVI}
\operatorname{Iso}(\mathfrak{C}_{\rm{VI}}) \to \mathscr{O}_{\rm{VI}}(\mathbb{C})\cong \mathbb{C}^4
\end{equation}
from the set of isomorphism classes of objects in $\mathfrak{C}_{\rm{VI}}$ to complex points of $\mathscr{O}_{\rm{VI}}$.
\end{corollary}

\begin{proof}



By Proposition \ref{prop:oblomkov}, the mapping $\Phi$ associates the same complex point on $\mathscr{O}_{\rm{VI}}$ to two cubic surfaces if and only if they are isomorphic in $\mathfrak{C}_{\rm{VI}}$. 
Therefore, we have an induced injective map from $\operatorname{Iso}(\mathfrak{C}_{\rm{VI}})$ to  $\mathscr{O}_{\rm{VI}}(\C)$.
Note further that $\Phi$ is  surjective since any value $\Phi(\mathcal{M}_w)=o\in \mathscr{O}_{\rm{VI}}(\C)$ can be realised by an appropriate choice of $w\in \mathscr{W}_{\rm{VI}}$, which establishes the corollary.
\end{proof}
\begin{remark} \label{rem:promotingtomodulistack}
    \Cref{cor:modulispacePVI} identifies the complex points of \(\mathscr O_{\mathrm{VI}}\) canonically with the isomorphism classes of objects of \(\mathfrak{ C}_{\mathrm{VI}}\). 
    We expect that \Cref{cor:modulispacePVI} can be promoted to a description of a moduli stack of $\mathfrak{C}_{\rm{VI}}$, taking into account families over arbitrary base schemes.
\end{remark}

\subsection{Lines on the monodromy surface}
For $w \in \mathscr{W}_{\rm{VI}}\setminus \mathscr{W}_{\rm{VI}}^{\operatorname{sing}}$, the projective completion $\overline{\mathcal{M}}_w$ of the monodromy surface for $\pain{VI}$ is a smooth projective cubic surface, and thus by the Cayley-Salmon theorem contains 27 lines. 

The lines on the Jimbo-Fricke cubic surface were first written down in \cite{mklimes2024}. They correspond to different forms of partial reducibility of monodromy representations \cites{mklimes2024,ORS}, which in turn correspond to truncations of the generic asymptotic expansions of solutions around critical points under suitable parameter assumptions, see \cite{guzzetti12}.

The problem of computing lines on a smooth projective cubic surface, as is the case for other problems in enumerative geometry, has a Galois-theoretic interpretation, as worked out by Jordan \cite{jordan1870} and revisited by Harris \cite{harris1979}.
In the case of a smooth projective cubic $\mathcal{S}\subset \p^3$, this is in terms of the field extension of the function field $\C(\mathcal{S})$ of the cubic to that of the incidence variety of lines on $\mathcal{S}$, see the survey \cite{sottileyahl}.

In our setting, this can be seen in terms of the need to extend the field $\C(\mathscr{W}_{\rm{VI}})$ in order to write lines on $\mathcal{M}_w$ rationally. 
We provide a description of the lines using a different field extension than the one used in \cite{mklimes2024}, which is minimal in allowing the lines to be written down rationally, see Remark \ref{rem:rational}.

Starting from the parameter space $\Theta_{\mathrm{VI}}$ for $\pain{VI}$, introduce  
$$\mathscr{U}_{\rm{VI}} = \left\{ u=(u_1,u_2,u_3,u_4)\in (\C^*)^{\times 4} \right\},$$
where 
\begin{equation}\label{eq:pvitheta_to_u}
 u_1 = e^{2\pi i(\theta_1+\theta_t)}, 
\quad  u_2 = e^{2\pi i(\theta_0+\theta_1)},
\quad  u_3 = e^{2\pi i(\theta_0+\theta_t)},
\quad  u_4 = e^{2\pi i(\theta_1+\theta_{\infty})}.   
\end{equation}

In terms of these parameters, the lines on the surface $\mathcal{M}_w$ have the following expressions.
\begin{equation}\label{eq:linesPVI}
\begin{aligned}
    L_k : \quad x_1 &= b_k+\frac{1}{b_k},\quad &&b_k x_2+x_3=\frac{w_3b_k-w_2}{b_k-b_k^{-1}}, \quad &&k=1,\dots,8, \\
    L_k :\quad  x_2 &= b_k+\frac{1}{b_k},\quad &&b_k x_3+x_1=\frac{w_1b_k-w_3}{b_k-b_k^{-1}}, \quad &&k=9,\dots,16, \\
    L_k : \quad x_3 &= b_k+\frac{1}{b_k},\quad &&b_k x_1+x_2=\frac{w_2b_k-w_1}{b_k-b_k^{-1}}, \quad &&k=17,\dots,24, 
\end{aligned}
\end{equation}
where the $b_k$, $k=1,\dots,24$, are given by
\begin{equation}\label{eq:defbpvi}
\begin{aligned}
    b_1 &= \frac{u_3}{u_2}, &
    b_2 &= \frac{u_2}{u_3}, &
    b_3 &= \frac{u_3u_4}{u_1}, &
    b_4 &= \frac{u_1}{u_3u_4}, &
    b_5 &= u_1, &
    b_6 &= \frac{1}{u_1}, &
    b_7 &= \frac{u_2}{u_4}, &
    b_8 &= \frac{u_4}{u_2},\\
    b_9 &= \frac{u_1}{u_3}, &
    b_{10} &= \frac{u_3}{u_1}, &
    b_{11} &= \frac{u_3u_4}{u_2}, &
    b_{12} &= \frac{u_2}{u_3u_4}, &
    b_{13} &= u_2, &
    b_{14} &= \frac{1}{u_2}, &
    b_{15} &= \frac{u_1}{u_4}, &
    b_{16} &= \frac{u_4}{u_1},\\
    b_{17} &= \frac{u_2}{u_1}, &
    b_{18} &= \frac{u_1}{u_2}, &
    b_{19} &= u_4, &
    b_{20} &= \frac{1}{u_4}, &
    b_{21} &= u_3, &
    b_{22} &= \frac{1}{u_3}, &
    b_{23} &= \frac{u_1u_2}{u_3u_4}, &
    b_{24} &= \frac{u_3u_4}{u_1u_2}.
\end{aligned}
\end{equation}
We will also use $L_k$ to denote the closure of $L_k$ in $\mathbb{P}^3$, $1\leq k\leq 24$. The remaining three lines on the projective cubic are
\begin{equation*}
\begin{aligned}
    L_{25}=\{X\in\mathbb{P}^3 : X_0=0,X_1=0\},\\
    L_{26}=\{X\in\mathbb{P}^3 : X_0=0,X_2=0\},\\
    L_{27}=\{X\in\mathbb{P}^3 : X_0=0,X_3=0\}.
\end{aligned}
\end{equation*}


The following lemma shows that $\C(\mathscr{U}_{\rm{VI}})$ is indeed a field extension of $\C(\mathscr{W}_{\rm{VI}})$.
\begin{lemma}     \label{lem:PvimorphismUtoW}
The definitions of the $u_i,w_i$ in terms of $\theta_j$ yields the field homomorphism 
    \begin{equation} \label{pvifieldextensionUW}
    \C(\mathscr{W}_{\rm{VI}}) \longrightarrow \C(\mathscr{U}_{\rm{VI}}) ,
    \end{equation}
   specified by 
\begin{align*}
	w_1&=u_1+\frac{1}{u_1}+\frac{u_2}{u_3}+\frac{u_3}{u_2}+\frac{u_2}{u_4}+\frac{u_4}{u_2}+\frac{u_1}{u_3 u_4}+\frac{u_3 u_4}{u_1},\\
	w_2&=u_2+\frac{1}{u_2}+\frac{u_1}{u_3}+\frac{u_3}{u_1}+\frac{u_1}{u_4}+\frac{u_4}{u_1}+\frac{u_2}{u_3 u_4}+\frac{u_3 u_4}{u_2},\\
    w_3&=u_3+\frac{1}{u_3}+\hspace{0.5mm}u_4\hspace{0.5mm}+\frac{1}{u_4}+\frac{u_1}{u_2}+\frac{u_2}{u_1}+\frac{u_1 u_2}{u_3 u_4}+\frac{u_3 u_4}{u_1 u_2},\\
	w_4&=4
    +u_3 u_4
    +\frac{1}{u_3 u_4}
    +\frac{u_3}{u_4}
    +\frac{u_4}{u_3}  
    +\frac{u_1}{u_2 u_3}
    +\frac{u_2 u_3}{u_1}
    +\frac{u_1}{u_2 u_4}
    +\frac{u_2 u_4}{u_1}
    +\frac{u_3}{u_1 u_2}
    +\frac{u_1 u_2}{u_3}
    +\frac{u_4}{u_1 u_2}
    +\frac{u_1 u_2}{u_4}\\
    &\quad  +\frac{u_2}{u_1 u_3}
    +\frac{u_1 u_3}{u_2}
    +\frac{u_2}{u_1 u_4}
    +\frac{u_1 u_4}{u_2}
    +\frac{u_1^2}{u_3 u_4}
    +\frac{u_3 u_4}{u_1^2}
    +\frac{u_2^2}{u_3 u_4}
    +\frac{u_3 u_4}{u_2^2}
    +\frac{u_3^2 u_4}{u_1 u_2}
    +\frac{u_1 u_2}{u_3^2 u_4}
    +\frac{u_3 u_4^2}{u_1 u_2}
    +\frac{u_1 u_2}{u_3 u_4^2}.
\end{align*}
 \end{lemma}
 \begin{proof}
 Introducing the variables $t_j=e^{2\pi i \theta_j}$, $1\leq j\leq 4$, the definitions of $u_i,w_i$ in terms of $t_j$ give field homomorphisms $\C(\mathscr{U}_{\rm{VI}}) \to  \mathbb{C}(t_1,t_2,t_3,t_4)$, and $\C(\mathscr{W}_{\rm{VI}}) \to  \mathbb{C}(t_1,t_2,t_3,t_4)$. Direct computation in the field $\mathbb{C}(t_1,t_2,t_3,t_4)$, or really $\mathbb{Q}(t_1,t_2,t_3,t_4)$, shows that \eqref{pvifieldextensionUW} is the unique homomorphism making the following diagram commute:
 \begin{equation*}
     \begin{tikzcd}[sep=tiny]
        \C(\mathscr{W}_{\rm{VI}}) \arrow[dd] \arrow[dr]& \\
        & \qquad \mathbb{C}(t_1,t_2,t_3,t_4) \\
         \C(\mathscr{U}_{\rm{VI}}) \arrow[ur]
     \end{tikzcd}
 \end{equation*}
 and the lemma follows.
 \end{proof}
The singular locus $\mathscr{W}_{\rm{VI}}^{\operatorname{sing}}\subset \mathscr{W}_{\rm{VI}}$ of the family $\mathcal{M}\to \mathscr{W}_{\rm{VI}}$ pulls back to $\mathscr{U}_{\rm{VI}}^{\operatorname{sing}}\subset \mathscr{U}_{\rm{VI}}$  defined by 
\begin{equation*}
\begin{aligned}
    &\left(u_1- u_2 u_3\right)
    \left(u_2- u_1 u_3\right) 
    \left(u_3- u_1 u_2\right) 
    \left(u_1- u_2 u_4\right)
    \left(u_2- u_1 u_4\right) 
    \left(u_4- u_1 u_2\right) \times \\ 
   &\left(u_3-u_4\right) 
    \left(u_3 -u_4^{-1}\right) 
     \left(u_1^2-u_3 u_4\right) 
    \left(u_2^2-u_3 u_4\right)    
    \left(u_3^2 -u_1 u_2u_4^{-1}\right) 
    \left( u_4^2-u_1 u_2u_3^{-1}\right)=0.
\end{aligned}
\end{equation*}

\begin{remark}
    When $u\in \mathscr{U}_{\rm{VI}}^{\operatorname{sing}}$, some of the lines are merged. 
For example, when $u_3=u_4$, the following pairs of lines coincide: 
\begin{equation*}
    (L_1,L_8), \quad (L_2,L_7), \quad (L_9,L_{16}),\quad (L_{10},L_{15}),\quad (L_{19},L_{21}),\quad (L_{20},L_{22}).
\end{equation*}
Under the Riemann-Hilbert correspondence, $\mathscr{U}_{\rm{VI}}^{\operatorname{sing}}$ corresponds to parameter values for the Riccati solutions of the sixth Painlev\'e equation \cite{inabadynamics}. 
\end{remark}

\begin{remark} \label{rem:linesmonodromyPVI}
    Note that the lines and their characterisation in terms of partial reducibility of monodromy are given in \cites{mklimes2024,ORS}. 
    The relevance of the lines in the asymptotic theory of $\pain{VI}$ can already be observed in Jimbo \cite{jimbo1982}, and is fully worked out by Guzzetti \cites{guzzetti2006,guzzetti12}, who also provides the corresponding partial reducibility of monodromy. 
    There are also six distinguished conics, which correspond to logarithmic asymptotics of solutions of $\pain{VI}$ \cite{guzzetti2008loga}.
    These curves are given by 
    $$X_i=\pm 2X_0, \qquad (X_j\pm X_k)^2 - X_0(w_j X_j+x_kX_k) + (w_4+4 \mp 2w_i)X_0^2=0,\qquad \{i,j,k\}=\{1,2,3\}, $$
 for $1\leq i\leq 3$.
    These can be characterised geometrically in terms of hyperplane sections of $\overline{\mathcal{M}}_w$ that decompose into a line at infinity and a conic tangent to it.
\end{remark}

\subsection{Combinatorial monodromy}
~
We now give our first description of monodromy of the family $\mathcal{M}\to\mathscr{W}_{\rm{VI}}$, which is combinatorial and concerns permutations of lines which preserve their intersections.

\begin{definition} \label{def:pviintersectiongraph}
    Let $u \in \mathscr{U}_{\rm{VI}}\setminus \mathscr{U}_{\rm{VI}}^{\operatorname{sing}}$ and consider the lines $L_1,\dots,L_{27}$ on $\overline{\mathcal{M}}_{w}$, for $w$ given in terms of $u$ as in \Cref{lem:PvimorphismUtoW}.
    Form the intersection graph of lines on $\overline{\mathcal{M}}_{w}$ as a 2-coloured graph as follows.
    The vertices are the indices $1,\dots,{27}$, with those corresponding to the triangle at infinity coloured red, and the remaining coloured blue. 
    Two vertices $i,j$ are connected by an edge if and only if the lines $L_i$ and $L_j$ intersect in $\mathbb{P}^3$. Edges among red vertices are coloured red and others are coloured blue.
    We denote the intersection graph by $\mathcal{G}_{\rm{VI}}$, and the subgraph generated by the red vertices by $\mathcal{G}_{\rm{VI}}^{\infty}$.
\end{definition}
The intersection graph can be computed directly from the explicit expressions for the lines $L_k$, $1\leq k\leq 27$, in equation \eqref{eq:linesPVI}, yielding the coloured graph in Figure  \ref{fig:linesgraphpvi}. It is not surprising that the result is independent of $u\in \mathscr{U}_{\rm{VI}}\setminus \mathscr{U}_{\rm{VI}}^{\operatorname{sing}}$, since, for generic $u$, $\overline{\mathcal{M}}_{w}$ is a smooth projective cubic without Eckardt points \cite{dolgachevclassicalAG}, so that the intersection graph of lines must be the dual of the famous Schl\"afli graph \cite{schlafli1858}.

\begin{figure}[htb]
    \centering
    \includegraphics[width=1\linewidth]{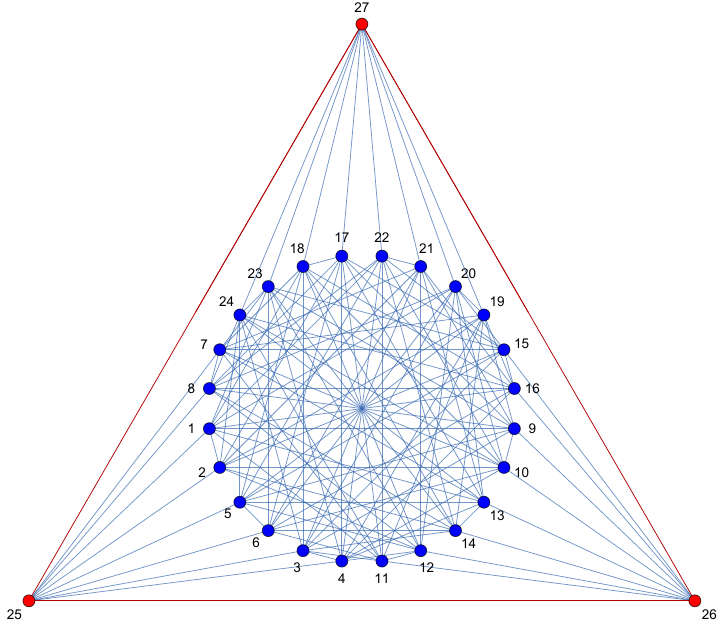}
    \caption{Intersection graph $\mathcal{G}_{\rm{VI}}$ of lines on the monodromy surface for $\pain{VI}$,
    with in red the subgraph $\mathcal{G}_{\rm{VI}}^\infty$ generated by the lines at infinity $(L_{25},L_{26},L_{27})$, as defined in Definition \ref{def:pviintersectiongraph}.
    }
    \label{fig:linesgraphpvi}
\end{figure}

Considering the intersection graph above $\mathscr{W}_{\rm{VI}}$, instead of $\mathscr{U}_{\rm{VI}}$, leads to a combinatorial version of monodromy of the family $\mathcal{M} \to \mathscr{W}_{\rm{VI}}$, made precise in the following definition.

\begin{definition} \label{def:combinatorialmonodromyPVI}
    The \emph{combinatorial monodromy} of the family $\mathcal{M}\to\mathscr{W}_{\rm{VI}}$ is the set of permutations of lines on $\mathcal{M}_w$ that preserve the intersection configuration, in other words, the group
    \begin{equation*}
        \operatorname{Fix}_{\operatorname{Aut}(\mathcal{G}_{\rm{VI}})}(\mathcal{G}_{\rm{VI}}^\infty)
    \end{equation*}
    of automorphisms of the graph $\mathcal{G}_{\rm{VI}}$ in \Cref{fig:linesgraphpvi} that pointwise fix the vertices $\{{25},{26},{27}\}$.
\end{definition}

We will show that the combinatorial monodromy of $\mathcal{M}\to \mathscr{W}_{\rm{VI}}$ forms the underlying finite Weyl group $W(D_4)$ of the symmetry group $\widetilde{W}(D_4^{(1)})$ of $\pain{VI}$, which we consider using the usual presentation generated by simple reflections encoded in the $D_4$ Dynkin diagram: 
 \begin{equation*}
 	W(D_4) =
 	\left\langle r_1,r_2,r_3,r_4 \quad {\bigg |}\quad 
 	r_{i}^{2} = 1,  \quad
 	\begin{aligned}
 		(r_{i} r_{j})^2 &= 1&  &\text{ when }
 		\raisebox{-0.15in}{\begin{tikzpicture}[
 				elt/.style={circle,draw=black!100,thick, inner sep=0pt,minimum size=1.5mm}]
 				\path   ( 0,0) 	node  	(ai) [elt] {}
 				( 0.5,0) 	node  	(aj) [elt] {};
 				\draw [black] (ai)  (aj);
 				\node at ($(ai.south) + (0,-0.2)$) 	{\small ${i}$};
 				\node at ($(aj.south) + (0,-0.2)$)  {\small ${j}$};
 		\end{tikzpicture}}\\
 		(r_{i} r_{j})^3 &= 1& &\text{ when }
 		\raisebox{-0.12in}{\begin{tikzpicture}[
 				elt/.style={circle,draw=black!100,thick, inner sep=0pt,minimum size=1.5mm}]
 				\path   ( 0,0) 	node  	(ai) [elt] {}
 				( 0.5,0) 	node  	(aj) [elt] {};
 				\draw [black, thick] (ai) -- (aj);
 				\node at ($(ai.south) + (0,-0.2)$) 	{\small ${i}$};
 				\node at ($(aj.south) + (0,-0.2)$)  {\small ${j}$};
 		\end{tikzpicture}}
 	\end{aligned} 
 	\right\rangle
    \qquad 
    \quad
       \raisebox{-20pt}{
\begin{tikzpicture}[scale=.4,elt/.style={circle,draw=black!100,thick, inner sep=0pt,minimum size=1.5mm}]
		\path 	(-1.25,0) 	node 	(d1) [elt ] {}
		        ( 0,0) 	node  	(d2) [elt  ] {}
		        ( 1,-1) 	node  	(d3) [elt  ] {}
		        (1,1) 	node 	(d4) [elt ] {};
		\node at ($(d1.west) + (-.3,+0.0)$) 	{ \tiny ${1}$};
		\node at ($(d2.east) + (+.3,+0.0)$) 	{ \tiny ${2}$};
		\node at ($(d3.east) + (+.3,+0.0)$) 	{ \tiny ${3}$};
		\node at ($(d4.east) + (+.3,+0.0)$) 	{ \tiny ${4}$};

		\draw [black,line width=1pt ] (d2) -- (d1);
		\draw [black,line width=1pt ] (d3) -- (d2) -- (d4);
		\node at ($(d2.east) + (+0,-2)$) 	{ \small ${D_4}$};

	\end{tikzpicture} .}
 \end{equation*}


\begin{proposition} \label{prop:combinatorialmonodromyPVI}
The combinatorial monodromy of the family $\mathcal{M}\to\mathscr{W}_{\rm{VI}}$ is isomorphic to $W(D_4)$, with an explicit isomorphism 
    \begin{equation*}
     W(D_4)\xrightarrow{\sim}   \operatorname{Fix}_{\operatorname{Aut}(\mathcal{G}_{\rm{VI}})}(\mathcal{G}_{\rm{VI}}^\infty), \quad r\mapsto \tilde{r},
    \end{equation*}
defined by sending the generators $r_k$, $1\leq k\leq 4$, to
\begin{align*}
	\tilde{r}_1&=(3\,\, 7)\,(4\, \, 8)\,(11\,\, 15)\,(12\,\, 16)\,(19\,\, 23)\,(20\,\, 24),\\
	\tilde{r}_2&=(5\,\, 8)\,(6\, \, 7)\,(13\,\, 16)\,(14\,\, 15)\,(21\,\, 24)\,(22\,\, 23),\\
	\tilde{r}_3&=(1\,\, 5)\,(2\, \, 6)\,(9\,\, 14)\,(10\,\, 13)\,(19\,\, 24)\,(20\,\, 23),\\
	\tilde{r}_4&=(3\,\, 8)\,(4\, \, 7)\,(9\,\, 13)\,(10\,\, 14)\,(17\,\, 22)\,(18\,\, 21).
\end{align*}
\end{proposition}
\begin{proof}
        It is well-known that a smooth projective cubic surface can be realised as $\p^2$ blown up at six points in general position.        
        Applying this to $\overline{\mathcal{M}}_{w}$  allows us to translate the problem of computing the subgroup $\operatorname{Fix}_{\operatorname{Aut}(\mathcal{G}_{\rm{VI}})}(\mathcal{G}_{\rm{VI}}^\infty)$ of graph automorphisms to an algebraic one in the Picard group $\Pic(\overline{\mathcal{M}}_w)$.
       
        Fix $w \in \mathscr{W}_{\rm{VI}}\setminus \mathscr{W}_{\rm{VI}}^{\operatorname{sing}}$.
        For any set of six lines on $\overline{\mathcal{M}}_w$ of which none intersect - a `sixer' in the language of \cite{dolgachevclassicalAG} - 
        there is a birational morphism $\pi :\overline{\mathcal{M}}_w\to \p^2$, which contracts these six lines.
    This gives 
    $$\Pic(\overline{\mathcal{M}}_w) \cong \Z \mathcal{H} \oplus \Z \E_1 \oplus \Z \E_2\oplus \Z \E_3\oplus \Z \E_4\oplus \Z \E_5\oplus \Z \E_6,$$
    where $\h$ corresponds to the linear equivalence class of the pullback by $\pi$ of a hyperplane divisor on $\p^2$ and $\E_1,\dots,\E_6$ correspond to classes of the six lines contracted by $\pi$.
    The exceptional curves on $\overline{\mathcal{M}}_w$ are the 27 lines on $\overline{\mathcal{M}}_w$.
    On the level of $\Pic(\overline{\mathcal{M}}_w)$, these correspond to the subset 
    $$\operatorname{EX} = \{\F\in \Pic(\overline{\mathcal{M}}_w) ~|~ \F\cdot\F =\F \cdot \mathcal{K}_{\overline{\mathcal{M}}_w}= -1\},$$
    where $\mathcal{K}_{\overline{\mathcal{M}}_w}$ is the canonical divisor class of $\overline{\mathcal{M}}_w$ and $\cdot $ is the intersection form on $\Pic(\overline{\mathcal{M}}_w)$.
    The elements of $\operatorname{EX}$ are all of the form
    $$\E_i, \quad \h - \E_j - \E_k, \quad 2\h - \E_{1}-\dots-\E_6 + \E_{\ell},\quad \text{ where } \quad i,j,k,\ell\in\{1,\dots,6\}, j\neq k,$$
    so images in $\p^2$ of the lines under $\pi$ are either points onto which lines in the sixer are contracted, lines between a pair of these points, or conics passing through five of these six points.
    
    Graph automorphisms in $\operatorname{Aut}(\mathcal{G}_{\rm{VI}})$, being permutations of lines that preserve the intersection configuration, then correspond to permutations of $\operatorname{EX}$ induced by Cremona isometries of $\Pic(\overline{\mathcal{M}}_w)$, which form the finite Weyl group $W(E_6)$ \cite{dolgachevclassicalAG}.
    This can be described in terms of the elements of $\Pic(\overline{\mathcal{M}}_w)$,
    $$\alpha_1 = \E_6-\E_5, \quad \alpha_2=\E_5-\E_4, \quad \alpha_3=\E_4-\E_3, \quad \alpha_4=\E_3-\E_2, \quad \alpha_5=\E_2-\E_1, \quad \alpha_6 = \h -\E_4-\E_5-\E_6,$$
    which play the role of simple roots of the $E_6$ root system.
    The enumeration of these simple roots corresponds to the enumeration of the $E_6$ Dynkin diagram given in \Cref{fig:dynkinE6}.
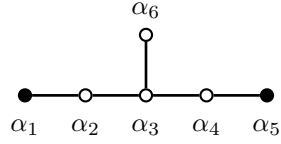
\begin{figure}[htb]
    \centering
            \begin{tikzpicture}[elt/.style={circle,draw=black!100,fill=black!100,thick, inner sep=0pt,minimum size=1.6mm},aff/.style={circle,draw=black!100,thick, inner sep=0pt,minimum size=1.6mm},scale=0.8]
 		\path 	(-2,0) 	node 	(a1) [elt, label={[xshift=0pt, yshift = -21 pt] $\alpha_1$} ] {}
                    (-1,0) 	node 	(a2) [aff, label={[xshift=0pt, yshift = -21 pt] $\alpha_2$} ] {}
                    (0,0) 	node 	(a3) [aff, label={[xshift=0pt, yshift = -21 pt] $\alpha_3$} ] {}
                    (1,0) 	node 	(a4) [aff, label={[xshift=0pt, yshift = -21 pt] $\alpha_4$} ] {}
                    (2,0) 	node 	(a5) [elt, label={[xshift=0pt, yshift = -21 pt] $\alpha_5$} ] {}
                    (0,1) 	node 	(a6) [aff, label={[xshift=0pt, yshift = +0 pt] $\alpha_6$} ] {}
 		       ;
 		\draw [black,line width=1pt ] (a1) -- (a2) -- (a3) -- (a4) -- (a5);
 		\draw [black,line width=1pt ] (a6) -- (a3);
 	\end{tikzpicture}
   \caption{Dynkin diagram $E_6$ with $D_4$ subdiagram indicated by unfilled nodes.}
    \label{fig:dynkinE6}
\end{figure}
    The reflections in these simple roots are defined as lattice automorphisms of $\Pic(\overline{\mathcal{M}}_w)$ by 
    \begin{equation} \label{eq:reflectionformulapic}
        r_{\alpha_i} (\F) = \F + (\F\cdot \alpha_i)\alpha_i,
    \end{equation}
    and generate the finite Weyl group of type $E_6$, so 
    \begin{equation*}
        \operatorname{Aut}(\overline{\mathcal{G}}_{\rm{VI}}) \cong \langle r_{\alpha_1},\dots r_{\alpha_6}\rangle \cong W(E_6),
    \end{equation*}
    where $\overline{\mathcal{G}}_{\rm{VI}}$ indicates the graph obtained from $\mathcal{G}_{\rm{VI}}$ by forgetting about the colouring.
    
    Then $\operatorname{Fix}_{\operatorname{Aut}(\mathcal{G}_{\rm{VI}})}(\mathcal{G}_{\rm{VI}}^\infty)$ corresponds to a subgroup of $W(E_6)$, which we compute as follows. 
    The choice of sixer can be made such that the lines at infinity correspond to the following three elements of $\operatorname{EX}\subset \Pic(\overline{\mathcal{M}}_w)$:
    $$ \mathcal{L}^{\infty}_{1} =  \E_1, \quad \mathcal{L}^{\infty}_{2} = \h - \E_1-\E_6, \quad \mathcal{L}^{\infty}_{3} = 2\h - \E_1-\E_2-\E_3-\E_4-\E_5,$$
    where $\mathcal{L}^{\infty}_{1}+\mathcal{L}^{\infty}_{2}+\mathcal{L}^{\infty}_{3}=-\mathcal{K}_{\overline{\mathcal{M}}_w}$. 
    Explicitly, the choice of sixer can be made, for example, such that $\E_1,\dots,\E_6$ correspond to $L_{25}, L_{18},L_{20},L_{22},L_{24},L_9$.
    
    Then $\operatorname{Fix}_{\operatorname{Aut}(\mathcal{G}_{\rm{VI}})}(\mathcal{G}_{\rm{VI}}^\infty)$ is given by elements of $W(E_6)$ that fix pointwise the subset $\{\mathcal{L}^{\infty}_{1},\mathcal{L}^{\infty}_{2},\mathcal{L}^{\infty}_{3}\}$.
    To show that this subgroup is isomorphic to $W(D_4)$, we first note that the subset of $\Z \alpha_1 + \dots + \Z \alpha_6$ orthogonal to $\{\mathcal{L}^{\infty}_{1},\mathcal{L}^{\infty}_{2},\mathcal{L}^{\infty}_{3}\}$ is spanned by     
    $$\beta_1 = \alpha_6=\h-\E_4-\E_5-\E_6, \quad \beta_2 =\alpha_2=\E_5-\E_4, \quad \beta_3 =\alpha_3=\E_4-\E_3, \quad \beta_4=\alpha_4= \E_3-\E_2.$$
    This is isomorphic as a lattice to the root lattice of the $D_4$ root system, with $(\beta_i\cdot\beta_j)=-B_{i,j}$, $i,j=1,\dots,4$, given by the Cartan matrix  
    $$ B = \begin{bmatrix}
 2 & 0 & -1 & 0 \\
 0 & 2 & -1 & 0 \\
 -1 & -1 & 2 & -1 \\
 0 & 0 & -1 & 2 
    \end{bmatrix}.$$
    The choice of $\beta_1,\beta_2,\beta_3,\beta_4$ corresponds to the $D_4$ subdiagram in \Cref{fig:dynkinE6}.
    The reflections in these roots, defined by the formula \eqref{eq:reflectionformulapic}, generates the subgroup
    $$\langle r_{\beta_1},r_{\beta_2},r_{\beta_3},r_{\beta_4}\rangle \cong W(D_4),$$
    all elements of which fix pointwise $\{\mathcal{L}^{\infty}_{1},\mathcal{L}^{\infty}_{2},\mathcal{L}^{\infty}_{3}\}$.

To show that this exhausts $\operatorname{Fix}_{\operatorname{Aut}(\mathcal{G}_{\rm{VI}})}(\mathcal{G}_{\rm{VI}}^\infty)$, we consider the action of $W(E_6)$ on 45 triangles of lines on $\overline{\mathcal{M}}_w$ \cite{dolgachevclassicalAG}.
The order of $W(E_6)$ is $51,840$ and it acts transitively on the set of triangles, so the setwise stabiliser of a given triangle has order $51,840/45 =1152$. 
Vertices of any triangle are also acted on transitively by $W(E_6)$, so the pointwise stabiliser of a triangle has order $1152/6=192$, which is the order of $W(D_4)$.

\end{proof}

\begin{remark} \label{rem:combinatorialremarkPVI}
    
    The affine lines are naturally partitioned into subsets $\{L_1,
    \dots, L_8\}$, $\{L_9,
    \dots, L_{16}\}$, and $\{L_{17},
    \dots, L_{24}\}$ according to which line at infinity they intersect.  
    By definition, the combinatorial monodromy must respect this partition and
    we observe that $W(D_4)$ acts faithfully when restricted to any of these three sets of lines.
\end{remark}

The Dynkin diagram automorphisms also induce, via their nontrivial action on $\mathcal{M}\to\mathscr{W}_{\rm{VI}}$, permutations of lines, and some of these move the lines at infinity.
\begin{proposition}
    The action of on $\mathcal{M}\to \mathscr{W}_{\rm{VI}}$ of $\Aut(D_4^{(1)})\cong \mathfrak{S}_4$ in \Cref{prop:PVI:dynkinautosunderRH} induces the following automorphisms of the graph $\mathcal{G}_{\rm{VI}}$:
\begin{align*}
	\sigma_{(03)}&= (3\,\, 4)\,(5\,\, 6)\,(7\,\, 8)\,(9\,\, 17)\,(10\,\, 18)\,(11\,\, 20)\,(12\,\, 19)\,(13\,\, 22)\,(14\,\, 21)\,(15\,\, 24)\,(16\,\, 23)\,(26\,\,27),\\
	\sigma_{(34)}&=(1\,\, 17)\,(2\,\, 18)\,(3\,\, 20)\,(4\,\, 19)\,(5\,\, 22)\,(6\,\, 21)\,(7\,\, 24)\,(8\,\, 23)\,(11\,\, 12)\,(13\,\, 14)\,(15\,\, 16)\,(25\,\,27),\\
	\sigma_{(04)}&=(1\,\, 9)\,(2\,\, 10)\,(3\,\, 12)\,(4\,\, 11)\,(5\,\, 14)\,(6\,\, 13)\,(7\,\, 16)\,(8\,\, 15)\,(19\,\, 20)\,(21\,\, 22)\,(23\,\, 24)\,(25\,\,26),\\
	\sigma_{(03)(14)}&=(1\,\, 2)\,(3\,\, 4)\,(9\,\, 11)\,(10\,\, 12)\,(13\,\, 15)\,(14\,\, 16)\,(17\,\, 20)\,(18\,\, 19)\,(21\,\, 23)\,(22\,\, 24),\\
	\sigma_{(04)(13)}&=(1\,\, 3)\,(2\,\, 4)\,(5\,\, 7)\,(6\,\, 8)\,(9\,\, 12)\,(10\,\, 11)\,(13\,\, 15)\,(14\,\, 16)\,(17\,\, 18)\,(19\,\, 20)\,,\\
	\sigma_{(01)(34)}&=(1\,\, 4)\,(2\,\, 3)\,(5\,\, 7)\,(6\,\, 8)\,(9\,\, 10)\,(11\,\, 12)\,(17\,\, 19)\,(18\,\, 20)\,(21\,\, 23)\,(22\,\, 24).
\end{align*}
In particular, the cyclic subgroup of order 3 generated by $\sigma_{(043)}$ acts according to 
$$\sigma_{(043)}=(1\,\,9\,\, 17)\,(2\,\,10\,\, 18)\,(3\,\,11\,\, 19)\,(4\,\,12\,\, 20)\,(5\,\,13\,\, 21)\,(6\,\,14\,\, 22)\,(7\,\,15\,\, 23)\,(8\,\,16\,\, 24)\,( 25\,\,26\,\,27).$$


\end{proposition}
\begin{proof}
    The permutations are induced through a lift of the action of $\Aut(D_4^{(1)})$ on $\mathscr{W}_{\rm{VI}}$ to an action on $\mathscr{U}_{\rm{VI}}$. 
    We define this using the action on $\mathscr{A}_{\rm{VI}}\cong\Theta_{\rm{IV}}$ in \Cref{tab:PVI:symmetry:varsandparams} through the map from $\mathscr{A}_{\rm{VI}}\cong\Theta_{\rm{IV}}$  to $\mathscr{U}_{\rm{VI}}$ in \cref{eq:pvitheta_to_u}.
Explicitly, this is given by 
\begin{align*}
	\sigma_{(03)}&: &&\tilde{u}_1=u_1,&\quad &\tilde{u}_2 = u_3, &\quad 
    &\tilde{u}_3 = u_2,&\quad 
    &\tilde{u}_4 = \frac{u_3 u_4}{u_2},\\
	\sigma_{(34)}&: &&\tilde{u}_1=u_3,&\quad &\tilde{u}_2 = u_2,&\quad 
    &\tilde{u}_3 = u_1,&\quad 
    &\tilde{u}_4 = \frac{u_3 u_4}{u_1},\\
	\sigma_{(04)}&: &&\tilde{u}_1=u_2,&\quad &\tilde{u}_2 =u_1,&\quad 
    &\tilde{u}_3 =u_3,&\quad 
    &\tilde{u}_4 = u_4,\\
	\sigma_{(03)(14)}&: &&\tilde{u}_1=u_1,&\quad &\tilde{u}_2 = - \frac{u_1}{u_4},&\quad 
    &\tilde{u}_3 = - \frac{u_1 u_2}{u_3 u_4},&\quad 
    &\tilde{u}_4 = - \frac{u_1}{u_2},\\
	\sigma_{(04)(13)}&: &&\tilde{u}_1=- \frac{u_2}{u_4},&\quad 
    &\tilde{u}_2 = - \frac{u_1}{u_4},&\quad 
    &\tilde{u}_3 = u_3,&\quad 
    &\tilde{u}_4 = \frac{1}{u_4},\\
	\sigma_{(01)(34)}&: &&\tilde{u}_1=-\frac{u_2}{u_4},&\quad &\tilde{u}_2 = u_2,&\quad 
    &\tilde{u}_3 = - \frac{u_1 u_2}{u_3 u_4}, &\quad 
    &\tilde{u}_4 = - \frac{u_2}{u_1}.
\end{align*}
        The induced permutations of lines are derived from this and the action on $x_i$ in \Cref{tab:PVI:symmetry:ws} by direct computation.
\end{proof}

\subsection{Algebraic monodromy}

\label{subsec:algmonodromyPVI}
We now proceed to a second description of the monodromy of the family $\mathcal{M} \to \mathscr{W}_{\rm{VI}}$, which is as the Galois group of a field extension coming from the incidence variety of lines on the monodromy surface, in the same spirit as the Galois groups of enumerative problems formulated in \cite{harris1979}.

\begin{definition}\label{def:coveringpvi}
    The incidence variety of lines on $\mathcal{M}_w$ is  
    \begin{equation*}
        \Gamma \xrightarrow{\rho} \mathscr{W}_{\rm{VI}}, \quad  \Gamma = \left\{ (w,\ell) \in \mathscr{W}_{\rm{VI}}\times \mathbb{G}(1,3) ~|~ \ell\subset \overline{\mathcal{M}}_w, \,\,\ell\not\subset \overline{\mathcal{M}_w}\setminus\mathcal{M}_w \right\},
    \end{equation*}
    where $\mathbb{G}(1,3)$ is the Grassmannian of projective lines in $\p^3$.
\end{definition}
Then $\rho$ gives a branched covering of $\mathscr{W}_{\rm{VI}}$, and letting $\Gamma^{\operatorname{ns}}$ be the restriction of $\Gamma$ over $\mathscr{W}_{\rm{VI}}\setminus\mathscr{W}_{\rm{VI}}^{\operatorname{sing}}$, we have a covering $\rho : \Gamma^{\operatorname{ns}} \to \mathscr{W}_{\rm{VI}}\setminus\mathscr{W}_{\rm{VI}}^{\operatorname{sing}}$.
The fibre $\Gamma_w$ over $w\in\mathscr{W}_{\rm{VI}}\setminus \mathscr{W}_{
\rm{VI}}^{\operatorname{sing}}$ consists of 24 points, corresponding to the lines $L_1,\dots,L_{24}$ on $\mathcal{M}_w$. 

To define the algebraic monodromy of the family $\mathcal{M} \to \mathscr{W}_{\rm{VI}}$, we construct a field extension $K_{\Gamma}$ of $\C(\mathscr{W}_{\rm{VI}})$ that can be thought of as the smallest field extension necessary to write all lines rationally. 
We define this field through the incidence variety of lines $\Gamma$, which, as we will see, is not irreducible.
\begin{definition} \label{def:KgammaPVI}
    For any irreducible component $\Gamma_r$ of $\Gamma$, pulling back the function field of $\mathscr{W}_{\rm{VI}}$ via $\rho|_{\Gamma_{r}}$ we get an algebraic field extension $\C(\mathscr{W}_{\rm{VI}})\cong\rho|_{\Gamma_{r}}^*\C(\mathscr{W}_{\rm{VI}})\subset \C(\Gamma_r)$.
    We define the field extension $K_{\Gamma}$ of $\C(\mathscr{W}_{\rm{VI}})$ to be the smallest subfield in the algebraic closure of $\C(\mathscr{W}_{\rm{VI}})$ containing all of the normal closures of the function fields $\C(\Gamma_r)$ (their {compositum}).
\end{definition}
We will later see that $K_{\Gamma}$ is itself normal in the algebraic closure of $\C(\mathscr{W}_{\rm{VI}})$.

 \begin{definition}
 Consider (the normal closure of) the algebraic field extension $\C(\mathscr{W}_{\rm{VI}}) \subset K_{\Gamma}$.
   The \emph{algebraic monodromy} of the family $\mathcal{M}\to \mathscr{W}_{\rm{VI}}$ is the Galois group of this field extension:
    \begin{equation*}
        \operatorname{Gal}\left( K_{\Gamma}/\C (\mathscr{W}_{\rm{VI}})\right).
    \end{equation*}
\end{definition}

The following proposition shows that the algebraic monodromy also forms $W(D_4)$.
\begin{proposition}\label{prop:algebraicmonodromyPVI} 
    The field extension $\C(\mathscr{W}_{\rm{VI}})\subset \C(\mathscr{U}_{\rm{VI}})$, defined in \Cref{lem:PvimorphismUtoW}, is a Galois extension and the algebraic monodromy of $\mathcal{M}\to\mathscr{W}_{\rm{VI}}$ is 
    \begin{equation*}
\operatorname{Gal}\left( K_{\Gamma}/\C (\mathscr{W}_{\rm{VI}})\right)\cong \operatorname{Gal}\left(\C(\mathscr{U}_{\rm{VI}})/\C(\mathscr{W}_{\rm{VI}})\right)\cong W(D_4).
    \end{equation*}
    Here, the first isomorphism comes from a $\C(\mathscr{W}_{\rm{VI}})$-linear isomorphism of fields 
    $K_{\Gamma} \to \C(\mathscr{U}_{\rm{VI}})$. 
    The second isomorphism comes from the action of $W(D_4)$ on $\C(\mathscr{U}_{\rm{VI}})$ given by 
    \begin{equation} \label{pvi:actionofD4onU}
\begin{aligned}
    r_1 &: &&u_1\to u_1, \quad 
    &&u_2\to u_2, \quad 
    &&u_3\to u_3, \quad 
    &&u_4 \to \frac{u_1 u_2}{u_3 u_4}, \\
    r_2 &: &&u_1\to \frac{u_4}{u_2}, \quad 
    &&u_2\to \frac{u_4}{u_1}, \quad 
    &&u_3\to \frac{u_3 u_4}{u_1 u_2}, \quad 
    &&u_4\to u_4. \\
    r_3 &: &&u_1\to \frac{u_3}{u_2}, \quad 
    &&u_2\to \frac{u_3}{u_1}, \quad 
    &&u_3 \to u_3, \quad 
    &&u_4\to \frac{u_3 u_4}{u_1 u_2}, \\
    r_4 &: &&u_1 \to u_1, \quad 
    &&u_2\to \frac{u_1}{u_3}, \quad 
    &&u_3\to \frac{u_1}{u_2}, \quad
    &&u_4 \to u_4,
\end{aligned}
\end{equation}
which keeps $(w_1,w_2,w_3,w_4)$ fixed and induces the same permutations of lines as in \Cref{prop:combinatorialmonodromyPVI}.
\end{proposition}

\begin{proof}
    We first establish the $\C(\mathscr{W}_{\rm{VI}})$-linear isomorphism of fields between $K_{\Gamma}$ and $\C(\mathscr{U}_{\rm{VI}})$. 
    To do this, we describe the incidence variety using (dual) Pl\"ucker coordinates on $\mathbb{G}(1,3)$ in the following way.
For a projective line in $\p^3$ defined as the intersection of two hyperplanes
\begin{equation*}
    \ell = V( a_0 X_0+a_1X_1+a_2X_2+a_3X_3, b_0 X_0+b_1X_1+b_2X_2+b_3X_3),
\end{equation*}
the Pl\"ucker coordinates of $\ell$ are $[P_{0,1}:P_{0,2}:P_{0,3}:P_{1,2}:P_{2,3}:P_{3:1}]\in V(P_{0,1}P_{2,3}+ P_{0,3}P_{1,2}+P_{0,2}P_{1,3})\subset \p^5$, defined by 
\begin{equation*}
    P_{i,j}=a_i b_j - a_j b_i, \quad i,j=0,\dots,3.
\end{equation*}

We work with the incidence variety as an algebraic set in $\mathscr{W}_{\rm{VI}}\times \p^5$ via Pl\"ucker coordinates of lines.
These can be computed from the expressions for lines in \eqref{eq:linesPVI} and are given in \Cref{tab:PVI:plucker:1to24}.
In particular, by the last three columns of the table,
\begin{equation*}
    P_{2,3}P_{3,1}P_{1,2}=0
\end{equation*}
on $\Gamma$. Corresponding to the three components of $\Gamma$ defined by $P_{2,3}=0$, $P_{3,1}=0$ and $P_{1,2}=0$,
are respectively the sets of lines
 $\{L_1,\ldots, L_8\}$, $\{L_9,\ldots, L_{16}\}$ and $\{L_{17},\ldots, L_{24}\}$.
Each component is a variety in $\mathscr{W}_{\rm{VI}}\times \p^5$ of dimension 4.

The lines $L_1,\dots,L_8$ correspond to points in the hypersurface $P_{2,3}=0$ with $P_{3,1}\neq0$, so to study this part of $\Gamma$ we restrict to the affine chart with coordinates $(p_{0,1},p_{0,2},p_{0,3},p_{1,2})\in\C^4$, 
$$[P_{0,1}:P_{0,2}:P_{0,3}:P_{1,2}:P_{2,3}:P_{3:1}] = [p_{0,1}:p_{0,2}:p_{0,3}:p_{1,2}:0:1].$$
The condition that $\ell\subset \overline{\mathcal{M}}_w$ and also the Pl\"ucker relation is written in these coordinates as
\begin{equation} \label{eq:pluckereqs1}
    \begin{gathered}
        p_{0,3} p_{0,2}^2+p_{1,2} p_{0,2}^2+p_{0,3}^2 p_{1,2} =0, \\
        w_3 p_{0,3} p_{1,2}^2 p_{0,2}-w_2 p_{0,3}^2 p_{1,2}^2-p_{0,1} p_{0,3} p_{0,2}^2-2 p_{0,1} p_{1,2} p_{0,2}^2=0,\\
        w_1 p_{0,3}^2 p_{1,2} p_{0,2}-w_3 p_{0,1} p_{0,3} p_{1,2} p_{0,2}+w_4 p_{0,3}^2 p_{1,2}^2+p_{0,1}^2 p_{0,2}^2+p_{0,3}^2 p_{0,2}^2=0,\\
        p_{0,2}+p_{0,3} p_{1,2} = 0.
    \end{gathered}
\end{equation}
All lines $L_1\dots,L_8$ have $p_{0,1},p_{0,2},p_{0,3},p_{1,2}\neq 0$, for generic $w$, and under this assumption the equations \eqref{eq:pluckereqs1} reduce, letting $p_{1,2}=-b$, to a single equation $\Lambda_1(b)=0$, where $\Lambda_1(b)$ is the degree eight polynomial
\begin{equation} \label{pvi:Lambda1}
    \Lambda_1(b)= b^8-w_1b^7 +w_4 b^6 + (w_1- w_2 w_3) b^5+ (w_2^2 +w_3^2-2w_4-2) b^4 +(w_1- w_2 w_3)b^3+ w_4 b^2- w_1 b+1.
\end{equation}
The Pl\"ucker coordinates of lines $L_9,\dots,L_{16}$ are given by the solutions of equations that similarly reduce to $\Lambda_2(b)=0$, where 
\begin{equation} \label{pvi:Lambda2}
    \Lambda_2(b)= b^8-w_2b^7 +w_4b^6 + (w_2 - w_1 w_3) b^5+(w_1^2+ w_3^2-2w_4-2)b^4+(w_2- w_1 w_3)b^3+ w_4 b^2 - w_2 b+1,
\end{equation}
while $L_{17},\dots,L_{24}$ lead to $\Lambda_3(b)=0$, where 
\begin{equation} \label{pvi:Lambda3}
    \Lambda_3(b)= b^8-w_3b^7 + w_4 b^6+ (w_3  - w_1 w_2) b^5+( w_1^2+ w_2^2-2  w_4-2)b^4 +(  w_3-w_1 w_2)b^3+ w_4 b^2 - w_3b+1.
\end{equation}

Recalling the definition of $b_k$, $1\leq k\leq 24$, in equation \eqref{eq:defbpvi}, the three polynomials \eqref{pvi:Lambda1}, \eqref{pvi:Lambda2}, \eqref{pvi:Lambda3} split over $\C(\mathscr{U}_{\rm{VI}})$, with roots $b_1,\dots,b_8$ for $\Lambda_1$,  $b_9,\dots,b_{16}$ for $\Lambda_2$, and $b_{17},\dots,b_{24}$ for $\Lambda_3$.
Therefore the irreducible components $\Gamma_1,\Gamma_2$, and $\Gamma_3$, corresponding to $L_1,\dots,L_8$, $L_9,\dots,L_{16}$, and  $L_{17},\dots,L_{24}$ respectively, have function fields whose normal closures are $\C(b_1,\dots,b_8)$, $\C(b_7,\dots,b_{16})$, and 
$\C(b_{17},\dots,b_{24})$ respectively.
This means that $K_{\Gamma}\cong\mathbb{C}(b_1,\dots, b_{24}) \cong \C(\mathscr{U}_{\rm{VI}})$.
    Further, since $\C(\mathscr{U}_{\rm{VI}
    })$ is the splitting field of the polynomial $\Lambda_1(b)\Lambda_2(b)\Lambda_3(b)$ over $\mathbb{C}(\mathscr{W}_{\rm{VI}})$, it is a Galois extension of $\mathbb{C}(\mathscr{W}_{\rm{VI}})$.

\begingroup

\setlength{\tabcolsep}{15pt} 
\renewcommand{\arraystretch}{1.75} 

\begin{table}[h]
    \begin{equation*}
    \begin{array}{c||c|c|c|c|c|c|}
             & P_{0,1}   & P_{0,2}    & P_{0,3}   & P_{1,2}   &  P_{2,3}   & P_{3,1}        \\
        \hline\hline 
        L_1  & \frac{u_1}{u_2}+u_3+\frac{u_3 u_4}{u_1 u_2}+\frac{1}{u_4}     & -\frac{u_3^2}{u_2^2}-1     & -\frac{u_2}{u_3}-\frac{u_3}{u_2}   & \frac{u_3}{u_2}  & 0   & -1           \\
        \hline
        L_2  &  \frac{u_2}{u_1}+\frac{u_1 u_2}{u_3 u_4}+u_4+\frac{1}{u_3}     & -\frac{u_2^2}{u_3^2}-1     & -\frac{u_2}{u_3}-\frac{u_3}{u_2}   & \frac{u_2}{u_3}  & 0   & -1           \\
        \hline
        L_3  & \frac{u_2}{u_1}+u_3+\frac{u_3 u_4}{u_1 u_2}+u_4     & -\frac{u_3^2 u_4^2}{u_1^2}-1  & -\frac{u_1}{u_3 u_4}-\frac{u_3 u_4}{u_1}  & \frac{u_3 u_4}{u_1}  & 0   & -1           \\
        \hline
        L_4  & \frac{u_1}{u_2}+\frac{u_2 u_1}{u_3 u_4}+\frac{1}{u_3}+\frac{1}{u_4}    &-\frac{u_1^2}{u_3^2 u_4^2}-1  & -\frac{u_1}{u_3 u_4}-\frac{u_3 u_4}{u_1}  & \frac{u_1}{u_3 u_4}  & 0   & -1           \\
        \hline
        L_5  & \frac{u_1}{u_2}+\frac{u_2 u_1}{u_3 u_4}+u_3+u_4    &-u_1^2-1  & -u_1-\frac{1}{u_1}  & u_1  & 0   & -1           \\
        \hline
        L_6  & \frac{u_2}{u_1}+\frac{u_3 u_4}{u_1 u_2}+\frac{1}{u_3}+\frac{1}{u_4}    &-\frac{1}{u_1^2}-1  & -u_1-\frac{1}{u_1}  & \frac{1}{u_1}  & 0   & -1           \\
        \hline
        L_7  & \frac{u_2}{u_1}+\frac{u_1 u_2}{u_3 u_4}+u_3+\frac{1}{u_4}    &-\frac{u_2^2}{u_4^2}-1  & -\frac{u_2}{u_4}-\frac{u_4}{u_2}  & \frac{u_2}{u_4}  & 0   & -1           \\
        \hline
        L_8  & \frac{u_1}{u_2}+\frac{u_3 u_4}{u_1 u_2}+u_4+\frac{1}{u_3}    &-\frac{u_4^2}{u_2^2}-1  & -\frac{u_2}{u_4}-\frac{u_4}{u_2}  & \frac{u_4}{u_2}  & 0   & -1           \\
        \hline\hline 
        L_9  & -\frac{u_3^2}{u_1^2}-1    &\frac{u_2}{u_1}+u_3+\frac{u_3 u_4}{u_1 u_2}+\frac{1}{u_4}  & -\frac{u_1}{u_3}-\frac{u_3}{u_1} & -\frac{u_3}{u_1}  & 1   & 0           \\
        \hline
        L_{10}  & -\frac{u_1^2}{u_3^2}-1 & \frac{u_1}{u_2}+\frac{u_2 u_1}{u_3 u_4}+u_4+\frac{1}{u_3} & -\frac{u_1}{u_3}-\frac{u_3}{u_1}  & -\frac{u_1}{u_3}  & 1   & 0           \\
        \hline
        L_{11} &-\frac{u_2^2}{u_3^2 u_4^2}-1    &\frac{u_2}{u_1}+\frac{u_1 u_2}{u_3 u_4}+\frac{1}{u_3}+\frac{1}{u_4}  & -\frac{u_2}{u_3 u_4}-\frac{u_3 u_4}{u_2} &-\frac{u_2}{u_3 u_4}  & 1   & 0           \\
        \hline
        L_{12}  & -\frac{u_3^2 u_4^2}{u_2^2}-1  &\frac{u_1}{u_2}+u_3+\frac{u_3 u_4}{u_1 u_2}+u_4  & -\frac{u_2}{u_3 u_4}-\frac{u_3 u_4}{u_2} & -\frac{u_3 u_4}{u_2} & 1   & 0          \\
        \hline
        L_{13}  & -\frac{1}{u_2^2}-1  &\frac{u_1}{u_2}+\frac{u_3 u_4}{u_1 u_2}+\frac{1}{u_3}+\frac{1}{u_4}  & -u_2-\frac{1}{u_2} & -\frac{1}{u_2}  & 1   & 0           \\
        \hline
        L_{14}  & -u_2^2-1  & \frac{u_2}{u_1}+\frac{u_1 u_2}{u_3 u_4}+u_3+u_4 & -u_2-\frac{1}{u_2} & -u_2  & 1   & 0           \\
        \hline
        L_{15}  & -\frac{u_4^2}{u_1^2}-1  & \frac{u_2}{u_1}+\frac{u_3 u_4}{u_1 u_2}+u_4+\frac{1}{u_3} & -\frac{u_1}{u_4}-\frac{u_4}{u_1} & -\frac{u_4}{u_1}  & 1   & 0           \\
        \hline
        L_{16}  & -\frac{u_1^2}{u_4^2}-1  & \frac{u_1}{u_2}+\frac{u_2 u_1}{u_3 u_4}+u_3+\frac{1}{u_4} & -\frac{u_1}{u_4}-\frac{u_4}{u_1}  & -\frac{u_1}{u_4} & 1   & 0          \\
        \hline \hline
        L_{17}  & -\frac{u_2^2}{u_1^2}-1  & -\frac{u_1}{u_2}-\frac{u_2}{u_1} &\frac{u_2}{u_3 u_4}+u_2+\frac{u_3}{u_1}+\frac{u_4}{u_1} & 0 & -1   & \frac{u_2}{u_1}         \\
        \hline
        L_{18}  &-\frac{u_1^2}{u_2^2}-1  & -\frac{u_1}{u_2}-\frac{u_2}{u_1} & \frac{u_1}{u_3}+\frac{u_1}{u_4}+\frac{u_3 u_4}{u_2}+\frac{1}{u_2}  & 0 & -1   & \frac{u_1}{u_2}         \\
        \hline
        L_{19}  & -u_4^2-1 & -u_4-\frac{1}{u_4} & \frac{u_1}{u_3}+u_2+\frac{u_3 u_4}{u_2}+\frac{u_4}{u_1} & 0 & -1   & u_4          \\
        \hline
        L_{20}  & -\frac{1}{u_4^2}-1  & -u_4-\frac{1}{u_4} & \frac{u_1}{u_4}+\frac{1}{u_2}+\frac{u_2}{u_3 u_4}+\frac{u_3}{u_1} & 0 & -1   & \frac{1}{u_4}          \\
        \hline
        L_{21}  & -u_3^2-1 & -u_3-\frac{1}{u_3} & \frac{u_1}{u_4}+u_2+\frac{u_3 u_4}{u_2}+\frac{u_3}{u_1} & 0 & -1   & u_3        \\
        \hline
        L_{22}  & -\frac{1}{u_3^2}-1   & -u_3-\frac{1}{u_3} & \frac{u_1}{u_3}+\frac{1}{u_2}+\frac{u_2}{u_3 u_4}+\frac{u_4}{u_1} & 0 & -1   & \frac{1}{u_3}          \\
        \hline
        L_{23}  & -\frac{u_1^2 u_2^2}{u_3^2 u_4^2}-1 & -\frac{u_1 u_2}{u_3 u_4}-\frac{u_3 u_4}{u_1 u_2} & \frac{u_1}{u_3}+\frac{u_1}{u_4}+u_2+\frac{u_2}{u_3 u_4}  & 0 & -1   & \frac{u_1 u_2}{u_3 u_4}          \\
        \hline
        L_{24}  & -\frac{u_3^2 u_4^2}{u_1^2 u_2^2}-1  & -\frac{u_1 u_2}{u_3 u_4}-\frac{u_3 u_4}{u_1 u_2} & \frac{u_4 u_3}{u_2}+\frac{u_3}{u_1}+\frac{u_4}{u_1}+\frac{1}{u_2} & 0 & -1   & \frac{u_3 u_4}{u_1 u_2}          \\
        \hline
    \end{array}
    \end{equation*}
    \caption{Pl\"ucker coordinates of lines $L_1,\dots,L_{24}$ on the monodromy surface for $\pain{VI}$}
    \label{tab:PVI:plucker:1to24}
\end{table}
\endgroup

Elements of $\operatorname{Gal}\left(\C(\mathscr{U}_{\rm{VI}})/\C(\mathscr{W}_{\rm{VI}})\right)$ give permutations of $b_1,\dots,b_{24}$, which in turn induce permutations of the lines $L_1,\dots,L_{24}$ which necessarily preserve incidence relations among them, yielding a homomorphism 
\begin{equation} \label{eq:galtofix}
\operatorname{Gal}\left(\C(\mathscr{U}_{\rm{VI}})/\C(\mathscr{W}_{\rm{VI}})\right)\rightarrow \operatorname{Fix}_{\operatorname{Aut}(\mathcal{G}_{\rm{VI}})}(\mathcal{G}_{\rm{VI}}^\infty).
\end{equation}
The kernel is trivial, since any nontrivial element of $\operatorname{Gal}\left(\C(\mathscr{U}_{\rm{VI}})/\C(\mathscr{W}_{\rm{VI}})\right)$ acts nontrivially on $\mathbb{C}(\mathscr{U}_{\rm{VI}})=\C(b_1,\dots,b_{24})$. 
On the other hand, the actions of $r_1,r_2,r_3,r_4$ on $\C(\mathscr{U}_{\rm{VI}})$ in \Cref{pvi:actionofD4onU} give elements of $\operatorname{Gal}\left(\C(\mathscr{U}_{\rm{VI}})/\C(\mathscr{W}_{\rm{VI}})\right)$
which induce the permutations of lines as claimed, and we have the required isomorphism \eqref{eq:galtofix}.
\end{proof}

\begin{remark}\label{rem:rational}
    Proposition \ref{prop:algebraicmonodromyPVI} implies that the field extension from $\C(\mathscr{W}_{\rm{VI}})$ to $\C(\mathscr{U}_{\rm{VI}})$ is the minimal one required to write all lines on $\mathcal{M}_w$ rationally.
\end{remark}

It is striking that the Galois group $\operatorname{Gal}\left(\C(\mathscr{U}_{
\rm{VI}})/\C(\mathscr{W}_{\rm{VI}})\right)$ is not larger, since in the proof of \Cref{prop:algebraicmonodromyPVI}, it was shown to be the splitting field  of three palindromic octic polynomials, and the Galois group of a general palindromic octic polynomial is $C_2\wr \mathfrak{S}_4$ and already strictly contains $W(D_4)$. 
This can be explained as follows.
Firstly, we note that the splitting field of the three polynomials over $\C(\mathscr{W}_{\rm{VI}})$ is already exhausted by that of any one of them, and it suffices to consider $\Lambda_i(b)\in \C[w_1,w_2,w_3,w_4,b]$ for a single $i\in\{1,2,3\}$. 
Combinatorially, this is explained by the faithful action of $W(D_4)$ restricted to any one of the three subsets of eight lines in  \Cref{rem:combinatorialremarkPVI}.

Indeed, let $K_i$ be the the splitting field of $\Lambda_i(b)$ over $\C(\mathscr{W}_{\rm{VI}})$ and $G_i$ the corresponding Galois group for $i=1,2,3$.
Writing $K=K_1\cap K_2\cap K_3$, so that 
\begin{equation*}
\operatorname{Gal}\left(\C(\mathscr{U}_{\rm{VI}})/\C(\mathscr{W}_{\rm{VI}})\right)\cong \{(g_1,g_2,g_3)\in G_1\times G_2\times G_3: g_1|_K=g_2|_K=g_3|_K\},
\end{equation*}
we have the following result.
\begin{proposition} \label{prop:galoisprojection}
For $i=1,2,3$, the homomorphism obtained by projecting to  $G_i$,
    \begin{equation}\label{eq:projectionG1}
 \operatorname{Gal}\left(\C(\mathscr{U}_{\rm{VI}})/\C(\mathscr{W}_{\rm{VI}})\right)\rightarrow G_i, \quad (g_1,g_2,g_3)\mapsto g_i,
\end{equation}
is an isomorphism.
\end{proposition}
\begin{proof}
We use the explicit descriptions of $K_1 = \C(b_1,\dots,b_{8})$, $K_2 = \C(b_9,\dots,b_{16})$, $K_3 = \C(b_{17},\dots,b_{24})$ in \Cref{eq:defbpvi}, 
to show that
\begin{equation} \label{eq:squaresinL}
    \mathbb{C}(b_1^2,\ldots, b_{24}^2)\subseteq K.
\end{equation}
To see this, note that, for example,
\begin{equation*}
   u_1^2=b_5^2=b_9 b_{15} b_{11} b_{16}=b_{18}b_{19} b_{21}b_{23},
\end{equation*}
so that $b_5^2\in K$. Similarly,
\begin{equation*}
    u_2^2=b_2 b_3 b_5 b_7=b_{13}^2=b_{17} b_{19} b_{21} b_{23},
\end{equation*}
so that $b_{13}^2\in K$. The others follow analogously.

We show that the fact \eqref{eq:squaresinL} implies  the claimed isomorphism in the example of $G_1$. The others are analogous. 
Since $u_2^2\in K$, if $(g_1,g_2,g_3)$ is mapped to the identity in $G_1$, i.e. $g_1=1$, it follows that $g_1(u_2^2)=u_2^2$ and therefore $g_2(u_2^2)=u_2^2$. But $g_2$ has to map $u_2$ to another root of $\Lambda_2(b)$, so that necessarily $g_2(u_2)=u_2$. Similarly we see that $g_2(b_k)=b_k$ for $9\leq k\leq 16$ so that $g_2=1$. A similar argument shows that $g_3=1$ so that the kernel of the surjective homomorphism \eqref{eq:projectionG1} is indeed trivial.\end{proof}




In light of \Cref{prop:galoisprojection}, we can consider $\operatorname{Gal}\left(\C(\mathscr{U}_{\rm{VI}})/\C(\mathscr{W}_{\rm{VI}})\right)$ as the splitting field of $\Lambda_1(b)$ over $\C(\mathscr{W}_{\rm{VI}})$, without loss of generality.
What distinguishes $\Lambda_1(b)$ from a generic palindromic octic polynomial is the following.
\begin{proposition}
    The polynomial $\Lambda_1(b)$ is a palindromic octic polynomial with $\Lambda_1(1)\Lambda_1(-1)$ being a square in the base field, 
    so the Galois group   of its splitting field over $\mathscr{W}_{\rm{VI}}$ is $W(D_4)$.
\end{proposition}
\begin{proof}
    Since $\Lambda_1(b)$ is palindromic, its roots come in pairs of reciprocals, 
    \begin{equation*}
      \{r_1,r_1^{-1}\}=\{b_1,b_2\},\quad
      \{r_2,r_2^{-1}\}=\{b_3,b_4\},\quad
      \{r_3,r_3^{-1}\}=\{b_5,b_6\},\quad
      \{r_4,r_4^{-1}\}=\{b_7,b_8\}.
    \end{equation*}
    The values of $\Lambda_1(b)$ at $b=1$ and $b=-1$ are
    \begin{equation*}
    \begin{aligned}
        \Lambda_1(+1) &= \frac{ (r_1-1)^2(r_2-1)^2(r_3-1)^2(r_4-1)^2}{r_1r_2r_3r_4}  = (w_2 - w_3)^2,\\
      \Lambda_1(-1) &=  \frac{ (r_1+1)^2(r_2+1)^2(r_3+1)^2(r_4+1)^2}{r_1r_2r_3r_4}  = (w_2 + w_3)^2.
  \end{aligned}
    \end{equation*}
Therefore their product is given by 
    $$\Lambda_1(1) \Lambda_1(-1) = (r_1-r_1^{-1})^2(r_2-r_2^{-1})^2(r_3-r_3^{-1})^2(r_4-r_4^{-1})^2 = (w_2^2-w_3^2)^2,$$
    which is a square in the base field.
    Making the definite choice $r_i=b_{2i-1}$, $1\leq i\leq 4$, we get 
    $$(r_1-r_1^{-1})(r_2-r_2^{-1})(r_3-r_3^{-1})(r_4-r_4^{-1}) = w_2^2-w_3^2.$$ 
    Thus, the Galois group of the splitting field of $\Lambda_1(b)$ over $\mathscr{W}_{\rm{VI}}$ consists of permutations of the four pairs of roots, wreathed with even numbers of inversions, recovering the description of $W(D_4)$ as the group of signed permutations of four elements with an even number of sign changes.
\end{proof}

\subsection{Analytic monodromy} \label{subsec:analyticmonodromyPVI}
~
We now give our third formulation of $W(D_4)$ as the monodromy of the family $\overline{\mathcal{M}}\to \mathscr{W}_{\rm{VI}}$, which is analytic and concerns deformation of lines on $\mathcal{M}_w$ along loops.

Consider the nonsingular locus of the family of projective surfaces $\overline{\mathcal{M}}\to\mathscr{W}_{\rm{VI}}$ with the complex analytic topology, i.e.
$\mathscr{W}_{\rm{VI}}\setminus \mathscr{W}_{\rm{VI}}^{\operatorname{sing}}$. Recall the covering map $\rho : \Gamma^{\operatorname{ns}} \to \mathscr{W}_{\rm{VI}}\setminus\mathscr{W}_{\rm{VI}}^{\operatorname{sing}}$, see Definition \ref{def:coveringpvi}. Take any $w_{*}\in \mathscr{W}_{\rm{VI}}\setminus \mathscr{W}_{\rm{VI}}^{\operatorname{sing}}$ and enumerate the points in the fibre above $w_{*}$ by $l_1,\ldots,l_{27}$. Any loop $\gamma\in \pi_1( \mathscr{W}_{\rm{VI}}\setminus \mathscr{W}_{\rm{VI}}^{\operatorname{sing}} ; w_{*})$,  induces a permutation of the points in the fibre, and thus an element $s_\gamma$ of $\mathfrak{S}_{27}$. This gives a homomorphism
\begin{equation}\label{eq:monodromyhom}
    \pi_1( \mathscr{W}_{\rm{VI}}\setminus \mathscr{W}_{\rm{VI}}^{\operatorname{sing}} ; w_{*})\rightarrow \mathfrak{S}_{27}, \qquad \gamma\mapsto s_\gamma.
\end{equation}
\begin{definition}
    The \emph{analytic monodromy} of the family $\mathcal{M}\to \mathscr{W}_{\rm{VI}}$ is the subgroup of $\mathfrak{S}_{27}$ given by the image of the homomorphism \eqref{eq:monodromyhom},
    defined up to overall conjugation, within $\mathfrak{S}_{27}$, induced by changing enumeration of lines and choice of base-point.
\end{definition}

The following shows that the analytic monodromy matches with the combinatorial and algebraic ones established above.

\begin{proposition} \label{prop:analyticmonodromyPVI}
    The analytic monodromy of the family $\mathcal{M}\to \mathscr{W}_{\rm{VI}}$ is $W(D_4)$.
\end{proposition}
\begin{proof}
    Take a $w_* \in \mathscr{W}_{\rm{VI}}\setminus \mathscr{W}_{\rm{VI}}^{\operatorname{sing}}$.
    Make a choice of $u_*\in \mathscr{U}_{\rm{VI}}\setminus \mathscr{U}_{\rm{VI}}^{\operatorname{sing}}$ lying above $w_*$ under the morphism $\mathscr{U}_{\rm{VI}}\to \mathscr{W}_{\rm{VI}}$ from \Cref{lem:PvimorphismUtoW}, which determines a numbering of the lines on $\overline{\mathcal{M}}_{w_*}$ according to equations \eqref{eq:linesPVI}.
    Further, take $\theta_* \in \Theta_{\rm{VI}}$ lying above $u_*$ under the map $\theta \mapsto u$ as defined by equation \eqref{eq:pvitheta_to_u}.
    Denote the preimage of $\mathscr{U}_{\rm{VI}}^{\operatorname{sing}}$ under this map by $\Theta_{\rm{VI}}^{\operatorname{sing}}$. It is a countable union of hyperplanes and its complement in $\Theta_{\rm{VI}}$ is given by
    \begin{equation*}
        \Theta_{\rm{VI}}\setminus \Theta_{\rm{VI}}^{\operatorname{sing}}=\{\theta\in \mathbb{C}^4:2\theta_j\notin\mathbb{Z}\text{ for $j\in \{0,t,1,\infty\}$ and }
\epsilon_1\theta_0+\epsilon_2\theta_t+\epsilon_3\theta_1+\epsilon_4\theta_\infty\notin\mathbb{Z}\text{ for $\epsilon\in \{\pm 1\}^4$}\}.
    \end{equation*}
In particular, $\Theta_{\rm{VI}}\setminus \Theta_{\rm{VI}}^{\operatorname{sing}}$ is path-connected.
    
    Thus, for each of the generators $r_i$, $i=1,\dots,4$, of $W(D_4)$, there exists a path $\gamma_i$ in $\Theta_{\rm{VI}}\setminus \Theta_{\rm{VI}}^{\operatorname{sing}}$ with starting point $\theta_*$ and end point $r_i(\theta_*)$, where we recall the action of $r_i$ on $\Theta_{\rm{VI}}$ defined through the parameter correspondence \eqref{eq:rootvarstotheta} and Table \ref{tab:PVI:symmetry:varsandparams}.
    
    Then, each $\gamma_i$ descends to a path in $\mathscr{U}_{\rm{VI}}\setminus \mathscr{U}_{\rm{VI}}^{\operatorname{sing}}$ from $u_*$ to $r_i(u_*)$, where the latter is defined in \eqref{pvi:actionofD4onU}.
    It descends down further to a loop in $\mathscr{W}_{\rm{VI}}\setminus\mathscr{W}_{\rm{VI}}^{\operatorname{sing}}$, based at $w_*$, that permutes the lines according to the action of $W(D_4)$ in \Cref{prop:combinatorialmonodromyPVI}.

    This shows that the analytic monodromy of $\mathcal{M}\to \mathscr{W}_{\rm{VI}}$ contains the combinatorial monodromy. On the other hand, since the incidence relations among lines are preserved under continuation along paths that avoid the singular locus $\mathscr{W}_{\rm{VI}}^{\operatorname{sing}}$, the analytic monodromy is a subset of the combinatorial monodromy and the two must therefore coincide.
\end{proof}

\Cref{thm:PVIRHsymmetries} and \Cref{prop:PVI:dynkinautosunderRH} establish \Cref{mainthm:conjugatedsymmetries},   \Cref{prop:oblomkov} and \Cref{cor:modulispacePVI} establish \Cref{mainthm:categoryandmodulispace},
and \Cref{prop:combinatorialmonodromyPVI,prop:algebraicmonodromyPVI,prop:analyticmonodromyPVI} establish \Cref{mainthm:mon} in the case of Painlev\'e-VI.

\section{The fourth Painlev\'e equation}\label{sec:PIV}

We next consider the fourth Painlev\'e equation $\pain{IV}$, in the form
\begin{equation*} 
\Pfour:\quad y_{tt}=\displaystyle\frac{1}{2y}y_t^2+ \frac{3}{2} y^3 + 4\, t\,y^2+2(t^2+1-2\theta_\infty)y-\frac{8\theta_0^2}{y}.
\end{equation*}
Denote the space of parameters by 
\begin{equation*}
    \Theta_{\rm{IV}} = \left\{ \theta=(\theta_0, \theta_{\infty}) \in \C^2\right\}.
\end{equation*}
We again use root variables, the space of which in this case we denote by
\begin{equation*}
    \mathscr{A}_{\rm IV} = \left\{ a = (a_0,a_1,a_2) \in \C^3~|~ a_0+a_1+a_2=1\right\},
\end{equation*}
and we relate $\Theta_{\rm{IV}}$ and $\mathscr{A}_{\rm{IV}}$ according to 
\begin{equation} \label{eq:thetatorootvarsPIV}
    \theta_0 = \frac{a_1}{2}, \quad \theta_{\infty} = 1- \frac{a_1}{2}-a_2.
\end{equation}
We work with the Hamiltonian form of $\pain{IV}$ given by
\begin{equation} \label{eq:hamPIV}
    \left\{ 
    \begin{aligned}
        f_t &= + \frac{\partial H_{\rm{IV}}}{\partial g} = f ( 2g - f - 2t) - 2 a_1, \\
        g_t &= - \frac{\partial H_{\rm{IV}}}{\partial f} = g ( 2f - g + 2 t) +2a_2, 
    \end{aligned}
    \right.
\end{equation}
with Hamiltonian 
\begin{equation*}
    H_{\rm{IV}} = f g^2 - f^2 g - 2 t f g - 2 a_1 g - 2 a_2 f.
\end{equation*}
This coincides with that provided by Okamoto in \cite{okamotopolynomialhamiltonians}, up to scaling and notation, and eliminating $g(t)$ from \eqref{eq:hamPIV} leads to $\pain{IV}$ for $f(t)$.

\subsection{B\"acklund transformations}
In the case of $\pain{IV}$, the B\"acklund transformation symmetries form the extended affine Weyl group of type $A_2^{(1)}$, which we write as
\begin{equation*}
    \widetilde{W}(A_2^{(1)}) := W(A_2^{(1)}) \rtimes \Aut (A_2^{(1)}) \cong \left\langle r_0, r_1, r_2 \right\rangle \rtimes \left\langle  \sigma_{(12)}, \sigma_{(02)} \right\rangle,  
    \quad
       \raisebox{-20pt}{
\begin{tikzpicture}[scale=.4,elt/.style={circle,draw=black!100,thick, inner sep=0pt,minimum size=1.5mm}]
		\path 	(-1,-1) 	node 	(a1) [elt ] {}
		        ( 0,.3) 	node  	(a0) [elt  ] {}
		        ( 1,-1) 	node  	(a2) [elt  ] {};
		\node at ($(a1.west) + (-.3,+0.0)$) 	{ \tiny ${1}$};
		\node at ($(a2.east) + (+.3,+0.0)$) 	{ \tiny ${2}$};
		\node at ($(a0.north) + (+0,+0.3)$) 	{ \tiny ${0}$};

		\draw [black,line width=1pt ] (a1) -- (a2) -- (a0) -- (a1);
		\node at ($(d2.east) + (+0,-2)$) 	{ \small ${A_2^{(1)}}$};

	\end{tikzpicture} }
\end{equation*}
where here and below we denote by $\sigma_{s}\in \Aut(A_2^{(1)})$ the element corresponding to the permutation $s$ of the indices $\{0,1,2\}$ of nodes of the $A_2^{(1)}$ Dynkin diagram, in cycle notation.
The B\"acklund transformations of system \eqref{eq:hamPIV} corresponding to these generators are given in \Cref{tab:PIV:symmetry:varsandparams}.
We also give the action of the generator $\sigma_{(012)}$ of the subgroup, isomorphic to $\Z/3\Z$, of rotational Dynkin diagram automorphisms.
There is also the symmetry $\vartheta$, which leaves the equation invariant without changing parameters.

\begingroup

\setlength{\tabcolsep}{15pt} 
\renewcommand{\arraystretch}{1.5} 

\begin{table}[h]
\makebox[\textwidth][c]{%
$   \begin{array}{c||c|c||c|c|c||c|c||c|}
        w       & \tilde{f}     & \tilde{g}     & \tilde{a}_0   & \tilde{a}_1   & \tilde{a}_2  & \tilde{\theta}_0 & \tilde{\theta}_{\infty} & \tilde{t} \\
        \hline\hline 
        r_0    & f - \frac{2 a_0}{f-g+2t}    & g + \frac{2 a_0}{f-g+2t}  & -a_0   & a_1+a_0   & a_2 + a_0  & \frac{\theta _{\infty }}{2}+\frac{\theta _0}{2} & \frac{3 \theta _0}{2}-\frac{\theta _{\infty }}{2} & t  \\
        \hline
        r_1    & f    & g - \frac{2 a_1}{f}  & a_0+a_1   & -a_1   & a_1+a_2  & - \theta_0 & \theta_{\infty} & t  \\\hline
        r_2    & f + \frac{2a_2}{g}   & g  & a_0+a_2   & a_1+a_2   & -a_2 & \frac{1}{2}+\frac{\theta _0}{2}-\frac{\theta _{\infty }}{2} & \frac{3}{2}-\frac{3 \theta _0}{2}-\frac{\theta _{\infty }}{2}  & t  \\    
        \hline \hline
        \sigma_{(12)}   & i g    & i f  & a_0   & a_2   & a_1 & \frac{1}{2} -\frac{\theta _0}{2} -\frac{\theta _{\infty }}{2} &  \frac{1}{2}-\frac{3 \theta _0}{2}+\frac{\theta_{\infty}}{2}  & -i t \\
        \sigma_{(02)}   &  -i f    & i (g-f-2t) & a_2   & a_1   & a_0  & \theta_0 & 1-\theta _{\infty }  & -i t \\    
        \sigma_{(01)}   & i (g-f-2t)    & i g  & a_1   & a_0   & a_2  & -\frac{\theta _0}{2} +\frac{\theta _{\infty }}{2}& \frac{3 \theta _0}{2}+\frac{\theta _{\infty }}{2}  & i t \\    
        \sigma_{(012)}   & - g    & f- g + 2 t  & a_1   & a_2   & a_0  &\frac{1}{2}-\frac{\theta _0}{2} -\frac{\theta _{\infty }}{2}& \frac{1}{2}+\frac{3 \theta _0}{2}-\frac{\theta _{\infty }}{2}  & t \\
        \hline \hline  
        \vartheta & - f    & - g  & a_0   & a_1   & a_2  & \theta_0 & \theta_{\infty}  & -  t 
    \end{array}$    
    }
    \caption{Symmetries of $\pain{IV}$ on variables $(f,g)$, parameters $\theta$, root variables $a$ and independent variable $t$.}
    \label{tab:PIV:symmetry:varsandparams}
\end{table}
\endgroup

The translation part of $\widetilde{W}(A_2^{(1)})$, corresponding to the weight lattice of $A_2$, is generated by 
    $$T_1=\sigma_{(021)}r_1r_2 : (a_0,a_1,a_2)\mapsto(a_0-1,a_1+1,a_2), \quad(\theta_0,\theta_{\infty})\mapsto(\theta_0+\tfrac{1}{2},\theta_{\infty}-\tfrac{1}{2}),$$
    and 
    $$T_2= \sigma_{(012)}r_2r_1 
    : (a_0,a_1,a_2)\mapsto(a_0-1,a_1,a_2+1),\quad(\theta_0,\theta_{\infty})\mapsto(\theta_0,\theta_{\infty}-1).$$
    The translation part of $W(A_2^{(1)})$, corresponding to the root lattice of $A_2$ and not requiring extension by Dynkin diagram automorphisms, is generated by 
    $$T_1^2T_2^{-1}: (a_0,a_1,a_2) \mapsto (a_0-1,a_1+2,a_2-1), \quad (\theta_0,\theta_{\infty})\mapsto(\theta_0+1,\theta_{\infty}),$$
    and 
    $$T_1^{-1}T_2^{2} : (a_0,a_1,a_2) \mapsto (a_0-1,a_1-1,a_2+2), \quad (\theta_0,\theta_{\infty})\mapsto(\theta_0-\tfrac{1}{2},\theta_{\infty}-\tfrac{3}{2}).$$

    \begin{remark}
        The transformations given in \Cref{tab:PIV:symmetry:varsandparams} define an action of a larger group than $\widetilde{W}(A_2^{(1)})$, and there is nontrivial interaction between involutions in $\Aut(A_2^{(1)})$ and the symmetry $\vartheta$. The subgroup generated by $W(A_2^{(1)})$ and $\sigma_{(012)}$ commutes with $\vartheta$, and the the remaining relations are
$$\sigma_{(01)}^2=\sigma_{(12)}^2=\sigma_{(02)}^2 = \vartheta, \quad \sigma_{(01)}\sigma_{(02)}=\sigma_{(021)},\quad \sigma_{(01)}\sigma_{(12)}=\sigma_{(012)}, \quad \sigma_{(12)}\sigma_{(02)}=\vartheta  \sigma_{(012)}.$$
    \end{remark}
    
\subsection{Initial value space}

In this case we have a family of Sakai surfaces 
\begin{equation*}
    \overline{\mathcal{X}} \to \mathscr{T}_{\rm{IV}}\times \mathscr{A}_{\rm{IV}},
\end{equation*}
where $\mathscr{T}_{\rm{IV}} = \C$ is the independent variable space of $\pain{IV}$.
The surfaces can be constructed explicitly through a sequence of eight point blowups from $\p^1 \times\p^1$ which can be derived through the usual procedure applied to the system \eqref{eq:hamPIV}. 
The fibre $\overline{\mathcal{X}}_{t,a}$ is a Sakai surface of surface type $E_6^{(1)}$, with anticanonical divisor $D_{t,a}$ whose irreducible components intersect according to the $E_6^{(1)}$ Dynkin diagram.
Removing the support of this anticanonical divisor from each fibre gives a family $\mathcal{X}\to\mathscr{T}_{\rm{IV}}\times \mathscr{A}_{\rm{IV}}$.

\begin{definition}[Initial value space for $\pain{IV}$]
    The initial value space at $t\in\mathscr{T}_{\rm{IV}}$ for $\pain{IV}$ with parameters $a\in \mathscr{A}_{\rm{IV}}$ is the fibre $\mathcal{X}_{t,a}$ of the family $\mathcal{X}_a \to \mathscr{T}_{\rm{IV}}$.
\end{definition}

For each $w \in \widetilde{W}(A_2^{(1)})$, the corresponding B\"acklund transformation of $\pain{IV}$ gives an automorphism of the family $\mathcal{X}$, and in particular we have an automorphism
$$w : \mathscr{T}_{\rm{IV}}\times \mathscr{A}_{\rm{IV}} \to \mathscr{T}_{\rm{IV}}\times \mathscr{A}_{\rm{IV}},\quad (t,a)\mapsto (\tilde{t},\tilde{a}),$$
and an isomorphism
$$w : \mathcal{X}_{t,a}\to \mathcal{X}_{\tilde{t},\tilde{a}}.$$

\subsection{Associated linear problem}

The Painlev\'e IV equation governs isomonodromy deformations of a rank two system of linear ODEs with one Fuchsian singularity and one irregular singularity of Poincar\'e rank two.
We consider this in the form
\begin{subequations}\label{eq:pivlinearsystem}
    \begin{align}
    Y_z&=AY, & A&=zA_1+A_0+z^{-1}A_{-1},\label{eq:pivlinearsystem1}\\
    Y_t&=BY, & B&=zB_1+B_0,
\end{align}
\end{subequations}
where, for fixed $(\theta_0,\theta_{\infty})\in \Theta_{\rm{IV}}$, the matrices $A_1,A_0,A_{-1}\in \mathfrak{sl}_2(\mathbb{C})$ are such that
		\begin{equation}\label{eq:pivlinconditions}
A_1=\sigma_3:=\begin{bmatrix}
        1 & 0\\
        0 & -1
    \end{bmatrix},\quad \operatorname{Spec}(A_{-1})=\{+\theta_0,-\theta_0\},\quad |A|=-(z+t)^2+2\theta_\infty+\mathcal{O}(z^{-1}) \quad (z\rightarrow \infty),
	\end{equation}
 and $B_1,B_0\in \mathfrak{sl}_2(\mathbb{C})$ are given by
\begin{equation*}
    B_1=\sigma_3,\quad B_0=A_0-t\,\sigma_3.
\end{equation*}
Introduce coordinates $\{f,g,k\}$ on the space of matrices $A$ satisfying the conditions \eqref{eq:pivlinconditions} according to 
\begin{equation*}
 A_{12}=k\left(1-\frac{f}{2z}\right),\quad A_{11}|_{z=f/2}=g-\frac{2 \theta _0}{f}-\frac{f}{2}-t.  
\end{equation*}
Then the compatibility of the pair \eqref{eq:pivlinearsystem} yields
the first-order system \eqref{eq:hamPIV} with parameters $a$ given in terms of $\theta_0,\theta_{\infty}$ by \Cref{eq:thetatorootvarsPIV}, as well as 
\begin{equation*}
   \frac{k_t}{k}=-(f+2t). 
\end{equation*}

\subsection{Derivation of monodromy surface}

The monodromy surface for $\pain{IV}$ is formed of equivalence classes of generalised monodromy data for the linear system \eqref{eq:pivlinearsystem1}, i.e. both usual monodromy as well as Stokes data associated with the irregular singularity.
To derive this, begin with the unique formal solution of the linear system \eqref{eq:pivlinearsystem1} of the form
\begin{equation}\label{eq:formalsolPIV}
    Y_{\operatorname{form}}(z)=P(z) e^{(\frac{1}{2}z^2+t z)\sigma_3}z^{-\theta_\infty \sigma_3},
\end{equation}
where $P(z)$ is a power series around $z=\infty$,
\begin{equation} \label{eq:matrixPasymptoticPIV}
    P(z)=I+\sum_{n=1}^\infty z^{-n}U_n.
\end{equation}
We consider the function $z^{-\theta_\infty \sigma_3}$ as a single-valued function on the universal covering space $\widetilde{\mathbb{C}^*}$ and correspondingly define Stokes sectors
\begin{equation*}
    \Sigma_k=\left\{\left|\arg z- \frac{(k-1)\pi}{2}\right|<\frac{\pi}{4}\right\}\subseteq \widetilde{\mathbb{C}^*} \qquad (k\in\mathbb{Z}),
\end{equation*}
within which $e^{\frac{1}{2}z^2+t z}$ is alternately exponentially small or large as $z\rightarrow \infty$.

For any $k\in\mathbb{Z}$, there exists a unique solution $Y_k$ of the linear problem that satisfies
\begin{equation*}
    Y_k(z)\sim Y_{\operatorname{form}}(z)\qquad (z\in \Sigma_k\cup \Sigma_{k+1}, z\rightarrow \infty).
\end{equation*}
The Stokes phenomenon is then embodied by the following relation among these solutions,
\begin{equation}\label{eq:pivstokes}
    Y_{k+1}(z)=Y_k(z)S_k,
\end{equation}
with Stokes matrices
\begin{equation*}
    S_k=\begin{bmatrix}
        1 & 0\\
        s_k & 1
    \end{bmatrix}\,\,\, \text{if $k$ even},\qquad
    S_k=\begin{bmatrix}
        1 & s_k\\
        0 & 1
    \end{bmatrix}\,\,\, \text{if $k$ odd},
\end{equation*}
for some $s_k\in\mathbb{C}$, called a Stokes multiplier, for $k\in\mathbb{Z}$.

By the definition of $Y_k(z)$,
\begin{equation*}
    Y_{k+4}(e^{2\pi i}z)=Y_k(z)D_\infty^{-1},\qquad D_\infty=e^{2\pi i \theta_\infty \sigma_3},
\end{equation*}
and therefore
\begin{equation}\label{eq:pivstokescyclic}
    S_{k+4}=D_\infty S_k D_\infty^{-1}\qquad (k\in\mathbb{Z}).
\end{equation}
Furthermore, it follows from this and \eqref{eq:pivstokes} that the monodromy of $Y_k(z)$ around $z=\infty$ is given by
\begin{equation*}
      Y_k(e^{-2\pi i}z)=Y_k(z)M_\infty^{(k)},\qquad M_\infty^{(k)}=S_kS_{k+1}S_{k+2}S_{k+3} D_\infty.  
\end{equation*}
The corresponding monodromy matrix $M_\infty^{(k)}$ satisfies
\begin{equation}\label{eq:pivmonodromycyclic}
    M_\infty^{(k+1)}=S_k^{-1} M_\infty^{(k)} S_k,\qquad
    M_\infty^{(k+4)}=D_\infty M_\infty^{(k)} D_\infty^{-1},
\end{equation}
by equation \eqref{eq:pivstokescyclic}.

Since $z=0$ and $z=\infty$ are the only singularities of the linear system \eqref{eq:pivlinearsystem}, monodromy of $Y_k(z)$ around $z=0$ in the clockwise direction is given by $M_\infty^{(k)}$ as well. 
Furthermore, as $\operatorname{Spec}(A_{-1})=\{+\theta_0,-\theta_0\}$, see equation \eqref{eq:pivlinconditions}, this means that the eigenvalues of $M_\infty^{(k)}$ must be $e^{\pm 2\pi i \theta_0}$ and therefore
\begin{equation}\label{eq:pivtracecondition}
    \operatorname{Tr}M_\infty^{(k)}=2\cos 2\pi \theta_0,
\end{equation}
for $k\in\mathbb{Z}$. 
Due to \eqref{eq:pivmonodromycyclic}, equation \eqref{eq:pivtracecondition} is equivalent to one and the same condition on the Stokes multipliers for all $k\in\mathbb{Z}$, which is the vanishing of the quartic polynomial
\begin{equation} \label{eq:pivquarticpolynomiali1}
    i_v:=s_1 s_2 s_3 s_4  +s_1 s_2+s_1 s_4+s_3 s_4+v_{\infty }^{-6} s_2 s_3 - v_{\infty }^{-3}(v_0+v_0^{-1})+ v_{\infty }^{-6}+1,
\end{equation}
in which we have used the notation
\begin{equation}\label{eq:defiv0vinfPIV}
    v= (v_0,v_{\infty}),\qquad v_0=e^{2 i \pi \theta_0},\quad v_\infty=e^{\frac{2 i \pi \theta_\infty}{3}}.
\end{equation}
This quartic first appeared in \cite[Eq. (2.27)]{fokaszhou}.

We then interpret the space of monodromy data
 for the linear problem \eqref{eq:pivlinearsystem} with $\theta_0,\theta_{\infty}$ fixed as the affine variety\footnote{It is possible to also include connection data in the derivation of the monodromy surface, see Remark \ref{rem:includeconnection}.}
\begin{equation} \label{eq:monodromyspaceMpiv}
    M_v = \operatorname{Spec} \C [s_1,s_2,s_3,s_4]/(i_v).
\end{equation}

The freedom of scaling the gauge factor $k$ by an arbitrary nonzero complex number, $k\mapsto c k$, is equivalent to rescaling the coefficient matrix of the linear system \eqref{eq:pivlinearsystem} by
\begin{equation*}
    A\mapsto c^{\frac{1}{2}\sigma_3} A c^{-\frac{1}{2}\sigma_3},
\end{equation*}
which correspondingly rescales canonical solutions near infinity by
\begin{equation*}
    Y_k(z)\mapsto c^{\frac{1}{2}\sigma_3} Y_k(z) c^{-\frac{1}{2}\sigma_3}\qquad (k\in\mathbb{Z}),
\end{equation*}
and consequently 
\begin{equation*}
    S_k\mapsto c^{\frac{1}{2}\sigma_3} S_k c^{-\frac{1}{2}\sigma_3}\qquad (k\in\mathbb{Z}),
\end{equation*}
which is equivalent to
\begin{equation} \label{eq:pivscalingactiononsj}
 (s_1,s_2,s_3,s_4)\mapsto  (c^{-1}s_1,c\, s_2,c^{-1}s_3,c\, s_4).
\end{equation}

To describe the space of equivalence classes of monodromy data up to this scaling, analogously to the character variety associated with $\pain{VI}$, we take the GIT quotient of the variety \eqref{eq:monodromyspaceMpiv} by this action of $\C^*$.
We consider $R=\C [s_1,s_2,s_3,s_4]/(i_v)$, with $i_v$ in \eqref{eq:pivquarticpolynomiali1}, and consider the action of $\C^*$ on $R$ as in \eqref{eq:pivscalingactiononsj}.
The ring of invariants is generated by 
\begin{equation}\label{eq:ytosPIV}
    y_1 = s_1 s_4 ,\quad y_2 = s_1 s_2, \quad y_3 = s_3 s_4, \quad y_4 = s_2 s_3.
\end{equation}
Introducing the notation
\begin{equation} \label{eq:xtoyPIV}
    x_1 = v_{\infty}^{-2}(1+v_{\infty}^6 y_1), \quad x_2 = v_{\infty}^2( 1 + y_2 ),\quad x_3 = v_{\infty}^2( 1 + y_3 ), \quad x_4 = v_{\infty}^{-2}(1+y_4),
\end{equation}
we arrive at the following description of the ring of invariants: 
$$ R^{\C^*} =  \C[x_1,x_2,x_3,x_4]/I_{w},$$ where the relations among $x_1,x_2,x_3,x_4$ correspond to the ideal 
\begin{equation}\label{piv:idealIw}
    I_{w} = (x_2x_3+x_1+x_4-w_1, \,\,x_1 x_4+x_2+x_3-w_2),
\end{equation} 
in which
\begin{equation*} 
    w_1 =  v_{\infty}\left(v_0  + v_0^{-1}\right)+ {v_{\infty}^{-2}}, \qquad 
w_2 = v_{\infty}^{-1}\left(v_0+v_0^{-1}\right)+v_{\infty}^2.
\end{equation*}

The affine variety $\operatorname{Spec} R^{\C^*}$ has canonical projective completion being a smooth Segre surface. 
It was found in \cite{JMR} as an affine surface embedded in $\mathbb{A}^6$, described by
$$z_1+z_2+z_3+z_4+z_5+z_6=0,\quad z_4=1,\quad z_3 z_4 - \lambda_1 z_1 z_2 = 0,  \quad z_5 z_6 - \lambda_2 z_1z_2=0,$$
see \cite[Sec. 4.2.2]{JMR}. This is related to $\Spec R^{\C^*}$ above by 
$$x_1=  c\,  v_{\infty }z_2 +v_{\infty }^{-2}, \quad 
x_2 = c\, v_{\infty}^{-1}{z}_5+v_{\infty }^2, \quad
x_3 = c\, v_{\infty}^{-1} z_6 +v_{\infty }^2, \quad 
x_4 = c\, v_{\infty } z_1 + v_{\infty }^{-2},$$
in which $c = v_{\infty }^3+v_{\infty }^{-3}-v_0-v_0^{-1}$, with the parameters $\lambda_1,\lambda_2$ related to $v_0,v_{\infty}$ according to 
$$\lambda_1 =  v_{\infty }^6-v_{\infty }^3(v_0+v_0^{-1})+1,\quad \lambda_2 = v_{\infty }^{-6}-v_{\infty }^{-3} (v_0 + v_0^{-1})+1.$$
Note that in the same paper, monodromy surfaces for other Painlev\'e equations were realised as affine Segre surfaces, and this one corresponding to $\pain{IV}$ is distinguished in having smooth projective completion with a rectangle of lines at infinity.

\begin{definition} \label{def:monodromysurfacePIV}
    The monodromy surface of $\pain{IV}$ is the embedded affine Segre surface
    \begin{equation*}
    \begin{aligned}
        \mathcal{M}_w &= \Spec \C[x_1,x_2,x_3,x_4]/I_{w}, \\  
        I_w &= (x_2x_3+x_1+x_4-w_1, \,\,x_1 x_4+x_2+x_3-w_2),
    \end{aligned}
    \end{equation*}
    where $w \in \mathscr{W}_{\rm{IV}} = \left\{ w=(w_1,w_2) \in \C^2 \right\}$.
    This gives the family
    \begin{equation*}
        \mathcal{M}\to\mathscr{W}_{\rm{IV}},
    \end{equation*}
    with fibre over $w\in\mathscr{W}_{\rm{IV}}$ being $\mathcal{M}_{w}$.
\end{definition}
We regard $\mathcal{M}_w\subset\mathbb{A}^4$ as an embedded affine variety.
Under the embedding 
        \begin{equation*}
        \begin{aligned}
            \mathbb{A}^4 &\to \p^4,\\
            (x_1,x_2,x_3,x_4) &\mapsto [1: x_1:x_2:x_3:x_4],
        \end{aligned}
        \end{equation*}
        its projective completion is the Segre surface $\overline{\mathcal{M}}_w$ given by 
        \begin{equation*}
            \begin{aligned}
            X_2X_3+X_0(X_1+X_4)-w_1 X_0^2 &= 0, \\
            X_1 X_4+ X_0(X_2+X_3)-w_2 X_0^2 &= 0.           
            \end{aligned}
        \end{equation*}
\begin{remark}
We have the following properties of $\mathcal{M}_w$.
    \begin{itemize}
        \item The hyperplane section at infinity of $\overline{\mathcal{M}}_w$ is 
        \begin{equation*}
            X_0 = 0, \quad X_2 X_3 = 0 ,\quad X_1 X_4 = 0,
        \end{equation*}
        which is the union of four lines that pairwise intersect forming a rectangle, and consists only of smooth points of $\overline{\mathcal{M}}_w$.
        \item The singular locus $\mathscr{W}^{\operatorname{sing}}_{\rm{IV}}\subset \mathscr{W}_{\rm{IV}}$, where $\mathcal{M}_w$ is singular, is given by $Q(w)=0$, where
        \begin{equation*} 
    Q(w)= w_1^2 w_2^2 - 4 ( w_1^3+ w_2^3) +18 w_1 w_2-27.
    \end{equation*}
    \end{itemize}
\end{remark}

As an affine variety, the surface $\mathcal{M}_w$ is isomorphic to the affine cubic surface
\begin{equation}\label{eq:affinecubicPIV}
    \begin{aligned}
        \mathcal{C}_w &= \Spec \C[x_1,x_2,x_3]/\tilde{I}_w, \\  
        \tilde{I}_w &= (x_1 x_2 x_3+x_1^2 -w_1 x_1 -x_2-x_3 +w_2),
    \end{aligned}
\end{equation}
which is the model of the monodromy surface for $\pain{IV}$ which appeared in \cite[Sec. 3.7]{putsaito}, but scaled differently, see \Cref{rem:comparisonpivPS} below.
The isomorphism is provided by
\begin{equation} \label{eq:isomsegretocubicPIV}
\begin{aligned}
    \mathcal{M}_w &\to \mathcal{C}_w, \\
    (x_1,x_2,x_3,x_4) &\mapsto (x_1,x_2,x_3). 
\end{aligned}
\end{equation}
The canonical projective completion $\overline{\mathcal{C}}_w\subset \p^3$ of the affine cubic surface is not smooth. 
At infinity lies a triangle of lines, of which two corners form $A_1$ singularities of the projective cubic, see \cite[Table 4]{JMR}.

\begin{remark}\label{rem:comparisonpivPS}
Note that we have chosen a slightly different parametrisation, and our cubic is related to the one in \cite{putsaito}, which we write as 
\begin{equation*}
\tilde{x}_1\tilde{x}_2\tilde{x}_3 + \tilde{x}_1^2 - \left(\tilde{s}_2^2+\tilde{s}_1\tilde{s}_2\right)\tilde{x}_1-\tilde{s}_2^2 \tilde{x}_2 - \tilde{s}_2^2 \tilde{x}_3+\tilde{s}_2^2+\tilde{s}_1\tilde{s}_2^3 = 0, 
\end{equation*}
according to
\begin{equation*}
    \tilde{x}_1= v_{\infty}^{-4} x_1, \quad 
    \tilde{x}_2= v_{\infty}^{-2} x_2, \quad 
    \tilde{x}_3= v_{\infty}^{-2} x_3,
\end{equation*}
with the parameters $\tilde{s}_1,\tilde{s}_2$ related to ours by
\begin{equation*}
    \tilde{s}_1 =v_0+v_0^{-1},\quad 
    \tilde{s}_2 = v_\infty^{-3},
\end{equation*}
where we recall that $v_0$ and $v_\infty$ are defined in \eqref{eq:defiv0vinfPIV}.
We choose our particular normalisation since the coefficients $w_1$ and $w_2$ are invariant under the $W(A_2^{(1)})$ symmetry. 
\end{remark}

\begin{remark} \label{rem:includeconnection}
    In the Riemann-Hilbert theory of $\pain{IV}$, an important role is also played by the canonical solution of the linear problem around $z=0$ and connection matrices between this solution and the canonical solutions around infinity \cite{fokas}. 
    If we extend the space of monodromy data in \eqref{eq:monodromyspaceMpiv} with connection data\footnote{For an example of how connection data are included, see Appendix \ref{app:derivationofmonodromysurfaceFN} where this is done in the derivation of the monodromy surface for the Flaschka-Newell Lax pair of $\Ptwo$.}, we get, after taking the GIT quotient, four additional variables $\{x_5,x_6,x_7,x_8\}$ satisfying
    \begin{equation*}
        x_5= x_1x_3,\quad x_6 = x_3 x_4\quad x_7 =x_1 x_2,\quad x_8 = x_2 x_4.
    \end{equation*}
 After projective completion, the resulting surface in $\p^8$ is also a del Pezzo surface of degree four with a rectangle of lines at infinity, but not anti-canonically embedded.
 Away from infinity it is isomorphic to $\mathcal{M}_w$, and the lines at infinity are contained in a copy of $\p^3$ given by the intersection of five hyperplanes in $\p^8$. 
 These lines at infinity are related to those in $\overline{\mathcal{M}}_w$ by a monomial transformation between copies of $\p^3$, which swaps the four lines with their four points of pairwise intersection.
\end{remark}

\subsection{Riemann-Hilbert map and symmetries}

The Riemann-Hilbert map in this case is 
\begin{equation*}
    \mathcal{X}_{t,a} \xrightarrow{~~\operatorname{RH}_{t,a}~~} \mathcal{M}_{w(a)},    
\end{equation*}
which is a biholomorphism for generic $a$ and $t$ \cite{puttop2013piv}. 
It is defined analogously to the $\pain{VI}$ case outlined in \Cref{subsec:pvi:riemannhilbertmapandsymmetries},  by computing the monodromy data for the linear system with a local solution of $\pain{IV}$ inserted.
The map $a\mapsto w(a)$ from $\mathscr{A}_{\rm{IV}}\cong\Theta_{\rm{IV}}$ to $\mathscr{W}_{\rm{IV}}$ is defined via the intermediate parameters $v_0,v_{\infty}$, and is given explicitly by
\begin{equation} \label{eq:wtothetaandrootvarsPIV}
    \begin{aligned}
    w_1 &= v_{\infty}\left(v_0  + v_0^{-1}\right)+ {v_{\infty}^{-2}}\\
    &=e^{-2\pi i \theta_0 + \frac{2 \pi i}{3}\theta _{\infty }}
    +e^{\frac{-4  \pi i}{3}   \theta _{\infty }}
    +e^{ \frac{2 \pi i}{3}   \theta _{\infty } +2 \pi i \theta _0} \\
    &= e^{\frac{2\pi i }{3} \left(a_0+2a_1\right)}
    +e^{\frac{2  \pi i}{3}   \left(a_0- a_1\right)}
    +e^{\frac{-2\pi i}{3}   \left(2a_0 + a_1\right)},\\
    w_2 &= v_{\infty}^{-1}\left(v_0+v_0^{-1}\right)+v_{\infty}^2\\
    &=e^{ 2 \pi i \theta _0-\frac{2 \pi i}{3}\theta _{\infty }}
    +e^{\frac{4\pi i }{3}   \theta _{\infty }}
    +e^{ \frac{-2  \pi i}{3}  \theta _{\infty } - 2 \pi i \theta _0} \\
    &= e^{\frac{-2 \pi i}{3}   \left(a_0+2 a_1\right)}
    +e^{\frac{2\pi i}{3}  \left(a_1- a_0\right)}
    +e^{\frac{2 \pi i}{3}   \left(2a_0+a_1\right)}.
    \end{aligned}
\end{equation}

We will consider the action of the extended affine Weyl group $\widetilde{W}(A_2^{(1)})$ extended further by the symmetry $\vartheta$, as given in \Cref{tab:PIV:symmetry:varsandparams}, conjugated by the Riemann-Hilbert map, which gives an action of 
$\widetilde{W}(A_2^{(1)})\rtimes \Z/2\Z$
on the family $\mathcal{M} \to \mathscr{W}_{\rm{IV}}$.

\begin{theorem} \label{thm:PIVsymmetriesunderRH}
    For any $g\in W(A_2^{(1)})$, the corresponding symmetry of $\pain{IV}$ conjugates to the identity under the Riemann-Hilbert map.
    The subgroup generated by $\Aut(A_2^{(1)})$ and $\langle\vartheta\rangle \cong \Z/2\Z$ acts nontrivally, as described in \Cref{tab:PIV:symmetry:dynkinautos}.
\end{theorem}

\begingroup
\setlength{\tabcolsep}{15pt} 
\renewcommand{\arraystretch}{1.5} 
\begin{table}[h]
    \begin{equation*}
    \begin{array}{c||c|c|c|c||c|c|}
              & \tilde{x}_1    & \tilde{x}_2     & \tilde{x}_3    & \tilde{x}_4   & \tilde{w}_1   & \tilde{w}_2           \\
        \hline\hline 
        r_0   &   x_1    &  x_2    &  x_3 &    x_4 & w_1   &  w_2             \\
        r_1   &   x_1    &  x_2    &  x_3 &    x_4 & w_1   &  w_2             \\
        r_2   &   x_1    &  x_2    &  x_3 &    x_4 & w_1   &  w_2             \\
        \hline\hline
        \sigma_{(12)}   &   \zeta_3^{-1} x_2    &  \zeta_3 x_4    &  \zeta_3 x_1 &   \zeta_3^{-1} x_3 & \zeta_3^{-1} w_2   & \zeta_3 w_1             \\
        \sigma_{(02)}   & \zeta_3 x_2    &  \zeta_3^{-1} x_4    &  \zeta_3^{-1} x_1 &   \zeta_3 x_3 & \zeta_3 w_2   & \zeta_3^{-1} w_1             \\
        \sigma_{(01)}   & x_3    & x_1     & x_4 & x_2 & w_2   & w_1      
        \\
        \sigma_{(012)}   &  \zeta_3^{-1} x_1  &  \zeta_3 x_2    & \zeta_3 x_3 & \zeta_3^{-1} x_4 & \zeta_3^{-1} w_1   &       \zeta_3 w_2
        \\
        \hline\hline
        \vartheta   &  x_4  &   x_3    & x_2 & x_1 &  w_1   &      w_2
    \end{array}
    \end{equation*}
    \caption{Action of symmetries of $\pain{IV}$ on $\mathcal{M}\rightarrow\mathscr{W}_{\rm{IV}}$ via conjugation by the Riemann-Hilbert map.}
    \label{tab:PIV:symmetry:dynkinautos}
\end{table}
\endgroup

We will prove \Cref{thm:PIVsymmetriesunderRH} in several steps.
Regarding the action on parameters, we first note the following, which can be verified by direct computation.
\begin{lemma}\label{lem:weylactionwPIV}
    For any $g\in \widetilde{W}(A_2^{(1)})\rtimes\Z/2\Z$, the action of $g$ on $\mathscr{A}_{\rm{IV}}\cong\Theta_{\rm{IV}}$ induces via the correspondence \eqref{eq:wtothetaandrootvarsPIV} the transformation of $\mathscr{W}_{\rm{IV}}$ as written in \Cref{tab:PIV:symmetry:dynkinautos}.
    In particular, if $g$ is an element of the affine Weyl group $W(A_2^{(1)})$ then it acts as the identity on $\mathscr{W}_{\rm{IV}}$.
\end{lemma}

Regarding the action on coordinates $x_i$, we first discuss the symmetries that can easily be extended to the linear problem, in the following lemma.
\begin{lemma} \label{lem:PIVsymmetriesunderRHdirect}
The action of $\widetilde{W}(A_2^{(1)})\rtimes \Z/2\Z$ on $\mathcal{M}\to \mathscr{W}_{\rm{IV}}$ obtained through conjugation by the Riemann-Hilbert map satisfies the following.
\begin{enumerate}
    \item The generators of the translation part of $\widetilde{W}(A_2^{(1)})$ act as
    \begin{align*}
        T_1 : (x_1,x_2,x_3,x_4)&\mapsto \left( \zeta_3 x_1, \zeta_3^{-1} x_2, \zeta_3^{-1} x_3, \zeta_3 x_4\right), \quad \left(w_1,w_2\right)\mapsto \left( \zeta_3 w_1, \zeta_3^{-1} w_2\right),\\
        T_2 : (x_1,x_2,x_3,x_4)&\mapsto \left( \zeta_3^{-1}x_1, \zeta_3 x_2,\zeta_3 x_3,\zeta_3^{-1} x_4\right), \quad \left(w_1,w_2\right)\mapsto \left( \zeta_3^{-1} w_1, \zeta_3 w_2\right),
    \end{align*}
    in which $\zeta_3=e^{\frac{2\pi i}{3}}$.
    \item The translation part of $W(A_2^{(1)})$ generated by $T_1^2T_2^{-1}$ and $T_1^{-1}T_2^{2}$ acts trivially.
    \item The element $r_1$ acts trivially.
    \item The generator $\vartheta$ of $\Z/2\Z$ acts as 
    \begin{equation*}
        \vartheta : (x_1,x_2,x_3,x_4)\mapsto \left( x_4,x_3,x_2,x_1\right), \quad \left(w_1,w_2\right)\mapsto \left( w_1, w_2\right),
    \end{equation*}
    \item The element $\sigma_{(01)}r_1r_0$ acts as  
    \begin{equation*}
        \sigma_{(01)}r_1r_0 : (x_1,x_2,x_3,x_4)\mapsto \left( x_3,x_1,x_4,x_2\right), \quad \left(w_1,w_2\right)\mapsto \left( w_2, w_1\right).
    \end{equation*}
\end{enumerate}
\end{lemma}
\begin{proof}
Regarding (1) and (2), the translation part of the extended affine Weyl group $\widetilde{W}(A_2^{(1)})$ can be realised as Schlesinger transformations of the linear problem \eqref{eq:pivlinearsystem1}, as shown in \cite{fokasmugan1988}. In particular, the corresponding Schlesinger transformation for $T_1$ takes the form
\begin{equation*}
    Y(z)\mapsto \widetilde{Y}(z)=R(z)Y(z),\quad A(z)\mapsto \widetilde{A}(z)=R(z)A(z)R(z)^{-1}+R_z(z)R(z)^{-1},
\end{equation*}
where $R(z)$ is a matrix function of the form
\begin{equation*}
    R(z)=z^{\frac{1}{2}}\begin{bmatrix}
        0 & 0\\
        0 & 1
    \end{bmatrix}+z^{-\frac{1}{2}}\begin{bmatrix}
        1 & r_{12}\\
        r_{21} & r_{12}r_{21}
    \end{bmatrix}.
\end{equation*}
Correspondingly, the canonical solutions at infinity are transformed as
\begin{equation*}
    Y_k(z)\mapsto \widetilde{Y}_k(z)=R(z)Y_k(z)\qquad (k\in\mathbb{Z}).
\end{equation*}
This means that the Stokes data are left completely invariant under the transformation, that is, $\tilde{s}_k=s_k$ for $1\leq k\leq 4$. On the other hand, the parameters $v_0$ and $v_\infty$ are transformed as follows,
\begin{equation*}
    \tilde{v}_0=-v_0,\quad \tilde{v}_\infty=\zeta_6^{-1}v_\infty,\quad \zeta_6=e^{\frac{\pi i}{3}}.
\end{equation*}
Correspondingly, we see that
\begin{equation*}
    \widetilde{x}_1=\zeta_3x_1,\quad \widetilde{x}_2=\zeta_3^{-1}x_2,\quad 
     \widetilde{x}_3=\zeta_3^{-1}x_3,\quad \widetilde{x}_4=\zeta_3x_4,\quad \widetilde{w}_1=\zeta_3w_1,\quad \widetilde{w}_2=\zeta_3^{-1}w_2.
\end{equation*}
The action of $T_2$ is worked out similarly.

For the translations that lie in the unextended affine Weyl group $W(A_2^{(1)})\subset \widetilde{W}(A_2^{(1)})$, the corresponding Schlesinger transformations act trivially on equivalence classes of generalised monodromy data. Indeed, one may check directly that $T_1^2T_2^{-1}$ and $T_1^{-1}T_2^2$ leave the $x_k$, $1\leq k\leq 4$ and $(w_1,w_2)$ invariant using the expressions in (1).

Regarding (3), the element $r_1$ leaves $A(z)$ completely invariant, and therefore also the canonical solutions at infinity and hence also the Stokes data. As $v_0$ and $v_\infty$ are also invariant under $r_1$, it follows that $r_1$ acts completely trivially on the $x_k$, $1\leq k\leq 4$ and $(w_1,w_2)$.

Regarding (4) and (5), we consider the transformation
\begin{equation*}
    Y(z)\mapsto \widetilde{Y}(z)=\sigma_1 Y(-iz),\quad A(z)\mapsto \widetilde{A}(z)=-i\sigma_1A(-iz)\sigma_1.
\end{equation*}
This transformation realises the action of $\sigma_{(01)}r_1r_0$ on $A(z)$ with $$k\mapsto \widetilde{k}=\frac{ i\left(q(4 p-q-2 t)+4 \theta _{\infty }\right)}{2 k}=\frac{i \left(f (g-f-2 t)+2 \theta _{\infty }-2 \theta _0\right)}{k}.$$
Correspondingly, we have the following transformation of the formal solution \eqref{eq:formalsolPIV},
\begin{equation*}
   \widetilde{Y}_{\operatorname{form}}(z)=\sigma_1 Y_{\operatorname{form}}(-iz)\sigma_1 e^{\frac{1}{2}i\pi  \theta_\infty\sigma_3}.
\end{equation*}
The canonical solutions at infinity then transform as
\begin{equation*}
   \widetilde{Y}_{k}(z)=\sigma_1 Y_{k-1}(-iz)\sigma_1 e^{\frac{1}{2}i\pi  \theta_\infty\sigma_3}.
\end{equation*}
Therefore, the Stokes matrices transform as
\begin{equation*}
   \widetilde{S}_{k}=e^{-\frac{1}{2}i\pi  \theta_\infty\sigma_3}\sigma_1 S_{k-1}\sigma_1 e^{\frac{1}{2}i\pi  \theta_\infty\sigma_3}.
\end{equation*}
It follows that
\begin{equation*}
        \widetilde{s}_1=v_\infty^{\frac{9}{2}}s_4, \quad 
        \widetilde{s}_2=v_\infty^{\frac{3}{2}}s_1, \quad 
        \widetilde{s}_3=v_\infty^{-\frac{3}{2}}s_2,\quad
        \widetilde{s}_4=v_\infty^{\frac{3}{2}}s_3.
\end{equation*}
We thus obtain
\begin{equation*}
        \widetilde{y}_1=v_\infty^{6}y_3, \quad
        \widetilde{y}_2=v_\infty^{6}y_1,\quad
        \widetilde{y}_3=y_4, \quad
        \widetilde{y}_4=y_2.
\end{equation*}
Combining this with the transformations of the parameters   $\tilde{v}_0=v_0$ and $\tilde{v}_\infty=v_\infty^{-1}$, yields
\begin{equation*}
        \widetilde{x}_1=x_3, \quad
        \widetilde{x}_2=x_1,\quad 
        \widetilde{x}_3=x_4, \quad 
        \widetilde{x}_4=x_2,\quad 
        \widetilde{w}_1=w_2,\quad
         \widetilde{w}_2=w_1.
\end{equation*}
This proves part (5). 
Finally, part (4) follows by applying (5) twice,  since $\vartheta$ corresponds to the transformation
\begin{equation*}
    Y(z)\mapsto\widetilde{Y}(z)=Y(-z),\quad A(z)\mapsto \widetilde{A}(z)=- A(-z),
\end{equation*}
with $k\mapsto \widetilde{k}=-k$, which is obtained by applying the transformation realising $\sigma_{(01)}r_1r_0$ twice, so we have
$(\sigma_{(01)}r_1r_0)^2=\vartheta,$
finishing the proof of the proposition.
\end{proof}

In the next lemma, we show that $r_0$ acts trivially on the monodromy surface, which together with \Cref{lem:PIVsymmetriesunderRHdirect} will suffice to obtain \Cref{thm:PIVsymmetriesunderRH}. To prove that $r_0$ acts trivially, we follow the strategy used in \cite{inabaiwasakisaito} to establish Theorem \ref{thm:PVIRHsymmetries}, namely, by keeping track of the asymptotics of the general solution near a critical point, as a function of monodromy data, under the relevant B\"acklund transformation.
\begin{lemma} \label{lem:PIVsymmetryunderRHasymptotic}
The symmetry $r_0$ acts on $\mathcal{M}$ as the identity when conjugated with the Riemann-Hilbert map.
\end{lemma}
\begin{proof}
  The asymptotics of the general solution of $\Pfour$ are given in \cite{kapaev1996} in terms of the monodromy of a linear problem derived by Kitaev \cite{kitaev85}. The first step in the proof is to recall the relation between this linear problem and \eqref{eq:pivlinearsystem}, which can be found in \cite{kapaevhubert}. Introducing the new independent variable $\xi=\frac{1}{2}z^2$, the transformation
  \begin{equation*}
    Y^{\text{Ki}}(\xi)=\begin{bmatrix}
        \xi^{-\frac{1}{2}} & 0\\
        0 & \xi^{\frac{1}{2}}
    \end{bmatrix}\begin{bmatrix}
        1 & -i\, u\\
        0 & 1
    \end{bmatrix}Y(\tfrac{1}{2}\xi^2),\qquad u=\frac{i\,k}{f-g+2 t},
  \end{equation*}
  gives
  \begin{align*}
      Y_\xi^{\text{Ki}}=A^{\text{Ki}}Y^{\text{Ki}},\\
      Y_t^{\text{Ki}}=B^{\text{Ki}}Y^{\text{Ki}},   
  \end{align*}
with
\begin{align*}
    A^{\text{Ki}}&=(\tfrac{1}{2}\xi^2+(t+u\,v)\xi+\alpha\,\xi^{-1})\sigma_3+i(u\,\xi^2+2\,t\,u+u_t)\sigma_++i(v\,\xi^2+2\,t\,v-v_t)\sigma_-,\\
     B^{\text{Ki}}&=(\tfrac{1}{2}\xi^2+u\,v)\sigma_3+i\,u\,\xi\,\sigma_++i\,v\,\xi\,\sigma_-,
\end{align*}
where $\sigma_+$/$\sigma_-$ denote the strictly upper/lower triangular $2\times 2$ matrices with their only non-trivial entry equal to $1$. This is the linear problem for $\Pfour$ by Kitaev \cite{kitaev85}.

We have the following further formulas relating the variables occurring in both Lax pairs,
\begin{align*}
    \alpha&=2\theta_0-\frac{1}{2},\\
    \beta&=2\theta_0+2\theta_\infty=u_tv-uv_t+2tuv-u^2v^2,\\
    v&=-i\,k^{-1}((f-g+2t)f+2\theta_0-2\theta_\infty),
\end{align*}
and $y^{\text{Ki}}=u\,v$ satisfies $\Pfour$ with parameter values $$\widetilde{\theta}_0=\tfrac{1}{2}(\theta_0+\theta_\infty),\quad 
\widetilde{\theta}_\infty=\tfrac{1}{2}(3\theta_0-\theta_\infty).$$
The last change of parameters coincides with the action of $r_0$ on them and $y^{\text{Ki}}$ is related to $(f,g)$ correspondingly,
\begin{align*}
    &f+\frac{2(\theta_0-\theta_\infty)}{2t+f-g}=y^{\text{Ki}},\\
    &g-f-t=\frac{\theta_0+\theta_\infty}{y^{\text{Ki}}}+\frac{y_t^{\text{Ki}}}{2y^{\text{Ki}}}-\frac{y^{\text{Ki}}}{2}.
\end{align*}

By applying \cite[Th. 4(ii)]{kapaev1996} to $y^{\text{Ki}}$
we obtain the asymptotics of the general solution $y^{\text{Ki}}$ of $\Pfour$ as $t\rightarrow +\infty$ and $t\rightarrow-\infty$ along the real line, and correspondingly for $f$ and $g$. The result is as follows. Under the conditions $ x_1x_3,x_2x_4\notin \mathbb{R}$,
the leading order behaviours of $y^{\text{Ki}}$, $f$ and $g$ as $t\rightarrow \pm\infty$, $t\in\mathbb{R}$, are given by 
\begin{equation}\label{eq:kap_asymptotics}
    \frac{y^{\text{Ki}}}{2t}+\frac{1}{3}\sim R_\pm(t),\qquad
    \frac{f}{2t}+\frac{1}{3}\sim R_\pm(t), \qquad
    \frac{g}{2t}-\frac{1}{3}\sim \zeta_6^{-1}R_\pm(t),
\end{equation}
where $R_\pm(t)$ are the following asymptotically vanishing functions
\begin{equation*}
    R_\pm(t)=c_\pm^{-1}\zeta_6e^{\frac{1}{\sqrt{3}}i\,t^2}(2\sqrt{3}t^2)^{-\rho_\pm}.
\end{equation*}
Here the exponent $\rho_\pm\in\mathbb{C}$, $0<\Re \rho_\pm<\frac{1}{2}$, and multipliers $c_\pm\in\mathbb{C}^*$, are uniquely determined in terms of $\{x_1,x_2,x_3,x_4\}$, and vice versa, through
\begin{align*}
    \rho_+&=\frac{1}{2\pi i}\log(x_2x_4-1), &  \rho_-&=\frac{1}{2\pi i}\log(x_1x_3-1),\\
    x_2&=\frac{\sqrt{2\pi}}{\Gamma(\frac{1}{2}+\rho_+)} e^{\frac{i\pi\rho_+}{2}}c_+^{-1}, &     x_1&=\frac{\sqrt{2\pi}}{\Gamma(\frac{1}{2}+\rho_-)} e^{\frac{i\pi\rho_-}{2}}c_-^{-1},\\
    x_4&=\frac{\sqrt{2\pi}}{\Gamma(\frac{1}{2}-\rho_+)} e^{\frac{i\pi\rho_+}{2}}c_+, &  x_3&=\frac{\sqrt{2\pi}}{\Gamma(\frac{1}{2}-\rho_-)} e^{\frac{i\pi\rho_-}{2}}c_-.
\end{align*}
We note, in particular, that the leading order asymptotics have no explicit dependence on the parameters $\theta_0,\theta_\infty$.

Next, $r_0$ conjugated with the Riemann-Hilbert map induces a biholomorphism from $\mathcal{M}_w$ to itself,
\begin{equation*}
    \mathfrak{r}_0:\mathcal{M}_w\rightarrow \mathcal{M}_w, \quad x\mapsto \widetilde{x}.
\end{equation*}
Denoting $f=f(t;\theta_0,\theta_\infty,x)$ and $y^{\text{Ki}}=y^{\text{Ki}}(t;\theta_0,\theta_\infty,x)$, recall that these are related by the symmetry $r_0$, which we write as
\begin{equation*}
    y^{\text{Ki}}(t;\theta_0,\theta_\infty,x)=f(t;\widetilde{\theta}_0,\widetilde{\theta}_\infty,\widetilde{x}).
\end{equation*}
Denoting $R_\pm=R_\pm(t;x)$, and assuming that $ x_1x_3,x_2x_4,\widetilde{x}_1\widetilde{x}_3,\widetilde{x}_2\widetilde{x}_4\notin \mathbb{R}$, applying the asymptotic formulas \eqref{eq:kap_asymptotics} to both sides, yields
\begin{equation*}
    R_\pm(t;x)=R_\pm(t;\widetilde{x}),
\end{equation*}
and therefore $x=\widetilde{x}$. This means that $\mathfrak{r}_0$ acts as the identity on the dense open subset $\{x\in\mathcal{M}_w: x_1x_3,x_2x_4\notin \mathbb{R}\}$ and thus equals the identity globally, finishing the proof of the lemma.
\end{proof}

\begin{remark}\label{rem:kapaevgenasymp} For the proof of Lemma \ref{lem:PIVsymmetryunderRHasymptotic}, we only required the generic leading order behaviour of solutions on the real line, given in equation \eqref{eq:kap_asymptotics}. The full content of \cite[Th. 4(ii)]{kapaev1996} translates to the following: for $t$ either real or purely imaginary, the large $t$ asymptotics of $f$ are described by
\begin{equation*}\displaystyle
    \left(\frac{f}{2t}+\frac{1}{3}\right)^{-1}=\begin{cases}\displaystyle
    1+\frac{\sqrt{2\pi}}{\Gamma(\frac{1}{2}-\rho_0)}\frac{e^{\frac{i\pi\rho_0}{2}}}{x_4}m_0(t)+\frac{\sqrt{2\pi}}{\Gamma(\frac{1}{2}+\rho_0)}\frac{e^{\frac{i\pi\rho_0}{2}}}{x_2}m_0(t)^{-1}+\mathcal{O}(t^{-1+2|\Re \rho_0|}) & \text{as $t\rightarrow+\infty$,}\\
   \displaystyle 1+\frac{\sqrt{2\pi}}{\Gamma(\frac{1}{2}-\rho_1)}\frac{e^{\frac{i\pi\rho_1}{2}}}{x_3}m_1(t)+\frac{\sqrt{2\pi}}{\Gamma(\frac{1}{2}+\rho_1)}\frac{e^{\frac{i\pi\rho_1}{2}}}{x_4}m_1(t)^{-1}+\mathcal{O}(t^{-1+2|\Re \rho_1|}) & \text{as $t\rightarrow+i\infty$,}\\
      \displaystyle 1+\frac{\sqrt{2\pi}}{\Gamma(\frac{1}{2}-\rho_2)}\frac{e^{\frac{i\pi\rho_2}{2}}}{x_1}m_2(t)+\frac{\sqrt{2\pi}}{\Gamma(\frac{1}{2}+\rho_2)}\frac{e^{\frac{i\pi\rho_2}{2}}}{x_3}m_2(t)^{-1}+\mathcal{O}(t^{-1+2|\Re \rho_2|}) & \text{as $t\rightarrow-\infty$,}\\
      \displaystyle 1+\frac{\sqrt{2\pi}}{\Gamma(\frac{1}{2}-\rho_{3})}\frac{e^{\frac{i\pi\rho_3}{2}}}{x_2}m_3(t)+\frac{\sqrt{2\pi}}{\Gamma(\frac{1}{2}+\rho_3)}\frac{e^{\frac{i\pi\rho_3}{2}}}{x_1}m_3(t)^{-1}+\mathcal{O}(t^{-1+2|\Re \rho_3|}) & \text{as $t\rightarrow-i\infty$,}\\
    \end{cases}
\end{equation*}
where
\begin{align*}
    m_0(t)&=e^{\frac{\pi i}{3}}e^{\frac{i}{\sqrt{3}}t^2}(2\sqrt{3}t^2)^{-\rho_0}, & \rho_0&=\frac{1}{2\pi i}\log(x_2x_4-1),\\
   m_1(t)&=e^{-\frac{\pi i}{3}}e^{-\frac{i}{\sqrt{3}}t^2}(-2\sqrt{3}t^2)^{-\rho_1}, & \rho_1&=\frac{1}{2\pi i}\log(x_4x_3-1),\\
   m_2(t)&=e^{\frac{\pi i}{3}}e^{\frac{i}{\sqrt{3}}t^2}(2\sqrt{3}t^2)^{-\rho_2}, & \rho_2&=\frac{1}{2\pi i}\log(x_3x_1-1),\\
   m_3(t)&=e^{-\frac{\pi i}{3}}e^{-\frac{i}{\sqrt{3}}t^2}(-2\sqrt{3}t^2)^{-\rho_3}, & \rho_3&=\frac{1}{2\pi i}\log(x_1x_2-1),
\end{align*}
with $-\frac{1}{2}<\Re\rho_j<\frac{1}{2}$ for $0\leq j\leq 3$. The respective conditions of validity for these asymptotic formulas are
\begin{equation*}
    x_2x_4-1\notin \mathbb{R}_{\leq 0},\quad  x_4x_3-1\notin \mathbb{R}_{\leq 0},\quad  x_3x_1-1\notin \mathbb{R}_{\leq 0},\quad  x_1x_2-1\notin \mathbb{R}_{\leq 0}.
\end{equation*}
Furthermore, note that, e.g. as $t\rightarrow+\infty$, the asymptotics are at leading order real, if and only if  $\overline{x}_2=x_4$ and $x_2x_4-1\in\mathbb{R}_{>0}$, so that in particular $\Re\rho_0=0$.
\end{remark}

\begin{proof}[Proof of \Cref{thm:PIVsymmetriesunderRH}]
    From \Cref{lem:PIVsymmetriesunderRHdirect}, we have that $r_1$ acts trivially on $\mathcal{M}\to \mathscr{W}_{\rm{IV}}$.
    It follows from the same lemma that $r_0r_1r_2r_1$ acts trivially, since it is in the translation part of $W(A_2^{(1)})$ generated by $T_1^{2}T_2^{-1}$ and $T_1^{-1}T_2^2$.
    Combining this with \Cref{lem:PIVsymmetryunderRHasymptotic} we have that the whole of $\langle r_0,r_1, r_2\rangle \cong W(A_2^{(1)})$ acts trivially.

    To prove the second part, we note that \Cref{lem:PIVsymmetriesunderRHdirect} gives the action of $\vartheta$. The action of the Dynkin diagram automorphisms in $\Aut(A_2^{(1)})=\langle \sigma_{(01)},\sigma_{(012)}\rangle$ are worked out as follows. 
    Since $r_0$ and $r_1$ act trivially, it follows from the action of  $\sigma_{(01)}r_1r_0$ given in \Cref{lem:PIVsymmetriesunderRHdirect} that $\sigma_{(01)}$ acts as described in \Cref{tab:PIV:symmetry:dynkinautos}.
    We obtain the actions of elements of the cyclic subgroup generated by $\sigma_{(012)}$ through their expressions in terms of $T_2,r_1,r_2$, 
    $$\sigma_{(012)} = T_2r_1r_2.$$ 
    Actions of the remaining Dynkin diagram automorphisms are obtained from those of $\sigma_{(01)}$ and $\sigma_{(012)}$, leading to the expressions in \Cref{tab:PIV:symmetry:dynkinautos}.
\end{proof}

\begin{remark} \label{rem:commutingRHsPIV}
Based on the literature, it is apparently difficult to realise the full symmetry group of $\pain{IV}$ on the level of the $2\times2$ linear problem. 
In \cite{puttop2013piv}, the symmetries were realised on the level of a $3\times3$ linear problem, and a Riemann-Hilbert map to the corresponding monodromy surface 
$\widehat{\operatorname{RH}}_{t,a}:\mathcal{X}_{t,a} \rightarrow \widehat{\mathcal{M}}_{\widehat{w}},$ 
was constructed. 
This induces a commutative diagram of biholomorphisms
\begin{equation*}
    \begin{tikzcd}
    & \mathcal{M}_{w(a)} \arrow[dd,leftrightarrow]\\
        \mathcal{X}_{t,a}  \arrow[ur, "\operatorname{RH}_{t,a}"] \arrow[dr,swap, "\widehat{\operatorname{RH}}_{t,a}"]& \\
    & \widehat{\mathcal{M}}_{\widehat{w}(a)}         
    \end{tikzcd}
\end{equation*}
providing an alternative pathway to prove \Cref{thm:PIVsymmetriesunderRH}, by working out the induced map between $\mathcal{M}_w$  and $\widehat{\mathcal{M}}_{\widehat{w}}$.
In the context of $\pain{VI}$, the induced map between two different monodromy surfaces was worked out in \cite{deganoguzzetti}, relating the trace coordinates on monodromy for a Fuchsian $2\times 2$ linear system with Stokes data for an irregular $3\times3$ system.
\end{remark}

\subsection{Moduli space}
The category of monodromy surfaces for $\pain{IV}$ is that of embedded affine Segre surfaces with a rectangle of lines at infinity.
The parameter space $\mathscr{A}_{\rm{IV}}$ of $\pain{IV}$ modulo the action of $W(A_2^{(1)})\rtimes \Aut(A_2^{(1)})$ provides the corresponding moduli space of objects.

\begin{definition}[Category $\mathfrak{C}_{\rm IV}$ of monodromy surfaces for $\pain{IV}$]
\label{def:categoryPIV}
Define the category $\mathfrak{C}_{\rm IV}$ with
\begin{itemize}
    \item an object being an embedded affine Segre surface $\mathcal{V}\subset \mathbb{A}^4$, such that its complement in its projective completion $\overline{\mathcal{V}}\setminus \mathcal{V}\subseteq \p^4$ is the union of four lines, which intersect like a rectangle, and consists of only smooth points of $\overline{\mathcal{V}}$. 
    \item a morphism between objects $\mathcal{V}_1$ and $\mathcal{V}_2$ being an affine linear $L \in \operatorname{End}(\mathbb{A}^4)$ such that $L(\mathcal{V}_1)\subseteq \mathcal{V}_2$.
\end{itemize}
\end{definition}

The family $\mathcal{M}\to \mathscr{W}$ from \Cref{def:monodromysurfacePIV} represents isomorphism classes of objects in $\mathfrak{C}_{\rm IV}$ in the following way.
\begin{proposition} \label{prop:oblomkovstylenormalformPIV}
    For any object in $\mathcal{V}$ in $\mathfrak{C}_{\rm IV}$, there exists a $w\in \mathscr{W}_{\rm{IV}}$, unique up to the actions of $\operatorname{Aut}(A_2^{(1)})$ in \Cref{tab:PIV:symmetry:dynkinautos},
  such that $\mathcal{V}$ is isomorphic to $\mathcal{M}_w$ in $\mathfrak{C}_{\rm IV}$.  The corresponding isomorphism is unique up to composition with automorphisms of $\mathcal{M}_w$.
  The automorphism group of $\mathcal{M}_w$ in $\mathfrak{C}_{\rm{IV}}$ is generated by $\langle\varpi_1,\varpi_2\rangle$, described by
      \begin{equation} \label{eq:PIVxsymmetriesvarpi}
    \varpi_{1} : x_1\leftrightarrow x_4,\quad \varpi_2 : x_2 \leftrightarrow x_3, \qquad \varpi_1\varpi_2 = \vartheta,
\end{equation}
  and $\operatorname{Stab}_{\Aut(A_2^{(1)})}\left(w\right)$, the latter being trivial for generic $w$.

\end{proposition}

\begin{proof}
The proof is analogous to that of \Cref{prop:oblomkov}.
Let $\mathcal{V}\subset \mathbb{A}^4$ be an object in $\mathfrak{C}_{\rm IV}$ and denote its projective completion by $\overline{\mathcal{V}} \subset \mathbb{P}^4$. 
The hyperplane section $\overline{\mathcal{V}}\setminus \mathcal{V}$ by assumption  consists of four lines intersecting according to a rectangle, so can be written as
\begin{equation*}
    L_1 L_4 =0,\quad L_2 L_3=0,\quad X_0=0,
\end{equation*}
where each $L_k\in \mathbb{C}[X_0,X_1,X_2,X_3,X_4]$ is homogeneous of degree 1. 
An affine map puts $\mathcal{V}$ into the form
\begin{equation}\label{eq:refinedsegre1aff}
\begin{aligned}
    x_1 x_4+ b_4x_4+b_3x_3+b_2x_2+b_1x_1+b_0 &= 0, \\
    x_2 x_3+ c_4x_4+c_3x_3+c_2x_2+c_1x_1+c_0&=0,
\end{aligned}
\end{equation}
for some $b_{k},c_{k}\in\mathbb{C}$, $1\leq k\leq 4$.
Applying $x_1\mapsto x_1-b_4,x_2\mapsto x_2-c_3,x_3\mapsto x_3-c_2,x_4\mapsto x_4-b_1$, we may eliminate the linear terms in $x_2,x_3$ and in $x_1,x_4$ from the first and second equations respectively in \eqref{eq:refinedsegre1aff}, yielding
\begin{equation*}
\begin{aligned}
    x_1 x_4 +d_3 x_3+d_2 x_2-e_1 &= 0, \\
    x_2 x_3+d_4 x_4+d_1 x_1-e_2    &=0,
\end{aligned}\end{equation*}
for some $d_k\in\mathbb{C}$, $1\leq k\leq 4$, and $e_1,e_2\in\mathbb{C}$.
Further scaling, 
$x_k\mapsto \gamma_k x_k$ for $k=1,2,3$, rescales the coefficients as
\begin{equation*}
    d_1\mapsto d_1 \frac{\gamma_1}{\gamma_2\gamma_3},\quad 
    d_2\mapsto d_2 \frac{\gamma_2}{\gamma_1\gamma_4},\quad
    d_3\mapsto d_3 \frac{\gamma_3}{\gamma_1\gamma_4},\quad
    d_4\mapsto d_4 \frac{\gamma_4}{\gamma_2\gamma_3}, \quad
   e_1\mapsto e_1 \frac{1}{\gamma_1\gamma_4},\quad 
    e_2\mapsto e_2 \frac{1}{\gamma_2\gamma_3}.
\end{equation*}
Since $\overline{\mathcal{V}}\setminus \mathcal{V}$ consists of only smooth points of $\overline{\mathcal{V}}$, the coefficients $d_k$, $1\leq k\leq 4$, are all nonzero and any of the three solutions $(\gamma_1,\gamma_2,\gamma_3,\gamma_4)\in(\mathbb{C}^*)^4$ to
$$
\begin{aligned}\frac{d_1}{d_4}\gamma_1^3&=d_1d_2d_3d_4, & 
\frac{d_2}{d_3}\gamma_2^3&=d_1d_2d_3d_4, &
\frac{d_3}{d_2}\gamma_3^3&=d_1d_2d_3d_4, &
\frac{d_4}{d_1}\gamma_4^3&=d_1d_2d_3d_4,\\
\frac{d_4}{d_1}\gamma_1&=\gamma_4, & \frac{d_2}{d_3}\gamma_3&=\gamma_2, & \gamma_2\gamma_4&=d_1d_3, & &
\end{aligned}$$
put the pair of equations into the form 
of the two generators of the ideal in \Cref{piv:idealIw} 
with 
$$w_1 = \frac{e_2}{\gamma_2\gamma_3}, \quad 
w_2 = \frac{e_1}{\gamma_1\gamma_4}.$$
This gives, for any $\mathcal{V}$, a $w \in \mathscr{W}_{\rm VI}$ and an isomorphism $\mathcal{V}\to \mathcal{M}_w$ determined up to the action on $\mathcal{M}\to \mathscr{W}_{\rm{IV}}$ of $\Aut(A_2^{(1)})$ and the automorphisms $\varpi_1,\varpi_2$ in \eqref{eq:PIVxsymmetriesvarpi}. 
This comes through the freedom of choice of enumeration of $L_1,L_2,L_3,L_4$ and the choice of $(\gamma_1,\gamma_2,\gamma_3,\gamma_4)$, which correspond exactly to the actions of $\operatorname{Aut}(A_2^{(1)})$ and $\vartheta$ on $x_i$ and $w_i$ in \Cref{tab:PIV:symmetry:dynkinautos} as well as of the automorphisms \eqref{eq:PIVxsymmetriesvarpi}.

To show that $\mathcal{M}_w$ is isomorphic to $\mathcal{M}_{\tilde{w}}$ if and only if $w$ and $\tilde{w}$ are related by the action of some $\sigma\in \Aut(A_2^{(1)})$ as given in \Cref{tab:PIV:symmetry:dynkinautos}, we compute directly as follows. 
Explicitly, up to permutation of coordinates via the action of $\langle \varpi_1,\varpi_2\rangle \cong \Z/2\Z\times\Z/2\Z$, to match the rectangle of lines at infinity $\overline{\mathcal{M}}_{w}\setminus \mathcal{M}_{w}$ and $\overline{\mathcal{M}}_{\tilde{w}}\setminus \mathcal{M}_{\tilde{w}}$ the isomorphism must be either of the form $(x_1,x_2,x_3,x_4)\mapsto (A_1 x_1,A_2 x_2, A_3x_3, A_4x_4)$ for some nonzero $A_1,A_2,A_3,A_4$, or of the form $(x_1,x_2,x_3,x_4)\mapsto (B_2 x_2,B_1 x_1, B_4x_4, B_3x_3)$ for some nonzero $B_1,B_2,B_3,B_4$. 
In the former case, requiring that $I_w$ is sent to $I_{\tilde{w}}$ gives 
\begin{equation*}
    A_4=A_1,\quad A_3=A_2, \qquad A_1=A_2^2,\quad A_2=A_1^2\qquad \tilde{w}_1 = A_1 w_1, \quad \tilde{w}_2 = A_2 w_2, 
\end{equation*}
the solutions of which correspond to the action of $\langle \sigma_{(012)}\rangle$.
In the latter case, we require 
\begin{equation*}
    B_4=B_1,\quad B_3=B_2, \qquad B_1=B_2^2,\quad B_2=B_1^2\qquad \tilde{w}_1 = B_2 w_2, \quad \tilde{w}_2 = B_1 w_1, 
\end{equation*}
the solutions of which correspond to the action of the coset $\sigma_{(01)}\langle \sigma_{(012)}\rangle$, which accounts for the remaining part of $\operatorname{Aut}(A_2^{(1)})$.
An automorphism of $\mathcal{M}_w$ then must come from a combination of $\langle \varpi_1,\varpi_2\rangle$ and the action of some $\sigma \in \Aut(A_2^{(1)})$ such that $\sigma(w)=w$, so generically the isomorphism $\mathcal{V}\to \mathcal{M}_w$ is unique up to $\langle \varpi_1,\varpi_2\rangle$ for fixed $w$.
\end{proof}

To describe the moduli space of $\mathfrak{C}_{\rm{IV}}$, from \Cref{prop:oblomkovstylenormalformPIV} we get a mapping 
\begin{equation} \label{eq:mapPhiPIV}
\begin{gathered}
    \Phi : \operatorname{ob}(\mathfrak{C}_{\rm{IV}})\to \mathscr{O}_{\rm{IV}}:= \mathscr{W}_{\rm{IV}}\sslash \Aut(A_2^{(1)}) = \Spec \C[o_1,o_2], \\
    o_1 = w_1 w_2, \quad o_2 = w_1^3 + w_2^3. 
\end{gathered}
\end{equation}
The singular locus $\mathscr{W}_{\rm{IV}}^{\operatorname{sing}}\subset \mathscr{W}_{\rm{IV}}$ corresponds to $\mathscr{O}_{\rm{IV}}^{\operatorname{sing}}\subset\mathscr{O}_{\rm{IV}}$ given by 
$$ o_1^2+18 o_1 - 4 o_2 -27 = 0.$$
The following is proved using \Cref{prop:oblomkovstylenormalformPIV} in an analogous way to \Cref{cor:modulispacePVI} was proved in the $\pain{VI}$ case.
\begin{corollary} \label{cor:modulispacePIV}
    The map $\Phi$ given in \eqref{eq:mapPhiPIV} induces a canonical bijection 
    $$\operatorname{Iso}(\mathfrak{C}_{\rm{IV}}) \to \mathscr{O}_{\rm{IV}}(\C)\cong\C^2,$$
    where $\mathscr{O}_{\rm{IV}} = \mathscr{W}_{\rm{IV}} \sslash \operatorname{Aut}(A_2^{(1)})$.
\end{corollary}

\begin{remark} \label{rem:mysterysymmetryPIV}
    As long as $\theta$ lies in 
        $$\Theta_{\rm{IV}}\setminus \Theta_{\rm{IV}}^{\operatorname{sing}} = \left\{ \theta = (\theta_0,\theta_{\infty})\in \C^2 : 2\theta_0\not\in  \Z, \,\,\theta_0+\theta_{\infty}\not\in \Z, \,\,\theta_0-\theta_{\infty}\not\in \Z \right\},$$
so that $w\in \mathscr{W}_{\rm{IV}}\setminus \mathscr{W}_{\rm{IV}}^{\operatorname{sing}}$, the automorphisms $\varpi_1,\varpi_2$ of $\mathcal{M}_w$ introduced in \cref{eq:PIVxsymmetriesvarpi} conjugate under the Riemann-Hilbert map $\operatorname{RH}_{t,a}$ to biholomorphic involutions of the (open) initial value space $\mathcal{X}_{t,a}$.
These must be transcendental, and as far as we know no such symmetries of $\pain{IV}$ have appeared in the literature.
\end{remark}

\subsection{Lines on the monodromy surface}

For $w\in \mathscr{W}_{\rm{IV}}\setminus \mathscr{W}_{\rm{IV}}^{\operatorname{sing}}$, $\mathcal{M}_{w}$ is a smooth Segre surface and as such contains 16 lines.
To write these down, introduce
$$\mathscr{U}_{\rm{IV}} = \left\{ u = (u_1,u_2) \in (\C^{*})^{\times 2} \right\},$$
where 
\begin{equation}\label{eq:thetatouPIV}
    u_1=e^{ \frac{2\pi i}{3}\theta_\infty+ 2 \pi i\theta_0}, \quad u_2=e^{\frac{2\pi i}{3}\theta_\infty-2\pi i\theta_0}.
\end{equation}
The surface $\overline{\mathcal{M}}_w$ has, in addition to the four lines at infinity, 12 affine lines, which are given in terms of these parameters by
\begin{equation} \label{eq:linesPIV}
\begin{aligned}
    L_k : x_1 &= b_k,\quad &&x_2= \frac{1}{b_k},\quad &x_3  + b_k x_4  = w_2 - b_k^{-1}, \quad &&k=1,2,3,\\
    L_k : x_1 &= b_k,\quad &&x_3= \frac{1}{b_k},\quad &x_2  + b_k x_4  = w_2 - b_k^{-1}, \quad &&k=4,5,6,\\
    L_k : x_4 &= b_k,\quad &&x_2= \frac{1}{b_k},\quad &x_1 + b_k^{-1} x_3 = w_1 - b_k, \quad &&k=7,8,9,    \\
    L_k : x_4 &= b_k,\quad &&x_3= \frac{1}{b_k},\quad &x_1 + b_k^{-1} x_2  = w_1 - b_k, \quad &&k=10,11,12,    
    \end{aligned}
    \end{equation}
where $b_k$, $k=1,\dots,12$, are given by
\begin{equation*} 
    b_k = u_1, \text{ if } k = 1 \mod 3 , \quad 
    b_k = u_2, \text{ if } k = 2 \mod 3, \quad 
    b_k = \frac{1}{u_1u_2}, \text{ if } k = 3 \mod 3.
\end{equation*}    
We will also use $L_k$, $k=1,\dots,12$, to denote the closures of the above affine lines in $\p^4$ under the embedding.
We denote the four lines at infinity by
\begin{equation} \label{eq:linesatinfinityPIV}
\begin{aligned}
    L_{13} &= \{X \in \p^4 : X_0=0, X_1=0,X_2=0\},\\
    L_{14} &= \{X \in \p^4 : X_0=0, X_1=0,X_3=0\},\\
    L_{15} &= \{X \in \p^4 : X_0=0, X_3=0,X_4=0\},\\
    L_{16} &= \{X \in \p^4 : X_0=0, X_2=0,X_4=0\}.
\end{aligned}
\end{equation}

To describe the field extension from $\C(\mathscr{W}_{\rm{IV}})$ to $\C(\mathscr{U}_{\rm{IV}})$ used to write the lines on $\mathcal{M}_w$, we use the following lemma, which is proved similarly to \Cref{lem:PvimorphismUtoW}, using the auxiliary field $\C(t_1,t_2)$, where $t_1=e^{\frac{2\pi i }{3}\theta_\infty},t_2=e^{2\pi i \theta_0}$.
\begin{lemma}     \label{lem:morphismUtoW}
The definitions of the $u_i,w_i$ in terms of $\theta_j$ yields the field homomorphism 
    \begin{equation*}
    \C(\mathscr{W}_{\rm{IV}}) \longrightarrow \C(\mathscr{U}_{\rm{IV}}),
    \end{equation*}
    defined by
\begin{align*}
    w_1&= u_1 + u_2 + \frac{1}{u_1 u_2}, \\
    w_2&= \frac{1}{u_1}+\frac{1}{u_2} +u_1 u_2,
\end{align*}
 \end{lemma}
 
In terms of the parameters $u$, the singular locus $\mathscr{W}_{\rm{IV}}^{\operatorname{sing}}\subset \mathscr{W}_{\rm{IV}}$ of the family $\mathcal{M}\to \mathscr{W}_{\rm{IV}}$ pulls back to $\mathscr{U}_{\rm{IV}}^{\operatorname{sing}}\subset \mathscr{U}_{\rm{IV}}$  defined by 
\begin{equation*}
\left(u_1 - u_2 \right)^2 \left(u_1 - u_1^{-1}u_2^{-1} \right)^2 \left(u_2 - u_1^{-1}u_2^{-1} \right)^2 = 0.
\end{equation*}

\begin{remark}
    When $u\in \mathscr{U}_{\rm{IV}}^{\operatorname{sing}}$, some of the lines are merged. 
For example, when $u_1=u_2$, the following pairs of lines coincide: 
\begin{equation*}
    (L_1,L_2), \quad (L_4,L_5), \quad (L_7,L_{8}),\quad (L_{10},L_{11}).
\end{equation*}
Under the Riemann-Hilbert correspondence, $\mathscr{U}_{\rm{IV}}^{\operatorname{sing}}$ corresponds to parameter values for the Riccati solutions of the fourth Painlev\'e equation \cites{kapaev1996,kapaev1998,saitoterajimanodalcurvesriccati}.

\end{remark}

    We will now describe the relation between the models of the monodromy surface for $\pain{IV}$ as embedded affine Segre and cubic surfaces respectively.
    Let $w \in \mathscr{W}_{\rm{IV}}\setminus \mathscr{W}_{\rm{IV}}^{\operatorname{sing}}$, and consider the Segre surface as 
    \begin{equation*}
        \begin{aligned}
            \overline{\mathcal{M}}_w &= \operatorname{Proj} \C[X_0,X_1,X_2,X_3,X_4]/\overline{I}_w, \\
            \overline{I}_w&=(X_2 X_3+X_0(X_1+X_4)-w_1 X_0^2, X_1X_4+X_0(X_2+X_3)-w_2 X_0^2),
        \end{aligned}
    \end{equation*}
    and the singular projective cubic surface as 
    \begin{equation*}
        \begin{aligned}
            \overline{\mathcal{C}}_w &= \operatorname{Proj} \C[Y_0,Y_1,Y_2,Y_3]/\overline{J}_w, \\
            \overline{J}_w &= \left( Y_1 Y_2 Y_3 + Y_0Y_1^2 - Y_0^2(Y_2 +Y_3 + w_1 Y_1)+w_2Y_0^3 \right).
        \end{aligned}
    \end{equation*}

    Note that there is one more affine line on the monodromy surface when realised as a cubic as in \eqref{eq:affinecubicPIV}.
    This corresponds to a distinguished conic on  $\overline{\mathcal{M}}_w$, of which there are in fact four. 
    In the next lemma, we describe geometrically the relation between $\overline{\mathcal{M}}_w$ and $\overline{\mathcal{C}}_w$, which in particular will be used to explain these distinguished conics.

    \begin{lemma} \label{lem:segretocubicPIV}
Let $w\in \mathscr{W}_{\rm{IV}}\setminus \mathscr{W}_{\rm{IV}}^{\operatorname{sing}}$ and 
    \begin{equation} \label{eq:rhobirationalsegretocubicPIV}
    \begin{aligned}
        \rho : \overline{\mathcal{M}}_w &\dashrightarrow \overline{\mathcal{C}}_w,\\
        [X_0:X_1:X_2:X_3:X_4] &\mapsto [Y_0:Y_1:Y_2:Y_3]= [X_0:X_1:X_2:X_3].
    \end{aligned}
    \end{equation}
     be the extension of the isomorphism $\mathcal{M}_w\to\mathcal{C}_w$ in \eqref{eq:isomsegretocubicPIV}.
    Denote by $\operatorname{Bl}_q : \widetilde{\overline{\mathcal{M}}}_w\to \overline{\mathcal{M}}_w$ the blowup at the point $q\in \overline{\mathcal{M}}_w$ defined by $[X_0:X_1:X_2:X_3:X_4]=[0:0:0:0:1]$.
    Then the map $\widetilde{\overline{\mathcal{M}}}_w \to \overline{\mathcal{C}}_w$ defined by the following commutative diagram is a minimal resolution of the two $A_1$ singularities of $\overline{\mathcal{C}}_w$:
    \begin{equation*} 
        \begin{tikzcd}
            &   \widetilde{\overline{\mathcal{M}}}_w \arrow[dl,swap,  "\operatorname{Bl}_{q}"] \arrow[dr] & \\
            \overline{\mathcal{M}}_w  \arrow[rr,dashed, "\rho"]& & \overline{\mathcal{C}}_w
        \end{tikzcd}
    \end{equation*}
    Further, 
    when $\overline{\mathcal{M}}_w$ is realised as $\p^2$ blown up at five points $b_1,\dots,b_5$, with the blowup morphism contracting two of the disjoint lines at infinity, the point $q$ lies on the intersection of one of these with one of the two others.
\end{lemma}

\begin{proof}
Being a smooth del Pezzo surface of degree 4, $\overline{\mathcal{M}}_w$ can be realised as $\p^2$ blown up at 5 points.
    For any choice of a set of five disjoint lines on $\mathcal{M}_w$ we have a birational morphism 
    $$\pi : \overline{\mathcal{M}}_w \to \p^2,$$
    which contracts these five lines, and the Picard group is 
    $$\Pic(\overline{\mathcal{M}}_w)\cong \Z\mathcal{H}\oplus \Z \mathcal{E}_1\oplus\Z\mathcal{E}_2\oplus \Z \mathcal{E}_3\oplus \Z \mathcal{E}_4\oplus \Z \mathcal{E}_5,$$
    and the canonical divisor class is written as 
    $$\mathcal{K}_{\overline{\mathcal{M}}_w} = -3 \mathcal{H}+\mathcal{E}_1+\mathcal{E}_2+\mathcal{E}_3+\mathcal{E}_4 + \mathcal{E}_5.$$
    The exceptional curves on $\overline{\mathcal{M}}_{w}$ correspond to the following elements of $\Pic(\overline{\mathcal{M}}_w)$:
    $$\mathcal{E}_i,\quad \mathcal{H}-\mathcal{E}_j-\mathcal{E}_k, \quad 2 \mathcal{H}-\E_1-\E_2-\E_3-\E_4-\mathcal{E}_5, \qquad \text{where} \quad i,j,k\in \{1,\dots,5\}, j\neq k,$$
    so a line on $\mathcal{M}_w$ is mapped by $\pi$ to either one of the five points, a line joining a pair of them, or the unique conic passing through all five.

    There are sixteen choices of a set of five disjoint lines on $\overline{\mathcal{M}}_w$, all of which include one of the pairs $(L_{13},L_{15})$ and $(L_{14},L_{16})$ of disjoint lines at infinity. 
    We choose such that the lines at infinity correspond to the following elements of $\Pic(\overline{\mathcal{M}}_w)$:
    \begin{equation*}
        \mathcal{L}_{1}^{\infty} = \mathcal{E}_1,\quad 
        \mathcal{L}_{2}^{\infty} = \mathcal{H}-\mathcal{E}_1-\mathcal{E}_2,\quad 
        \mathcal{L}_3^{\infty}=\mathcal{E}_2,\quad 
        \mathcal{L}_4^{\infty} = 2\mathcal{H}-\mathcal{E}_1- \mathcal{E}_2-\mathcal{E}_3- \mathcal{E}_4-\mathcal{E}_5.
    \end{equation*}
    Explicitly the five lines can be chosen to such that $\mathcal{E}_{1},\dots,\mathcal{E}_{5}$ correspond to $L_{13},L_{15},L_{9},L_{7},L_{8}$,
    and the morphism in this case becomes
    \begin{equation*}
            \pi : 
             \left[X_0:X_1:X_2:X_3:X_4\right] \mapsto \left[ X_0  : X_0  - u_1 u_2 X_4 : X_0 - \tfrac{1}{u_1u_2}X_2\right],
    \end{equation*}
    with the lines contracted to the points 
    \begin{gather*}
        b_1 = [0:1:0],\quad b_2 = [0:0:1], \quad b_3 = [1:0:0],\\ \quad b_4 =[1: 1-u_1^2u_2:u_1u_2-u_1^{-1}],\quad b_5 =[1:1-u_1u_2^2:u_1u_2-u_2^{-1}].
    \end{gather*}
    Then consider the blowup $\operatorname{Bl}_q : \widetilde{\overline{\mathcal{M}}}_w \to \overline{\mathcal{M}}_w$ centred at the point 
$$q : [X_0:X_1:X_2:X_3:X_4]=[0:0:0:0:1],
$$
and denote the class of the exceptional divisor of this by $\mathcal{F}$ so 
$$\Pic( \widetilde{\overline{\mathcal{M}}}_w) \cong \Z\mathcal{H}\oplus \Z \mathcal{E}_1\oplus\Z\mathcal{E}_2\oplus \Z \mathcal{E}_3\oplus \Z \mathcal{E}_4\oplus \Z \mathcal{E}_5\oplus \Z \mathcal{F},$$
where we abuse notation in using the same symbols for generators of $\Pic(\overline{\mathcal{M}}_w)$ as for their pullbacks under the blowup.  
Note that $q$ lies at the intersection of the lines $L_{13}$ and $L_{14}$ on $\overline{\mathcal{M}}_w$ and the strict transforms of these correspond to  
$$ \mathcal{E}_1 - \mathcal{F}, \qquad \mathcal{H}-\mathcal{E}_1-\mathcal{E}_2-\mathcal{F}.$$
Then $\widetilde{\overline{\mathcal{M}}}_w$ is smooth, and a direct computation shows that the map $\rho\circ \operatorname{Bl}_q : \widetilde{\overline{\mathcal{M}}}_w \rightarrow \overline{\mathcal{C}}_w$ is a birational morphism which contracts these $-2$ curves onto the two $A_1$ singularities of $\overline{\mathcal{C}}_w$, which are at 
$$p_1 : [Y_0:Y_1:Y_2:Y_3]=[0:0:0:1], \quad p_2: [Y_0:Y_1:Y_2:Y_3]=[0:0:1:0],$$
and this is a minimal resolution. 
We give an illustration of the relation between the hyperplane sections at infinity $\overline{\mathcal{M}}_w\setminus \mathcal{M}_w$ and $\overline{\mathcal{C}}_w\setminus\mathcal{C}_w$ under $\rho : \overline{\mathcal{M}}_w \dashrightarrow \overline{\mathcal{C}}_w$ in \Cref{fig:segretocubicdivisoratinfinity}. 
\end{proof}

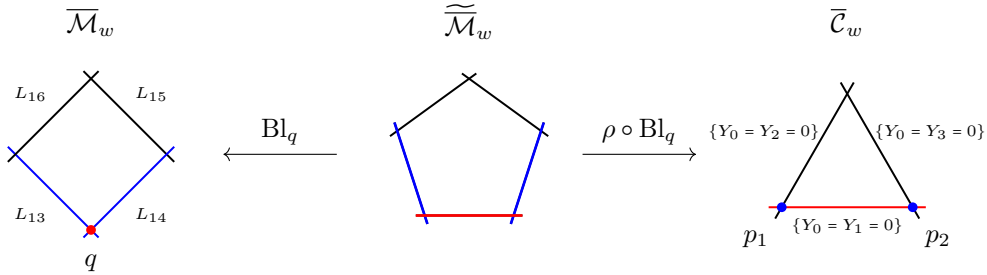
\begin{figure}[htb]
       \begin{equation*}
    \begin{tikzpicture}[redpoint/.style={circle,draw=red!100,fill=red!100,thick, inner sep=0pt,minimum size=1mm},bluepoint/.style={circle,draw=blue!100,fill=blue!100,thick, inner sep=0pt,minimum size=1mm}]
    

\coordinate (T) at (0,1);
\coordinate (R) at (1,0);
\coordinate (B) at (0,-1);
\coordinate (L) at (-1,0);

\def\eps{0.1}
\draw[line width=.75pt] ($(T)!-\eps!(R)$) -- ($(R)!-\eps!(T)$);
\draw[line width=.75pt, blue] ($(R)!-\eps!(B)$) -- ($(B)!-\eps!(R)$);
\draw[line width=.75pt, blue] ($(B)!-\eps!(L)$) -- ($(L)!-\eps!(B)$);
\draw[line width=.75pt] ($(L)!-\eps!(T)$) -- ($(T)!-\eps!(L)$);
\node[redpoint, label={[yshift=-4pt]below:$q$}] at (B) {};

\node at (-.8,-.8) {\tiny $L_{13}$};
\node at (.8,-.8) {\tiny $L_{14}$};
\node at (.8,.8) {\tiny $L_{15}$};
\node at (-.8,.8) {\tiny $L_{16}$};

\draw[<-] (1.75,0) -- node[pos=0.5, above] {$\operatorname{Bl}_q$} (3.25,0);


\begin{scope}[xshift=5cm]
\foreach \i in {1,...,5} {
    \coordinate (P\i) at ({90 + 72*(\i-1)}:1);
}

\foreach \i/\j in {1/2,2/3,3/4,4/5,5/1} {
    \draw[line width=.75pt]
      ($(P\i)!-\eps!(P\j)$)
      --
      ($(P\j)!-\eps!(P\i)$);
}

    \draw[line width=1pt, blue] ($(P2)!-\eps!(P3)$)
      --
      ($(P3)!-\eps!(P2)$);

      \draw[line width=1pt, blue] ($(P4)!-\eps!(P5)$)
      --
      ($(P5)!-\eps!(P4)$);

    \draw[line width=1pt, red] ($(P3)!-\eps!(P4)$)
      --
      ($(P4)!-\eps!(P3)$);
\draw[->] (1.5,0) -- node[pos=0.5, above] {$\rho\circ\operatorname{Bl}_q$}(3,0);

\end{scope}


\begin{scope}[xshift=10cm,yshift=-.2cm]
\def\eps{0.10}

\foreach \i in {1,...,3} {
    \coordinate (P\i) at ({90 + 120*(\i-1)}:1);
}
    \draw[line width=.75pt]
      ($(P1)!-\eps!(P2)$)
      --
      ($(P2)!-\eps!(P1)$);
    \draw[line width=.75pt, red]
      ($(P2)!-\eps!(P3)$)
      --
      ($(P3)!-\eps!(P2)$);
    \draw[line width=.75pt]
      ($(P3)!-\eps!(P1)$)
      --
      ($(P1)!-\eps!(P3)$);
\node at (1.1,0.5) {\tiny $\{Y_0=Y_3=0\}$}; 
\node at (-1.1,0.5) {\tiny $\{Y_0=Y_2=0\}$}; 
\node at (0,-0.75) {\tiny $\{Y_0=Y_1=0\}$}; 

\node[bluepoint, label={[yshift=-4pt]below left:$p_1$}] at (P2) {};
\node[bluepoint, label={[yshift=-4pt]below right:$p_2$}] at (P3) {};

\end{scope}


\node at (0,1.75) {$\overline{\mathcal{M}}_w$};
\node at (5,1.75) {$\widetilde{\overline{\mathcal{M}}}_w$};
\node at (10,1.75) {$\overline{\mathcal{C}}_w$};
    \end{tikzpicture}
\end{equation*}
\caption{Relation of hyperplane sections at infinity under $\rho : \overline{\mathcal{M}}_w \dashrightarrow \overline{\mathcal{C}}_w$ for $\pain{IV}$ in \Cref{lem:segretocubicPIV}.}
\label{fig:segretocubicdivisoratinfinity}
\end{figure}

We denote the additional line on $\mathcal{C}_w$, which does not correspond to a line on $\mathcal{M}_w$, by
\begin{equation*}
    L_{0} =\{ x_1 = 0,\, x_2+x_3= w_2\}\subset \mathcal{C}_w.
\end{equation*}
This extra line does not require the field extension from $\C(\mathscr{W}_{\rm{IV}})$ to $\C(\mathscr{U}_{\rm{IV}})$ to write, and will play no role in the description of monodromy that follows. 
This extra line $L_0$ corresponds under $\rho$ to 
$$\rho^* L_0 = \{w_2 X_0X_3 +X_0X_4 =w_1 X_0^2+X_3^2\}\cap\overline{\mathcal{M}}_w.$$
Then $\rho^* L_0$ is characterised as follows relative to the point $q\in \overline{\mathcal{M}}_w$ and the points $b_1,\dots,b_{5}\in \p^2$ to which the birational morphism $\pi : \overline{\mathcal{M}}_w\to \p^2$ contracts the lines $L_{13},L_{15},L_{9},L_{7},L_{8}$.
The curve $\rho^* L_0$ passes through $q$, and it becomes a conic in $\p^2$ passing through $b_1,b_3,b_4,b_5$, so all images of contracted lines except for that of $L_{15}$, so $\rho^* L_0$ corresponds to
\begin{equation} \label{eq:fakelineinPicPIV}
2\h -\mathcal{E}_1-\mathcal{E}_3-\mathcal{E}_4-\mathcal{E}_5-\mathcal{F}\in\Pic(\widetilde{\overline{\mathcal{M}}}_w).
\end{equation}
Thus it is a $-1$ curve on $\widetilde{\overline{\mathcal{M}}}_w$ and its image in $\p^2$ is the unique conic which passes through $b_1,b_3,b_4,b_5$ and which has strict transform containing $q$.
Further, this is the unique conic in $\p^2$ which corresponds to a line on $\overline{\mathcal{C}}_w$.
Such a conic must come from a rational curve on $\widetilde{\overline{\mathcal{M}}}_w$ of self-intersection $-1$ disjoint from the $-2$ curves, so corresponds to 
$$\mathcal{D}= 2\h-a_1 \E_1-a_2 \E_2-\cdots -a_5 \E_5 - b \F\in \Pic(\widetilde{\overline{\mathcal{M}}}_w),$$
for some $a_1,\dots,a_5,b\in\Z$ such that 
$$\mathcal{D}\cdot\mathcal{D} = -1, \quad 
\mathcal{D}\cdot\mathcal{K}_{\widetilde{\overline{\mathcal{M}}}_w}= -1 \quad 
\mathcal{D}\cdot(\E_1-\F)=0, \quad 
\mathcal{D}\cdot(\h-\E_1-\E_2-\F)=0,$$
in which the canonical divisor class of $\widetilde{\overline{\mathcal{M}}}_w$ is 
$$\mathcal{K}_{\widetilde{\overline{\mathcal{M}}}_w} = -3 \mathcal{H}+\mathcal{E}_1+\mathcal{E}_2+\mathcal{E}_3+\mathcal{E}_4 + \mathcal{E}_5+\mathcal{F}.$$
The unique solution $a_1,\dots,a_5,b\in\Z$ gives the element corresponding to $L_0$ in \eqref{eq:fakelineinPicPIV}.

    In addition to $\rho$ in \eqref{eq:rhobirationalsegretocubicPIV}, there are three more ways to map birationally from the Segre surface to different copies of the cubic, corresponding to different choices of $x_1,\dots,x_4$ to eliminate. 
    The corresponding copies of $\overline{\mathcal{C}}_w$ will be related by combinations of $\varpi_1,\varpi_2$ in \eqref{eq:PIVxsymmetriesvarpi} and $\sigma_{(01)}$ in \Cref{tab:PIV:symmetry:dynkinautos}.
    Each of the four birational maps will produce a different line, coming from a distinguished curve on $\overline{\mathcal{M}}_w$, which corresponds to a distinguished conic in $\p^2$ under the blowup projection.
    We denote these by
    \begin{equation*} 
    \begin{aligned}
        &C_1 : \quad x_1 = 0, \qquad x_2 +x_3 = w_2, \qquad x_2 x_3+x_4=w_1, \\
        &C_2 : \quad x_2 = 0, \qquad x_1 +x_4 = w_1, \qquad x_1 x_4+x_3=w_2, \\
        &C_3 : \quad x_3 = 0, \qquad x_1 +x_4 = w_1, \qquad x_1 x_4+x_2=w_2, \\
        &C_4 : \quad x_4 = 0, \qquad x_2 +x_3 = w_2, \qquad x_2 x_3+x_1=w_1.
    \end{aligned}
    \end{equation*}

We next explain the meaning of the lines $L_k$, $k=1,\dots,12$,  and the curves $C_{\ell}$, $\ell=1,\dots,4$, 
in terms of generalised monodromy data and solutions of the linear problem.
In order to do this, we introduce the following notation for the solutions $Y_k(z)$, which we recall are asymptotic to the formal solution $Y_{\operatorname{form}}(z)$ given in \cref{eq:formalsolPIV} as $z\to\infty$ in $\Sigma_{k}\cup\Sigma_{k+1}$.
Denote the columns of the matrix function in \eqref{eq:matrixPasymptoticPIV} according to $P(z)=\begin{pmatrix}
        P_+(z) & P_-(z)
    \end{pmatrix}.$
Define $\Psi_k(z)$ as the unique column vector solution of \eqref{eq:pivlinearsystem1} satisfying
\begin{equation*}
    \Psi_k(z)\sim P_\pm(z)e^{\pm(\frac{1}{2}z^2+tz)}z^{\mp \theta_\infty},
\end{equation*}
with sign $\pm 1=(-1)^k$ determined by the parity of $k$,
as $z\rightarrow \infty$ in $\Sigma_{k-1}\cup\Sigma_k\cup\Sigma_{k+1}$. Then we have, for $k\in\mathbb{Z}$,
\begin{equation*}
    Y_k=\begin{pmatrix}
        \Psi_k,\Psi_{k+1}
    \end{pmatrix}\quad \text{if $k$ even},\qquad
     Y_k=\begin{pmatrix}
        \Psi_{k+1},\Psi_k
    \end{pmatrix}\quad \text{if $k$ odd,}
\end{equation*}
and the Stokes phenomenon takes the form
\begin{equation}\label{eq:stokesmult}
    \Psi_{k+1}=\Psi_{k-1}+s_k\Psi_k.
\end{equation}

\begin{proposition}\label{prop:curves_explained}
    The lines $L_1,\dots,L_{12}$ correspond to the following subsets of the space $M$ of Stokes multipliers given in \Cref{eq:monodromyspaceMpiv}:
    \begin{equation}\label{eq:stokesdatalinesPIV}
        \begin{aligned}
            L_1 : &\quad  \left\{s_2 + e^{2\pi i( \theta_{\infty}-\theta_{0}) } s_4 = 0, 
            \quad 1 + s_1s_2 = e^{-2\pi i (\theta_{\infty}+\theta_0)} \right\}, \\
            L_2 : &\quad  \left\{ s_2 + e^{2\pi i( \theta_{\infty}+\theta_{0}) } s_4 = 0,
            \quad 1 + s_1s_2 = e^{-2\pi i (\theta_{\infty}-\theta_0)} \right\},\\
            L_3 : &\quad \left\{ s_1=0 \right\}\cap M, \\
            L_4 : &\quad  \left\{s_3 + e^{2\pi i( \theta_{\infty}-\theta_{0}) } s_1 = 0, 
            \quad 1 + s_3s_4= e^{-2\pi i (\theta_{\infty}+\theta_0)} \right\}, \\
            L_5 : &\quad  \left\{s_3 + e^{2\pi i( \theta_{\infty}+\theta_{0}) } s_1 = 0, 
            \quad 1 + s_3s_4= e^{-2\pi i (\theta_{\infty}-\theta_0)} \right\}, \\
            L_6 : &\quad \left\{ s_4=0 \right\}\cap M, \\
            L_7 : &\quad \left\{s_1 + e^{2\pi i( \theta_{\infty}+\theta_0) } s_3 = 0, 
            \quad 1 + s_2 s_3 = e^{-2\pi i (\theta_{\infty}+\theta_{0})}\right\}, \\
            L_8 : &\quad \left\{s_1 + e^{2\pi i( \theta_{\infty}-\theta_0) } s_3 = 0, 
            \quad 1 + s_2 s_3 = e^{-2\pi i (\theta_{\infty}-\theta_{0})}\right\},\\
            L_9 : &\quad \left\{ s_2=0\right\} \cap M,\\
            L_{10} : &\quad \left\{ s_2 + e^{2\pi i( \theta_{\infty}+\theta_0) } s_4 = 0,
            \quad  1 + s_3 s_4 = e^{-2\pi i (\theta_{\infty}+\theta_{0})}\right\}, \\
            L_{11} : &\quad \left\{ s_2 + e^{2\pi i( \theta_{\infty}-\theta_0) } s_4 = 0,
            \quad  1 + s_3 s_4 = e^{-2\pi i (\theta_{\infty}-\theta_{0})}\right\} \\
            L_{12} : &\quad \left\{ s_3=0\right\} \cap M.            
        \end{aligned}
    \end{equation}
        Further, the curves $C_1,\dots,C_4$ correspond to the following:
    \begin{equation} \label{eq:stokesdataconicsPIV}
    \begin{aligned}
            &C_1 : \left\{ s_1 s_4 + e^{-4 \pi i \theta_{\infty}}=0\right\}\cap M, \\
            &C_2 : \left\{ s_1 s_2 + 1=0\right\}\cap M, \\
            &C_3 : \left\{ s_3 s_4 + 1=0\right\}\cap M, \\
            &C_4 : \left\{ s_2 s_3 + 1=0\right\}\cap M. 
    \end{aligned}
    \end{equation}
From an analytic point of view, they each correspond to a different quantisation condition on the spectral equation:
 \begin{enumerate}
     \item $\Psi_{k+1}=c\,\Psi_{k-1}$, for some $c\in\mathbb{C}^*$, is equivalent to $s_k=0$, in which case $c=1$, corresponding to lines $L_3,L_9,L_{12},L_{6}$ for the respective values $k=1,2,3,4$.
     \item $\Psi_{k+2}=c\,\Psi_{k-1}$, for some $c\in\mathbb{C}^*$, is equivalent to $1+s_ks_{k+1}=0$, in which case $c=s_{k+1}$, corresponding to conics $C_2,C_4,C_3,C_1$ for the respective values $k=1,2,3,4$.
     \item Finally, $\Psi_k(e^{-2\pi i}z)=c\,\Psi_k(z)$, for some $c\in\mathbb{C}^*$, is equivalent to $s_{k+1}+s_{k+3}+s_{k+1}s_{k+2}s_{k+3}=0$, in which case there are two possible values for the eigenvalue $c$ of the analytic continuation operator, $c=e^{\pm 2\pi i\theta_0}$. Then, the condition $\Psi_k(e^{-2\pi i}z)=e^{+ 2\pi i\theta_0}\,\Psi_k(z)$ corresponds to lines $ L_{11}, L_{4}, L_2,L_7$ for the respective values $k=1,2,3,4$. The condition $\Psi_k(e^{-2\pi i}z)=e^{- 2\pi i\theta_0}\,\Psi_k(z)$ corresponds to lines $ L_{10}, L_{5}, L_1,L_8$ for the respective values $k=1,2,3,4$.
 \end{enumerate}
    \end{proposition}

\begin{proof}
    For the conditions on Stokes data corresponding to the lines and curves, the formulas in \eqref{eq:stokesdatalinesPIV} and \eqref{eq:stokesdataconicsPIV} can be derived directly from the expressions for $x_i$ in terms of Stokes multipliers that follow from \eqref{eq:ytosPIV} and \eqref{eq:xtoyPIV}. 

    Regarding the analytic interpretation of the lines and special conics, note that (1) follows directly from equation \eqref{eq:stokesmult}.
    To obtain (2), we apply \eqref{eq:stokesmult} twice, yielding
    \begin{equation*}
        \Psi_{k+2}=s_{k+1}\Psi_{k-1}+(1+s_k s_{k+1})\Psi_k.
    \end{equation*}
This means that $\Psi_{k+2}$ is a multiple of  $\Psi_{k-1}$ if and only if $1+s_k s_{k+1}=0$, in which case $\Psi_{k+2}=s_{k+1}\Psi_{k-1}$.
Finally, to obtain (3), we apply \eqref{eq:stokesmult} thrice, yielding
    \begin{equation*}
        \Psi_{k+4}=(1+s_{k+2}s_{k+3})\Psi_{k-1}+(s_{k+1}+s_{k+3}+s_{k+1}s_{k+2}s_{k+3})\Psi_k.
    \end{equation*}
Thus, $\Psi_{k+4}$ is a multiple of  $\Psi_k$ if and only if 
\begin{equation}\label{eq:triple_product}
   s_{k+1}+s_{k+3}+s_{k+1}s_{k+2}s_{k+3}=0,
\end{equation}
in which case $\Psi_{k+4}=(1+s_{k+2}s_{k+3})\Psi_{k}$. On the other hand, we know that the eigenvalues of the operator of analytic continuation around $z=0$ are $e^{\pm 2\pi i\theta_0}$. Combining this with
\begin{equation*}
   e^{(-1)^k2\pi i\theta_\infty} \Psi_{k+4}(z)=\Psi_{k}(e^{-2\pi i}z),
\end{equation*}
yields $$1+s_{k+2}s_{k+3}=e^{2\pi i(\pm \theta_0+(-1)^k\theta_\infty)}.$$
The last equation and \eqref{eq:triple_product} cut out the different lines $L_j$, $j\in\{1,2,4,5,7,8,10,11\}$, in $\mathcal{M}_w$ as described in the proposition. 
\end{proof}

\begin{remark}
The asymptotic descriptions of the general solution of $\pain{IV}$ for large real and imaginary times in  Remark \ref{rem:kapaevgenasymp} degenerate on the lines and distinguished conics. Kapaev \cite{kapaev1996} provided degenerate asymptotics of the solution for each of these special cases, collected in different theorems.
\begin{enumerate}
    \item Degenerate asymptotics corresponding to the  conics $C_1,C_2,C_3,C_4$ are described in  \cite[Th. 6]{kapaev1996}.
    \item Degenerate asymptotics corresponding to the lines $L_3,L_6,L_9,L_{12}$ are described in  \cite[Th. 7]{kapaev1996}.
    \item Degenerate asymptotics corresponding to the lines $L_2,L_5,L_7,L_{10}$ are described in  \cite[Th. 8]{kapaev1996}.
    \item Degenerate asymptotics corresponding to the lines $L_1,L_4,L_8,L_{11}$ are described in  \cite[Th. 9]{kapaev1996}.
\end{enumerate}

\begin{remark}
The automorphism group of $\mathcal{M}_w$ is maximal when $w_1=w_2=0$, generated by $\langle\varpi_1,\varpi_2\rangle$ and $\Aut(A_2^{(1)})$, see Proposition \ref{prop:oblomkovstylenormalformPIV}. In this case, there is a unique global fixed point, $x=(0,0,0,0)$, where all four distinguished conics $C_k$, $1\leq k\leq 4$, intersect. The corresponding solutions of $\pain{IV}$ under the Riemann-Hilbert map are the Okamoto rational solutions \cites{kapaev1998,maro18}.
\end{remark}

\end{remark}

\subsection{Combinatorial monodromy}

We define the monodromy of the family $\mathcal{M}\to\mathscr{W}_{\rm{IV}}$ first in terms of permutations of lines that preserve intersections as follows. 
These correspond to automorphisms of the famous Clebsch graph, which encodes the intersection configuration of lines on a generic Segre surface, that fix the nodes corresponding to a rectangle of lines at infinity.

\begin{definition} \label{def:PIVintersectiongraph}
    Let $u\in\mathscr{U}_{\rm{IV}}\setminus\mathscr{U}_{\rm{IV}}^{\operatorname{sing}}$ and consider the lines $L_1,\dots,L_{16}$ on $\overline{\mathcal{M}}_w$, for $w\in \mathscr{W}_{\rm{IV}}$ given as in \Cref{lem:morphismUtoW}.
    Form the intersection graph of lines on $\overline{\mathcal{M}}_w$, with vertices corresponding to lines at infinity coloured red and the remaining coloured blue, with edges encoding pairwise intersections, similarly to \Cref{def:pviintersectiongraph}. Denote this intersection graph by $\mathcal{G}_{\rm{IV}}$, and the the subgraph generated by the red vertices by $\mathcal{G}_{\rm{IV}}^\infty$.
\end{definition}
The intersection graph $\mathcal{G}_{\rm{IV}}$ can be computed from the expressions for the lines on $\overline{\mathcal{M}}_w$ given in \eqref{eq:linesPIV} and \eqref{eq:linesatinfinityPIV}, and is shown in \Cref{fig:linesgraphpiv}.
\begin{figure}
    \centering
    \includegraphics[width=0.6\linewidth]{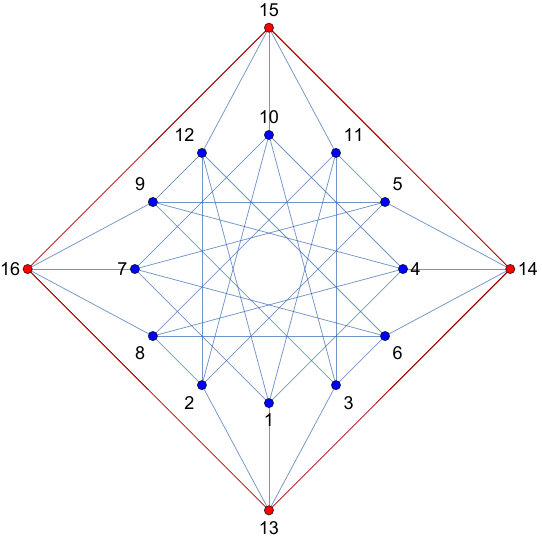}
    \caption{Graph $\mathcal{G}_{\rm{IV}}$ encoding intersections of lines on the monodromy variety corresponding to $\pain{IV}$,
    with in red the subgraph $\mathcal{G}_{\rm{IV}}^\infty$ induced by the vertices $(L_{13},L_{14},L_{15},L_{16})$.
    }
    \label{fig:linesgraphpiv}
\end{figure}

\begin{definition} \label{def:combinatorialmonodromyPIV}
The \emph{combinatorial monodromy} of the family $\mathcal{M}\to\mathscr{W}_{\rm{IV}}$ is the group 
$$\operatorname{Fix}_{\operatorname{Aut}(\mathcal{G}_{\rm{IV}})}\left( \mathcal{G}_{\rm{IV}}^{\infty}\right)$$
of automorphisms of the graph $\mathcal{G}_{\rm{IV}}$ in \Cref{fig:linesgraphpiv} that pointwise fix the vertices $\{13,14,15,16\}$.
\end{definition}
We will show that the combinatorial monodromy of $\mathcal{M}\to \mathscr{W}_{\rm{IV}}$ is isomorphic to  
$$W(A_2) = \langle r_1,r_2~|~ r_1^2=r_2^2=1, \,\,r_1r_2r_1=r_2 r_1 r_2 \rangle \cong \mathfrak{S}_3,$$
which forms the underlying finite Weyl group of the symmetry group of $\pain{IV}$.

\begin{proposition} \label{prop:combinatorialmonodromyPIV}
    The combinatorial monodromy of the family $\mathcal{M}\to \mathscr{W}_{\rm{IV}}$ is isomorphic to $W(A_2)$, with an explicit isomorphism
    \begin{equation*}
        W(A_2) \xrightarrow{\sim}   \operatorname{Fix}_{\operatorname{Aut}(\mathcal{G}_{\rm{IV}})}(\mathcal{G}_{\rm{IV}}^\infty), \quad r\mapsto \tilde{r},
    \end{equation*}
defined by sending the generators $r_1,r_2$ to
\begin{align*}
    \tilde{r}_1&= (1\,\, 2)\,(4\,\,5)\,(7\,\,8)\,(10\,\,11),\\
    \tilde{r}_2&= (1\,\, 3)\,(4\,\,6)\,(7\,\,9)\,(10\,\,12).
\end{align*}
\end{proposition}
\begin{proof}
        Similarly to in the proof of \Cref{prop:combinatorialmonodromyPVI} for $\pain{VI}$, using the realisation of $\overline{\mathcal{M}}_w$ as the blowup of $\p^2$ at five points as in the proof of \Cref{lem:segretocubicPIV}, the combinatorial monodromy can be realised as a group of lattice automorphisms of $\Pic(\overline{\mathcal{M}}_w)$.
        
        Recall that the lines on $\overline{\mathcal{M}}_w$ correspond to the set $\operatorname{EX}\subset\Pic(\overline{\mathcal{M}}_w)$ of classes of exceptional curves, formed by the sixteen elements 
        $$\mathcal{E}_i,\quad \mathcal{H}-\mathcal{E}_j-\mathcal{E}_k, \quad 2 \mathcal{H}-\E_1-\E_2-\E_3-\E_4-\mathcal{E}_5, \qquad \text{where} \quad i,j,k\in \{1,\dots,5\}, j\neq k.$$
        The group of Cremona isometries of $\Pic(\overline{\mathcal{M}}_w)$ is the Weyl group $W(D_5)$ forming the monodromy of del Pezzo surfaces of degree 4 \cites{dolgachevclassicalAG,manincubicforms}. 
        With the choice of birational morphism $\pi : \overline{\mathcal{M}}_w \to \p^2$ as in \Cref{lem:segretocubicPIV}, this can be generated by reflections 
            $$r_{\alpha_i} (\F) = \F + (\F\cdot \alpha_i)\,\alpha_i,$$
        associated to the simple roots 
        $$ \alpha_1 = \mathcal{E}_1-\mathcal{E}_2, \quad \alpha_2 = \mathcal{E}_2-\mathcal{E}_3, \quad \alpha_3=\mathcal{E}_3-\mathcal{E}_4 \quad \alpha_4=\mathcal{E}_4-\mathcal{E}_5, \quad \alpha_5=\mathcal{H}-\mathcal{E}_1-\mathcal{E}_2-\mathcal{E}_3.$$
       A permutation of lines fixing the lines at infinity corresponds to an automorphism of $\Pic(\overline{\mathcal{M}}_w)$ fixing the elements 
        \begin{equation*}
        \mathcal{E}_1,\quad 
        \mathcal{H}-\mathcal{E}_1-\mathcal{E}_2,\quad 
        \mathcal{E}_2,\quad 
        2\mathcal{H}-\mathcal{E}_1- \mathcal{E}_2-\mathcal{E}_3-\mathcal{E}_4-\mathcal{E}_5.
    \end{equation*}
        The subgroup of such automorphisms is generated by reflections 
    corresponding to roots orthogonal to these, of which it suffices to take  
    $$\beta_1 = \alpha_4=\E_4-
    \E_5,\qquad \beta_2=\alpha_3 = \E_3-\E_4.$$
    These play the role of simple roots of the $A_2$ root system, corresponding to the $A_2$ subdiagram of $D_5$ in \Cref{fig:dynkinD5}. 
    The reflections $r_{
    \beta_1},r_{\beta_2}$ acting on $\operatorname{EX} \subset \Pic(\overline{\mathcal{M}}_w)$ induce the permutations of lines $\tilde{r}_1,\tilde{r}_2$ in the proposition.
\end{proof}

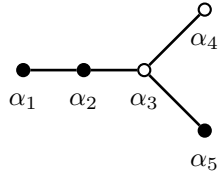
\begin{figure}[htb]
    \centering
            \begin{tikzpicture}[elt/.style={circle,draw=black!100,fill=black!100,thick, inner sep=0pt,minimum size=1.6mm},aff/.style={circle,draw=black!100,thick, inner sep=0pt,minimum size=1.6mm},scale=0.8]
 		\path 	(-2,0) 	node 	(a1) [elt, label={[xshift=0pt, yshift = -21 pt] $\alpha_1$} ] {}
                    (-1,0) 	node 	(a2) [elt, label={[xshift=0pt, yshift = -21 pt] $\alpha_2$} ] {}
                    (0,0) 	node 	(a3) [aff, label={[xshift=0pt, yshift = -21 pt] $\alpha_3$} ] {}
                    (1,1) 	node 	(a4) [aff, label={[xshift=0pt, yshift = -21 pt] $\alpha_4$} ] {}
                    (1,-1) 	node 	(a5) [elt, label={[xshift=0pt, yshift = -21 pt] $\alpha_5$} ] {}
 		       ;
 		\draw [black,line width=1pt ] (a1) -- (a2) -- (a3) -- (a4);
 		\draw [black,line width=1pt ] (a3) -- (a5);
 	\end{tikzpicture}
   \caption{Dynkin diagram $D_5$ with $A_2$ subdiagram indicated by unfilled nodes.}
    \label{fig:dynkinD5}
\end{figure}

\begin{remark} \label{rem:combinatorialremarkPIV}
    As in the case of the Jimbo-Fricke cubic in \Cref{rem:combinatorialremarkPVI}, the affine lines are naturally partitioned into subsets according to which line at infinity they intersect.  
    By definition, the combinatorial monodromy must respect this partition and we observe that $W(A_2)$ acts faithfully when restricted to any of these four sets of three lines.
\end{remark}

The nontrivial action of the Dynkin diagram automorphisms and the symmetry $\vartheta$ induce permutations of lines, including those at infinity.

\begin{proposition}
    The action of $\Aut(A_2^{(1)})$ and $\langle \vartheta \rangle$ on $\mathscr{A}_{\rm{IV}}\cong\Theta_{\rm{IV}}$ in \Cref{tab:PIV:symmetry:varsandparams} induces the action on $\mathscr{U}_{\rm{IV}}$ given by 
        \begin{equation*}
        \begin{aligned}
            \sigma_{(01)} &: (u_1,u_2)\mapsto (\tilde{u}_1,\tilde{u}_2), \qquad \tilde{u}_1 = u_1 u_2, &  
            &\tilde{u}_2 = \frac{1}{u_2}, \\
            \sigma_{(12)} &: (u_1,u_2)\mapsto (\tilde{u}_1,\tilde{u}_2), \qquad \tilde{u}_1 =\zeta_3^{-1} \frac{1}{u_1} , &
            &\tilde{u}_2 = \zeta_3^{-1} u_1 u_2, \\            
            \sigma_{(02)} &: (u_1,u_2)\mapsto (\tilde{u}_1,\tilde{u}_2), \qquad \tilde{u}_1 =\zeta_3 \frac{1}{u_2} , & 
            &\tilde{u}_2 = \zeta_3\frac{1}{u_1}, \\
            \vartheta &: (u_1,u_2) \mapsto (\tilde{u}_1,\tilde{u}_2), \qquad \tilde{u}_1=u_1, & &\tilde{u}_2=u_2.
            \end{aligned}
    \end{equation*} 
    This, together with the action on $\mathcal{M}\to \mathscr{W}_{\rm{IV}}$ given in \Cref{tab:PIV:symmetry:dynkinautos} 
    induces the following automorphisms of the graph $\mathcal{G}_{\rm{IV}}$:
    \begin{align*}
        \sigma_{(01)} &= (1\,\,6\,\,10\,\,9)\,(2\,\,5\,\,11\,\,8)\,(3\,\,4\,\,12\,\,7)\,(13\,\,16\,\,15\,\,14),\\
        \sigma_{(12)} &= (1\,\,7\,\,10\,\,4)\,(2\,\,9\,\,11\,\,6)(3\,\,8\,\,12\,\,5)\,\,(13\,\,14\,\,15\,\,16), \\
        \sigma_{(02)} &= (1\,\,8\,\,10\,\,5)\,(2\,\,7\,\,11\,\,4)\,(3\,\,9\,\,12\,\,6)\,(13\,\,14\,\,15\,\,16),\\
        \vartheta &= (1\,\,10)\,(2\,\,11)\,(3\,\,12)\,(4\,\,7)\,(5\,\,8)\,(6\,\,9)\,(13\,\,15)\,(14\,\,16).
    \end{align*}
In particular, the cyclic subgroup $\langle \sigma_{(012)} \rangle$ acts trivially on the rectangle of lines at infinity and preserves the partition of the affine lines as in \Cref{rem:combinatorialremarkPIV}, with 
$$\sigma_{(021)}= (1\,\,2\,\,3)\,(4\,\,5\,\,6)\,(7\,\,8\,\,9)\,(10\,\,11\,\,12).$$
Further the automorphisms $\varpi_1,\varpi_2$ of $\overline{\mathcal{M}}_w$ written in \eqref{eq:PIVxsymmetriesvarpi}, extended to act trivially on $\mathscr{U}_{\rm{IV}}$, induce the following:
\begin{align*}
    \varpi_1 &= (1\,\,7)\,(2\,\,8)\,(3\,\,9)\,(4\,\,10)\,(5\,\,11)\,(6\,\,12)\,(13\,\,16)\,(14\,\,15),\\
    \varpi_2 &= (1\,\,4)\,(2\,\,5)\,(3\,\,6)\,(7\,\,10)\,(8\,\,11)\,(9\,\,12)(13\,\,14)\,(15\,\,16).
\end{align*}
\end{proposition}
\begin{proof}
     
    We define the action on $\mathscr{U}_{\rm{IV}}$ through the map from $\mathscr{A}_{\rm{IV}}\cong\Theta_{\rm{IV}}$  to $\mathscr{U}_{\rm{IV}}$ defined by $u_1=e^{ \frac{2\pi i}{3}\theta_\infty+ 2 \pi i\theta_0},  u_2=e^{\frac{2\pi i}{3}\theta_\infty-2\pi i\theta_0}$.
        The induced permutations of lines are derived by direct computation.
\end{proof}

\subsection{Algebraic monodromy}
\label{subsec:algmonodromyPIV}
We will describe the monodromy of the family $\mathcal{M}\to \mathscr{W}_{\rm{IV}}$ algebraically in terms of an extension of the function field $\C(\mathscr{W}_{\rm{IV}})$ coming from the incidence variety of lines on the monodromy surface.

\begin{definition}
    The incidence variety of lines on $\mathcal{M}_w$ is 
    \begin{equation*}
        \Gamma\xrightarrow{\rho} \mathscr{W}_{\rm{IV}}, \quad\Gamma = \{ (w,\ell)\in \mathscr{W}_{\rm{IV}} \times \mathbb{G}(1,4)~|~\ell \subset \overline{\mathcal{M}}_w, \,\, \ell \not\subset \overline{\mathcal{M}}_w\setminus \mathcal{M}_w\},
    \end{equation*}
    where $\mathbb{G}(1,4)$ is the Grassmannian of projective lines in $\p^4$.
\end{definition}

The restriction to the nonsingular locus of $\mathcal{M}\to\mathscr{W}_{\rm{IV}}$ gives a covering $\rho : \Gamma^{\operatorname{ns}}\to \mathscr{W}_{\rm{IV}}\setminus \mathscr{W}_{\rm{IV}}^{\operatorname{sing}}$.
The fibre over $w\in \mathscr{W}_{\rm{IV}}\setminus \mathscr{W}_{\rm{IV}}^{\operatorname{sing}}$ consists of 12 points corresponding to the affine lines $L_{1},\dots,L_{12}$ on $\mathcal{M}_{w}$.

Similarly to the case of $\pain{VI}$ described in \Cref{subsec:algmonodromyPVI}, $\Gamma$ is reducible, and we define $K_{\Gamma}$ to be the compositum of the normal closures of the function fields of its irreducible components within the algebraic closure of $\C(\mathscr{W}_{\rm{IV}})$. 

\begin{definition}
    The \emph{algebraic monodromy} of the family $\mathcal{M}\to \mathscr{W}_{\rm{IV}}$ is the Galois group of the (normal closure of the) field extension $\C(\mathscr{W}_{\rm{IV}})\subset K_{\Gamma}$:
    \begin{equation*}
        \operatorname{Gal}\left( K_{\Gamma}/ \C(\mathscr{W}_{\rm{IV}})\right).
    \end{equation*}
\end{definition}
The algebraic monodromy here also forms $W(A_2)$, the underlying finite Weyl group of the symmetry group of $\pain{IV}$.

\begin{proposition}  \label{prop:algebraicmonodromyPIV}
    The algebraic monodromy of $\mathcal{M}\to \mathscr{W}_{\rm{IV}}$ is 
    $$\operatorname{Gal}\left( K_{\Gamma}/ \C(\mathscr{W}_{\rm{IV}})\right) \cong \operatorname{Gal}\left( \C(\mathscr{U}_{\rm{IV}})/ \C(\mathscr{W}_{\rm{IV}})\right) \cong W(A_2),$$
    where the first isomorphism comes from a $\C(\mathscr{W}_{\rm{IV}})$-linear isomorphism $K_{\Gamma}\cong \C(\mathscr{U}_{\rm{IV}})$,  and the second isomorphism is given by the action of $W(A_2)$ on $\C(\mathscr{U}_{\rm{IV}})$ defined by 
    \begin{equation*}
        \begin{aligned}
    r_1 &: &&u_1\to u_2, \quad 
    &&u_2\to u_1, \\
    r_2 &: &&u_1\to \frac{1}{u_1u_2}, \quad 
    &&u_2\to u_2,
        \end{aligned}
    \end{equation*}
    which keeps $(w_1,w_2)$ fixed and induces the same permutations of lines as in \Cref{prop:combinatorialmonodromyPIV}.
\end{proposition}

\begin{proof}
    We proceed similarly to in the proof of \Cref{prop:algebraicmonodromyPVI}. 
    We first derive the defining equations of $\Gamma$ in Pl\"ucker coordinates.
    For a projective line in $\p^4$ defined as the intersection of three hyperplanes
    \begin{equation*}
        \ell = \left\{ [X_0:X_1:X_2:X_3:X_4]\in \p^4  : 
        \begin{aligned}
            &a_0 X_0+a_1X_1+a_2X_2+a_3X_3 + a_4 X_4=0 \\
            &\hspace{0.2mm}b_0\hspace{0.2mm} X_0+\hspace{0.2mm}b_1\hspace{0.2mm}X_1+\hspace{0.2mm}b_2\hspace{0.2mm}X_2+\hspace{0.2mm}b_3\hspace{0.2mm}X_3 + \hspace{0.2mm}b_4\hspace{0.2mm} X_4=0 \\
            &\hspace{0.2mm}c_0\hspace{0.2mm} X_0+\hspace{0.2mm}c_1\hspace{0.2mm}X_1+\hspace{0.2mm}c_2\hspace{0.2mm}X_2+\hspace{0.2mm}c_3\hspace{0.2mm}X_3 + \hspace{0.2mm}c_4\hspace{0.2mm} X_4=0 
        \end{aligned}
        \right\},
    \end{equation*}  
    we take Pl\"ucker coordinates of $\ell$ to be $[P_{0,1}:P_{0,2}:P_{0,3}:P_{0,4}:P_{1,2}:P_{1,3}:P_{1,4}:P_{2,3}:P_{2,4}:P_{3,4}]$, where $P_{i,j}$ is the $3\times 3$ minor of the matrix 
    $$\left(\begin{array}{ccccc}a_0 & a_1 & a_2 & a_3 & a_4 \\b_0 & b_1 & b_2 & b_3 & b_4 \\c_0 & c_1 & c_2 & c_3 & c_4\end{array}\right),$$
    with respect to the two columns corresponding to indices $0\leq i<j\leq 4$.
    The Pl\"ucker coordinates of the lines are given in \Cref{tab:PIV:plucker}.
    Under the Pl\"ucker embedding the incidence variety $\Gamma$ becomes an algebraic set in $\mathscr{W}_{\rm{IV}}\times \p^9$.
    From \Cref{tab:PIV:plucker}, it is contained in 
    $$ P_{1,2} P_{1,3}P_{2,4}P_{3,4}=0.$$
    The first component $P_{1,2}$ contains the points corresponding to $L_1,L_2,L_3$, for which $P_{1,2}=0, P_{0,4}\neq0$, so to study this part we restrict to the affine chart $(p_{0,1},p_{0,2},p_{0,3},p_{1,3},p_{1,4},p_{2,3},p_{2,4},p_{3,4})\in\C^8$ given by
    \begin{gather*} [P_{0,1}:P_{0,2}:P_{0,3}:P_{0,4}:P_{1,2}:P_{1,3}:P_{1,4}:P_{2,3}:P_{2,4}:P_{3,4}]= \\
    [p_{0,1}:p_{0,2}:p_{0,3}:1:0:p_{1,3}:p_{1,4}:p_{2,3}:p_{2,4}:p_{3,4}].\end{gather*}
    The defining equations for $\Gamma$ in this chart, including the Pl\"ucker relations, become
    \begin{equation*} 
        \begin{gathered}
            p_{0,1}=0, \quad p_{0,2}p_{0,3}=0, \quad
            p_{0,2}p_{1,3}=0, \quad p_{0,2}p_{1,4}=0, \\
            p_{1,3}- p_{0,3}p_{1,4}=0,\quad 
            p_{2,3} - p_{0,3}p_{2,4}+ p_{0,2}p_{3,4},\quad
            p_{1,4}p_{2,3}-p_{1,3}p_{2,4} = 0,\\
            p_{0,3}p_{2,4}+ p_{0,2}p_{3,4}=1,\quad 
            p_{0,2} = p_{0,3}+p_{1,4},\\
            p_{1,4}+p_{2,4}p_{3,4} + w_1=0,\quad p_{3,4}- p_{2,4} + w_2 = 0.
        \end{gathered}
    \end{equation*}
    The lines $L_1,L_2,L_3$ have $p_{0,3}\neq 0$, for generic $w$, and in this case the equations reduce to the following single equation for $b=p_{0,3}$:
    $$b^3 - w_1 b^2 + w_2 b-1=0,$$
    which splits over $\C(\mathscr{U}_{\rm{IV}})$ with roots $u_1,u_2, \tfrac{1}{u_1u_2}$. 
    Equations for Pl\"ucker coordinates of the remaining three groups of lines $\{L_{4},L_5,L_6\}$, $\{L_7,L_8,L_9\}$ and $\{L_{10},L_{11},L_{12}\}$ also reduce to the same polynomial, 
    and similarly to the proof of \Cref{prop:algebraicmonodromyPVI} we find that $K_{\Gamma}\cong \C(\mathscr{U}_{\rm{IV}})$.
    The claimed correspondence between elements of $\operatorname{Gal}\left(\C(\mathscr{U}_{\rm{IV}})/\C(\mathscr{W}_{\rm{IV}})\right)$ and permutations of lines is checked directly.
\end{proof}

\begingroup

\setlength{\tabcolsep}{0pt} 
\renewcommand{\arraystretch}{1.5} 

\begin{table}[h]
\makebox[\textwidth][c]{%
    $\begin{array}{c||c|c|c|c|c|c|c|c|c|c|}
             & P_{0,1}   & P_{0,2}    & P_{0,3}   & P_{0,4} & P_{2,1}   & P_{3,1} &  P_{1,4}   & P_{2,3}   & P_{4,2}    & P_{4,3}     \\
        \hline\hline 
        L_1 &  0 & 0 & u_1^2 u_2 & u_1 u_2& 0& u_1^3 u_2& -u_1^2 u_2 & u_1 u_2 & -u_2 & u_1^2 u_2^2 + u_1  \\
        \hline 
        L_2 &  0 & 0 & u_1 u_2^2 & u_1 u_2 & 0 & u_1 u_2^3 & -u_1 u_2^2 & u_1 u_2 & -u_1 & u_1^2u_2^2 + u_2 \\
        \hline 
        L_3 & 0 & 0 & u_1 u_2 & u_1^2 u_2^2 & 0 & 1 & -u_1 u_2 & u_1^2 u_2^2 & -u_1^3 u_2^3 & u_1^2u_2+u_1u_2^2 \\
        \hline
        L_{4} &  0 & u_1^2 u_2 & 0 & u_1 u_2 & -u_1^3 u_2 & 0 & u_1^2 u_2 & u_1 u_2 & u_1^2 u_2^2+u_1& -u_2 \\
        \hline 
        L_{5} & 0 & u_1 u_2^2 & 0 & u_1 u_2 & -u_1 u_2^3 & 0 & u_1 u_2^2 & u_1 u_2 & u_1^2u_2^2 + u_2 & -u_1\\
        \hline 
        L_{6} & 0 & u_1 u_2 & 0 & -u_1^2 u_2^2 & 1 & 0 & u_1 u_2 & u_1^2 u_2^2 & u_1^2u_2+ u_1u_2^2 & -u_1^3 u_2^3 \\
        \hline
        L_{7} & -u_1 u_2 & 0 & u_1^2 u_2 & 0 & -u_2 &  u_1^2 u_2^2+u_1 & u_1^2 u_2 & u_1 u_2 & 0 & u_1^3 u_2 \\
        \hline 
        L_{8} & -u_1 u_2 & 0 & u_1 u_2^2 & 0 & -u_1 &  u_1^2u_2^2+u_2 & u_1 u_2^2 & u_1 u_2&  0 & u_1 u_2^3 \\
        \hline
        L_{9} & -u_1^2 u_2^2 & 0 & u_1 u_2 & 0 & -u_1^3 u_2^3 & u_1^2u_2+u_1u_2^2 & u_1 u_2 & u_1^2 u_2^2 & 0 & 1
        \\
        \hline
        L_{10} & u_1 u_2 & u_1^2 u_2 & 0 & 0 &  u_1^2 u_2^2+u_1 & -u_2 & -u_1^2 u_2 & u_1 u_2 & u_1^3 u_2 & 0 \\
        \hline
        L_{11} & u_1 u_2 & u_1 u_2^2 & 0 & 0 &  u_1^2u_2^2+u_2 & -u_1 & -u_1 u_2^2 & u_1 u_2 & u_1 u_2^3 & 0\\
        \hline
        L_{12} & u_1^2 u_2^2 & u_1 u_2 & 0 & 0 & 
        u_1^2u_2+u_1u_2^2 
        & -u_1^3 u_2^3 & -u_1 u_2& u_1^2 u_2^2 & 1 & 0 \\
        \hline
    \end{array}$
    }
    \caption{Pl\"ucker coordinates of lines $L_1,\dots,L_{12}$ on the monodromy surface for $\pain{IV}$.
    }
    \label{tab:PIV:plucker}
\end{table}
\endgroup
\subsection{Analytic monodromy}
Our final result of this section shows that the combinatorial monodromy of the family $\mathcal{M}\to\mathscr{W}_{\rm{IV}}$ is induced by deformations over loops in $\mathscr{W}_{\rm{IV}}$.

Similarly to in \cref{subsec:analyticmonodromyPVI}, take $w_* \in \mathscr{W}_{\rm{IV}}\setminus \mathscr{W}_{\rm{IV}}^{\operatorname{sing}}$ and consider the homomorphism
\begin{equation} \label{eq:monodromyhomPIV}
    \pi_1(\mathscr{W}_{\rm{IV}}\setminus \mathscr{W}_{\rm{IV}}^{\operatorname{sing}} ; w_*) \to \mathfrak{S}_{16},
\end{equation}
which sends a loop to the permutation of lines on $\overline{\mathcal{M}}_{w_*}$ induced by deformation over it.

\begin{definition}
    The \emph{analytic monodromy} of the family $\mathcal{M}\to\mathscr{W}_{\rm{IV}}$ is the subgroup of $\mathfrak{S}_{16}$ given by the image of the homomorphism \eqref{eq:monodromyhomPIV}, defined up to overall conjugation, within $\mathfrak{S}_{16}$, induced by changing enumeration of lines and choice of basepoint.
\end{definition}

Similarly to in the proof of \Cref{prop:analyticmonodromyPVI}, we will show that the analytic monodromy exhausts the combinatorial monodromy group and therefore the two must coincide.

\begin{proposition} \label{prop:analyticmonodromyPIV}
    The analytic monodromy of the family $\mathcal{M}\to\mathscr{W}_{\rm{IV}}$ is $W(A_2)$.
\end{proposition}
\begin{proof}
    Choose $u_*\in\mathscr{U}_{\rm{IV}}\setminus \mathscr{U}_{\rm{IV}}^{\operatorname{sing}}$ lying above $w_*$, which determines an enumeration of the lines on $\overline{\mathcal{M}}_w$. Choose a $\theta_*\in\Theta_{\rm{IV}}$ lying above $u_*$ under the map $\theta\mapsto u$ defined by \cref{eq:thetatouPIV}. 
    Then $\mathscr{U}_{\rm{IV}}\setminus\mathscr{U}_{\rm{IV}}^{\operatorname{sing}}$ lifts under this same map to
    $$\Theta_{\rm{IV}}\setminus \Theta_{\rm{IV}}^{\operatorname{sing}} = \left\{ \theta = (\theta_0,\theta_{\infty})\in \C^2 : 2\theta_0\not\in  \Z, \,\,\theta_0+\theta_{\infty}\not\in \Z, \,\,\theta_0-\theta_{\infty}\not\in \Z \right\},$$
    which is path-connected.
    Then for each of the generators $r_i$, $i=1,2$ of $W(A_2)$, one can choose a path in $\Theta_{\rm{IV}}$ from $\theta_*$ to $r_i(\theta_*)$ that avoids $\Theta_{\rm{IV}}^{\operatorname{sing}}$ and descends to a loop in $\mathscr{W}_{\rm{IV}}\setminus \mathscr{W}_{\rm{IV}}^{\operatorname{sing}}$, inducing the permutation of lines corresponding to $r_i$ in \Cref{prop:combinatorialmonodromyPIV}.
\end{proof}

\Cref{thm:PIVsymmetriesunderRH} establishes \Cref{mainthm:conjugatedsymmetries},   \Cref{prop:oblomkovstylenormalformPIV} and \Cref{cor:modulispacePIV} establish \Cref{mainthm:categoryandmodulispace},
and \Cref{prop:combinatorialmonodromyPIV,prop:algebraicmonodromyPIV,prop:analyticmonodromyPIV} establish \Cref{mainthm:mon} in the case of Painlev\'e-IV.

\section{The second Painlev\'e equation}\label{sec:PII}

We consider the second Painlev\'e equation $\pain{II}$ in the form 
\begin{equation*}
\pain{II}:\quad y_{tt}=2\,y^3+t\,y+\theta+\tfrac{1}{2},
\end{equation*}
and denote its parameter space by
\begin{equation*}
    \Theta_{\rm{II}} = \left\{ \theta \in \C \right\}.
\end{equation*}
The root variables in this case form the space
\begin{equation*}
    \mathscr{A}_{\rm{II}} = \left\{ (a_0, a_1) \in \C^2 ~|~ a_0 + a_1 = 1 \right\}.
\end{equation*}
which is related to $\Theta_{\rm{II}}$ according to 
\begin{equation*}
    a_0=-\theta, \quad a_1 = 1 + \theta,
\end{equation*}
We use the Hamiltonian form of $\pain{II}$ given by
\begin{equation} \label{eq:hamPII}
    \left\{
    \begin{aligned}
        f_t &= f^2 + g + \frac{t}{2}  = \frac{\partial H}{\partial g}, \\
        g_t &= -2 f g +a_1 - 1 = -\frac{\partial H}{\partial f},\\
    \end{aligned}
    \right.
    \qquad 
    H = \frac{1}{2}g^2 + f^2 g + f + \frac{t}{2}g - a_1 f.
\end{equation} 
Eliminating $g(t)$ from \eqref{eq:hamPII} leads to $\pain{II}$ for $f(t)$.

\subsection{B\"acklund transformations}

The B\"acklund transformations for $\pain{II}$ form the extended affine Weyl group of type $A_1^{(1)}$, which we write as
\begin{equation*}
    \widetilde{W}(A_1^{(1)}) \defeq W(A_1^{(1)}) \rtimes \Aut(A_1^{(1)}) \cong \left< r_0, r_1\right>\rtimes\left< \sigma \right>,\quad
       \raisebox{-15pt}{
\begin{tikzpicture}[scale=.4,elt/.style={circle,draw=black!100,thick, inner sep=0pt,minimum size=1.5mm}]
		\path 	(-1,-1) 	node 	(a0) [elt ] {}
		        ( 1,-1) 	node  	(a1) [elt  ] {};
		\node at ($(a0.west) + (-.3,+0.0)$) 	{ \tiny ${0}$};
		\node at ($(a1.east) + (+.3,+0.0)$) 	{ \tiny ${1}$};

		\draw [black,line width=1pt ] (a0) -- (a1) node[midway,above ]{$\infty$};
		\node at ($(d2.east) + (+0,-2)$) 	{ \small ${A_1^{(1)}}$};

	\end{tikzpicture} }
\end{equation*}
where the generators $r_0$, $r_1$ are free, and  $\sigma\in\Aut(A_1^{(1)}) \cong\mathfrak{S}_2$ corresponds to the permutation of the two nodes of the $A_1^{(1)}$ Dynkin diagram, acting by conjugation of the reflections as $\sigma r_0=r_1 \sigma$.
The B\"acklund transformations of system \eqref{eq:hamPII} corresponding to these generators are  given in \Cref{tab:PII:symmetry:varsandparams}. 
There is also a cyclic subgroup $\langle \vartheta\rangle\cong \Z/3\Z$ of symmetries which leave the equation invariant without changing parameters, which we also give in \Cref{tab:PII:symmetry:varsandparams}. 

\begingroup

\setlength{\tabcolsep}{15pt} 
\renewcommand{\arraystretch}{1.75} 

\begin{table}[h]
    \begin{equation*}
    \begin{array}{c||c|c||c|c||c|c|}
               & \tilde{f}     & \tilde{g}     & \tilde{a}_0   & \tilde{a}_1  & \tilde{\theta} & \tilde{t}\\
        \hline\hline 
        r_0    & f + \frac{a_0}{g}    & g  & -a_0   & 2a_0+a_1   & -\theta &  t  \\
        \hline
        r_1    & f + \frac{a_1}{g+2f^2+t}    & g+2 f^2 - 2 \big(f + \frac{a_1}{g+2f^2+t}\big)^2  & a_0+2a_1   & -a_1   & -2 -\theta   & t
        \\
        \hline \hline
        \sigma   & - f    & - g - 2 f^2 - t & a_1   & a_0   & -1-\theta    & t \\
        \hline\hline       
        \vartheta   & \zeta_3^{-1} f    & \zeta_3 g & a_0   & a_1   & \theta    &  \zeta_3 t\\
        \hline
        \vartheta^{-1}   & \zeta_3 f   & \zeta_3^{-1} g  & a_0   & a_1   & \theta    & \zeta_3^{-1} t \\
        \hline
    \end{array}
    \end{equation*}
    \caption{Symmetries of $\pain{II}$ on variables $(f,g)$, parameter $\theta$, root variables $a$ and independent variable $t$.}
    \label{tab:PII:symmetry:varsandparams}
\end{table}
\endgroup

\subsection{Initial value space}

For $\pain{II}$ we have a family of Sakai surfaces
\begin{equation*}
    \overline{\mathcal{X}} \to \mathscr{T}_{\rm{II}}\times \mathscr{A}_{\rm{II}},
\end{equation*}
where $\mathscr{T}_{\rm{II}}= \C$ is the independent variable space of the system \eqref{eq:hamPII}.
The fibre $\overline{\mathcal{X}}_{t,a}$ is a Sakai surface of type $E_7^{(1)}$, with irreducible components of its effective anticanonical divisor $D_{t,a}$ intersecting according to the $E_7^{(1)}$ Dynkin diagram.
Denote the family obtained by removing the support of this divisor from each fibre of $\overline{\mathcal{X}}$ by $\mathscr{\mathcal{X}}\to \mathscr{T}_{\rm{II}}\times \mathscr{A}_{\rm{II}}$.

\begin{definition}[Initial value space for $\pain{II}$]
    The initial value space at $t\in \mathscr{T}_{\rm{II}}$ for $\pain{II}$ with parameters $a\in\mathscr{A}_{\rm{II}}$ is the fibre $\mathcal{X}_{t,a}$ of the family $\mathcal{X}_a\to\mathscr{T}_{\rm{II}}$.
\end{definition}

For each $w \in \widetilde{W}(A_2^{(1)})$, the corresponding B\"acklund transformation gives an automorphism of the family $\mathcal{X}$ and we have an automorphism
$$w : \mathscr{T}_{\rm{II}}\times \mathscr{A}_{\rm{II}} \to \mathscr{T}_{\rm{II}}\times \mathscr{A}_{\rm{II}}, \qquad (t,a)\mapsto (\tilde{t},\tilde{a}),$$
and corresponding isomorphisms 
$$w : \mathcal{X}_{t,a} \to \mathcal{X}_{\tilde{t},\tilde{a}}.$$

\subsection{Associated linear problems}

The Painlev\'e II equation governs isomonodromic deformations within two different classes of rank two systems of linear ODEs on $\p^1$. 
The first class concerns rank two systems with a single irregular singularity of Poincar\'e rank 3, considered by Jimbo and Miwa \cite{jimbomiwaII1981}.
The second class concerns rank two systems with an irregular singularity of the same type in addition to a single regular singularity, with an overall symmetry constraint, considered by Flaschka and Newell \cite{flaschkanewell}. 
The two linear problems are related through an integral transform \cite{joshikitaevtreharneP1P2}, which causes a change in singularity structure.
These lead to different models of the monodromy surface of $\pain{II}$ as affine cubic surfaces \cites{putsaito,chekhovdecorated}, which are isomorphic \cite[Prop. 5.9]{JMR}. 
We will show that the two linear problems lead to the same family of affine del Pezzo surfaces.

\subsubsection{The Jimbo-Miwa linear problem}
In \cite{jimbomiwaII1981}, Jimbo and Miwa derived a Lax pair for the Painlev\'e II equation, given by
\begin{subequations} \label{eq:JMlaxpair}
    \begin{align}
    Y_z&=A^{\text{JM}}Y, & A^{\text{JM}}&=z^2\sigma_3+zA_1^{\text{JM}}+A_0^{\text{JM}} \label{eq:JMlaxpair1}\\
    Y_t&=B^{\text{JM}}Y, & B^{\text{JM}}&=\tfrac{1}{2}z\sigma_3+ \tfrac{1}{2}A_1^{\text{JM}}, \label{eq:JMlaxpair2}
\end{align}
\end{subequations}
in which \eqref{eq:JMlaxpair2} defines an isomonodromic deformation of $\eqref{eq:JMlaxpair1}$ with respect to $t$.
The matrix $A^{\text{JM}}$ is characterised analytically by the conditions that
\begin{equation*}
    \operatorname{Tr}A_0^{\text{JM}}= \operatorname{Tr}A_1^{\text{JM}}=0,\quad |A^{\text{JM}}|=- \left( z^4+tz^2+2\theta z\right)+\mathcal{O}(1)\quad (z\rightarrow \infty),
\end{equation*}
for fixed $\theta$.
Parametrising the matrices by $f,g,k$ according to 
\begin{equation*}
A_1^{\text{JM}}=\begin{bmatrix}
    0 & k\\
    -2 g/k & 0\\
\end{bmatrix},\qquad
A_0^{\text{JM}}=\begin{bmatrix}
    g+\frac{1}{2}t &-f\,k\\
    -2(f g - \theta)/k & -g-\frac{1}{2}t
\end{bmatrix},
\end{equation*}
the compatibility of gives the system \eqref{eq:hamPII} for $f,g$, as well as
\begin{equation*}
    \frac{k_t}{k}=-f.
\end{equation*}

\subsubsection{The Flaschka-Newell linear problem}
In \cite{flaschkanewell}, Flaschka and Newell derived the Lax pair
\begin{subequations} \label{eq:FNlaxpair}
    \begin{align}
    Y_z&=A^{\text{FN}}Y, & A^{\text{FN}}&=4z^2\sigma_3+zA_1^{\text{FN}}+A_0^{\text{FN}}+z^{-1}A_{-1}^{\text{FN}} \label{eq:FNlaxA}\\
    Y_t&=B^{\text{FN}}Y, & B^{\text{FN}}&=-z\,\sigma_3-\tfrac{1}{4}A_1^{\text{FN}},
\end{align}
\end{subequations}
where $A^{\operatorname{FN}}$ is characterised by
\begin{equation*}
    \operatorname{Tr}A^{\text{FN}}=0,\quad \operatorname{Spec}(A_{-1}^{\operatorname{FN}}) = \{\theta+\tfrac{1}{2}, -\theta-\tfrac{1}{2}\},\quad |A^{\text{FN}}|=-16\,z^4+8\,t\,z^2+\mathcal{O}(1)\quad (z\rightarrow \infty)
\end{equation*}
and 
\begin{equation}\label{eq:Asymmetry}
    A^{\text{FN}}(-z)=\sigma_1 A^{\text{FN}}(z)\sigma_1.
\end{equation}
Parametrising the matrices by $y,r$ according to
\begin{equation*}
A_1^{\text{FN}}=-4\,y\, \sigma_1,\quad
A_0^{\text{FN}}=\begin{bmatrix}
    -t-2y^2 & 2r\\
    -2 r & t+2y^2
\end{bmatrix},\quad A_{-1}^{\text{FN}}=-(\theta+\tfrac{1}{2})\sigma_1,
\end{equation*}
leads to the compatibility conditions
\begin{equation*}
    y_t=r,\qquad r_t=t\,y+2y^3+\tfrac{1}{2}+\theta,
\end{equation*}
which is equivalent to $\pain{II}$ for $y$.

\subsection{Derivation of monodromy surface} \label{subsec:spaceofmonodromydataPII}

We now derive the monodromy surface for $\pain{II}$ as an embedded affine del Pezzo surface of degree six with a triangle of conics at infinity.
We will do this here using the Jimbo-Miwa linear problem \eqref{eq:JMlaxpair}, and provide the derivation of the same monodromy surface from the Flaschka-Newell linear problem in \Cref{app:derivationofmonodromysurfaceFN}.

Begin with the unique formal solution of the linear system \eqref{eq:JMlaxpair1} of the form
\begin{equation}\label{eq:formalsolPII}
    Y_{\text{form}}(z)=P(z)e^{(\frac{1}{3}z^3+\frac{1}{2}tz)\sigma_3}z^{\theta \sigma_3},
\end{equation}
where $P(z)$ is a power series around $z=\infty$,
\begin{equation}\label{eq:matrixPasymptoticPII}
    P(z)=I+\sum_{n=1}^\infty z^{-n}U_n.
\end{equation}
Define Stokes sectors by
\begin{equation*}
    \Sigma_k=\left\{\left|\arg z- \frac{(k-1)\pi}{3}\right|<\frac{\pi}{6}\right\}\subseteq \widetilde{\mathbb{C}^*}, \qquad (k \in \Z)
\end{equation*}
within which $e^{\frac{1}{3}z^3+\frac{1}{2}tz}$ is alternately exponentially small or large as $z\to \infty$, corresponding to even and odd $k$ respectively.

For any $k\in\mathbb{Z}$, there exists a unique solution $Y_k$ of the linear problem that satisfies
\begin{equation*}
    Y_k(z)\sim Y_\text{form}(z)\qquad (z\in \Sigma_k\cup \Sigma_{k+1}, z\rightarrow \infty),
\end{equation*}
and we have the Stokes phenomenon
\begin{equation*}
    Y_{k+1}(z)=Y_k(z)S_k,
\end{equation*}
where
\begin{equation*}
    S_k=\begin{bmatrix}
        1 & 0\\
        s_k & 1
    \end{bmatrix}\,\,\, \text{if $k$ even},\qquad
    S_k=\begin{bmatrix}
        1 & s_k\\
        0 & 1
    \end{bmatrix}\,\,\, \text{if $k$ odd}.
\end{equation*}
The $Y_k(z)$ are single-valued analytic matrix functions on $\mathbb{C}$, thus we have the cyclic property
\begin{equation*}
    Y_{k+6}(z)=Y_k(z)e^{2\pi i \theta \sigma_3}.
\end{equation*}
and correspondingly
\begin{align}
    &S_{k+6}=e^{-2\pi i \theta \sigma_3}S_k e^{2\pi i \theta \sigma_3},\label{eq:stokesPIIeqscycle}\\ 
    &S_k\cdot S_{k+1}\cdot\ldots\cdot S_{k+5}=e^{2\pi i \theta\sigma_3},\label{eq:stokesPIIeqs}
\end{align}
for $k\in\mathbb{Z}$. By the cyclic relation \eqref{eq:stokesPIIeqscycle}, equation \eqref{eq:stokesPIIeqs} is equivalent to the same set of four conditions on the Stokes multipliers for any choice of $k\in\mathbb{Z}$, which is the vanishing of the following four polynomials, in which 
$u=e^{\pi i \theta}$,
\begin{equation} \label{eq:idealpiv}
    \begin{aligned}
        i_1 &= 1-u^{-2}+ s_2 s_3+s_2 s_5+s_4 s_5+s_2 s_3 s_4 s_5 ,\\
        i_2 &= s_1+s_3+s_5 + s_1 s_2 s_3 +s_1 s_2 s_5 +s_1  s_4 s_5 +s_3 s_4 s_5   +s_1 s_2 s_3 s_4 s_5 ,\\
        i_3 &= s_2+s_4+s_6+ s_2 s_3 s_4 +s_2 s_3 s_6  +s_2 s_5 s_6 +s_4 s_5 s_6  +s_2 s_3 s_4 s_5 s_6, \\
        i_4 &= 1-u^{+2} +s_1 s_2+s_3 s_4 +s_1 s_4+s_1 s_6+s_3 s_6 +s_5 s_6\\
        & \;\;\,\,+s_1 s_2 s_3 s_4+s_1 s_2 s_3 s_6+s_1 s_2 s_5 s_6+s_1 s_4 s_5 s_6+s_3 s_4 s_5 s_6+s_1 s_2 s_3 s_4 s_5 s_6.
    \end{aligned}
\end{equation}
 We interpret the set of monodromy data for the linear problem \eqref{eq:JMlaxpair1} with $\theta$ fixed as the affine variety 
\begin{equation}\label{eq:monodromyspaceMPII} 
    M_u = \operatorname{Spec} \C [s_1,\dots,s_6]/I_u,
\end{equation}
where $I_u = (i_1,i_2,i_3,i_4)$ the ideal generated by the four polynomials in \eqref{eq:idealpiv}.

We note that the generators $i_1,i_2,i_3,i_4$ are algebraically dependent, satisfying the relation
\begin{equation*}
    i_2 i_3 = i_1 i_4 + u^{+2} i_1 + u^{-2} i_4,
\end{equation*}
and that the dimension of $M_u$ is three.

Rescaling of the auxiliary variable $k$ additional to $f,g$ in the parametrisation of $A^{\operatorname{JM}}$,
\begin{equation*}
    k\mapsto \widetilde{k}=c\, k, \qquad c\in \C^*,
\end{equation*}
leads to the following scaling of the coefficient matrix, fundamental solutions and Stokes matrices respectively,
\begin{align*}
    &A^{\operatorname{JM}}\mapsto \widetilde{A}^{\operatorname{JM}}=c^{\tfrac{1}{2}\sigma_3}Ac^{-\tfrac{1}{2}\sigma_3},\\
    &Y_{k}\mapsto \widetilde{Y}_k=c^{\tfrac{1}{2}\sigma_3}Y_kc^{-\tfrac{1}{2}\sigma_3},\\
    &S_k\mapsto \widetilde{S}_k=c^{\tfrac{1}{2}\sigma_3}S_k c^{-\tfrac{1}{2}\sigma_3}.
\end{align*}
This induces the $\C^*$-action on $M_u$ given by 
\begin{equation*}
    s_{2n}\mapsto \widetilde{s}_{2n}=c^{-1}\, s_{2n},\quad s_{2n-1}\mapsto \widetilde{s}_{2n-1}=c^{+1} s_{2n-1},\qquad (1\leq n\leq 3,\,c\in\mathbb{C}^*).
\end{equation*}
To take the GIT quotient by this $\C^*$-action, we note that the ring of invariants in $R=\C[s_1,\dots,s_6]/I_u$ under the $\C^*$-action is generated by 
\begin{equation} \label{eq:ytosPII}
    \begin{aligned}
    y_1 &= s_1 s_2, &\quad &y_2 = s_1 s_4, &\quad &y_3 = s_1 s_6, \\
    y_4 &= s_2 s_3, &\quad &y_5 = s_3 s_4, &\quad &y_6 = s_3 s_6, \\
    y_7 &= s_2 s_5, &\quad &y_8 = s_4 s_5, &\quad &y_9 = s_5 s_6,
    \end{aligned}
\end{equation}
among which there are the following three linear relations in $R$,
\begin{equation*}
    y_1+1=u^{+2}( y_8+1),\quad y_5+1=u^{-2} (y_3+ u^{4}) , \quad y_9+1=u^{+2}(y_4+1).
\end{equation*}
Introducing 
\begin{equation}\label{eq:xtoyPII}
    \begin{aligned}
    x_1&=u^{+1} \left(y_8+1\right), & \quad 
    x_2&=u^{-1} \left(y_9+1\right), &\quad 
    x_3&=u^{-1}\left(y_5+1\right),
     \\
    x_4&=1-u^{-2} y_6,&\quad 
    x_5&= 1- u^{-2}y_2,&\quad 
    x_6&=1-u^{+2} y_7, 
    \end{aligned}
\end{equation}
the ring of invariants $R^{\C^*}$ is realised as 
\begin{equation}\label{eq:ringofinvariants}
    R^{\C^*} \simeq \C[x_1,\dots,x_6]/K_w,
\end{equation}
where $K_w=(k_0,k_1,k_2,k_3)$ is generated by
\begin{equation*}
    \begin{aligned}
        k_0 &= x_1 x_2 x_3 - x_1 -x_2 -x_3  + w, \\
        k_1 &= x_4 - x_2 x_3, \\
        k_2 &= x_5 - x_1 x_3, \\
        k_3 &= x_6 - x_1 x_2, 
    \end{aligned}
\end{equation*}
in which
\begin{equation*}
    w=u + u^{-1}=2\cos(\pi \theta) ,\qquad u=e^{\pi i \theta}.
\end{equation*}
We then arrive at the following description of the monodromy surface.

\begin{definition} \label{def:monodromysurfacePII}
    The monodromy surface of $\pain{II}$ is the embedded affine variety
    \begin{equation*}
    \begin{aligned}
        \mathcal{M}_w &= \Spec \C[x_1,x_2,x_3,x_4,x_5,x_6]/K_{w}, \\  
        K_w &= (x_1 x_2 x_3 - x_1 -x_2 -x_3  + w, \,\,
        x_4 - x_2 x_3,\,\,
        x_5 - x_1 x_3, \,\,
        x_6 - x_1 x_2),
    \end{aligned}
    \end{equation*}
    where $w\in \mathscr{W}_{\rm{II}}:= \left\{ w\in \C\right\}$.
    This gives the family
    \begin{equation*}
        \mathcal{M}\to\mathscr{W}_{\rm{II}},
    \end{equation*}
    with fibre over any $w\in\mathscr{W}_{\rm{II}}$ being $\mathcal{M}_{w}$.
\end{definition}

Under the embedding 
\begin{equation} \label{eq:embeddingPII}
    \begin{aligned}
        \mathbb{A}^6 &\rightarrow \p^6,\\
        (x_1,x_2,x_3,x_4,x_5,x_6) &\mapsto [1:x_1:x_2:x_3:x_4:x_5:x_6],
    \end{aligned}
\end{equation}
the projective completion of $\mathcal{M}_w$ can be described as follows.
\begin{lemma} \label{lem:projectivecompletionPII}
The projective completion of $\mathcal{M}_w$ under  \eqref{eq:embeddingPII} is given by
\begin{equation*}
    \begin{gathered}
    \overline{\mathcal{M}}_{w} = \operatorname{Proj} \C[X_0,X_1,\dots,X_6]/\overline{K}_w,      \qquad \operatorname{deg}(X_i)=1, \qquad \overline{K}_w= \left( \bar{k}_1,\dots,\bar{k}_9\right),
    \end{gathered}
\end{equation*} 
where
\begin{equation*}
            \begin{aligned}
            \bar{k}_1 &= X_0 X_4- X_2 X_3,\\
            \bar{k}_2 &= X_0 X_5- X_1 X_3,\\
            \bar{k}_3 &= X_0 X_6- X_1 X_2,\\
            \bar{k}_4 &= w X_0^2-X_0 X_1 -X_0 X_2 - X_0 X_3 +X_1 X_4,\\
            \bar{k}_5 &= w X_0^2-X_0 X_1 -X_0 X_2 - X_0 X_3 +X_2 X_5,\\
            \bar{k}_6 &= w X_0^2-X_0 X_1 -X_0 X_2 - X_0 X_3 +X_3 X_6,\\
            \bar{k}_7 &= w X_0 X_1-X_1^2-X_0 X_5-X_0 X_6+X_5 X_6, \\
            \bar{k}_8 &= w X_0 X_2-X_2^2-X_0 X_4-X_0 X_6+X_4 X_6, \\
            \bar{k}_9 &= w X_0 X_3-X_3^2-X_0 X_4-X_0 X_5+X_4 X_5.
        \end{aligned}
\end{equation*}
The hyperplane section at infinity is the union of the three conics 
  \begin{equation} \label{eq:conicsatinfinityPII}
        \begin{aligned}
            C^{\infty}_1 &: \quad X_0 = X_1=X_2=X_6 = 0 ,\quad X_4 X_5 = X_3^2, \\
            C^{\infty}_2 &: \quad X_0 = X_1=X_3=X_5 = 0 ,\quad X_4 X_6 = X_2^2, \\
            C^{\infty}_3 &: \quad X_0 = X_2=X_3=X_4 = 0 ,\quad X_5 X_6 = X_1^2. 
        \end{aligned}
        \end{equation}
\end{lemma}
\begin{proof} 
Let $\overline{K}$ denote the homogenisation of the ideal $K$. To prove the lemma, it needs to be shown that $\overline{K}$  is generated by $\overline{k}_1,\ldots,\overline{k}_9$ and that $\operatorname{Proj} \C[X_0,X_1,\dots,X_6]/\overline{K}_w$ is the closure of $\mathcal{M}_w$.  
Let $k_i$ denote $\overline{k}_i$ with $X_0=1$ and $X_j=x_j$ for $1\leq j\leq 6$, $1\leq i\leq 9$.
By definition $k_i\in K$ for $0\leq i\leq 3$. Furthermore,
\begin{equation*}
\begin{aligned}
    k_4&=k_0+x_1 k_1,\\
    k_5&=k_0+x_2 k_2,\\
    k_6&=k_0+x_3 k_3,\\
    k_7&= x_1 (k_0+x_2 k_2)+ (1-x_5) k_3+k_2,    \\
    k_8&= x_2 (k_0+x_3 k_3)+ (1-x_6) k_1+k_3,    \\
    k_9&= x_3 (k_0+x_1 k_1)+ (1-x_4) k_2+k_1.
\end{aligned}
\end{equation*}
hence $(\overline{k}_1,\ldots,\overline{k}_9)\subset \overline{K}_w$.
Restricting to the affine chart $X_0\neq0$, the equations for $\overline{\mathcal{M}}_w$ reduce to those of $\mathcal{M}_w$. Note furthermore, that $\mathcal{M}_w$ is irreducible, since, as an affine variety, it is isomorphic to the irreducible cubic surface $\operatorname{Spec}\C[x_1,x_2,x_3]/(x_1 x_2 x_3 - x_1 -x_2 -x_3  + w)\}$.

By direct calculation, the complement of $\mathcal{M}_w$ in $\overline{\mathcal{M}}_w$ is the union of the three curves \eqref{eq:conicsatinfinityPII}, and in particular is of dimension one. 
It follows that $\overline{\mathcal{M}}_w$ is irreducible and thus the projective closure of $\mathcal{M}_w$.
\end{proof}

\begin{remark} \label{rem:propertiesofmonodromysurfacePII}
We have the following properties of $\mathcal{M}_w$.
    \begin{itemize}
        \item The three conics at infinity \eqref{eq:conicsatinfinityPII}
        pairwise intersect forming a triangle, and consist only of smooth points of $\overline{\mathcal{M}}_w$.
        \item The singular locus $\mathscr{W}^{\operatorname{sing}}_{\rm{II}}\subset \mathscr{W}_{\rm{II}}$, where $\mathcal{M}_w$ is singular, is given by $$w^2=4.$$
    \end{itemize}
\end{remark}

The representation in \Cref{lem:projectivecompletionPII} of $\overline{\mathcal{M}}_w$ as an intersection of nine quadrics in $\p^6$ is related to the fact that it, like the projective completions of monodromy surfaces of $\pain{VI}$ and $\pain{IV}$ above, is a del Pezzo surface. 
This will be used below when we define the category of monodromy surfaces for $\pain{II}$ so we establish it in the following lemma.

\begin{lemma}\label{lem:delpezzodeg6PII}
    For $w \in \mathscr{W}_{\rm{II}}\setminus \mathscr{W}_{\rm{II}}^{\operatorname{sing}}$, the surface $\overline{\mathcal{M}}_w$ in \Cref{lem:projectivecompletionPII} is a smooth del Pezzo surface of degree 6. When $w \in  \mathscr{W}_{\rm{II}}^{\operatorname{sing}}$, it is a singular del Pezzo surface of degree 6 of type $(i')$ in the notation of \cite{dolgachevclassicalAG} containing three lines which meet at its single singularity of type $A_1$.
    
\end{lemma}
\begin{proof}
    We will relate the representation of $\overline{\mathcal{M}}_w$ in \Cref{lem:projectivecompletionPII} to the standard model of a smooth del Pezzo surface of degree 6 under its anticanonical embedding in $\p^6$, see \cite[Th. 8.4.9]{dolgachevclassicalAG}.
    Recall that this is given by the vanishing locus of the nine $2\times2$ minors $m_1,\dots,m_9$ of the matrix
    $$     \begin{bmatrix}
        g & d & b\\
        c & g & a\\
        e & f & g
    \end{bmatrix}.$$ 
    Denoting this by 
    \begin{equation} \label{eq:standardmodeldepPezzoPII}
        \overline{\mathcal{S}}=\Proj \C[a,b,c,d,e,f,g] / (m_1,\dots,m_9)\subset \p^6,
    \end{equation}
    we have a projective equivalence between $\overline{\mathcal{S}}$ and $\overline{\mathcal{M}}_w$, given by
    \begin{equation}\label{eq:projectivityStoMPII}
    \begin{bmatrix}
    a\\
    b\\
    c\\
    d\\
    e\\
    f\\
    g
    \end{bmatrix}
    = \begin{bmatrix}
 u+2u^{-1} & -u^{-2}-1 & -2 & -2 & u & u^{-1} & u^{-1} \\
 2 u+u^{-1} & -2 & -2 & -u^2-1 & u & u & u^{-1} \\
 u^2+2 & -2 u & -u-u^{-1} & -2 u & 1 & u^2 & 1 \\
 2 u+u^{-1} & -u^2-1 & -2 & -2 & u^{-1} & u & u \\
 u+2u^{-1} & -2 & -2 & -u^{-2}-1 & u^{-1} & u^{-1} & u \\
 u^{-2}+2 & -2 u^{-1} & -u-u^{-1} & -2 u^{-1} & 1 & u^{-2} & 1 \\
 u^2+u^{-2}+1 & -u - u^{-1} & -u - u^{-1} & -u-u^{-1} & 1 & 1 & 1 \\
\end{bmatrix}
    \begin{bmatrix}
    X_0\\
    X_1\\
    X_2\\
    X_3\\
    X_4\\
    X_5\\
    X_6
    \end{bmatrix},
    \end{equation}
    where again $u=e^{i \pi \theta}$, so $w=u+u^{-1}$, which can be verified by direct calculation\footnote{We obtained this map by starting with a general projectivity and imposing that the lines and their points of intersection on $\mathcal{M}_w$ get mapped to those on $\mathcal{S}$ and a further $4$ points on $\mathcal{M}_w$ get mapped into $\mathcal{S}$.}. The determinant of the matrix on the right-hand side is $(u^{-1}-u)^{11}$ and thus nonzero for $w \in \mathscr{W}_{\rm{II}}\setminus \mathscr{W}_{\rm{II}}^{\operatorname{sing}}$, and the first part of the lemma follows.

    For the second part, we use the following standard form of a degree 6 singular del Pezzo surface with three lines intersecting at its single singularity of type $A_1$:
    \begin{equation} \label{eq:standardmodeldepPezzoPIIsing}
        \overline{\mathcal{S}}_{\operatorname{sing}} =\Proj \C[a,b,c,d,e,f,g] / (m_1,\dots,m_9)\subset \p^6,
    \end{equation}
    where 
    \begin{equation*}
        \begin{aligned}
        m_1 &= a g - b f, \\
        m_2 &= a b -e g, \\
        m_3 &= a c -b g , 
        \end{aligned}
        \qquad
        \begin{aligned}
            m_4 &= a^2- e f , \\
        m_5 &= b^2-c  e, \\
        m_6 &= g^2 - c f
        \end{aligned}
        \qquad 
        \begin{aligned}
        m_7 &= a d - c f +  f g  , \\
        m_8 &= b d - c g  +  c f, \\
        m_9 &= d e - b g  +  b f  .
        \end{aligned}
    \end{equation*}
This model is obtained by considering the vector space of cubic forms in $Z_0,Z_1,Z_2$ vanishing at the three collinear points $[0:0:1], [0:1:0],[0:1:1]$ in $\p^2$, and will be used below in  \Cref{prop:oblomkovstylenormalformPII}.

    We have a projective equivalence between $\overline{\mathcal{S}}_{\operatorname{sing}}$ and $\overline{\mathcal{M}}_{w}$, $w=\pm 2$, given by
    \begin{equation} \label{eq:projectivityStoMsingPII}
        \begin{bmatrix}
    a\\
    b\\
    c\\
    d\\
    e\\
    f\\
    g
    \end{bmatrix}
    = 
\begin{bmatrix}
 -1 & \pm 1 & 0 & \pm 1 & 0 & -1 & 0 \\
 1 & \mp 1 & \mp 1 & 0 & 0 & 0 & 1 \\
 -1 & 0 & 0 & 0 & 0 & 0 & 1 \\
 0 & \pm 1 & 0 & 0 & 0 & 0 & 0 \\
 3 & \mp 2 & \mp 2 & \mp 2 & 1 & 1 & 1 \\
 -1 & 0 & 0 & 0 & 0 & 1 & 0 \\
 -1 & \pm 1 & 0 & 0 & 0 & 0 & 0 \\
\end{bmatrix}
    \begin{bmatrix}
    X_0\\
    X_1\\
    X_2\\
    X_3\\
    X_4\\
    X_5\\
    X_6
    \end{bmatrix}.
    \end{equation}
    The determinant of the matrix on the right-hand side is $\pm 1$, which finishes the proof of the second part of the lemma.
\end{proof}

Note that $\mathcal{M}_w$ is isomorphic, as an affine variety, to
\begin{equation*}
    \begin{aligned}
        \mathcal{C}_w &= \operatorname{Spec}\C[x_1,x_2,x_3]/(k_0),\\
        k_0&=x_1 x_2 x_3 - x_1 -x_2 -x_3  + w,
    \end{aligned}
\end{equation*}
which is the model of the monodromy surface for the Jimbo-Miwa linear problem for $\pain{II}$ as an affine cubic surface as in \cites{putsaito,chekhovdecorated}, but with a different scaling. see \Cref{rem:reltovdPScubicPII}. 
The isomorphism is provided by the projection 
\begin{equation} \label{eq:MwtocubicPII}
    \begin{aligned}
    \rho : \mathcal{M}_w &\rightarrow \mathcal{C}_w,\\
    \left(x_1,x_2,x_3,x_4,x_5,x_6\right) &\mapsto \left(x_1,x_2,x_3\right).
    \end{aligned}
\end{equation}
The projective completion $\overline{\mathcal{C}}_w$ has $A_1$ singularities at the three corners of a triangle of lines at infinity, see \cite[Table 4]{JMR}.

\begin{remark} \label{rem:reltovdPScubicPII}
    As in the $\pain{IV}$ case, we have a slightly different parametrisation of the affine cubic to that appearing in \cite{putsaito}, which we write as 
    \begin{equation*}
        \tilde{x}_1\tilde{x}_2\tilde{x}_3-\tilde{x}_1-\alpha  \tilde{x}_2 - \tilde{x}_3 + \alpha+1=0, \quad \alpha\in\C^*.
    \end{equation*}
    This can be related to $\mathcal{C}_w$ according to 
    $$ \tilde{x}_1=u x_1, \quad \tilde{x}_2= u^{-1} x_2, \quad \tilde{x}_3= u x_3,$$
    with parameter correspondence
    $$\alpha= u^2.$$
    We choose our particular normalisation since $w = u+u^{-1}$ is invariant under the action of ${W}(A_1^{(1)})$ by B\"acklund transformations.
    
\end{remark}

\subsection{Riemann-Hilbert map and symmetries}

We again write the Riemann-Hilbert map as 
\begin{equation*}
    \mathcal{X}_{t,a} \xrightarrow{\operatorname{RH}_{t,a}} \mathcal{M}_{w(a)}.
\end{equation*}
This is defined in the obvious analogous way to the $\pain{IV}$ and $\pain{VI}$ cases above, and is a biholomorphism for $a$ such that $a_1=\theta+1 \notin \Z$, see \cites{fokas,putsaito}.
The Riemann-Hilbert map associated to the Flaschka-Newell linear problem is also biholomorphic \cites{flaschkanewell,fokas, putsaito}, which will be relevant below, see \Cref{rem:commutingRHsPII}.

The map $\mathscr{A}_{\rm{II}}\cong\Theta_{\rm{II}}\to\mathscr{W}_{\rm{II}}$ is defined as above via the intermediate parameter $u = e^{\pi i \theta}$, which we recall is
\begin{equation*} 
    w = u+u^{-1} = 2 \cos (\pi \theta)
    .
\end{equation*}
We describe the result of conjugating the symmetries of $\pain{II}$ as given in \Cref{tab:PII:symmetry:varsandparams} by the Riemann-Hilbert map in the following theorem.

\begin{theorem} \label{thm:symmetryunderRHPII}
    For any $g\in W(A_1^{(1)})$, the corresponding symmetry of $\pain{II}$ conjugates to the identity under the Riemann-Hilbert map.
    The nontrivial Dynkin diagram automorphism $\sigma$ and the additional symmetries in $\langle\vartheta\rangle \cong\Z/3\Z$ act nontrivially, as described in \Cref{tab:PII:symmetriesunderRH}.
\end{theorem}
\begin{proof}
    We will lift the actions of $r_0$, $T$, and $\vartheta$ to the linear problem and compute the corresponding action on monodromy, which suffices to establish the theorem since $T=\sigma r_0 = r_1\sigma$.
    
    First, $r_0$ is realised on the linear problem \eqref{eq:JMlaxpair1} by 
    $$Y(z)\mapsto \widetilde{Y}(z)=\sigma_1 Y(-z),$$
    which transforms the coefficient matrix according to 
    $$A(z)\mapsto \widetilde{A}(z) = - \sigma_1 A(-z)\sigma_1,$$
    corresponding to the action of $r_0$ on $f,g$ as in \Cref{tab:PII:symmetry:varsandparams} as well as 
    $$ k \mapsto \tilde{k}= -\frac{2 g}{k}.$$
    The formal solution transforms as 
    $$ \widetilde{Y}_{\operatorname{form}}(z) = \sigma_1 Y_{\operatorname{form}}(e^{-\pi i}z)\sigma_1 e^{-\pi i \theta \sigma_3},$$
    and the solutions $Y_k$ transform as
 $$\widetilde{Y}_{k+3}(z) = \sigma_1 Y_k(e^{-\pi i }z)\sigma_1 e^{-\pi i \theta \sigma_3}.$$
Then the new Stokes matrices are
$$\widetilde{S}_{k+3}=e^{\pi i \theta \sigma_3}\sigma_1 S_k\sigma_1 e^{-\pi i \theta \sigma_3}$$
so that
$$\tilde{s}_{k+3}=\begin{cases}
    s_k e^{+2\pi i\theta} &\text{if $k$ even,}\\
    s_k e^{-2\pi i\theta} &\text{if $k$ odd.}
\end{cases}$$
    Then $x_i$, $1\leq i\leq 6$, defined in terms of Stokes multipliers through \eqref{eq:ytosPII} and \eqref{eq:xtoyPII} transform trivially as claimed.    
    The action of $r_0$ on $\theta$ leaves $w$ invariant and we have established the action of $r_0$ as in \Cref{tab:PII:symmetriesunderRH}.

    Next we lift the translation $T = \sigma r_0$ to the linear problem as a Schlesinger transformation, which can be found in \cite[Sec. 6.1]{fokas}.
    The corresponding Schlesinger transformation for $T$ takes the form
\begin{equation*}
    Y(z)\mapsto \widetilde{Y}(z)=G(z)Y(z),\quad A(z)\mapsto \widetilde{A}(z)=G(z)A(z)G(z)^{-1}+G_z(z)G(z)^{-1},
\end{equation*}
where $G(z)$ is given by 
\begin{equation*}
    G(z)=z\begin{bmatrix}
        0 & 0\\
        0 & 1
    \end{bmatrix}+\begin{bmatrix}
        0 & -\frac{k}{f} \\
 \frac{g}{k} & \frac{\theta }{g}-f
    \end{bmatrix},
\end{equation*}
which corresponds to the action of $T$ on $f,g$ obtained by composing those of $\sigma$ and $r_0$ in \Cref{tab:PII:symmetry:varsandparams}, as well as $k\mapsto\tilde{k}=\frac{2g}{k}$.
The solutions $Y_k(z)$ transform according to 
$$ \widetilde{Y}_k(z)=G(z) Y_k(z)$$ 
and the Stokes multipliers are left completely invariant.
The induced action on $x_i$, $1\leq i\leq 6$, is given by 
$$\tilde{x}_1 = - x_1, \quad \tilde{x}_2=-x_2, \quad \tilde{x}_3 = x_3, \quad \tilde{x}_4=x_4, \quad \tilde{x}_5=x_5, \quad \tilde{x}_6=x_6.$$
The action on $w$ can be computed from that of $T$ on $\theta$, namely $\tilde{\theta}=\theta-1$.
Since $r_0$ acts trivially, this allows us to deduce the action of $\sigma = T r_0$ given  in \Cref{tab:PII:symmetriesunderRH}.
This also shows that the action of $r_1$ is trivial, since $r_1 =T \sigma$.

Lastly, we derive the action of $\vartheta$. 
Consider the transformation
$$Y(z) \mapsto \widetilde{Y}(z)= Y(e^{-\frac{2\pi i}{3}}z),\quad A(z)\mapsto \widetilde{A}(z) = e^{-\frac{2\pi i}{3}}A(e^{-\frac{2\pi i }{3}}z),$$
which corresponds to the action of $\vartheta$ on $f,g$ as in \Cref{tab:PII:symmetry:varsandparams} as well as $k \mapsto \tilde{k}= e^{-\frac{2 \pi i}{3}} k.$
The formal solution transforms as 
    $$ \widetilde{Y}_{\operatorname{form}}(z) =  Y_{\operatorname{form}}(e^{-\frac{2\pi i}{3}}z) e^{\frac{2\pi i}{3} \theta \sigma_3},$$
    and the solutions $Y_k$ transform as
 $$\widetilde{Y}_{k+2}(z) = Y_{k}(e^{\frac{-2\pi i}{3} }z) e^{\frac{2\pi i}{3} \theta \sigma_3},$$
so the new Stokes matrices are 
$$\widetilde{S}_{k+2} = e^{-\frac{2\pi i}{3} \theta \sigma_3} S_k e^{\frac{2\pi i}{3} \theta \sigma_3}.$$
Accordingly, the Stokes multipliers transform as 
$$\tilde{s}_{k+2} = \begin{cases}
    s_k e^{+\frac{4\pi i}{3}\theta} &\text{if $k$ even,}\\
    s_k e^{-\frac{4\pi i}{3}\theta} &\text{if $k$ odd.}
\end{cases}$$
Then the invariants are transformed as
$$\tilde{y}_{1} = y_9, \quad 
\tilde{y}_2 = u^4 y_7, \quad
\tilde{y}_3 = u^4 y_8, \quad 
\tilde{y}_4 = { u^{-4} y_3}, \quad 
\tilde{y}_5 = y_1, \quad 
\tilde{y}_6 = y_2, \quad 
\tilde{y}_7 = { u^{-4} y_6}, \quad 
\tilde{y}_8 = y_4, \quad 
\tilde{y}_9 = y_5,
$$
which induces via the equations \eqref{eq:xtoyPII} the action of $\vartheta$ on $x_i$, $1\leq i\leq 6$, as given in \Cref{tab:PII:symmetriesunderRH}.
\end{proof}

\begingroup
\setlength{\tabcolsep}{15pt} 
\renewcommand{\arraystretch}{1.5} 
\begin{table}[h]
    \begin{equation*}
    \begin{array}{c||c|c|c|c|c|c||c|}
              & \tilde{x}_1    & \tilde{x}_2     & \tilde{x}_3    & \tilde{x}_4     & \tilde{x}_5   & \tilde{x}_6& \tilde{w}           \\
        \hline\hline 
        r_0   &   x_1    &  x_2    &  x_3 &    x_4 & x_5   &  x_6 & w             \\
        \hline
        r_1  &   x_1    &  x_2    &  x_3 &    x_4 & x_5   &  x_6 & w        \\
        \hline\hline
        \sigma   &   -x_1    &   -x_2    & -x_3   & x_4    &  x_5 &  x_6 & -w         
        \\
        \hline\hline
        \vartheta   &  x_2  &   x_3    & x_1 & x_5 & x_6  & x_4   &      w\\
        \hline
    \end{array}
    \end{equation*}
    \caption{Action of symmetries of $\pain{II}$ on $\mathcal{M}\rightarrow\mathscr{W}_{\rm{II}}$ under the Riemann-Hilbert map}
    \label{tab:PII:symmetriesunderRH}
\end{table}
\endgroup

\begin{remark} \label{rem:symmetriesonlinearproblemFNPII}
    The action of symmetries of $\pain{II}$ on monodromy data for the Flaschka-Newell linear problem is given in \cite[Eq. (2.6)]{kapaev91_essential}. 
    Via the identification of this monodromy data with points on $\mathcal{M}_w$ as detailed in \Cref{app:derivationofmonodromysurfaceFN}, we have the same action of $r_0,r_1,\sigma,\vartheta$ as given in \Cref{tab:PII:symmetriesunderRH}.

\end{remark}

\begin{remark} \label{rem:commutingRHsPII}
    While we show that the Jimbo-Miwa and Flaschka-Newell linear problems lead to the same monodromy surface, we have not established the relationship between the associated Riemann-Hilbert maps, which seems to be an open problem \cite[Sec. 2.2]{millerincreasing18}. 
    This problem can be translated into describing the induced map $\Upsilon$ from $\mathcal{M}_w$ to itself in the following commutative diagram:
    \begin{equation*}
    \begin{tikzcd}
    & \mathcal{M}_{w(a)} \arrow[dd,dotted, "\Upsilon"]\\
        \mathcal{X}_{t,a}  \arrow[ur, "\operatorname{RH}^{\operatorname{JM}}_{t,a}"] \arrow[dr,swap, "\operatorname{RH}^{\operatorname{FN}}_{t,a}"]& \\
    & \mathcal{M}_{w(a)}         
    \end{tikzcd}
\end{equation*}
We conjecture it to be the identity map, up to cyclic relabeling of coordinates (so as to be compatible with the action of $\vartheta$).

Note that for the increasing tritronqu\'ee solution $y=y(t;\theta)$ of $\pain{II}$ uniquely characterised by the asymptotic behaviour
\begin{equation*}
    y(t;\theta)\sim -\frac{1}{2}\sqrt{2t} \qquad \text{as }t\rightarrow \infty\text{ with }|\arg t|<\frac{2\pi}{3},
\end{equation*}
the corresponding monodromy data of both the Jimbo-Miwa linear problem  and the Flaschka-Newell linear problem  are available in the literature. 
From \cite{millerincreasing18} we see that for this solution $s_1^{\operatorname{JM}}=s_6^{\operatorname{JM}}=0$ and from \cite[Sec. 11.5]{fokas}\footnote{We note the correspondence $(s_1^{\operatorname{FN}},s_2^{\operatorname{FN}},s_3^{\operatorname{FN}})=(is_1^{\operatorname{FIKN}},-is_2^{\operatorname{FIKN}},is_3^{\operatorname{FIKN}})$ between our Stokes multipliers for the Flaschka-Newell linear problem and those for the gauge-transformed one used in \cite[Sec. 11]{fokas}.} we obtain $(s_1^{\operatorname{FN}},s_2^{\operatorname{FN}},s_3^{\operatorname{FN}})=(i u^{-1},iu,iu^{-1})$, so that in both cases the corresponding point on $\mathcal{M}_w$ is
\begin{equation*}
    x_1=u^{-1},\quad x_2=u,\quad x_3=u^{-1},\quad x_4=1,\quad
    x_5=u^{-2},\quad
    x_6=1.
\end{equation*}
This allows us to refine the conjecture to state that the induced map $\Upsilon$ is simply the identity.
\end{remark}

\subsection{Moduli space}

For $\pain{II}$, the category of monodromy surfaces is again one of embedded affine surfaces whose projective completions are del Pezzo of a certain degree, characterised by a special hyperplane section at infinity.

\begin{definition}[Category $\mathfrak{C}_{\rm{II}}$ of monodromy surfaces for $\pain{II}$]
Define the category $\mathfrak{C}_{\rm{II}}$ with 
\begin{itemize}
    \item an object being an embedded affine surface $\mathcal{V}\subset \mathbb{A}^6$, such that its projective completion $\overline{\mathcal{V}}\subset \p^6$ is a (possibly singular) del Pezzo surface of degree six, and further that the hyperplane section at infinity $\overline{\mathcal{V}}\setminus \mathcal{V}$ is the union of three conics which intersect like a triangle, and consists of only smooth points of $\overline{\mathcal{V}}$.
    \item a morphism between objects $\mathcal{V}_1$ and $\mathcal{V}_2$ being an affine linear $L\in\operatorname{End}(\mathbb{A}^6)$ such that $L(\mathcal{V}_1)\subseteq\mathcal{V}_2$.
\end{itemize}
\end{definition}
As in the cases of $\pain{VI}$ and $\pain{IV}$ above, the family $\mathcal{M}\to \mathscr{W}_{\rm{II}}$ from \Cref{def:monodromysurfacePII} represents isomorphism classes of objects in $\mathfrak{C}_{\rm{II}}$.

\begin{proposition} \label{prop:oblomkovstylenormalformPII}
    For any $\mathcal{V}\in \operatorname{ob}(\mathfrak{C}_{\rm II})$, there exists a $w\in \mathscr{W}_{\rm{II}}$, unique up to the action of $\operatorname{Aut}(A_1^{(1)})=\langle\sigma\rangle\cong\mathfrak{S}_2$ in \Cref{tab:PII:symmetriesunderRH}, such that $\mathcal{V}$ is isomorphic to $\mathcal{M}_w$ in $\mathfrak{C}_{\rm II}$.  
    The corresponding isomorphism is unique up to composition with automorphisms of $\mathcal{M}_w$ in $\mathfrak{C}_{\rm{II}}$. The automorphism group of $\mathcal{M}_w$ contains
    \begin{equation*}
    \begin{aligned}
    &\varpi_{1} : x_1\leftrightarrow x_2, \quad x_4 \leftrightarrow x_5 \\
    &\varpi_2 : x_2 \leftrightarrow x_3, \quad x_5 \leftrightarrow x_6,
    \end{aligned}
    \end{equation*}
    which form an action of the group $\langle \varpi_1,\varpi_2\rangle\cong\mathfrak{S}_3$, 
    and $\vartheta = \varpi_1\varpi_2$.
    The full automorphism group of $\mathcal{M}_w$ in $\mathfrak{C}_{\rm{II}}$ is then generated by $\mathfrak{S}_3$ and $\operatorname{Stab}_{\Aut(A_1^{(1)})}\left(w\right)$, the latter being trivial for generic $w$.

\end{proposition}

\begin{proof}
    Let $\mathcal{V}$ be an object in $\mathfrak{C}_{\rm{II}}$. 
    First consider the case when $\mathcal{V}$ is smooth, so $\overline{\mathcal{V}}\subset\p^6$ can be realised as the blowup of $\p^2$ at three non-collinear points, which we take to be at the corner points $[0:0:1],[0:1:0],[1:0:0]\in\p^2$.
    The Picard group is then realised as 
    $$\Pic(\overline{\mathcal{V}}) = \Z \h \oplus\Z \E_1\oplus\Z \E_2\oplus\Z \E_3,$$
    where $\mathcal{H}$ is the class of the pullback of a line on $\p^2$, and $\E_1,\E_2,\E_3$ come from the exceptional divisors.
    Since we are working with $\overline{\mathcal{V}}$ in its anticanonical embedding, any hyperplane section of $\overline{\mathcal{V}}$ is an effective anticanonical divisor, i.e. in the linear system of 
    $$ -\mathcal{K}_{\overline{\mathcal{V}}} = 3 \h-\E_1-\E_2-\E_3.$$
    By assumption that $\mathcal{V}$ is an object in $\mathfrak{C}_{\rm{II}}$, i.e. it has a hyperplane section consisting of a triangle of conics at infinity, the del Pezzo surface $\overline{\mathcal{V}}$ must admit an effective anticanonical divisor which decomposes into three rational curves that intersect like a triangle, and that are not among the lines on $\overline{\mathcal{V}}$.
    
    So there is an effective anticanonical divisor $D\in |-\mathcal{K}_{\overline{\mathcal{V}}}|$ that decomposes as 
    $$D = D_1+D_2+D_3,$$
    where $D_i$ are rational curves. 
    The genus formula gives $D_i^2 \geq -2$ for $i=1,2,3$.
        There are no $-2$-curves on a smooth del Pezzo surface of degree six, and since $D_i$ are not among the lines on $\overline{\mathcal{V}}$, we know $D_i^2\neq -1$.
    Thus $D_i^2\geq 0$ for $i=1,2,3$.
    Since the curves are required to intersect like a triangle we have $D_i\cdot D_j=1$, $1\leq i<j\leq 3$,
    so expanding $D^2=6$ gives 
$D_1^2 +D_2^2 +D_3^2=0,$
 from which we deduce $D_i^2=0$ for $i=1,2,3$, and the components must be in the linear equivalence classes
    $$ \h-\E_1,\quad \h-\E_2,\quad\h-\E_3.$$
    
    Hence the three conics forming the hyperplane section at infinity must be given by the strict transforms of lines on $\p^2$ passing through  $[0:0:1],[0:1:0],[1:0:0]$.
    Up to a projectivity of $\p^2$ that fixes these corner points, and the $\mathfrak{S}_3$ freedom of permuting the labels of the homogeneous coordinates $[Z_0:Z_1:Z_2]$,  these lines must be given by 
    \begin{equation} \label{eq:coniclinesstandard}
        Z_0-Z_1 = 0, \quad v Z_1 -  Z_2 =0, \quad Z_2-Z_0=0,
    \end{equation}
    for some $v\in \C^*$, unique up to $v\mapsto v^{-1}$, with $v\neq 1$, as the three conics do not intersect at a common point.
    So the object $\overline{\mathcal{V}}$ is isomorphic to the blowup of $\p^2$ at the three corner points, with the triangle of conics being sent to the three lines \eqref{eq:coniclinesstandard}, for some $v\in \C\setminus\{0,1\}$.
    We will show that this is isomorphic to $\overline{\mathcal{M}}_w$ for some $w\in \mathscr{W}_{\rm{II}}$.
    
    Recall the standard realisation of the del Pezzo surface $\overline{\mathcal{S}}$ in \eqref{eq:standardmodeldepPezzoPII}.
    The birational morphism 
    \begin{equation} \label{eq:blowdownStoP2}
    \begin{aligned}
        \pi : \overline{\mathcal{S}} &\to \p^2, \\
        [a:b:c:d:e:f:g] &\mapsto [Z_0:Z_1:Z_2] = [g : u^{-1} d : c], 
    \end{aligned}
    \end{equation} 
    contracts three lines on $\overline{\mathcal{S}}$ to the corner points on $\p^2$.
    By composing the projectivity $\overline{\mathcal{M}}_w\to \overline{\mathcal{S}}$ in \cref{eq:projectivityStoMPII} with $\pi : \overline{\mathcal{S}}\to\p^2$ in \eqref{eq:blowdownStoP2}, we get a birational morphism 
    realising $\overline{\mathcal{M}}_w$ as the blowup of $\p^2$ at corner points.
    In this case, the three conics at infinity are sent to the same three lines on $\p^2$ in \eqref{eq:coniclinesstandard} when $u^2=v$. 
    Then the desired isomorphism  from $\overline{\mathcal{V}}$  to $\overline{\mathcal{M}}_w$ is realised through the anticanonical embedding by the projectivity in \cref{eq:projectivityStoMPII}.
    This is guaranteed to be an isomorphism in $\mathfrak{C}_{\rm{II}}$ by the fact that both $\overline{\mathcal{V}}$ and $\overline{\mathcal{M}}_w$ are anticanonically embedded. Note that $w^2 = v+ 2 + v^{-1}$, so that
     $w$ is invariant under $v\to v^{-1}$ and unique up to the action of $\sigma$ as claimed.
    Note also that the condition $w \neq \pm 2$ for $\mathcal{M}_w$ to be smooth is equivalent to $v\neq 1$.

    For the required isomorphism in the case when $\overline{\mathcal{V}}$ is singular, it suffices to show that any $\mathcal{V}\in \operatorname{ob}(\mathfrak{C}_{\rm{II}})$ which is singular is isomorphic to the surface $\mathcal{M}_{w}$ with $w=\pm2$. 

    Let $\overline{\mathcal{V}} \in \operatorname{ob}(\mathfrak{C}_{\rm{II}})$ be singular. 
    We will first show that $\overline{\mathcal{V}}\subset \p^6$ must have exactly one singularity of type $A_1$, and this lies at the intersection of the three lines contained in $\overline{\mathcal{V}}$.
    Since $\overline{\mathcal{V}}$ is a singular del Pezzo surface of degree 6, its minimal resolution is isomorphic to the blowup of $\p^2$ at three points which are not in general position in the sense of \cite[Sec. 8.4]{dolgachevclassicalAG}, so they are either infinitely near, collinear, or both.
    Denote this minimal resolution by  
    $$ \widetilde{\overline{\mathcal{V}}}\to \overline{\mathcal{V}},$$ 
    and the birational morphism to $\p^2$ from the three blowups by 
    $$ \pi : \widetilde{\overline{\mathcal{V}}}\to \p^2.$$
    Then we have 
    $$\Pic(\widetilde{\overline{\mathcal{V}}}) = \Z \h\oplus  \Z \E_1\oplus \Z \E_2  \oplus \Z\E_3.$$
    By assumption any singularities of $\overline{\mathcal{V}}$ are away from the the three conics at infinity, so by a similar argument to the smooth case there must be an effective anticanonical divisor on $\widetilde{\overline{\mathcal{V}}}$, given by the sum of three irreducible curves of self-intersection $0$.
    Since $-\mathcal{K}_{\widetilde{\overline{\mathcal{V}}}}=3 \h - \E_1-\E_2-\E_3$, these must be representatives of the classes 
    $$ \h-\E_1, \quad \h-\E_2, \quad \h-\E_3.$$
    Since these must be effective classes, the points in $\p^2$ onto which the morphism $\pi$ contracts the curves corresponding to $\E_1,\E_2,\E_3$ must all be distinct.
    By the classification of \cite[Sec. 8.4.2]{dolgachevclassicalAG}, these three points must be collinear and $\overline{\mathcal{V}}$ contains exactly three lines, corresponding to the three $(-1)$-curves on $\widetilde{\overline{\mathcal{V}}}$ giving $\E_1,\E_2,\E_3$.
    Then there is a single $(-2)$-curve on $\widetilde{\overline{\mathcal{V}}}$, corresponding to 
    $$\h - \E_1-\E_2-\E_3,$$
    which is contracted onto the single singularity of type $A_1$ on $\overline{\mathcal{V}}$.
    Since $(\h-\E_1-\E_2-\E_3)\cdot \E_1=(\h-\E_1-\E_2-\E_3)\cdot \E_2=(\h-\E_1-\E_2-\E_3)\cdot \E_3= 1$, the images of the lines under $\widetilde{\overline{\mathcal{V}}}\to \overline{\mathcal{V}}$ must all intersect at the singularity.

    We will construct the desired isomorphism from $\overline{\mathcal{V}}$ to $\overline{\mathcal{M}}_{\pm 2}$ by showing that the anticanonical model of $\overline{\mathcal{V}}$ coincides with the surface $\overline{\mathcal{S}}_{\operatorname{sing}}$ in \eqref{eq:standardmodeldepPezzoPIIsing}, then using the projective equivalence given by \eqref{eq:projectivityStoMsingPII}.
    Up to a projective coordinate change on $\p^2$ we can take the points onto which $\pi$ contracts curves to be 
    \begin{equation*}
    b_1 : [Z_0:Z_1:Z_2] =[0:0:1] , \quad 
    b_2 :[Z_0:Z_1:Z_2] =[0:1:0], \quad 
    b_3 : [Z_0:Z_1:Z_2]=[0:1:1].
      \end{equation*}
The anticanonical linear system $|-\mathcal{K}_{\widetilde{\overline{\mathcal{V}}}}|$ can be computed by finding a basis for the vector space of cubic forms in $Z_0,Z_1,Z_2$ vanishing at the points $b_1,b_2,b_3$. 
    For this we take  
        \begin{equation*} a = Z_0^2 Z_2,\quad 
    b = Z_0^2 Z_1,\quad 
    c = Z_0 Z_1^2, \quad 
    d = Z_1^2 Z_2 - Z_1 Z_2^2, \quad
    e = Z_0^3, \quad 
    f = Z_0 Z_2^2, \quad 
    g = Z_0 Z_1 Z_2,\end{equation*}
    and the birational morphism $\phi  : \widetilde{\overline{\mathcal{V}}} \to \p^6$ coming from the anticanonical linear system gives 
        \begin{equation*} 
        \overline{\mathcal{S}}_{\operatorname{sing}} = \phi (\widetilde{\overline{\mathcal{V}}})=\Proj \C[a,b,c,d,e,f,g] / (m_1,\dots,m_9)\subset \p^6,
    \end{equation*}
    which is isomorphic in $\mathfrak{C}_{\rm{II}}$ to $\mathcal{M}_w$ with $w=\pm 2$, via the projectivity \eqref{eq:projectivityStoMsingPII}.
    The freedom of choice of sign in $w$ corresponds to the the action of $\sigma$ in \Cref{tab:PII:symmetriesunderRH}, and the isomorphism $\overline{\mathcal{V}}\to \overline{\mathcal{M}}_{\pm 2}$ is unique up to the automorphism group $\langle \varpi_1, \varpi_2\rangle \cong \mathfrak{S}_3$ of $\overline{\mathcal{M}}_{\pm 2}$ in $\mathfrak{C}_{\rm{II}}$.
    \end{proof}

From \Cref{prop:oblomkovstylenormalformPII} we get a mapping
\begin{equation} \label{eq:mapPhiPII}
\begin{gathered}
    \Phi : \operatorname{ob}(\mathfrak{C}_{\rm{II}})\to \mathscr{O}_{\rm{II}}:= \mathscr{W}_{\rm{II}}\sslash \Aut(A_1^{(1)}) = \Spec \C[o], \\
    o = w^2. 
\end{gathered}
\end{equation}
The singular locus $\mathscr{W}_{\rm{II}}^{\operatorname{sing}}\subset \mathscr{W}_{\rm{II}}$ corresponds to $\mathscr{O}_{\rm{II}}^{\operatorname{sing}}\subset\mathscr{O}_{\rm{II}}$ given by 
$$ o = 4.$$
Using \Cref{prop:oblomkovstylenormalformPII}, we obtain the following along the sames lines as \Cref{cor:modulispacePVI,cor:modulispacePIV}.
\begin{corollary} \label{cor:modulispacePII}
The map $\Phi$ given in \eqref{eq:mapPhiPII} induces a canonical bijection 
    $$\operatorname{Iso}(\mathfrak{C}_{\rm{II}}) \to \mathscr{O}_{\rm{II}}(\C)\cong\C,$$
    where $\mathscr{O}_{\rm{II}} = \mathscr{W}_{\rm{II}} \sslash \operatorname{Aut}(A_1^{(1)})$.
    \end{corollary}

\begin{remark} \label{rem:mysterysymmetryPII}
    For $\theta$ in 
        $$\Theta_{\rm{II}}\setminus \Theta_{\rm{II}}^{\operatorname{sing}} = \left\{ \theta \in \C : \theta\not\in  \Z \right\},$$
so that $w$ is away from the singular locus, the automorphisms $\varpi_1,\varpi_2$ of $\mathcal{M}_w$ introduced in \Cref{prop:oblomkovstylenormalformPII} conjugate under $\operatorname{RH}_{t,a}$ to biholomorphic involutions of $\mathcal{X}_{t,a}$.
These must be transcendental automorphisms of the initial value space, and as far as we know no such symmetries of $\pain{II}$ have appeared in the literature.
\end{remark}

\subsection{Lines on the monodromy surface}

To write down the lines on $\mathcal{M}_w$ for $w\in \mathscr{W}_{\rm{II}}\setminus \mathscr{W}_{\rm{II}}^{\operatorname{sing}}$, we let 
$$\mathscr{U}_{\rm{II}} = \left\{ u \in \C^{*} \right\},$$
and relate this to $\Theta_{\rm{II}}$ via $$u  = e^{\pi i \theta}.$$
The surface $\mathcal{M}_w$ has 6 affine lines, which can be written rationally in terms of the parameter $u$ by
\begin{equation} \label{eq:linesPII}
\begin{aligned}
    L_1 : \quad  x_1 &= u^{-1},\quad &&x_2= u,\quad &&x_6= 1,\quad &&x_4 = u\,x_3, \quad &&x_5= u^{-1} x_3,\\
    L_2 : \quad  x_3 &= u,\quad &&x_1= u^{-1},\quad &&x_5= 1,\quad &&x_6 = u^{-1} x_2, \quad &&x_4= u\,x_2,\\
    L_3 : \quad  x_2 &= u^{-1},\quad &&x_3= u,\quad &&x_4= 1,\quad &&x_5 = u\,x_1, \quad &&x_6= u^{-1} x_1,\\
    L_4 : \quad  x_1 &= u,\quad &&x_2= u^{-1},\quad &&x_6= 1,\quad &&x_4 = u^{-1} x_3, \quad &&x_5= u\,x_3,\\
    L_5 : \quad  x_3 &= u^{-1},\quad &&x_1= u,\quad &&x_5= 1,\quad &&x_6 = u\,x_2, \quad &&x_4= u^{-1} x_2,\\
    L_6 : \quad  x_2 &= u,\quad &&x_3= u^{-1},\quad &&x_4= 1,\quad &&x_5 = u^{-1} x_1, \quad &&x_6= u\,x_1.
    \end{aligned}
    \end{equation}


We will also use $L_k$, $k=1,\dots,6$ to denote the closures of the above lines in $\p^6$ under the embedding \eqref{eq:embeddingPII}. The hyperplane section of $\overline{\mathcal{M}}_w$ at infinity consists of the three conics as in \Cref{rem:propertiesofmonodromysurfacePII}. 
For the purpose of enumerating the intersection graph of lines and these three conics on $\overline{\mathcal{M}}_w$, we will denote these with the numbering
\begin{equation} \label{eq:conicsatinfinityPII2}
    \begin{aligned}
        C_{7}^{\infty} &= \{ X \in \p^6 \,:\, X_0=X_1=X_2=X_6=0, \quad X_4X_5=X_3^2 \}, \\
        C_{8}^{\infty} &= \{ X \in \p^6 \,:\, X_0=X_1=X_3=X_5=0, \quad X_4X_6=X_2^2 \}, \\
        C_{9}^{\infty} &= \{ X\in \p^6 \,:\, X_0=X_2=X_3=X_4=0, \quad X_5X_6=X_1^2 \}. 
    \end{aligned}
\end{equation}
The field extension used to write down the lines is described by the field homomorphism 
    \begin{equation} \label{lem:morphismUtoWPII}
        \C(\mathscr{W}_{\rm{II}}) \longrightarrow \C(\mathscr{U}_{\rm{II}}),
    \end{equation}
    defined by 
    $$w = u + u^{-1}.$$
The singular locus $\mathscr{W}_{\rm{II}}^{\operatorname{sing}}\subset \mathscr{W}_{\rm{II}}$ pulls back to $\mathscr{U}_{\rm{II}}^{\operatorname{sing}}\subset \mathscr{U}_{\rm{II}}$ defined by 
\begin{equation*}
    (u-1)^2(u+1)^2=0.
\end{equation*}

 \begin{remark}
     When $u\in \mathscr{U}^{\operatorname{sing}}_{\rm{II}}$, the following pairs of lines coincide:
     $$(L_1,L_4),\quad (L_2,L_5), \quad (L_3,L_6).$$
     Under the Riemann-Hilbert correspondence, $\mathscr{U}_{\rm{II}}^{\operatorname{sing}}$ corresponds to $\theta\in\Z$, or equivalently $a_0,a_1\in\Z$, which are parameter values for the Riccati solutions of $\pain{II}$ expressed in terms of Airy functions.
 \end{remark}

 On the affine cubic surface $\mathcal{C}_w$ there are three extra lines, which do not correspond to lines on $\mathcal{M}_w$.
As was the case for the extra line on the cubic for $\pain{IV}$, these can be written rationally in terms of $w$ and do not play a role in the monodromy of $\mathcal{M}_w$. 
These extra lines come from the following curves on $\mathcal{M}_w$ : 
\begin{equation*}
    \begin{aligned}
    &C_1 : \quad x_1 = x_5 =x_6 = 0,\quad x_2+x_3=w, \quad x_4 +x_2^2= w x_2,\\
    &C_2 : \quad x_2 = x_6 =x_4 = 0,\quad x_3+x_1=w, \quad x_5 +x_3^2= w x_3,\\
    &C_3 : \quad x_3 = x_4 =x_5 = 0,\quad x_1+x_2=w, \quad x_6 +x_1^2= w x_1.
    \end{aligned}
\end{equation*}
A geometric characterisation of these curves forms part of the following lemma, and their meaning in terms of monodromy data will be provided as part of  \Cref{prop:monodromydatalinescurvesexplainedPII}.

\begin{lemma} 
For any $w \in \mathscr{W}_{\rm{II}}\setminus \mathscr{W}_{\rm{II}}^{\operatorname{sing}}$, the birational morphism 
\begin{equation*}
    \begin{aligned}
        \pi : \overline{\mathcal{M}}_w&\to \p^2, \\
        [X_0:X_1:X_2:X_3:X_4: X_4:X_5:X_6] = X &\mapsto [Z_0,Z_1,Z_2]=[F_0(X):F_1(X):F_2(X)],
    \end{aligned}
\end{equation*}
where 
\begin{equation*}
    \begin{aligned}
    F_0(X) &= 
    \left(1-u^2\right) X_0
    +\left(u-u^{-1}\right) X_2
    +\left(u-u^{-1}\right) X_3
    +\left(u^{-2}-1\right) X_4,\\
    F_1(X) &= 
    \left(u^{-1}-u\right) X_0
    +\left(u^2-1\right) X_1
    +\left(1-u^{-2}\right) X_2
    +\left(u^{-1}-u\right) X_6,\\
    F_2(X) &= 
    \left(u^{-2}-1\right) X_0
    +\left(u-u^{-1}\right) X_1
    +\left(u-u^{-1}\right) X_3
    +\left(1-u^2\right) X_5,
    \end{aligned}
\end{equation*}
contracts the triple of disjoint lines $L_1,L_3,L_5$ on $\overline{\mathcal{M}}_w$ onto the following points of $\p^2$, 
$$L_1 \to b_1 = [1:0:-1], \quad 
L_3 \to b_2 = [0:1:-u], \quad 
L_5 \to b_3 = [1:-u:0].$$ 
The remaining lines $L_2,L_4,L_6$ are sent to the following lines between pairs of these points:
$$L_2 \to \{Z_0+ uZ_1+Z_2=0\}, \quad 
L_4 \to \{ uZ_0+Z_1+u^{-1}Z_2=0\}, \quad 
L_6 \to \{ Z_0 + u^{-1} Z_1 + Z_2=0\}.$$
The conics at infinity are sent by $\pi$ to the coordinate lines according to:
$$C_7^{\infty} \to  \{Z_1=0\}, \quad
C_8^{\infty} \to \{Z_2 = 0\}, \quad 
C_9^{\infty} \to \{Z_0 = 0\}.$$
The curves corresponding to the extra lines on $\overline{\mathcal{C}}_w$ are sent to the lines
$$C_1 \to  \{u Z_1 + Z_2=0\}, \quad
C_2 \to \{u Z_0 + Z_1\}, \quad 
C_3 \to \{Z_0+Z_2 = 0\}.$$
These are characterised uniquely as lines that join a corner point of $\p^2$ with the $b_i$ that lies on the opposite edge, with 
$$\pi(C_3) \ni b_1,[0:1:0], \quad 
\pi(C_1) \ni b_2, [1:0:0], \quad 
\pi(C_2)\ni b_3, [0:0:1].$$
\end{lemma}

\begin{proof}

The morphism $\pi$ is the standard way to project an anticanonically embedded del Pezzo surface of degree six to $\p^2$, adapted using the projectivity in \Cref{lem:delpezzodeg6PII}.
Here the choice of points $b_1,b_2,b_3$ onto which to contract is made so that the conics were sent to coordinate lines and so that the degeneration when $u=\pm1\in\mathscr{U}^{\operatorname{sing}}$ can be seen clearly.
The claims about where lines and conics are sent by the  morphism can be verified mostly by direct substitution.
The exceptions are some of the claims involving lines $L_1,L_2,L_6$, which are contained in the indeterminacy locus when $\pi$ is considered on $\p^6$, and require consideration of the defining equations of $\overline{\mathcal{M}}_w$. 
\end{proof}

In \Cref{fig:blowdownmodelPII}, we give an illustration of the images under the morphism from the previous lemma of the lines $L_1,\dots,L_6$, conics at infinity $C_7^{\infty},C_8^{\infty},C_9^{\infty}$, and conics $C_1,C_2,C_3$ coming from extra lines on the cubic.

\begin{figure}[htb]
       \begin{equation*}
    \begin{tikzpicture}[redpoint/.style={circle,draw=red!100,fill=red!100,thick, inner sep=0pt,minimum size=1mm},bluepoint/.style={circle,draw=blue!100,fill=blue!100,thick, inner sep=0pt,minimum size=1mm},blackpoint/.style={circle,draw=black!100,fill=black!100,thick, inner sep=0pt,minimum size=1mm},smallpoint/.style={circle,draw=black!100,fill=black!100,thick, inner sep=0pt,minimum size=.5mm}]
    






      


\begin{scope}[xshift=5.5cm,yshift=-.2cm, scale=2.75]
\def\epss{0.45}
\def\eps{0.10}
\def\epst{0.10}
\def\epsl{1.2}
\def\epsh{.02}
\def\epssh{.05}

\foreach \i in {1,...,3} {
    \coordinate (P\i) at ({90 + 120*(\i-1)}:1);
}
\draw[blue, line width=.75pt]
      ($(P1)!-\epss!(P2)$)
      --
      ($(P2)!-\epst!(P1)$) node[pos=1,below] {\tiny $\pi(C_8^{\infty}):Z_2=0$};
\draw[blue, line width=.75pt]
      ($(P2)!-\epss!(P3)$)
      --
      ($(P3)!-\epst!(P2)$) node[pos=1,right] {\tiny $\pi(C_7^{\infty}):Z_1=0$};   
\draw[blue, line width=.75pt]
      ($(P3)!-\epss!(P1)$)
      --
      ($(P1)!-\epst!(P3)$) node[pos=1,xshift=2pt,yshift=-1pt,above left] {\tiny $\pi(C_9^{\infty}):Z_0=0$};
      

\coordinate (B1) at (0.3,-0.5);
\coordinate (B2) at (0.3,0.5);
\coordinate (B3) at (-0.57735,0);

\draw[black, line width=.75pt]
      ($(B1)!-\eps!(B2)$)
      --
      ($(B2)!-\epsl!(B1)$);

\draw[black, line width=.75pt]
      ($(B2)!-\eps!(B3)$)
      --
      ($(B3)!-\epsl!(B2)$);      

\draw[black, line width=.75pt]
      ($(B3)!-\eps!(B1)$)
      --
      ($(B1)!-\epsl!(B3)$);     


\draw[black, dashed, line width=.75pt]
      ($(B1)!-\epsh!(P1)$)
      --
      ($(P1)!-\epsh!(B1)$);

\draw[black, dashed, line width=.75pt]
      ($(B2)!-\epsh!(P2)$)
      --
      ($(P2)!-\epsh!(B2)$);
\draw[black, dashed, line width=.75pt]
      ($(B3)!-\epsh!(P3)$)
      --
      ($(P3)!-\epsh!(B3)$);

\node[bluepoint] at (P1) {};
\node[bluepoint] at (P2) {};
\node[bluepoint] at (P3) {};
      
\node[redpoint,label={[xshift=+1pt,yshift=-0pt]below left:\tiny{\color{red}$\pi (L_1)$}}] at (B1) {};
\node[redpoint,label={[yshift=-1pt]right:\tiny{\color{red}$\pi (L_3)$}}] at (B2) {};
\node[redpoint,label={[yshift=0pt]left:\tiny{\color{red}$\pi (L_5)$}}] at (B3) {};

\coordinate (int1) at (0.3,1.52);
\coordinate (int2) at (-1.45, -0.497513);
\coordinate (int3) at (1.14172, -0.982487);

\node[blackpoint,label={[yshift=0pt]right:\tiny $\pi(L_2\cap C^{\infty}_8)$}] at (int1) {};
\node[blackpoint,label={[xshift=+1pt,yshift=0pt]above left:\tiny $\pi(L_4\cap C^{\infty}_7)$}] at (int2) {};
\node[blackpoint,label={[yshift=1pt]right:\tiny $\pi(L_6\cap C^{\infty}_9)$}] at (int3) {};

\node at (0.45,1) {\tiny $\pi(L_2)$};
\node at (0.65,-.85) {\tiny $\pi(L_6)$};
\node at (-1.15,-.2) {\tiny $\pi(L_4)$};

\node[label={[yshift=0pt]right:\tiny $\pi(C_1)$}] at (-0.725,-.35) {};
\node[label={[yshift=0pt]:\tiny $\pi(C_3)$}] at (-0.07,.5) {};
\node[label={[yshift=2pt]:\tiny $\pi(C_2)$}] at (0.55,-.45) {};


\node[smallpoint] at (-0.365,-0.0725) {};
\node[smallpoint] at (0.2575,-0.29) {};
\node[smallpoint] at (0.13,.355) {};


\node[smallpoint] at (-0.4,-0.1) {};
\node[smallpoint] at (0.3,-0.305) {};
\node[smallpoint] at (0.12,0.4) {};

\end{scope}
    \end{tikzpicture}
\end{equation*}
\caption{Images of lines and distinguished conics on the monodromy surface for $\pain{II}$ under the birational morphism $\pi : \overline{\mathcal{M}}_w\to \p^2$}
\label{fig:blowdownmodelPII}
\end{figure}
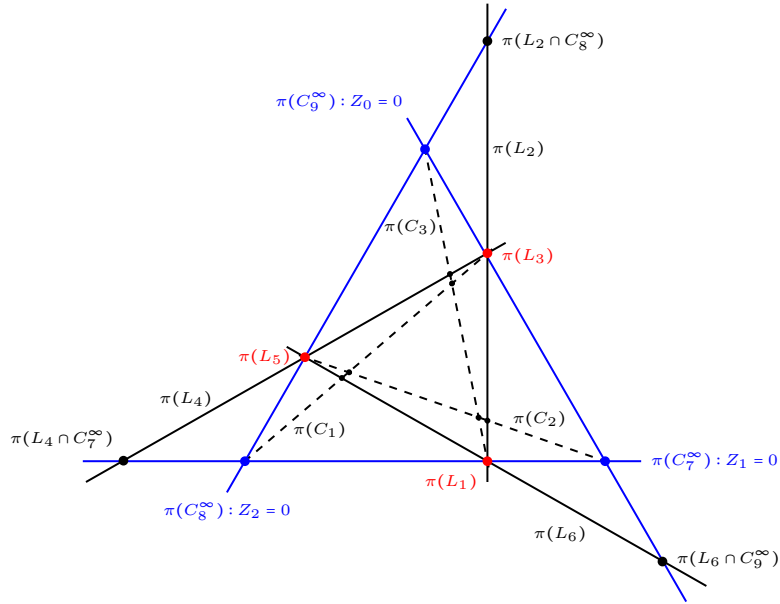

We next describe the movement of lines and conics under the map from $\overline{\mathcal{M}}_w$ to the cubic model of the monodromy surface for $\pain{II}$.
Consider the extension of the isomorphism of affine varieties in \eqref{eq:MwtocubicPII} to the del Pezzo surface $\overline{\mathcal{M}}_w$ and the projective completion of the cubic, which we write as 
\begin{equation*}
        \begin{aligned}
            \overline{\mathcal{C}}_w &= \operatorname{Proj} \C[Y_0,Y_1,Y_2,Y_3]/\overline{J}_w, \\
            \overline{J}_w &= \left( Y_1 Y_2 Y_3 - Y_0^2(Y_1+Y_2+Y_3)+w Y_0^3 \right).
        \end{aligned}
    \end{equation*}
    
\begin{proposition} \label{prop:delPezzotocubicPII}
    Let $w \in \mathscr{W}_{\rm{II}}\setminus \mathscr{W}_{\rm{II}}^{\operatorname{sing}}$
    and 
    \begin{equation*}
    \begin{aligned}
        \rho : \overline{\mathcal{M}}_w &\dashrightarrow \overline{\mathcal{C}}_w,\\
        [X_0:X_1:X_2:X_3:X_4:X_5:X_6] &\mapsto [Y_0:Y_1:Y_2:Y_3]= [X_0:X_1:X_2:X_3].
    \end{aligned}
    \end{equation*}
     be the extension of the isomorphism $\mathcal{M}_w\to\mathcal{C}_w$ in \eqref{eq:MwtocubicPII}.
   Denote the blow up at the three points $q_1,q_2,q_3\in \overline{\mathcal{M}}_w$ lying at the pairwise intersections of the three conics at infinity by $\widetilde{\overline{\mathcal{M}}}_w:=\operatorname{Bl}_{q_1,q_2,q_3}\overline{\mathcal{M}}_w$. 
   Then the map $\widetilde{\overline{\mathcal{M}}}_w\to \overline{\mathcal{C}}_w$ defined by the following commutative diagram is a minimal resolution of the three $A_1$ singularities of $\overline{\mathcal{C}}_w$:
    \begin{equation*} 
        \begin{tikzcd}
            &   \widetilde{\overline{\mathcal{M}}}_w \arrow[dl,swap,  "\operatorname{Bl}_{q_1,q_2,q_3}"] \arrow[dr] & \\
            \overline{\mathcal{M}}_w  \arrow[rr,dashed, "\rho"]& & \overline{\mathcal{C}}_w
        \end{tikzcd}
    \end{equation*}
\end{proposition}

\begin{proof}
    
    We will make use of computations in the Picard groups of $\overline{\mathcal{M}}_w$ and $\widetilde{\overline{\mathcal{M}}}_w$.
    The former can be described as 
$$\Pic(\overline{\mathcal{M}}_w)\cong \Z\mathcal{H}\oplus \Z \mathcal{E}_1\oplus\Z\mathcal{E}_2\oplus \Z \mathcal{E}_3,$$
with the intersection form given by 
$$\mathcal{H}\cdot \mathcal{H}=1, \quad \mathcal{E}_i \cdot \mathcal{E}_i=-1,\quad i \in \{1,2,3\},$$
and zero on all other pairs of generators.
The six lines on $\overline{\mathcal{M}}_{w}$ are the exceptional curves,  which intersect like a hexagon, and correspond to the following elements of $\Pic(\overline{\mathcal{M}}_w)$:
    $$L_1 :\mathcal{E}_1,\quad 
    L_2 :\mathcal{H}-\mathcal{E}_1-\mathcal{E}_2,\quad 
    L_3 : \mathcal{E}_2,\quad 
    L_4 :\mathcal{H}-\mathcal{E}_2-\mathcal{E}_3, \quad 
    L_5 :\mathcal{E}_3,\quad 
    L_6 : \mathcal{H}-\mathcal{E}_1-\mathcal{E}_3. $$
    The anticanonical divisor class is written as 
    $$-\mathcal{K}_{\overline{\mathcal{M}}_w} = 3 \mathcal{H}-\mathcal{E}_1-\mathcal{E}_2-\mathcal{E}_3,$$
    and the triangle of conics at infinity gives an effective anticanonical divisor whose three irreducible components correspond to the elements 
    $$C^{\infty}_7 : \mathcal{H}-\mathcal{E}_1,\quad 
    C^{\infty}_8 : \mathcal{H}-\mathcal{E}_3, \quad  
    C^{\infty}_9 : \mathcal{H}-\mathcal{E}_2.$$
    
        Now consider the blowup $\operatorname{Bl}_{q_1,q_2,q_3} : \widetilde{\overline{\mathcal{M}}}_w \to \overline{\mathcal{M}}_w$ 
        centred at the three points 
\begin{align*}
    q_1 &= C^{\infty}_8\cap C^{\infty}_9: [X_0:X_1:X_2:X_3:X_4:X_5:X_6]=[0:0:0:0:0:0:1],   \\ 
    q_2 &= C^{\infty}_7\cap C^{\infty}_9: [X_0:X_1:X_2:X_3:X_4:X_5:X_6]=[0:0:0:0:0:1:0],  \\ 
    q_3 &= C^{\infty}_7\cap C^{\infty}_8: [X_0:X_1:X_2:X_3:X_4:X_5:X_6]=[0:0:0:0:1:0:0].   
\end{align*}
Denote the classes of the exceptional divisors of the blowups of $q_1,q_2,q_3$ by $\mathcal{F}_1,\mathcal{F}_2,\mathcal{F}_3$, respectively, so 
$$\Pic( \widetilde{\overline{\mathcal{M}}}_w) \cong \Z\mathcal{H}\oplus \Z \mathcal{E}_1\oplus\Z\mathcal{E}_2\oplus \Z \mathcal{E}_3\oplus \Z \mathcal{F}_1\oplus \Z \mathcal{F}_2\oplus \Z \mathcal{F}_3.$$
Then the strict transforms of the three conics correspond to 
$$C^{\infty}_7 : \mathcal{H}-\mathcal{E}_1-\mathcal{F}_2-\mathcal{F}_3, \quad
C^{\infty}_8 : \mathcal{H}-\mathcal{E}_2-\mathcal{F}_1-\mathcal{F}_3, \quad
C^{\infty}_9 : \mathcal{H}-\mathcal{E}_3-\mathcal{F}_1-\mathcal{F}_2.$$
Then $\widetilde{\overline{\mathcal{M}}}_w$ is smooth, and a direct computation shows that the map $\rho\circ \operatorname{Bl}_{q_1,q_2,q_3} : \widetilde{\overline{\mathcal{M}}}_w \rightarrow \overline{\mathcal{C}}_w$ is a birational morphism which contracts these three $-2$ curves onto the $A_1$ singularities of $\overline{\mathcal{C}}_w$ at 
\begin{align*}
    p_1 : [Y_0:Y_1:Y_2:Y_3] &= [0:1:0:0], \\
p_2 : [Y_0:Y_1:Y_2:Y_3] &= [0:0:1:0], \\
p_3 : [Y_0:Y_1:Y_2:Y_3] &= [0:0:0:1],
\end{align*}
and this is a minimal resolution. 
We give an illustration of the relation between the hyperplane sections at infinity $\overline{\mathcal{M}}_w\setminus \mathcal{M}_w$ and $\overline{\mathcal{C}}_w\setminus\mathcal{C}_w$ under $\rho : \overline{\mathcal{M}}_w \dashrightarrow \overline{\mathcal{C}}_w$ in \Cref{fig:delpezzotocubicdivisoratinfinityPII}. \end{proof}

 \begin{figure}[htb]
       \begin{equation*}
    \begin{tikzpicture}[redpoint/.style={circle,draw=red!100,fill=red!100,thick, inner sep=0pt,minimum size=1mm},bluepoint/.style={circle,draw=blue!100,fill=blue!100,thick, inner sep=0pt,minimum size=1mm}]
    

\begin{scope}[scale=0.8,yshift=.2cm]
\def\eps{0.25}

\coordinate (q1) at (-1,0.6);
\coordinate (q2) at ( 1,0.6);
\coordinate (q3) at ( 0,-1);

\draw[line width=.75pt, blue]
  ($(q1)!-\eps!(q2)$)
  to[bend left=15]
  ($(q2)!-\eps!(q1)$);
\node at (+0,1.2) {\tiny $C_9^{\infty}$};

\draw[line width=.75pt, blue]
  ($(q1)!-\eps!(q3)$)
  to[bend right=35]
  ($(q3)!-\eps!(q1)$);
\node at (-1.3,-.5) {\tiny $C_8^{\infty}$};
\draw[line width=.75pt, blue]
  ($(q2)!-\eps!(q3)$)
  to[bend left=35]
  ($(q3)!-\eps!(q2)$);
\node at (+1.3,-.5) {\tiny $C_7^{\infty}$};

\node[redpoint] at ($(q1)+(-0.25,0.05)$) {};
\node[above left] at ($(q1)+(-0.25,0.05)$) {$q_1$};
\node[redpoint] at ($(q2)+(+0.25,0.05)$) {};
\node[above right] at ($(q2)+(+0.25,0.05)$) {$q_2$}; 
\node[redpoint] at ($(q3)+(0.0,-0.275)$) {};
\node[below] at ($(q3)+(0.0,-0.35)$) {$q_3$}; 
\end{scope}

\draw[<-] (1.85,0) -- node[pos=0.5, above] {$\operatorname{Bl}_{q_1,q_2,q_3}$} (3.25,0);


\begin{scope}[xshift=5cm]

\def\eps{0.10}

\foreach \i in {1,...,6}{
  \coordinate (P\i) at ({60*(\i-1)}:1);
}

\foreach \i/\j in {1/2,2/3,3/4,4/5,5/6,6/1}{
  \draw[line width=.75pt]
    ($(P\i)!-\eps!(P\j)$)
    --
    ($(P\j)!-\eps!(P\i)$);
}

        \draw[line width=1pt, red] ($(P1)!-\eps!(P2)$)
      --
      ($(P2)!-\eps!(P1)$);
      
      \draw[line width=1pt, blue] ($(P2)!-\eps!(P3)$)
      --
      ($(P3)!-\eps!(P2)$);

    \draw[line width=1pt, red] ($(P3)!-\eps!(P4)$)
      --
      ($(P4)!-\eps!(P3)$);
      
      \draw[line width=1pt, blue] ($(P4)!-\eps!(P5)$)
      --
      ($(P5)!-\eps!(P4)$);
    
    \draw[line width=1pt, red] ($(P5)!-\eps!(P6)$)
      --
      ($(P6)!-\eps!(P5)$);
      
      \draw[line width=1pt, blue] ($(P6)!-\eps!(P1)$)
      --
      ($(P1)!-\eps!(P6)$);

\draw[->] (1.5,0) -- node[pos=0.5, above] {$\rho\circ\operatorname{Bl}_{q_1,q_2,q_3}$}(3,0);

\end{scope}


\begin{scope}[xshift=10.5cm,yshift=-.2cm]
\def\eps{0.10}

\foreach \i in {1,...,3} {
    \coordinate (P\i) at ({90 + 120*(\i-1)}:1);
}
    \draw[line width=.75pt, red]
      ($(P1)!-\eps!(P2)$)
      --
      ($(P2)!-\eps!(P1)$);
    \draw[line width=.75pt, red]
      ($(P2)!-\eps!(P3)$)
      --
      ($(P3)!-\eps!(P2)$);
    \draw[line width=.75pt, red]
      ($(P3)!-\eps!(P1)$)
      --
      ($(P1)!-\eps!(P3)$);
\node at (1.1,0.5) {\tiny $\{Y_0=Y_1=0\}$}; 
\node at (-1.1,0.5) {\tiny $\{Y_0=Y_2=0\}$}; 
\node at (0,-0.75) {\tiny $\{Y_0=Y_3=0\}$}; 

\node[bluepoint,label={[yshift=4pt]above:$p_3$}] at (P1) {};
\node[bluepoint, label={[yshift=-4pt]below left:$p_1$}] at (P2) {};
\node[bluepoint, label={[yshift=-4pt]below right:$p_2$}] at (P3) {};

\end{scope}


\node at (0,2) {$\overline{\mathcal{M}}_w$};
\node at (5.1,2) {$\widetilde{\overline{\mathcal{M}}}_w$};
\node at (10.5,2) {$\overline{\mathcal{C}}_w$};
    \end{tikzpicture}
\end{equation*}
\caption{Relation of hyperplane sections at infinity under $\rho : \overline{\mathcal{M}}_w \dashrightarrow \overline{\mathcal{C}}_w$ for $\pain{II}$. }
\label{fig:delpezzotocubicdivisoratinfinityPII}
\end{figure}
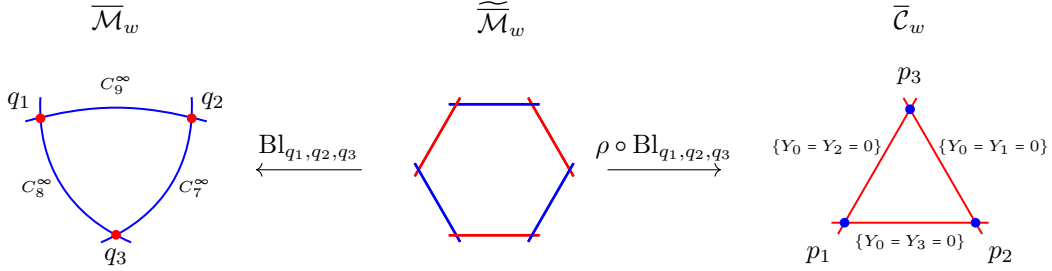

We next explain the meaning of the lines $L_k$, $1\leq k\leq 6$,  and the curves $C_{\ell}$, $1\leq \ell\leq3$, 
in terms of generalised monodromy data and solutions of the linear problem.
In order to do this, we introduce the following notation for the solutions $Y_k(z)$, which we recall are asymptotic to the formal solution $Y_{\operatorname{form}}(z)$ given in \cref{eq:formalsolPII} as $z\to\infty$ in $\Sigma_{k}\cup\Sigma_{k+1}$.
Denote the columns of the matrix function in \eqref{eq:matrixPasymptoticPII} according to $P(z)=\begin{pmatrix}
        P_+(z) & P_-(z)
    \end{pmatrix}.$
Define $\Psi_k(z)$ as the unique column vector solution of \eqref{eq:JMlaxpair1} satisfying
\begin{equation*}
    \Psi_k(z)\sim P_\pm(z)e^{\pm(\frac{1}{3}z^3+\frac{1}{2}tz)}z^{\mp \theta},
\end{equation*}
with sign $\pm 1=(-1)^k$ determined by the parity of $k$,
as $z\rightarrow \infty$ in $\Sigma_{k-1}\cup\Sigma_k\cup\Sigma_{k+1}$. Then we have, for $k\in\mathbb{Z}$,
\begin{equation*}
    Y_k=\begin{pmatrix}
        \Psi_k,\Psi_{k+1}
    \end{pmatrix}\quad \text{if $k$ even},\qquad
     Y_k=\begin{pmatrix}
        \Psi_{k+1},\Psi_k
    \end{pmatrix}\quad \text{if $k$ odd,}
\end{equation*}
and the Stokes phenomenon takes the form
\begin{equation*}
    \Psi_{k+1}=\Psi_{k-1}+s_k\Psi_k.
\end{equation*}
The following is proven by direct analogy with \Cref{prop:curves_explained}. 
\begin{proposition} \label{prop:monodromydatalinescurvesexplainedPII}
    The lines $L_1,\dots,L_{6}$ correspond to the following distinguished hyperplane sections of the the space $M$ of Stokes multipliers given in \Cref{eq:monodromyspaceMPII}:
    \begin{equation*}
    \begin{aligned}
    L_1 &: \left\{ s_2 = 0 \right\} \cap M, \quad
    &&L_2 : \left\{ s_3 = 0 \right\} \cap M, \quad
    &&L_3 : \left\{ s_4 = 0 \right\} \cap M, \\
    L_4 &: \left\{ s_5 = 0 \right\} \cap M, \quad 
    &&L_5 : \left\{ s_6 = 0 \right\} \cap M, \quad
    &&L_6 : \left\{ s_1 = 0 \right\} \cap M. 
    \end{aligned}
    \end{equation*}
The three distinguished conics on $\mathcal{M}_w$ that become lines on $\mathcal{C}_w$ correspond to the following:
\begin{equation*}
    C_1 : \left\{s_5 s_6+1=0 \right\} \cap M, \quad
    C_2 : \left\{s_1 s_2+1=0 \right\} \cap M, \quad 
    C_3 : \left\{s_3 s_4+1=0 \right\} \cap M.
\end{equation*}
From an analytic point of view, they each correspond to a different quantisation condition on the spectral equation:
 \begin{enumerate}
     \item $\Psi_{k+1}=c\,\Psi_{k-1}$, for some $c\in\mathbb{C}^*$, is equivalent to $s_k=0$, in which case $c=1$, corresponding to line $L_k,$ for $1\leq k\leq 6$.
     \item $\Psi_{k+2}=c\,\Psi_{k-1}$, for some $c\in\mathbb{C}^*$, is equivalent to $1+s_ks_{k+1}=0$, in which case $c=s_{k+1}$, corresponding to conic $C_k$, for $k=1,2,3$.
 \end{enumerate}
\end{proposition}

\begin{remark}
    Under the Riemann-Hilbert map defined by the Flaschka-Newell linear problem, each of the lines $L_k$, $1\leq k\leq 6$, corresponds to a one-parameter family of increasing tronqu\'ee solutions and their points of intersection correspond to the six increasing tritronqu\'ee solutions of $\pain{II}$ \cite[Sec. 11.5]{fokas}. Similarly, each of the distinguished conics $C_k$, $1\leq k\leq 3$, correspond to a one-parameter family of decreasing tronqu\'ee solutions and their points of intersection correspond to the three decreasing tritronqu\'ee solutions  of $\pain{II}$ \cite[Sec. 11.6]{fokas}. The six intersection points among the lines and distinguished conics correspond to simultaneous increasing tronqu\'ee and decreasing tronqu\'ee solutions, one of which is the Hastings-McLeod solution \cite[Sec. 11.7]{fokas}.
\end{remark}

\begin{remark}
    The automorphism group of $\mathcal{M}_w$ inside $\mathfrak{C}_{\rm{II}}$ is maximal when $w=0$, generated by $\langle\varpi_1,\varpi_2\rangle$ and $\Aut(A_1^{(1)})$, see Proposition \ref{prop:oblomkovstylenormalformPII}.  In this case, there is a unique global fixed point, $x=(0,0,0,0,0,0)$, where all three distinguished conics $C_k$, $1\leq k\leq 3$, intersect. The corresponding solutions of $\pain{II}$, under
    both instances of the Riemann-Hilbert correspondence induced by the Jimbo-Miwa and
     Flaschka-Newell Lax pairs, are the rational solutions, expressible in terms of Yablonskii-Vorob'ev polynomials \cite{millerrational}. The Stokes data in the two cases can be taken as 
     $s_k^{\operatorname{JM}}=i$, $1\leq k\leq 6$, and
$s_k^{\operatorname{FN}}=0$, $1\leq k\leq 3$, respectively.
\end{remark}

\subsection{Combinatorial monodromy}

We will describe the monodromy of the family $\overline{\mathcal{M}}\to \mathscr{W}_{\rm{II}}$ first in terms of a graph encoding intersections among lines on $\overline{\mathcal{M}}_w$ as well as the triangle of conics at infinity.
\begin{definition} \label{def:PIIintersectiongraph}
    Let $u\in\mathscr{U}_{\rm{II}}\setminus\mathscr{U}_{\rm{II}}^{\operatorname{sing}}$ and consider the lines $L_1,\dots,L_{6}$ on $\overline{\mathcal{M}}_w$ and the conics $C_7^{\infty},C_8^{\infty},C_9^{\infty}$, with $w\in \mathscr{W}_{\rm{II}}$ given by $w=u+u^{-1}$.
    Form the graph with six blue vertices corresponding to the lines and three red vertices corresponding to the conics, with edges encoding pairwise intersections.
     Denote this intersection graph by $\mathcal{G}_{\rm{II}}$, and the subgraph generated by the red vertices by $\mathcal{G}_{\rm{II}}^\infty$.
\end{definition}
The intersection graph $\mathcal{G}_{\rm{II}}$ can be computed from the expressions for the lines and conics on $\overline{\mathcal{M}}_w$ given in \eqref{eq:linesPII} and \eqref{eq:conicsatinfinityPII2}, and is shown in \Cref{fig:linesgraphpII}.
\begin{figure}
    \centering
    \includegraphics[width=0.35\linewidth]{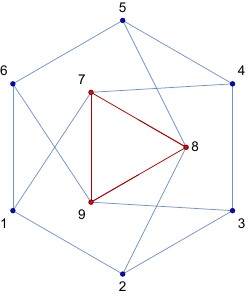}
    \caption{Graph $\mathcal{G}_{\rm{II}}$ encoding intersections of lines and the conics at infinity on the monodromy variety corresponding to $\pain{II}$,
    with in red the subgraph $\mathcal{G}_{\rm{II}}^\infty$ induced by the vertices $(C_{7}^\infty,C_{8}^\infty,C_{9}^\infty)$.
    }
    \label{fig:linesgraphpII}
\end{figure}

\begin{definition} \label{def:combinatorialmonodromyPII}
The \emph{combinatorial monodromy} of the family $\mathcal{M}\to\mathscr{W}_{\rm{II}}$ is the group 
$$\operatorname{Fix}_{\operatorname{Aut}(\mathcal{G}_{\rm{II}})}\left( \mathcal{G}_{\rm{II}}^{\infty}\right)$$
of automorphisms of the graph $\mathcal{G}_{\rm{II}}$ in \Cref{fig:linesgraphpII} that pointwise fix the vertices $\{7,8,9\}$.
\end{definition}
We will show that the combinatorial monodromy of $\mathcal{M}\to \mathscr{W}_{\rm{IV}}$ is isomorphic to  
$$W(A_1) = \langle r_1~|~ r_1^2=1\rangle \cong \mathfrak{S}_2,$$
which forms the underlying finite Weyl group of the symmetry group of $\pain{II}$.

\begin{proposition} \label{prop:combinatorialmonodromyPII}
    The combinatorial monodromy of the family $\mathcal{M}\to \mathscr{W}_{\rm{II}}$ is isomorphic to $W(A_1)$, with an explicit isomorphism
    \begin{equation*}
        W(A_1) \xrightarrow{\sim}   \operatorname{Fix}_{\operatorname{Aut}(\mathcal{G}_{\rm{II}})}(\mathcal{G}_{\rm{II}}^\infty), \quad r_1\mapsto \tilde{r}_1,
    \end{equation*}
defined by 
$$    \tilde{r}_1= (1\,\, 4)\,(2\,\,5)\,(3\,\,6).$$
\end{proposition}
\begin{proof}
Allowing permutations to act nontrivially on the red vertices, $\Aut(\mathcal{G}_{\rm{II}})$ is isomorphic to the dihedral group $\mathfrak{D}_6$ of order 12, i.e. the symmetry group of the hexagon encoding the intersections between the six lines on $\overline{\mathcal{M}}_{w}$.
It can be seen by inspection of the graph that the combinatorial monodromy $\operatorname{Fix}_{\operatorname{Aut}(\mathcal{G}_{\rm{II}})}(\mathcal{G}_{\rm{II}}^\infty)$ is the subgroup that respects the partition $\{(1,4),(2,5),(3,6)\}$ of the blue vertices into pairs, and is generated by $\tilde{r}_1$ as claimed.

This can also be seen on the level of Cremona isometries, using the realisation of $\overline{\mathcal{M}}_w$ as $\p^2$ blown up at three points as in the proof of \Cref{prop:delPezzotocubicPII}.
We have 
$$\Pic(\overline{\mathcal{M}}_w)\cong \Z\mathcal{H}\oplus \Z \mathcal{E}_1\oplus\Z\mathcal{E}_2\oplus \Z \mathcal{E}_3,$$
with the intersection form given by 
$$\mathcal{H}\cdot \mathcal{H}=1, \quad \mathcal{E}_i \cdot \mathcal{E}_i=-1,\quad i \in \{1,2,3\},$$
and zero on all other pairs of generators.
    The Cremona isometries of $\Pic(\overline{\mathcal{M}}_w)$ form $\mathfrak{D}_6$ and act by permutation on the set of exceptional classes in $\Pic(\overline{\mathcal{M}}_w)$:
    $$\operatorname{EX} = 
    \left\{\mathcal{E}_1,\,\,
    \mathcal{H}-\mathcal{E}_1-\mathcal{E}_2,\,\,
    \mathcal{E}_2,\,\,
    \mathcal{H}-\mathcal{E}_2-\mathcal{E}_3, \,\,
    \mathcal{E}_3, \,\,
    \mathcal{H}-\mathcal{E}_1-\mathcal{E}_3\right\}.$$
    This can be described as the finite Weyl group $W(A_1+A_2)\cong W(A_1)\times W(A_2)\cong \mathfrak{S}_2\times \mathfrak{S}_3\cong \mathfrak{D}_6$, generated by reflections 
    $$r_{\alpha_i}(\mathcal{F})=\mathcal{F}+\left( \mathcal{F}\cdot \alpha_i\right) \alpha_i,$$
    associated to the
    simple roots in $\Pic(\overline{\mathcal{M}}_w)$ given by
    $$\alpha_1 = \E_1-\E_2, \qquad \alpha_2 = \E_2-\E_3, \qquad \alpha_3 = \h-\E_1-\E_2-\E_3,$$
    of which $\alpha_1,\alpha_2$ generate the $A_2$ root subsystem, and $\alpha_3$ corresponds to the $A_1$ root subsystem.
    
While the conics at infinity are not among the lines on $\overline{\mathcal{M}}_w$, they correspond to the following elements of $\Pic(\overline{\mathcal{M}}_w)$:
$$ \h-\E_1, \quad \h-\E_2, \quad \h-\E_3.$$
Thus the combinatorial monodromy is the subgroup of $W(A_2)\times W(A_1)$ that fixes pointwise these elements. 
This subgroup is generated by the reflections associated to roots orthogonal to the elements corresponding to the conics, of which there are two, and to generate the subgroup it is sufficient to take 
$$ \beta_1 = \alpha_3 = \h-\E_1-\E_2-\E_3.$$
 This is the simple root for the $A_1$ root system, and the associated reflection $r_{\beta_1}$ acting on $\operatorname{EX}$ induces the permutation $\tilde{r}_1$ as claimed. 
\end{proof}

The nontrivial action of the Dynkin diagram automorphism $\sigma$, as well as the discrete symmetry $\vartheta$ also induce permutations of lines.

\begin{proposition}
    The action of $\langle\sigma\rangle\cong \Aut(A_1^{(1)})$ in $\mathscr{A}_{\rm{II}}\cong \Theta_{\rm{II}}$ in \Cref{tab:PII:symmetry:varsandparams} induces the action on $\mathscr{U}_{\rm{II}}$ given by 
    $$\sigma: u \mapsto \tilde{u}, \qquad\tilde{u} = - \frac{1}{u}.$$
    Together with the action on $\mathcal{M}\to \mathscr{W}_{\rm{II}}$ given in \Cref{tab:PII:symmetriesunderRH}, this induces the following automorphism of the graph $\mathcal{G}_{\rm{II}}$:
\begin{equation*}
     \sigma = (1\,\,4)\,(2\,\,5)\,(3\,\,6).
\end{equation*}       
The automorphisms $\varpi_1,\varpi_2$ of $\overline{\mathcal{M}}_w$, written in \Cref{prop:oblomkovstylenormalformPII}, induce the following when extended to act trivially on $\mathscr{U}_{\rm{II}}$:
\begin{align*}
    \varpi_1 &= (1\,\,4)\,(2\,\,3)\,(5\,\,6),\\
    \varpi_2 &= (1\,\,2)\,(3\,\,6)\,(4\,\,5),
\end{align*} 
and correspondingly $\vartheta=\varpi_1\varpi_2$, also acting trivially on $\mathscr{U}_{\rm{II}}$, consistent with its action on $\theta$, induces
\begin{equation*}
    \vartheta = (1\,\,3\,\,5)\,(2\,\,4\,\,6).
\end{equation*}
\end{proposition}
\begin{proof}
    Combining the actions on $\mathscr{U}_{\rm{II}}$ in the proposition with the actions on $x_i$, $1\leq i\leq 6$ in \Cref{tab:PII:symmetriesunderRH}, the claimed permutations are obtained directly from the expressions for the lines in \Cref{eq:linesPII}. 
\end{proof}

\subsection{Algebraic monodromy}

We will show that the field extension from $\C(\mathscr{W}_{\rm{II}})$ to $\C(\mathscr{U}_{\rm{II}})$ coincides with that coming from the incidence variety of lines on $\mathcal{M}_w$, and use this to show that algebraic monodromy is also given by $W(A_1)$.

\begin{definition}
    The incidence variety of lines on $\mathcal{M}_w$ is 
    \begin{equation*}
        \Gamma\xrightarrow{\rho} \mathscr{W}_{\rm{II}}, \quad\Gamma = \{ (w,\ell)\in \mathscr{W}_{\rm{II}} \times \mathbb{G}(1,6)~|~\ell \subset \overline{\mathcal{M}}_w, \,\, \ell \not\subset \overline{\mathcal{M}}_w\setminus \mathcal{M}_w\},
    \end{equation*}
    where $\mathbb{G}(1,6)$ is the Grassmannian of projective lines in $\p^6$.
\end{definition}

The restriction of $\rho$ to the nonsingular locus of $\mathcal{M}\to\mathscr{W}_{\rm{II}}$ gives a covering of $\mathscr{W}_{\rm{II}}\setminus \mathscr{W}_{\rm{II}}^{\operatorname{sing}}$, the fibre of which over $w$ consist of six points, corresponding to the lines on $\mathcal{M}_w$.
Similarly to the cases of $\pain{VI}$ and $\pain{IV}$ described in \Cref{subsec:algmonodromyPVI} and \Cref{subsec:algmonodromyPIV} respectively, $\Gamma$ is reducible. 
We define $K_{\Gamma}$ to be the compositum of the normal closures of the function fields of its irreducible components within the algebraic closure of $\C(\mathscr{W}_{\rm{II}})$. 

\begin{definition}
    The \emph{algebraic monodromy} of the family $\mathcal{M}\to \mathscr{W}_{\rm{II}}$ is the Galois group of the (normal closure of the) field extension $\C(\mathscr{W}_{\rm{II}})\subset K_{\Gamma}$:
    \begin{equation*}
        \operatorname{Gal}\left( K_{\Gamma}/ \C(\mathscr{W}_{\rm{II}})\right).
    \end{equation*}
\end{definition}
The algebraic monodromy here also forms $W(A_1)$, the underlying finite Weyl group of the symmetry group of $\pain{II}$.

\begin{proposition}  \label{prop:algebraicmonodromyPII}
    The algebraic monodromy of $\mathcal{M}\to \mathscr{W}_{\rm{IV}}$ is 
    $$\operatorname{Gal}\left( K_{\Gamma}/ \C(\mathscr{W}_{\rm{II}})\right) \cong \operatorname{Gal}\left( \C(\mathscr{U}_{\rm{II}})/ \C(\mathscr{W}_{\rm{II}})\right) \cong W(A_1),$$
    where the first isomorphism comes from a $\C(\mathscr{W}_{\rm{II}})$-linear isomorphism $K_{\Gamma}\cong \C(\mathscr{U}_{\rm{II}})$, and the second isomorphism is given by the action of $W(A_1)$ on $\C(\mathscr{U}_{\rm{II}})$ defined by 
    \begin{equation*}
    r_1 : u\to \frac{1}{u},
    \end{equation*}
    which keeps $w$ fixed and induces the same permutation of lines as in \Cref{prop:combinatorialmonodromyPII}.
\end{proposition}

\begin{proof}

We will establish the $\C(\mathscr{W}_{\rm{II}})$-linear isomorphism of fields by deriving the equations of $\Gamma$ in Pl\"ucker coordinates for $\mathbb{G}(1,6)$. 
Similarly to in the proof of \Cref{prop:algebraicmonodromyPIV}, we take the Pl\"{u}cker coordinates of a line in $\p^6$ to be
$P_{i,j}$, $0\leq i,j\leq 6$, where $P_{i,j}$ is the $5\times5$ minor of the matrix 
$$\begin{bmatrix}
    a_0 & a_1 & a_2 & a_3 & a_4 & a_5 &a_6\\
b_0 & b_1 & b_2 & b_3 & b_4 & b_5 &b_6 \\
c_0 & c_1 & c_2 & c_3 & c_4  & c_5 &c_6 \\
d_0 & d_1 & d_2 & d_3 & d_4  & d_5 &d_6 \\
e_0 & e_1 & e_2 & e_3 & e_4  & e_5 &e_6 
\end{bmatrix}
$$
with respect to the columns corresponding to indices $i,j$.
The line is given by the intersection of the five hyperplanes corresponding to linear forms in $X_0,\dots,X_6$ with coefficients given in the five rows of the matrix above.

Under the Pl\"ucker embedding, $\Gamma$ becomes an algebraic set in $\mathscr{W}_{\rm{II}}\times \p^{20}$. The Pl\"ucker coordinates of the lines $L_1,\dots,L_6$ can be computed directly using the expressions in \cref{eq:linesPII}, and from these we see that $\Gamma$ can be expressed as a disjoint union of three subsets, contained in quasiprojective subvarieties of $\mathscr{W}_{\rm{II}}\times \p^{20}$ as follows:
\begin{equation*}
    \begin{aligned}
        \Gamma_1 &\subset \left \{P_{0,3}\neq 0 ,  P_{0,1}=P_{0,2}=P_{0,6}=P_{1,2}=P_{1,6}=P_{2,6}=P_{3,4}=P_{3,5}=P_{4,5}=0 \right\}, \\ 
        \Gamma_2 &\subset \left \{P_{1,4}\neq 0 , P_{0,1}=P_{0,3}=P_{0,5}=P_{1,3}=P_{1,5}=P_{2,4}=P_{2,6}=P_{3,5}=P_{4,6}=0 \right\}, \\ 
        \Gamma_3 &\subset \left \{P_{2,5}\neq 0 ,P_{0,2}=P_{0,3}=P_{0,4}=P_{1,5}=P_{1,6}=P_{2,3}=P_{2,4}=P_{3,4}=P_{5,6}=0 \right\}.
    \end{aligned}
\end{equation*}
The components $\Gamma_1,\Gamma_2,\Gamma_3$ correspond to the pairs of lines $(L_1,L_4)$, $(L_2,L_5)$, $(L_3,L_6)$ respectively.
The defining equations of these components are computed similarly to in the proof of \Cref{prop:algebraicmonodromyPIV}.
For $\Gamma_1$, we have $P_{0,3}\neq 0$, so we work in the affine chart 
$$p_{i,j}=\frac{P_{i,j}}{P_{0,3}},$$ 
in which $\Gamma_1$ is given by 
\begin{equation*}
    \begin{gathered}
    p_{0,1}=0,\quad p_{0,2}=0, \quad p_{0,6}=0,\\ 
    p_{1,2}=0, \quad p_{1,6}=0, \quad p_{2,6}=0, \\ 
    p_{3,4}=0,\quad p_{3,5}=0, \quad p_{4,5}=0, \\
         p_{0,4}=-p_{2,3}, \quad p_{0,5} p_{2,3}= 1, \quad p_{1,3} p_{2,3}= -1, \quad p_{3,6}=-1, \\
         p_{2,3}^2-w p_{2,3}+1=0.
    \end{gathered}
\end{equation*}
For $\Gamma_2$, we have $P_{1,4}\neq 0$, so we work in the affine chart 
$$\tilde{p}_{i,j}=\frac{P_{i,j}}{P_{1,4}},$$ 
in which $\Gamma_2$ is given by 
\begin{equation*}
\begin{gathered}
    \tilde{p}_{0,1}=0,\quad \tilde{p}_{0,3}=0, \quad \tilde{p}_{0,5}=0, \\
    \tilde{p}_{1,3}=0,\quad \tilde{p}_{1,5}=0, \quad \tilde{p}_{2,4}=0,\\ \tilde{p}_{2,6}=0,\quad \tilde{p}_{3,5}=0, \quad \tilde{p}_{4,6}=0,\\
    \tilde{p}_{2,3}=\tilde{p}_{0,4}, \quad 
    \tilde{p}_{0,4} \tilde{p}_{0,6}=1, \quad
\tilde{p}_{0,4} \tilde{p}_{1,2}=-1, \quad 
\tilde{p}_{2,5}=-1, \\
\tilde{p}_{1,2}^2 - w \tilde{p}_{1,2} + 1=0.
\end{gathered}    
\end{equation*} 
For $\Gamma_3$, we have $P_{2,5}\neq 0$, so we work in the affine chart 
$$\hat{p}_{i,j}=\frac{P_{i,j}}{P_{2,5}},$$ 
in which $\Gamma_3$ is given by 
\begin{equation*}
    \begin{gathered}
    \hat{p}_{0,2}=0, \quad \hat{p}_{0,3}=0, \quad \hat{p}_{0,4}=0, \\
    \hat{p}_{1,5}=0, \quad \hat{p}_{1,6}=0, \quad \hat{p}_{2,3}=0, \\
    \hat{p}_{2,4}=0, \quad \hat{p}_{3,4}=0, \quad \hat{p}_{5,6}=0,\\
    \hat{p}_{1,3}=\hat{p}_{0,5}, \quad 
    \hat{p}_{0,5} \hat{p}_{0,6}=-1, \quad 
    \hat{p}_{0,5} \hat{p}_{1,2}=-1, \quad 
    \hat{p}_{1,4}=-1, \\
\hat{p}_{0,5}^2- w \hat{p}_{0,5}+1=0.
    \end{gathered}
\end{equation*}
Thus each component can be reduced to the same polynomial equation, 
$$ b^2 - w b + 1=0,$$
which splits over $\C(\mathscr{U}_{\rm{II}})$ via $w=u+u^{-1}$ with roots $u,u^{-1}$, so $K_{\Gamma}\cong \C(\mathscr{U}_{\rm{II}})$ as claimed.
The nontrivial element of $\operatorname{Gal}(\C(\mathscr{U}_{\rm{II}})/\C(\mathscr{W}_{\rm{II}}))$, given by $u\to\frac{1}{u}$, induces the permutation of lines $\tilde{r}_1$ in \Cref{prop:combinatorialmonodromyPII}, which is checked directly.
\end{proof}

\subsection{Analytic monodromy}

We will conclude this section by showing that the analytic monodromy of the family $\mathcal{M}\to\mathscr{W}_{\rm{II}}$ realises $W(A_1)$.

Take $w_* \in \mathscr{W}_{\rm{II}}\setminus \mathscr{W}_{\rm{II}}^{\operatorname{sing}}$ and consider the homomorphism
\begin{equation} \label{eq:monodromyhomPII}
    \pi_1(\mathscr{W}_{\rm{II}}\setminus \mathscr{W}_{\rm{II}}^{\operatorname{sing}} ; w_*) \to \mathfrak{S}_{6},
\end{equation}
which sends a loop to the permutation of lines on $\overline{\mathcal{M}}_{w_*}$ induced by deformation over it.

\begin{definition}
    The \emph{analytic monodromy} of the family $\mathcal{M}\to\mathscr{W}_{\rm{II}}$ is the subgroup of $\mathfrak{S}_{6}$ given by the image of the homomorphism \eqref{eq:monodromyhomPII}, defined up to overall conjugation, within $\mathfrak{S}_{6}$, induced by changing enumeration of lines and choice of basepoint.
\end{definition}

We will show that the analytic monodromy exhausts the combinatorial monodromy group.

\begin{proposition} \label{prop:analyticmonodromyPII}
    The analytic monodromy of the family $\mathcal{M}\to\mathscr{W}_{\rm{II}}$ is $W(A_1)$.
\end{proposition}
\begin{proof}
    The proof is analogous to that of \Cref{prop:analyticmonodromyPIV}. 
    Choose $u_*\in\mathscr{U}_{\rm{II}}\setminus \mathscr{U}_{\rm{II}}^{\operatorname{sing}}$ lying above $w_*$, and $\theta_*\in\Theta_{\rm{IV}}$ lying above $u_*$ under the map $\theta\mapsto u=e^{i \pi \theta}$.
    Then $\mathscr{U}_{\rm{II}}\setminus\mathscr{U}_{\rm{II}}^{\operatorname{sing}}$ lifts under this same map to
    $$\Theta_{\rm{II}}\setminus \Theta_{\rm{II}}^{\operatorname{sing}} = \left\{ \theta \in \C: \theta \not\in  \Z
    \right\},$$
    which is path-connected.
    Then for the generator $r_1$ of $W(A_1)$, one can choose a path in $\Theta_{\rm{IV}}$ from $\theta_*$ to $r_1(\theta_*)=-2 -\theta_*$ that avoids $\Theta_{\rm{II}}^{\operatorname{sing}}$ and descends to a loop in $\mathscr{W}_{\rm{II}}\setminus \mathscr{W}_{\rm{II}}^{\operatorname{sing}}$, inducing the permutation of lines corresponding to $r_1$ in \Cref{prop:combinatorialmonodromyPII}.
\end{proof}

 \Cref{thm:symmetryunderRHPII} establishes \Cref{mainthm:conjugatedsymmetries},   \Cref{prop:oblomkovstylenormalformPII} and \Cref{cor:modulispacePII} establish \Cref{mainthm:categoryandmodulispace},
and \Cref{prop:combinatorialmonodromyPII,prop:algebraicmonodromyPII,prop:analyticmonodromyPII} establish \Cref{mainthm:mon} in the case of Painlev\'e-II.

\section{The first Painlev\'e equation}\label{sec:PI}

While the first Painlev\'e equation 
\begin{equation} \label{eq:scalarPI}
    y_{tt} = 6 y^2 + t,
\end{equation}
does not contain parameters and so does not admit B\"acklund transformations, in this section we construct its associated monodromy surface as an embedded affine variety.
This is given by a del Pezzo surface of degree 5 with a pentagon of lines at infinity.

\subsection{Initial value space}

For $\pain{I}$, we consider the initial value space as obtained from the Hamiltonian form 
\begin{equation} \label{eq:hamPI}
    \left\{
    \begin{aligned} 
    q_t &= p = \frac{\partial H}{\partial p}, \\
    p_t &= 6 q^2 + t= - \frac{\partial H}{\partial q},
    \end{aligned}
    \right. \qquad
    H = \frac{1}{2} p^2 - 2 q^3 - t q.
\end{equation}
Eliminating $p$ from \eqref{eq:hamPI} leads to $\pain{I}$ as in \eqref{eq:scalarPI} for $q(t)$.

Since there are no parameters the family of surfaces is 
\begin{equation*}
    \overline{\mathcal{X}}\to \mathscr{T}_{\rm{I}},
\end{equation*}
where $\mathscr{T}_{\rm{I}}=\C$ is the independent variable space of the system \eqref{eq:hamPI}.
The fibre $\overline{\mathcal{X}}_t$ is a Sakai surface of type $E_8^{(1)}$, with unique effective anticanonical divisor $D_t$ given by a collection of $(-2)$-curves intersecting according to the $E_8^{(1)}$ Dynkin diagram.
Removing these curves from each fibre gives the family $\mathcal{X}\to\mathscr{T}_{\rm{I}}$.

There is a five-fold discrete symmetry of $\pain{I}$, generated by the following transformation that preserves the system \eqref{eq:hamPI}: 
\begin{equation} \label{eq:discreteonqpPI}
    \begin{gathered}
        \vartheta : (q,p) \mapsto (\tilde{q},\tilde{p}),\\
        \tilde{q}(t) = \zeta_5^2 q( \zeta_5 t), \quad \tilde{p}(t) = \zeta_5^{-2} p( \zeta_5 t).
    \end{gathered}
\end{equation} 

\subsection{Associated linear problem}

We will use the following Lax pair for $\pain{I}$, which is related by a simple scaling to that given by Jimbo and Miwa in \cite{jimbomiwaII1981}:
    \begin{subequations} \label{eq:laxpairPI}
            \begin{align}
                Y_z &= A Y, &   &A = z^2 A_2+z A_1+ A_{0}, \label{eq:laxpairPI1} \\
                Y_t &= B Y, &   &B = z B_1 + B_0. \label{eq:laxpairPI2}
            \end{align}
    \end{subequations}
The matrices $A_2,A_1,A_0\in\mathfrak{sl}_2(\C)$ are 
such that 
\begin{equation*}
    A_2 = \begin{bmatrix}
        0 & 2 \\
        0 & 0
    \end{bmatrix}, 
    \qquad (A_1)_{1,1}=0,
    \qquad |A| = - 4 z^3 - 2 t z + \mathcal{O}(1),\quad (z\to \infty),
\end{equation*}
so that the matrix $A$ is parametrised by 
\begin{equation}\label{eq:matrixAPI}
    A = z^2\begin{bmatrix}
        0 & 2 \\
        0 & 0
    \end{bmatrix}
    + z \begin{bmatrix}
        0 & 2 q \\
        2 & 0
    \end{bmatrix}
    + \begin{bmatrix}
         -p & 2q^2 + t\\
        -2 q &  p
    \end{bmatrix}.
\end{equation}
The matrices $B_1,B_0\in \mathfrak{sl}_{2}(\C)$ are such that 
\begin{equation} \label{eq:matrixBPI}
    B = z\begin{bmatrix}
        0 & 1\\
        0 & 0
    \end{bmatrix}
    + \begin{bmatrix}
        0 & 2 q \\
        1 & 0
    \end{bmatrix}.
\end{equation}
Then the compatibility condition of the pair of linear equations \eqref{eq:laxpairPI}, with $A$ and $B$ given in \eqref{eq:matrixAPI} and \eqref{eq:matrixBPI} respectively, gives the following differential equations for $q(t),p(t)$:
\begin{equation*}
    q_t = p, \quad p_t = 6 q^2 + t.
\end{equation*}
Eliminating $p(t)$ from these gives the first Painlev\'e equation for $q(t)$.

\subsection{Derivation of monodromy surface}

We follow the derivation by Kapaev in \cite{kapaev2004}.
There is a unique formal solution of equation \eqref{eq:laxpairPI1} of the form
$$ Y_{\operatorname{form}}(z) = z^{\frac{1}{4} \sigma_3}P(z) e^{( \frac{4}{5} z^{5/2} + t z^{1/2})\sigma_3},$$
where $P(z)$ is a power series in $z^{-1/2}$,
$$P(z) =  \sum_{n=0}^{\infty} z^{-n/2} U_n, \qquad U_0 =\frac{1}{\sqrt{2}}\begin{bmatrix}
    1 & 1 \\
    1 & -1 
\end{bmatrix}.$$
Define the Stokes sectors
$$ \Sigma_k =\left\{ \left| \operatorname{arg}z - \frac{2(k-1)\pi}{5}\right| < \frac{\pi}{5} \right\} \subseteq\widetilde{\C^*} , \qquad (k\in\Z), $$
within which $e^{( \frac{4}{5} z^{5/2} + t z^{1/2})}$ is alternately exponentially small or large as $z\to\infty$.

For each $k\in \Z$, there is a unique solution $Y_k$ of the linear problem \eqref{eq:laxpairPI1} such that 
$$Y_k(z)\sim Y_{\operatorname{form}}(z) \qquad (z\in \Sigma_k\cup \Sigma_{k+1}, z\to \infty),$$
and the Stokes phenomenon takes the form
$$Y_{k+1}(z) = Y_k(z) S_k,$$
where the Stokes matrices are given in terms of Stokes multipliers $s_k$,  $k\in \Z$, by
$$S_k = \begin{bmatrix}
    1 & 0 \\
    s_k & 1
\end{bmatrix},
 \text{ if } k \text{ even},
 \qquad
 S_k = \begin{bmatrix}
    1 & s_k \\
    0 & 1
\end{bmatrix},
 \text{ if } k \text{ odd}.
 $$

The solutions $Y_k(z)$ are analytic and single-valued on $\C$, and satisfy the cyclic property
\begin{equation} \label{eq:cyclicpropertyYPI}
    Y_{k+5}(z)=i Y_k(z)\sigma_1.
\end{equation}
This gives
\begin{equation} \label{eq:cyclicpropertySPI}
    S_{k+5}=\sigma_1S_k \sigma_1,
\end{equation}
and
$$S_{k+5} S_{k+4} S_{k+3} S_{k+2} S_{k+1} = \begin{bmatrix}
    0 & i\\
    i & 0
\end{bmatrix},$$
for $k\in\Z$.
Using the cyclic relation \eqref{eq:cyclicpropertySPI}, a set of Stokes data is determined by the five Stokes multipliers $s_1,s_2,s_3,s_4,s_5$ subject to the following five (non-independent) equations
\begin{equation*}\label{eq:stokesmultrelsPI}
s_1=i \left(s_3 s_4+1\right), \quad 
s_2=i \left(s_4 s_5+1\right), \quad \
s_3=i \left(s_1 s_5+1\right), \quad 
s_4=i \left(s_1 s_2+1\right),\quad 
s_5=i \left(s_2 s_3+1\right).
\end{equation*}
Setting 
$$s_k =  i x_k,$$
we have the following description of the monodromy surface for $\pain{I}$.
\begin{definition}\label{def:monodromysurfacePI}
    The monodromy surface of $\pain{I}$ is the embedded affine variety 
\begin{equation*}
        \mathcal{M} = \Spec\C[x_1,x_2,x_3,x_4,x_5]/ K, \qquad 
        K = (k_1,k_2,k_3,k_4,k_5),
\end{equation*}
where 
\begin{align*}
    k_1 &= x_1x_2+x_4-1, \\
    k_2 &= x_2x_3+x_5-1, \\
    k_3 &= x_3x_4+x_1-1, \\
    k_4 &= x_4x_5+x_2-1, \\
    k_5 &= x_1x_5+x_3-1.
\end{align*}
\end{definition}

Under the embedding 
\begin{equation*}
    \begin{aligned}
        \mathbb{A}^5 &\to \p^5, \\
        (x_1,x_2,x_3,x_4,x_5) &\mapsto [1:x_1:x_2:x_3:x_4:x_5],
    \end{aligned}
\end{equation*}
the projective completion of $\mathcal{M}$ is given by 
\begin{equation*}
    \overline{\mathcal{M}} = \operatorname{Proj}\C[X_0,X_1,X_2,X_3,X_4,X_5]/\overline{K}, \quad \operatorname{deg}(X_i)=1, \quad \overline{K} = (\overline{k}_1,\overline{k}_2,\overline{k}_3,\overline{k}_4,\overline{k}_5),
\end{equation*}
where 
\begin{align*}
\overline{k}_1 &= X_1X_2+X_0X_4-X_0^2, \\
\overline{k}_2 &= X_2X_3+X_0X_5-X_0^2, \\
\overline{k}_3 &= X_3X_4+X_0X_1-X_0^2, \\
\overline{k}_4 &= X_4X_5+X_0X_2-X_0^2, \\
\overline{k}_5 &= X_1X_5+X_0X_3-X_0^2.
\end{align*}

\begin{remark}
    As an affine variety, $\mathcal{M}$ is isomorphic to the cubic surface 
    $$ \mathcal{C} = \Spec \C[x_1,x_2,x_3]/( x_1x_2x_3-x_1-x_3+1).$$
    The form of the cubic appearing in \cite{putsaito}, namely 
$$\tilde{x}_1\tilde{x}_2\tilde{x}_3+\tilde{x}_1+\tilde{x}_2+1=0,$$
    is related to ours via 
    $$\tilde{x}_1=-x_1, \quad \tilde{x}_2 = -x_3, \quad \tilde{x}_3 = x_2.$$
\end{remark}

\begin{lemma} \label{lem:propertiesofMPI}
    \begin{itemize}
        \item The surface $\overline{\mathcal{M}}$ is a smooth del Pezzo surface of degree 5.
    \item  The hyperplane section at infinity is the union of the five lines
\begin{equation} \label{eq:linesatinfinityPI}
    \begin{aligned}
        L_{1}^{\infty} &: & X_0= 0 ,& &X_1 = 0 ,& & X_2 = 0 ,& &X_4= 0,\\
        L_{2}^{\infty} &: & X_0 = 0 ,& &X_1 = 0 ,& & X_3 = 0 ,& &X_4= 0,\\
        L_{3}^{\infty} &: & X_0 = 0 ,& &X_1 = 0 ,& & X_3 = 0 ,& &X_5= 0,\\
        L_{4}^{\infty} &: & X_0 = 0 ,& &X_2 = 0 ,& & X_3 = 0 ,& &X_5= 0,\\
        L_{5}^{\infty} &: & X_0 = 0 ,& &X_2 = 0 ,& & X_4 = 0 ,& &X_5= 0,
    \end{aligned}
\end{equation}
which intersect like a pentagon, i.e. a cycle of length 5.
\end{itemize}
\end{lemma}
\begin{proof}
    For the first part, note that the standard model of a smooth del Pezzo surface of degree 5 as the blowup of $\p^2$ at four points, no three of which are collinear, leads to the following realisation.
    Take the vector space of cubic forms in $Z_0,Z_1,Z_2$ vanishing at the four points 
    $$[0:0:1], \qquad [0:1:0], \qquad [1:0:0], \qquad [1:1:1],$$
    with basis 
    \begin{equation*}
        \begin{aligned}
        a&= Z_1Z_2^2- Z_0^2 Z_1, &\quad& 
        b=Z_1^2 Z_2-Z_0^2 Z_1, &\quad& 
        c=Z_0 Z_2^2-Z_0^2 Z_1, \\
        d&=Z_0 Z_1 Z_2-Z_0^2 Z_1, &\quad& 
        e=Z_0 Z_1^2-Z_0^2 Z_1, &\quad& 
        f=Z_0^2 Z_2-Z_0^2 Z_1.
        \end{aligned}
    \end{equation*}
    Then the anticanonical embedding of the blow up of $\p^2$ at the four points above is realised as 
    \begin{equation} \label{eq:standarddelPezzodeg5PI}
        \overline{\mathcal{S}} = \Proj \C[a,b,c,d,e,f,g] / I,
    \end{equation}
    where the homogeneous ideal $I = (I_1,I_2,I_3,I_4,I_5)$ is generated by 
\begin{align*}
        I_1 &= a d-b c+c e-d e,\\
        I_2 &= a e-b d+b f-e f,\\
        I_3 &= a f-c d+c e-e f,\\
        I_4 &= b f-c e+d e-d f,\\
        I_5 &= c d-a f+d f-d^2
    \end{align*}
    A direct calculation shows that $\overline{\mathcal{S}}$ is isomorphic to $\overline{\mathcal{M}}$ via the projectivity of $\p^5$ defined by
    \begin{equation} \label{eq:projectivityPI}
        X_0 =  d,\quad 
        X_1 = f, \quad 
        X_2 = 2 d-a, \quad 
        X_3 = e, \quad 
        X_4 = c-f, \quad 
        X_5 = b-e,
    \end{equation}
    with inverse 
    \begin{equation*}
        a= 2 X_0-X_2, \quad 
        b= X_3+X_5, \quad 
        c= X_1+X_4, \quad 
        d= X_0, \quad 
        e= X_3, \quad 
        f= X_1,
    \end{equation*}
    which finishes the proof of the first part.
    The second part is verified directly.
\end{proof}

An alternative characterisation of monodromy of the linear system \eqref{eq:laxpairPI1} is through the Nevanlinna theory of branched coverings of the Riemann sphere \cites{nevanlinna32,elfving1934}. This leads to a space of asymptotic values \cite{masoerobethe},
\begin{equation*}
    V_5=W_5/\operatorname{PSL}_2(\C),
\end{equation*}
where
\begin{equation*}
    W_5=\left\{ (p_1,p_2,p_3,p_4,p_5)\in (\p^1)^{\times 5}: p_k\neq p_{k+1} \text{ for $k\in\mathbb{Z}/5\Z$} \right\},
\end{equation*}
and $\operatorname{PSL}_2(\C)$ acts diagonally. Denoting cross-ratios by
\begin{equation*}
    (a,b;c,d)=\frac{(a-c)(b-d)}{(a-d)(b-c)},
\end{equation*}
the variety $V_5$ is isomorphic to $\mathcal{M}$ via
\begin{equation}\label{eq:iso_models}
    x_k=(p_{k+1},p_{k-2};p_{k-1},p_{k+2}),\quad k\in\mathbb{Z}/5\Z,
\end{equation}
where the inequality of consecutive asymptotic values ensures that each of these cross-ratios only takes finite values on $V_5$. We then have
\begin{equation*}
    \mathcal{M}_{0,5}\subset V_5\subset \overline{\mathcal{M}}_{0,5},
\end{equation*}
where $\mathcal{M}_{0,5}$ is the moduli space of smooth 5-pointed rational curves and its compactification $\overline{\mathcal{M}}_{0,5}$ is the moduli space of stable 5-pointed rational curves \cite{knudsen}. The latter is a smooth del Pezzo surface of degree 5 \cite[Sec. 4.3]{kapranov}. In fact, the isomorphism \eqref{eq:iso_models} between $\mathcal{M}$ and $V_5$ extends uniquely to one between $\overline{\mathcal{M}}$
and the moduli space $\overline{\mathcal{M}}_{0,5}$.

The Riemann-Hilbert map in this case is 
\begin{equation*}
    \mathcal{X}_t \xrightarrow{\operatorname{RH}_t} \mathcal{M},
\end{equation*}
defined by plugging a local solution of the Hamiltonian system \eqref{eq:hamPI} into the linear problem \eqref{eq:laxpairPI} and computing its Stokes data. This mapping is analytic and injective by general theory,
surjective \cites{kk,piwkb} and consequently its inverse is analytic by application of the analytic Fredholm alternative to a Riemann-Hilbert representation of $\Pone$ \cite{fokas}, so that $\operatorname{RH}_t$ is a biholomorphism.

\begin{proposition}
    The symmetry $\vartheta$ generating the five-fold symmetry of $\pain{I}$ conjugates under the Riemann-Hilbert map to the  automorphism of $\mathcal{M}$ defined by 
    \begin{equation*}
        x_k \mapsto \widetilde{x}_k=x_{k+2},\quad k\in\mathbb{Z}/5\mathbb{Z}.
    \end{equation*}
\end{proposition}
\begin{proof}
    The five-fold symmetry $\vartheta$ in \cref{eq:discreteonqpPI} is realised by the following transformation of the linear problem, \eqref{eq:laxpairPI},
\begin{equation}\label{eq:PIvarthetaonlin}
    Y(z)\mapsto \widetilde{Y}(z)=\zeta_5^{-2\sigma_3}Y(\zeta_5^{-2}z), \qquad A(z)\mapsto \widetilde{A}(z)=\zeta_5^{-2}\zeta_5^{-2\sigma_3} A(\zeta_5^{-2} z) \zeta_5^{2\sigma_3}.
\end{equation}
With regards to the canonical solutions around infinity, this means that
\begin{equation*}
\widetilde{Y}_k(z) =\zeta_5^{-2\sigma_3} Y_{k+2} (\zeta_5^{-2}z).
\end{equation*}
for $k\in\mathbb{Z}/5\mathbb{Z}$ and correspondingly the Stokes matrices transform as 
\begin{equation*}
    \tilde{S}_k = S_{k+2}, \quad k\in \Z,
\end{equation*}
so 
\begin{equation*}
    \tilde{s}_1=s_3,\quad 
    \tilde{s}_2=s_4,\quad
    \tilde{s}_3=s_5,\quad 
    \tilde{s}_4= s_1,\quad 
    \tilde{s}_5= s_2,
\end{equation*}
and the proposition follows.
\end{proof}

\subsection{Moduli space}

We show that the surface $\mathcal{M}$ represents the single isomorphism class in the following category of embedded affine del Pezzo surfaces.

\begin{definition}[Category $\mathfrak{C}_{\rm{I}}$ of monodromy surfaces for $\pain{I}$]
Define the category $\mathfrak{C}_{\rm{I}}$ with 
\begin{itemize}
    \item an object being an embedded affine surface $\mathcal{V}\subset \mathbb{A}^5$, such that its projective completion $\overline{\mathcal{V}}\subset \p^5$ is a del Pezzo surface of degree five, and further that the hyperplane section at infinity $\overline{\mathcal{V}}\setminus \mathcal{V}$ is the union of five lines, which intersect like a pentagon, consisting of smooth points of $\overline{\mathcal{V}}$.
    \item a morphism between objects $\mathcal{V}_1$ and $\mathcal{V}_2$ being an affine linear $L\in\operatorname{End}(\mathbb{A}^5)$ such that $L(\mathcal{V}_1)\subseteq\mathcal{V}_2$.
\end{itemize}
\end{definition}

\begin{proposition} \label{prop:oblomkovstylenormalformPI}
    Any $\mathcal{V}\in \operatorname{ob}(\mathfrak{C}_{\rm I})$ is isomorphic in $\mathfrak{C}_{\rm I}$ to $\mathcal{M}$.  
    The automorphism group of $\mathcal{M}$ in $\mathfrak{C}_{\rm I}$ is the dihedral group $\mathfrak{D}_5$, generated by the action of the five-fold symmetry $\vartheta$ as well as the automorphism
    $$\varpi : x_1 \mapsto x_1,\quad x_2\mapsto x_5, \quad x_3 \mapsto x_4, \quad x_4\mapsto x_3, \quad x_5 \mapsto x_2.$$

\end{proposition}

\begin{proof}
The idea of the proof is similar to that of \Cref{prop:oblomkovstylenormalformPII}.
For the first part, let $\mathcal{V}\in \operatorname{ob}(\mathfrak{C}_{\rm{I}})$. 
Begin by noting that the hyperplane section $\overline{\mathcal{V}}\setminus \mathcal{V}$ consisting of a cycle of five line implies that $\overline{\mathcal{V}}$ must be smooth. 
This follows from the classification of possibly singular del Pezzo surfaces of degree five as blow-ups of $\p^2$ at four points in \cite[Sec. 8.5]{dolgachevclassicalAG}.
The only cases where the surface is singular and has at least five lines correspond to blowup of $\p^2$ at four points $b_1,\dots,b_4$ in one of the following configurations:
\begin{enumerate}
    \item $b_1,\dots,b_4$ are distinct, with three collinear. In this case the surface has a single singularity of type $A_1$ and contains $7$ lines.
    \item two points $b_1,b_2$ are infinitely near, and no three of $b_1,\dots,b_4$ are collinear. In this case the surface has a single singularity of type $A_1$ and contains $7$ lines.
    \item two points $b_1,b_2$ are infinitely near, and three of the points $b_1,\dots,b_4$ are collinear. In this case the surface has two singularities of type $A_1$ and contains $5$ lines.
    \item two pairs $b_1,b_2$ and $b_3,b_4$ are infinitely near, and no three of the points $b_1,\dots,b_4$ are collinear. In this case the surface has two singularities of type $A_1$ contains $5$ lines.
\end{enumerate}
The fact that none of these cases lead to a surface containing five lines that intersect like a pentagon is verified directly.

Now since $\overline{\mathcal{V}}$ is smooth, it is the blowup of $\p^2$ at four points $b_1,\dots,b_4$, no three of which are collinear.
These can be sent via a projectivity to 
$$b_1 = [0:0:1], \quad b_2 = [0:1:0], \quad b_3=[1:0:0], \quad b_4 = [1:1:1],$$ and there is an isomorphism from $\overline{\mathcal{V}}$ to the surface $\overline{\mathcal{S}}$ in \cref{eq:standarddelPezzodeg5PI}, which can be chosen such that the hyperplane section $\overline{\mathcal{V}}\setminus \mathcal{V}$ is sent to $\{d=0\}\cap\overline{\mathcal{S}}$.
Then the isomorphism from $\overline{\mathcal{S}}$ to $\overline{\mathcal{M}}$, defined by the projectivity \eqref{eq:projectivityPI},  sends the hyperplane section $\{d=0\}\cap \overline{\mathcal{S}}$ to $\{X_0=0\}\cap \overline{\mathcal{M}}$. 
This induces an isomorphism from  $\mathcal{V}$ to $\mathcal{M}$ in $\mathfrak{C}_{\rm{I}}$.

For the second part, we describe the automorphism group of $\mathcal{M}$ in $\mathfrak{C}_{\rm{I}}$ as follows. 
Note that the automorphism group of $\overline{\mathcal{M}}$ as a projective variety is $\mathfrak{S}_5$, which can be regarded as the Weyl group $W(A_4)$ through its action on the Picard group $\Pic(\overline{\mathcal{M}})$. 
Via the realisation of $\overline{\mathcal{M}}$ as $\p^2$ blown up at four distinct points, this is described as
$$\Pic(\overline{\mathcal{M}}) = \Z \h +\Z\E_1 + \Z \E_2+ \Z \E_3 + \Z \E_4,$$
where $\E_1,\E_2,\E_3,\E_4$ correspond to a choice of four disjoint lines on $\overline{\mathcal{M}}$.
The Cremona isometries $\operatorname{Cr}(\overline{\mathcal{M}})$ form the group $W(A_4)\cong \mathfrak{S}_5$, generated by reflections 
            $$r_{\alpha_i} (\F) = \F + (\F\cdot \alpha_i)\,\alpha_i,$$
        associated to the simple roots 
$$ \alpha_1=\E_1-\E_2,\quad \alpha_2=\E_2-\E_3, \quad \alpha_3=\E_3-\E_4, \quad \alpha_4 = \h-\E_1-\E_2-\E_3.$$
There is an isomorphism of groups 
$\Aut(\overline{\mathcal{M}}) \cong \operatorname{Cr}(\overline{\mathcal{M}})$,
sending an automorphism to the induced action on $\Pic(\overline{\mathcal{M}})$.
The enumeration of exceptional divisors can be chosen such that the cycle of five lines corresponds to the following set of elements of $\Pic(\overline{\mathcal{M}})$:
\begin{equation} \label{eq:cycleoflinesPicPI}
    \operatorname{EX}^{\infty} = \left\{\h - \E_1 - \E_2, \,\, \E_2, \,\, \h - \E_2-\E_3, \,\, \E_3, \,\, \h-\E_3- \E_4\right\}.
\end{equation}
The subgroup of $\Aut(\overline{\mathcal{M}})$ that induces automorphisms of $\mathcal{M}$ in $\mathfrak{C}_{\rm{I}}$ consists of the elements that leave the set in \eqref{eq:cycleoflinesPicPI} invariant.
This is computed directly to be $$\left\{ w \in \operatorname{Cr}(\overline{\mathcal{M}}) ~|~ w (\operatorname{EX}^{\infty})= \operatorname{EX}^{\infty} \right\} \cong \mathfrak{D}_5 = \langle \sigma, \rho ~|~ \sigma^2=\rho^5=1, \sigma \rho=\rho^{4}\sigma\rangle,$$ 
with generators given in terms of the simple reflections $r_i = r_{\alpha_i}, i=1,\dots,4$ by 
$$\rho = r_2r_3r_1r_2r_3 r_4,\quad  \sigma = r_1r_4.$$
The fact that this subgroup of $\operatorname{Cr}(\overline{\mathcal{M}})$ is induced by the subgroup of $\operatorname{Aut}(\overline{\mathcal{M}})$ generated by $\vartheta$ and $\varpi$ is checked by direct calculation.
\end{proof}

We have the following directly from \Cref{prop:oblomkovstylenormalformPI}, by analogy with \Cref{cor:modulispacePVI,cor:modulispacePIV,cor:modulispacePII}.
\begin{corollary} \label{cor:modulispacePI}
    The moduli space of $\mathfrak{C}_{\rm{I}}$ is a point, $\mathscr{O}_{\rm{I}}=\operatorname{Spec} \C$.
\end{corollary}

\begin{remark} \label{rem:mysterysymmetryPI}
    The automorphism $\varpi$ of $\mathcal{M}$ introduced in \Cref{prop:oblomkovstylenormalformPI} conjugates under $\operatorname{RH}_{t}$ to a biholomorphic involution of $\mathcal{X}_{t}$. 
    As far as we know, a corresponding transcendental symmetry of $\pain{I}$ has not appeared in the literature.
\end{remark}

\subsection{Lines on the monodromy surface}

As a del Pezzo surface of degree 5, $\overline{\mathcal{M}}$ contains 10 lines.
Of these, there are 5 lines at infinity as in \eqref{eq:linesatinfinityPI},
as well as 5 affine lines, given by 
\begin{equation} \label{eq:linesaffPI}
    \begin{aligned}
        L_{1} &: & X_1 = X_0 ,& &X_2 = X_0 ,& & X_3+X_5 = X_0 ,& &X_4= 0,\\
        L_{2} &: & X_3 = X_0 ,& &X_4 = X_0 ,& & X_2+X_5 = X_0 ,& &X_1= 0,\\
        L_{3} &: & X_5 = X_0,& &X_1 = X_0 ,& & X_2+X_4 = X_0 ,& &X_3= 0,\\
        L_{4} &: & X_2 = X_0 ,& &X_3 = X_0 ,& & X_1+X_4 = X_0 ,& &X_5= 0,\\
        L_{5} &: & X_4 = X_0 ,& &X_5 = X_0 ,& & X_1+X_3 = X_0 ,& &X_2= 0.
    \end{aligned}
\end{equation}
The intersection graph of lines on $\overline{\mathcal{M}}$ is defined analogously to the cases in previous sections, being the coloured graph with blue vertices $1,\dots,5$ corresponding to the lines $L_1,\dots,L_5$ in \cref{eq:linesaffPI} and red vertices $6,\dots,10$ corresponding to the lines at infinity $L_{1}^{\infty},\dots,L_5^{\infty}$ in \cref{eq:linesatinfinityPI}, with edges indicating pairwise intersections. This graph $\mathcal{G}_{\rm{I}}$ (which when uncoloured is the famous Petersen graph) and the subgraph $\mathcal{G}_{\rm{I}}^{\infty}$ generated by the red vertices are shown in \Cref{fig:linesgraphpI}.

\begin{remark} \label{rem:monodromylinesexplainedPI}
    Under the Riemann-Hilbert map, each of the lines $L_k$, $1\leq k\leq 5$, corresponds to a one-parameter family of bitronqu\'ee solutions and their five points of intersection correspond to the five  tritronqu\'ee solutions of $\pain{I}$ \cites{kk,kapaev2004}. 
\end{remark}

\begin{remark}
    Recall that $\overline{\mathcal{M}}$ is isomorphic to the moduli space $\overline{\mathcal{M}}_{0,5}$ of stable rational curves with five marked points. 
    In terms of this model, the geometric meaning of the affine lines is given in \cite[Lem. 1]{masoerobethe}.
    On $\overline{\mathcal{M}}$, $x_k=0$ if and only if $p_{k-1}=p_{k+1}$ on $\overline{\mathcal{M}}_{0,5}$ under the isomorphism in $\eqref{eq:iso_models}$. 
    For the lines at infinity, these correspond to pairs of consecutive points coinciding,
    $p_k = p_{k-1}$.
\end{remark}

\begin{remark}
    While it is clear from \Cref{fig:linesgraphpI} that $\operatorname{Fix}_{\Aut(\mathcal{G}_{\rm{I}})}(\mathcal{G}_{\rm{I}}^{\infty})$ is trivial, the automorphisms that only set-wise fix $\mathcal{G}_{\rm{I}}^{\infty}$ form $\mathfrak{D}_5$. 
    The five-fold symmetry accounts for the cyclic part of this, with $\vartheta$ inducing the permutation
    $$\vartheta: (L_1\,\,L_2\,\,L_3\,\,L_4\,\,L_5)(L_1^{\infty}\,\,L_2^{\infty}\,\,L_3^{\infty}\,\,L_4^{\infty}\,\,L_5^{\infty}).$$
    The automorphism $\varpi$ of $\mathcal{M}$  introduced in \Cref{prop:oblomkovstylenormalformPI} induces the permutation of lines 
    $$\varpi: (L_2\,\,L_5)(L_3\,\,L_4)(L_2^{\infty}\,\,L_5^{\infty})(L_3^{\infty}\,\,L_4^{\infty}).$$

\end{remark}

\begin{figure}[htb]
    \centering
    \includegraphics[width=0.38\linewidth]{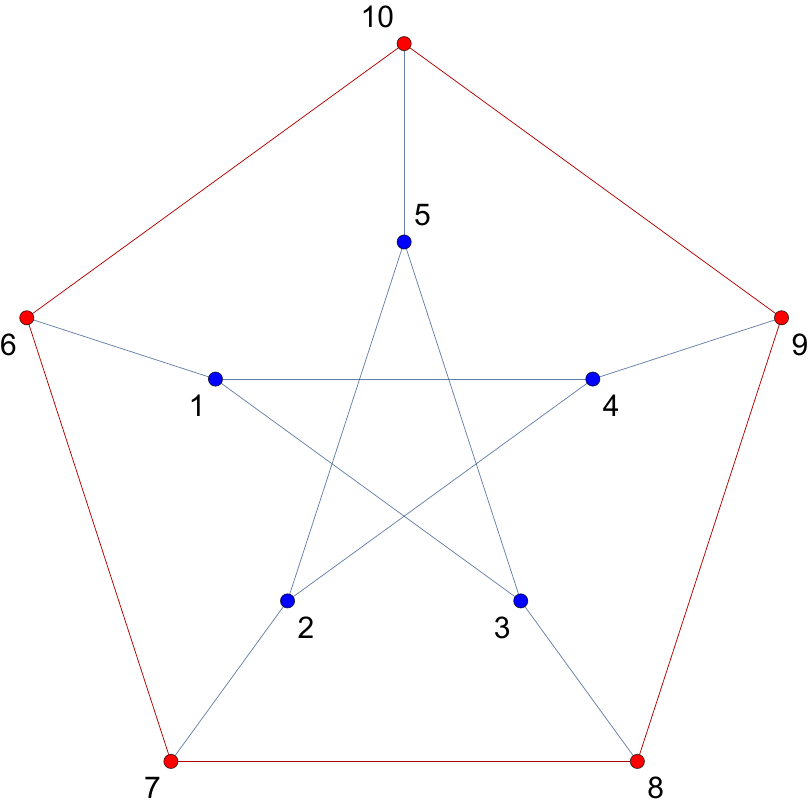}
    \caption{Graph $\mathcal{G}_{\rm{I}}$ encoding intersections of lines on the monodromy surface for $\pain{I}$,
    with in red the subgraph $\mathcal{G}_{\rm{I}}^\infty$ corresponding to lines at infinity. 
    }
    \label{fig:linesgraphpI}
\end{figure}

 \Cref{prop:oblomkovstylenormalformPI} and \Cref{cor:modulispacePI} establish  \Cref{mainthm:categoryandmodulispace} in the case of Painlev\'e-I.
 Theorems \ref{mainthm:conjugatedsymmetries} and \ref{mainthm:mon} are trivial in the case of Painlev\'e-I, since there are no nontrivial B\"acklund transformations.

\section{Conclusion} \label{sec:conclusion}
In this paper, we derived models of monodromy surfaces of $\pain{VI},\pain{IV},\pain{II},\pain{I}$ as embedded affine del Pezzo surfaces,  each characterised by a fixed degree and type of divisor at infinity.
This allowed us to resolve \Cref{main:problem} for these Painlev\'e equations by showing that the finite Weyl groups underlying their symmetry groups  appear on the right-hand side of the Riemann-Hilbert correspondence as monodromy groups of monodromy surfaces. 

It is natural ask what the categories of embedded affine del Pezzo surfaces are for the remaining Painlev\'e differential equations $\pain{V},\pain{III}^{D_6},\pain{III}^{D_7},\pain{III}^{D_8}$ as well as for discrete ones.
One interpretation of the program initiated in this paper is the formulation of a counterpart to the Sakai classification of discrete and differential Painlev\'e equations, on the monodromy surface side of the Riemann-Hilbert correspondence.

We expect that such a full classification exists, and remark that related conjectures were made in \cite{chekhovconj}.
The definitions of the categories associated with some $q$-discrete examples are obvious from \cites{joshirofqPVI,joshirofpiv,JMR}, and we conjecture that a $q$-Painlev\'e equation of surface type $A_k^{(1)}$ corresponds to a category of affine del Pezzo surfaces of degree $k+1$ with hyperplane section at infinity being an elliptic curve, for $0\leq k\leq 8$. 
Such a classification may extend beyond second-order discrete and differential Painlev\'e equations, and one may reasonably expect Fano varieties on the monodromy side.

\appendix

\section{Singular locus of the Jimbo-Fricke cubic}
\label{app:singularlocusofJimboFricke}
\begingroup\allowdisplaybreaks
In \Cref{rem:propertiesofMwPVI}, the singular locus $\mathscr{W}_{\rm{VI}}^{\operatorname{sing}}\subset \mathscr{W}_{\rm{VI}}$ of the monodromy surface $\overline{\mathcal{M}}_w$ for $\pain{VI}$, i.e. where the Jimbo-Fricke cubic is singular, was said to be given by $Q(w)=0$.
The polynomial $Q\in \mathbb{Q}[w_1,w_2,w_3,w_4]$ is
\begin{equation*}
    Q(w)= w_4^{5} + Q_4 w_4^4+Q_3 w_4^3+ Q_2 w_4^2+Q_1 w_4 + Q_0,
\end{equation*}
where
\begin{align*}
    Q_4 &= 16-\frac{1}{4}\left(w_1^2+w_2^2+w_3^2\right),
     \\
    Q_3 &=   96
    -\frac{11}{2} w_1 w_2 w_3 
    -12 (w_1^2+ w_2^2+ w_3^2)
    + \frac{1}{16}\left(w_1^2 w_2^2 +w_1^2 w_3^2 +w_2^2 w_3^2\right) ,
      \\
    Q_2  &= 256 
    -\frac{1}{64}w_1^2 w_2^2 w_3^2 
    + 22 w_1w_2w_3
     +\frac{5}{4} w_1w_2w_3\left(w_1^2+ w_2^2+w_3^2\right)
    \\
    &\quad 
   +2 \left(w_1^4+ w_2^4+ w_3^4\right)
    -88 \left(w_1^2 +w_2^2+ w_3^2\right)
    +\frac{49}{4}\left(  w_1^2w_2^2 +  w_1^2w_3^2 +w_2^2 w_3^2\right)
    ,
      \\
    Q_1  &= 
        256
     -\frac{11}{8} w_1^2 w_2^2 w_3^2  
   +184 w_1 w_2 w_3 
   -\frac{9}{32}w_1w_2w_3 (w_1^2 w_2^2 +  w_1^2 w_3^2 +w_2^2 w_3^2) 
    \\
    &\quad 
    -16w_1w_2w_3 (w_1^2 + w_2^2 +  w_3^2)
    +32 (w_1^4+ w_2^4+ w_3^4)
     -192 (w_1^2+w_2^2+ w_3^2)
     \\
    &\quad 
     -\frac{9}{4}(w_1^4 w_2^2 +w_1^4 w_3^2 
     +w_1^2 w_2^4 +w_2^4 w_3^2
     +w_1^2 w_3^4  +w_2^2 w_3^4)  
   +31 (w_1^2 w_2^2 + w_1^2 w_3^2 + w_2^2 w_3^2 )
    ,
     \\
    Q_0 &=  
     \frac{1}{16}w_1^3 w_2^3 w_3^3 
    +\frac{83}{4}w_1^2 w_2^2 w_3^2  
   + 32 w_1 w_2 w_3 
    -4 \left(w_1^6
    + w_2^6
    + w_3^6    
    \right)   
   +32 \left(
    w_1^4
    + w_2^4
    + w_3^4
    \right)    
    \\
    &\quad 
    -64 \left(w_1^2
    + w_2^2
    + w_3^2 
    \right)
     +3 w_1 w_2 w_3\left( w_1^4 
     + w_2^4 
    + w_3^4\right)
    -52 w_1 w_2 w_3 \left( w_1^2
    + w_2^2 
    + w_3^2 
    \right) 
    \\
    &\quad 
    +\frac{3}{32} w_1^2 w_2^2 w_3^2\left( 
    w_1^2
    + w_2^2 
    + w_3^2 
    \right)
    +\frac{3}{8} w_1w_2w_3\left(   w_1^2 w_2^2
    + w_1^2 w_3^2 
    +  w_2^2 w_3^2
    \right)
    +\frac{27}{64} \left(w_1^4 w_2^4 
    + w_1^4 w_3^4 
    +  w_2^4 w_3^4
    \right)    
    \\
    &\quad 
    -3 \left(  
    w_1^4 w_2^2 
    +w_1^4 w_3^2 
    +w_1^2 w_2^4 
    +w_2^4 w_3^2
    +w_1^2 w_3^4  
    +w_2^2 w_3^4
    \right)
    +60 \left( w_3^2 w_1^2
    + w_2^2 w_1^2
    + w_2^2 w_3^2
    \right).
    \end{align*}
After the quotient by $\operatorname{Aut}(D_4^{(1)})$ in \cref{eq:moduli_map}, the singular locus $\mathscr{O}^{\operatorname{sing}}_{\rm{VI}}\subset\mathscr{O}_{\rm{VI}}$ was said to be given by $P(o)=0$. The polynomial $P\in \mathbb{Q}[o_1,o_2,o_3,o_4]$ is 
\begin{equation*}
    P(o)= o_4^{10} + \sum_{n=0}^{9}P_n(o_1,o_2,o_3)o_4^{n},
\end{equation*}
where 

\begin{align*}
    P_{9} &= -\frac{o_1}{2}+ 32, \\
    P_{8} &= \frac{1}{16}o_1^2-32 o_1+\frac{1}{8}o_2+448, \\
    P_{7} &= 10 o_1^2-\frac{1}{32}o_1 o_2 -608 o_1+\frac{37 }{2}o_2-\frac{1}{32}o_3+3584, \\
    P_{6} &= 
    -o_1^3
    +316 o_1^2
    +\frac{1}{256}o_2^2
    -\frac{81}{8}o_1o_2 
    +\frac{1 }{128}o_1o_3
    \\
    &\quad -5632 o_1
    +210 o_2
    -20 o_3
    +17920, \\
    P_{5} &= 
    -72 o_1^3
    +3680 o_1^2
    +\frac{15 }{8}o_2^2
    +\frac{11}{8} o_2 o_1^2
    -\frac{493 }{2}o_1 o_2 
    +\frac{159 }{16}o_1 o_3 
    -\frac{1}{512}o_2 o_3
    \\
    &\quad 
    -29440 o_1
    +552 o_2
    +\frac{893 }{2}o_3
    +57344, \\
    P_{4} &= 
    6 o_1^4
    -1264 o_1^3
    +20576 o_1^2
    +\frac{1239 }{16}o_2^2
    +\frac{1}{4096}o_3^2
    +\frac{173}{2}  o_1^2o_2
    -\frac{5}{4}  o_1^2 o_3
    -\frac{63}{128}o_1 o_2^2  
    \\
    &\quad 
    -826 o_1 o_2 
    -\frac{3121 }{8}o_1 o_3 
    -\frac{343 }{128}o_2 o_3
    -90112 o_1
    -2208 o_2
    +3132 o_3
    +114688, \\
    P_{3} &= 
    224 o_1^4
    -7936 o_1^3
    -\frac{19}{2} o_1^3 o_2 
     +\frac{27 }{512}o_2^3
   +58880 o_1^2
     -464 o_2^2
   +\frac{133 }{256}o_3^2
     +\frac{87}{128} o_1 o_2 o_3 
   \\
    &\quad 
   +580 o_1^2 o_2 
    +\frac{223}{2} o_1^2 o_3 
    -\frac{369}{8} o_1 o_2^2 
    +3304 o_1 o_2 
    -1838 o_1 o_3 
    +\frac{679 }{16}o_2 o_3
    \\
    &\quad 
    -155648 o_1
    -13440 o_2
    -1592 o_3
    +131072, \\
    P_{2} &= 
    -16 o_1^5
    +1856 o_1^4
    -20224 o_1^3
    +\frac{891 }{128}o_2^3
    +80896 o_1^2
    +1239 o_2^2
    +\frac{3275 }{128} o_3^2
    -174 o_1^3 o_2 
    -\frac{83}{8} o_1^3 o_3 \\
    &\quad 
    +\frac{27}{4} o_1^2 o_2^2 
    -2320 o_1^2 o_2 
    +288 o_1^2 o_3 
    +\frac{891}{4}  o_1 o_2^2
    -\frac{189 }{2048}o_2^2 o_3
    -\frac{17}{128}  o_1 o_3^2
    -\frac{1017}{32} o_1 o_2 o_3 \\
    &\quad 
    +15776  o_1 o_2
    +2526 o_1 o_3 
    +\frac{1157 }{4}o_2 o_3
    -131072 o_1
    -18944 o_2
    -23424 o_3
    +65536, \\
    P_{1} &= 
    -256 o_1^5
    +3584 o_1^4
    -18432 o_1^3
    -\frac{891 }{32}o_2^3 
    +40960 o_1^2
    +480 o_2^2
    +\frac{2685 }{16}o_3^2
    +18 o_1^4 o_2 
    \\
    &\quad 
    +696 o_1^3 o_2 
    +5 o_1^3 o_3 
    -\frac{243}{128} o_1 o_2^3
    -\frac{27}{2} o_1^2 o_2^2 
    +\frac{81}{16} o_1^2 o_2 o_3 
    -\frac{1143}{8} o_1 o_2 o_3 
    \\
    &\quad 
    -5536 o_1^2 o_2 
    -738 o_1 o_2^2 
    -\frac{97}{16}  o_1 o_3^2
    +\frac{9}{256}  o_2 o_3^2
    +\frac{351}{256} o_2^2 o_3
    -1000 o_1^2 o_3 
     \\
    &\quad    
    +10368 o_1 o_2 
    +12272  o_1 o_3
    +\frac{1747 }{2}o_2 o_3
    -32768 o_1
    -2048 o_2
    -2688 o_3
    , 
    \\
    P_{0} &= 
    16 o_1^6
    -256 o_1^5
    +1536 o_1^4
     +\frac{729 }{4096}o_2^4
   -72 o_1^4 o_2 
    -3 o_1^4 o_3 
    -4096 o_1^3
    -\frac{27 }{8} o_2^3
    -\frac{1}{256} o_3^3
    -\frac{27}{8} o_1^3 o_2^2 
    \\
    &\quad 
    +\frac{3}{16}  o_1^2 o_3^2
    +108 o_1^2 o_2^2 
    +608 o_1^3 o_2 
    +122  o_1^3 o_3
    +\frac{81}{4} o_1^2 o_2 o_3 
     -\frac{81}{128} o_1 o_2^2 o_3 
    \\
    &\quad 
    -1408 o_1^2 o_2 
   -1664 o_1^2 o_3 
    +\frac{243}{32} o_1 o_2^3
    -126 o_1 o_2^2 
    -\frac{161}{8} o_1 o_3^2
    +\frac{45}{64} o_2 o_3^2
    -\frac{2133}{128} o_2^2 o_3
    \\
    &\quad
    -\frac{519}{2} o_1 o_2 o_3 
    +4096 o_1^2
    +16 o_2^2
    +\frac{4977 }{16}o_3^2
    +512 o_1 o_2 
    +1056 o_1 o_3 
    +218 o_2 o_3
    -1024 o_3
    .
\end{align*}
\endgroup

\section{Derivation of monodromy surface for $\pain{II}$ from Flaschka-Newell linear problem} \label{app:derivationofmonodromysurfaceFN}
In this Appendix, we derive the monodromy surface in 
 \Cref{def:monodromysurfacePII} from the Flaschka-Newell linear problem \eqref{eq:FNlaxA}.

There is a unique formal solution
\begin{equation*}
    Y_{\text{form}}(z)=P(z)e^{(\frac{4}{3}z^3-tz)\sigma_3},
\end{equation*}
where $P(z)$ is a power series in $z$ around $z=\infty$,
\begin{equation*}
    P(z)=I+\sum_{n=1}^\infty z^{-n}U_n(t).
\end{equation*}

We have Stokes sectors
\begin{equation*}
    \Sigma_k=\left\{\left|\arg z- \frac{(k-1)\pi}{3}\right|<\frac{\pi}{6}\right\}\subseteq \widetilde{\mathbb{C}^*},
\end{equation*}
where $k\in\mathbb{Z}$.

For any $k\in\mathbb{Z}$, there exists a unique solution $Y_k$ of the linear problem that satisfies
\begin{equation*}
    Y_k(z)\sim Y_\text{form}(z)\qquad (z\in \Sigma_k\cup \Sigma_{k+1}, z\rightarrow \infty),
\end{equation*}
Correspondingly, the Stokes phenomenon takes the form
\begin{equation*}
    Y_{k+1}(z)=Y_k(z)S_k,
\end{equation*}
where
\begin{equation*}
    S_k=\begin{bmatrix}
        1 & 0\\
        s_k & 1
    \end{bmatrix}\,\,\, \text{if $k$ even},\qquad
    S_k=\begin{bmatrix}
        1 & s_k\\
        0 & 1
    \end{bmatrix}\,\,\, \text{if $k$ odd},
\end{equation*}
for $k\in\mathbb{Z}$.

By the symmetry of the linear problem in \eqref{eq:Asymmetry},
\begin{equation*}
    Y_{k+3}(e^{\pi i}z)=\sigma_1Y_k(z)\sigma_1.
\end{equation*}
Therefore $S_{k+3}=\sigma_1S_k\sigma_1$ and hence $s_{k+3}=s_k$
for $k\in\mathbb{Z}$.

For arbitrary $\gamma\in\mathbb{C}^*$, there exists a solution
\begin{equation}\label{eq:FNY0sol}
    Y_{\underline{0}}(z)=G Q(z)z^{-(\frac{1}{2}+\theta)\sigma_3},\qquad G:=\frac{1}{\sqrt{2}}\begin{bmatrix}
        \gamma & -\gamma^{-1}\\
        \gamma & \gamma^{-1}
    \end{bmatrix},
\end{equation}
where $Q(z)$ is analytic at $z=0$ with $Q(z)=I+\mathcal{O}(z)$ as $z\rightarrow 0$. Upon fixing $\gamma$, $Q(z)$ is unique.

We introduce connection matrices,
\begin{equation*}
    Y_k(z)=Y_{\underline{0}}(z)E_k,
\end{equation*}
$k\in\mathbb{Z}$. For the sake of definitiveness, we impose that the complex powers in the defining equations of $Y_0(z)$ and $Y_{\underline{0}}(z)$ are principal on the same sheet $\mathbb{C}\setminus\mathbb{R}_{\leq 0}$. Note that each connection matrix lies in $\operatorname{SL}_2(\mathbb{C})$.

By symmetry \eqref{eq:Asymmetry},
\begin{equation*}
    Y_{\underline{0}}(e^{\pi i}z)=-i\, \sigma_1 Y_{\underline{0}}(z)u^{-\sigma_3},
\end{equation*}
where we again use the notation  $u=e^{\pi i \theta}$. This means that 
\begin{equation*}
    E_{k+3}=i\,u^{\sigma_3}E_k\, \sigma_1,
\end{equation*}
for $k\in\mathbb{Z}$.

We are thus left with three Stokes matrices $S_1,S_2,S_3$ and three connection matrices $E_1,E_2,E_3$, satisfying
\begin{equation}\label{eq:monrelations}
\begin{aligned}
        E_2=E_1S_1,\\
        E_3=E_2S_2,\\
        iu^{\sigma_3} E_1\,\sigma_1=E_3S_3.        
\end{aligned}
\end{equation}
Combining these three equations gives
\begin{equation*}
    S_1S_2S_3\sigma_1= E_1^{-1}iu^{\sigma_3}E_1,
\end{equation*}
and taking the traces of both sides shows that the Stokes multipliers satisfy the cubic equation
\begin{equation} \label{eq:cubicFN}
    s_1s_2s_3 +s_1+s_2+s_3- i (u+\tfrac{1}{u})=0,
\end{equation}
derived in Flaschka and Newell's original paper \cite[Eq. (3.24)]{flaschkanewell}.

The matrices $E_{k}$, $1\leq k\leq 3$, are defined up to the choice of $\gamma$ in equation \eqref{eq:FNY0sol}, whose freedom induces the $\mathbb{C}^*$-action
\begin{equation} \label{eq:scalingactionFNPII}
  S_k\mapsto S_k,\quad  E_k\mapsto r^{\sigma_3} E_k,\qquad (k=1,2,3),
\end{equation}
through $\gamma\mapsto \gamma/r$, since it scales the fundamental solution around zero as $Y_{\underline{0}}(z)\mapsto Y_{\underline{0}}(z)r^{-\sigma_3}$.

Writing
\begin{equation*}
    E_k=\begin{bmatrix}
        a_k & b_k\\
        c_k & d_k
    \end{bmatrix},
\end{equation*}
we have the equalities
\begin{equation}\label{eq:redundantvar}
\begin{aligned}
  a_2&=a_1, & c_2&=c_1, & b_3&=b_2,\\
 d_3&=d_2, & a_3&=i\,u\,b_1, & c_3&=i\,u^{-1}d_1,  
\end{aligned}
\end{equation}
which allow us to eliminate the variables $a_2,c_2,a_3,b_3,c_3,d_3$. Correspondingly, we interpret a set of monodromy data for the linear problem as a point on the affine variety
\begin{equation*}
    M = \operatorname{Spec} \C [s_1,s_2,s_3,a_1,b_1,c_1,d_1,b_2,d_2]/I,
\end{equation*}
where $I = (i_1,i_2,i_3,i_4,i_5,i_6,i_7)$ is the ideal generated by the seven elements
\begin{equation*}
    \begin{aligned}
        i_1=&\,a_1\,s_1+b_1-b_2, & i_4=&\,d_2\,s_2+c_1-i\, u^{-1}\, d_1,\\
        i_2=&\,c_1\,s_1+d_1-d_2, &     i_5=&\,b_1\,s_3-a_1-i\, u^{-1}\, b_2,\\
        i_3=&\,b_2\,s_2+a_1-i\, u\,b_1, & i_6=&\,d_1\,s_3-c_1-i\,u\, d_2,\\
        i_7=&\,a_1\,d_1-b_1\,c_1-1, &&
    \end{aligned}
\end{equation*}
where vanishing of the first six together with \eqref{eq:redundantvar} is equivalent to equations \eqref{eq:monrelations} and vanishing of $i_7$ corresponds to $|E_1|=1$.

Note that $\operatorname{dim}M=3$, and that the generic fibre of the projection from $M$ to the space $$\operatorname{Spec}\C[s_1,s_2,s_3] / (s_1s_2s_3 +s_1+s_2+s_3- i (u+\tfrac{1}{u}))$$ of Stokes multipliers subject to the cubic equation  \eqref{eq:cubicFN} is one-dimensional.

Induced by \eqref{eq:scalingactionFNPII} we have the $\C^*$-action on $M$ given by 
\begin{equation*}
(s_1,s_2,s_3,a_1,b_1,c_1,d_1,b_2,d_2)\mapsto (s_1,s_2,s_3,r\,a_1,r\,b_1,r^{-1}c_1,r^{-1}d_1,r\,b_2,r^{-1}d_2).
\end{equation*}
Denoting $R=\C[s_1,\dots,d_2]/I$, the ring of invariants in $R$ under the $\C^*$-action is generated by 
\begin{equation*}
    \begin{aligned}
    y_1 &= s_1, &\quad y_2 &= s_2, &\quad y_3 &= s_3, \\
    y_4 &= a_1\, c_1, &\quad y_5 &= b_2\,d_2, &\quad y_6 &= b_1\, d_1, \\
    y_7 &= a_1\, d_1, &\quad y_8 &= a_1\,d_2, &\quad y_9 &= b_1\, c_1,\\
    y_{10} &= b_1\,d_2, &\quad y_{11} &= b_2 \, c_1, &\quad y_{12} &= b_2\, d_1,
    \end{aligned}
\end{equation*}
among which there are the following six linear relations in $R$,
\begin{equation*}
\begin{aligned}
    (u-u^{-1})y_7&=u+i\,y_2, & (u-u^{-1})y_8&=-u^{-1}-i\,y_3, & (u-u^{-1})y_9&=u^{-1}+i\,y_2,\\
     (u-u^{-1})y_{10}&=-i+u^{-1}y_1, & (u-u^{-1})y_{11}&=-u-i\,y_3, & (u-u^{-1})y_{12}&=-i+u\,y_1.
\end{aligned}
\end{equation*}
These allow us to eliminate $y_k$, $7\leq k\leq 12$. Introducing 
\begin{equation*}
    \begin{aligned}
    x_1&=-i\,y_1, &\quad 
    x_2&=-i\,y_2, &\quad 
    x_3&=-i\,y_3, \\
    x_4&=1+i(u-u^{-1})y_4, &\quad 
    x_5&=1-i(u-u^{-1})y_5, &\quad 
    x_6&=1-i(u-u^{-1})y_6,
    \end{aligned}
\end{equation*}
the ring of invariants $R^{\C^*}$ is again realised by \eqref{eq:ringofinvariants}. In particular, the corresponding geometric invariant theory quotient leads to the same variety $\mathcal{M}_w$ as in \Cref{def:monodromysurfacePII}, again with $w=u+\tfrac{1}{u}$.


\bibliographystyle{amsalpha}

\bibliography{references}

\end{document}